\documentclass{coulonpaper}

\usepackage{tikz}

\usetikzlibrary{matrix,arrows,calc,backgrounds,decorations.markings,decorations.text, decorations.pathmorphing}

\title{Infinite periodic groups of even exponents}
\author{Rémi Coulon}
\begin{document}

\selectlanguage{english}

\distabstrue
\setmidfalse
\intvaldfrenchtrue
\setbftrue

\maketitle

\begin{abstract}
	We give a new proof that free Burnside groups of sufficiently large even exponents are infinite.
	The method is very flexible and can also be used to study (partially) periodic quotients of any group which admits an action on a hyperbolic space satisfying a weak form of acylindricity.
\end{abstract}

%
%
%

\medskip
{

}

\tableofcontents

%
\section{Introduction}
%

%
\subsection*{Historical background}
%

A group $G$ \emph{has exponent $n \in \N$}, if $g^n = 1$ for every $g \in G$.
In 1902, Burnside asked whether every finitely generated group of finite exponent was necessarily finite \cite{Burnside:1902vi}.
Despite its simplicity, this question remained open for a long time and motivated many developments in group theory.
The class of groups of exponent $n$ forms a group variety whose free elements are the \emph{free Burnside groups of exponent $n$}.
More concretely, the free Burnside group of rank $r$ and exponent $n$, that we denote by $\burn rn$, admits the following presentation
\begin{equation*}
	\burn rn = \left< a_1, \dots, a_r \mid x^n, \ \forall x\right>.
\end{equation*}
A major breakthrough in the subject was achieved by Novikov and Adian in 1968 \cite{Novikov:1968tp}.
They proved that $\burn rn$ is infinite provided $r\geq 2$ and $n$ is a sufficiently large \emph{odd} exponent.
Later Ol'shansk\u\i\ provided an alternative proof of the same result \cite{Olshanskii:1982ts}.
Since the first version of this article, an other strategy for odd exponents was also given by Atkarskaya, Rips, and Tent \cite{Atkarskaya:2024bu}.
Despite these progresses the case of \emph{even} exponents held up longer.
It was only in the early 90's that Ivanov \cite{Ivanov:1994kh} and Lysenok \cite{Lysenok:1996kw} independently proved that free Burnside groups of sufficiently large even exponents are also infinite.

The aforementioned results rely (more or less explicitly) on an iterated version of small cancellation theory (using combinatorics on words and/or reasoning on van Kampen diagrams).
Historically, the respective works of Novikov-Adian and Ol'shanski\u\i\ appeared before the theory of hyperbolic spaces/groups was formalized and developed by Gromov in \cite{Gromov:1987tk}.
However working with groups acting on hyperbolic spaces provides a perfect framework offering new insights into small cancellation theory.
See for instance \cite{Guirardel:2014aa, Coulon:2016uz} for a survey on the topic.
Using this geometric point of view, Delzant and Gromov revisited the Burnside problem making an \emph{explicit} use of hyperbolic geometry \cite{Delzant:2008tu}.
Nevertheless their work only applies to odd exponents.

\paragraph{Main results.}
In this article we provide a new approach to the free Burnside groups of \emph{even} exponents based on the geometrical ideas of Delzant and Gromov \cite{Delzant:2008tu}.
More precisely we provide an alternative proof for the statement below originally due to Ivanov~\cite{Ivanov:1994kh} and independently Lysenok~\cite{Lysenok:1996kw}.

\begin{theo}
\label{mth: free burnside}
	Let $r \geq 2$.
	There exists a critical exponent $n_0 \in \N$ such that for every integer $n \geq n_0$, the free Burnside group $\burn rn$ is infinite.
\end{theo}

Not only is our approach substantially shorter than the one of Lysenok and Ivanov (200 and 300 pages respectively) it also gives a way to produce (partially) periodic quotients of many groups as soon as they carry a certain form of negative curvature, far beyond the single instance of free groups.
Let us mention a few examples.
The next theorem is originally due to Ol'shanski\u\i\ and Ivanov \cite{Olshanskii:1991tt,Ivanov:1996va} answering a question of Gromov \cite{Gromov:1987tk}.
Given an arbitrary group $G$, we write $G^n$ for the (normal) subgroup of $G$ generated by the $n$-th power of all its elements and call $G/G^n$ the \emph{$n$-periodic quotient of $G$}.

\begin{theo}[Ol'shanski\u\i-Ivanov \cite{Ivanov:1996va}]
\label{res: periodic quotient hyperbolic group intro}
	Let $G$ be a non-elementary hyperbolic group.
	There exist $p,n_0 \in \N$, such that for every integer $n \geq n_0$ that is a multiple of $p$, the quotient $G / G^n$ is infinite.
	Moreover
	\begin{equation*}
		\bigcap_{n \geq 1} G^n = \{1\}.
	\end{equation*}
\end{theo}

As we explain in \autoref{res: exponent dependency hyperbolic group}, the values of $p$ and $n_0$ both depend on $G$ (or at least on ``certain parameters'' of the action on its Cayley graph).
\autoref{res: periodic quotient hyperbolic group intro} extends to the more general setting where $G$ is a group acting acylindrically on a Gromov hyperbolic space $X$ (see \autoref{sec: construction of periodic groups} for a precise definition).
More precisely, for arbitrarily large exponents $n \in \N$, we are able to produce a ``partially $n$-periodic'' quotient of $G$ (\autoref{res: partial periodic quotient - acyl}), i.e.\ a quotient $Q$ of $G$ such that 
\begin{enumerate}
	\item every elliptic subgroup of $G$ (for its action on $X$) embeds in $Q$,
	\item for every $q \in Q$, either $q$ is the image of an elliptic element of $G$ or $q^n = 1$.
\end{enumerate}

\begin{rema*}
	We did not formally define the notion of \emph{partially periodic group} here.
	Indeed, in this article, we use this expression quite loosely to refer to a group in which we tried to adjoint as many possible torsion relations as possible.
\end{rema*}

Applied to the mapping class group of a surface acting on the curve graph the above result yields the following statement, a weaker version of which was obtained in \cite[Theorem~1.2]{Coulon:2016if}, see \autoref{rem: comparison mcg result}.

\begin{theo}
\label{res: partial periodic quotient mcg}	
	Let $\Sigma$ be a compact surface of genus $g$ with $k$ boundary components such that $3g +k - 3 >1$.
	There exist $p,n_0 \in \N$ such that for every integer $n \geq n_0$ which is a multiple of $p$, there exists a quotient $Q$ of the mapping class group $\mcg \Sigma$ with the following properties.
	\begin{enumerate}
		\item If $E$ is a subgroup of $\mcg \Sigma$ that does not contain a pseudo-Anosov element, then the projection $\mcg \Sigma \onto Q$ induces an isomorphism from $E$ onto its image.
		\item Let $f$ be a pseudo-Anosov element of $\mcg \Sigma$.
		Either $f^n = 1$ in $Q$ or $f$ coincides in $Q$ with a periodic or a reducible element.
		\item There are infinitely many conjugacy classes  in $Q$ whose elements are not the image of a periodic or reducible element of $\mcg \Sigma$.
		Any non-trivial element in the kernel of $\mcg \Sigma \onto Q$ is pseudo-Anosov.
	\end{enumerate}
\end{theo}

Bass-Serre theory also provides examples of groups acylindrically on a tree for which our approach works (\autoref{res: quotient amalgamated products}).
For instance if $G = A\ast B$ is a free product, the corresponding quotient $Q$ corresponds to the \emph{$n$-periodic product} of $A$ and $B$, see for instance \cite{Adian:1976vg}.
Note that the same strategy could also be used to study the outer automorphism group of free Burnside groups of even exponents, extending  some other work of the author and his collaborators \cite{Coulon:2013up, Coulon:2018ac, Coulon:2017cs}.
Nevertheless to limit the length of the article we decided not to detail that part.

%
\subsection*{A geometrical approach}
%

Let us highlight a few important ideas involved in the proofs.
For simplicity we restrict our attention to free Burnside groups as this case already covers all the difficulties.
As shown by Ivanov and Lysenok, free Burnside groups of (sufficiently large) odd or even exponent have a considerably different algebraic structure.
For instance if $n$ is odd, every finite subgroup of $\burn rn$ is cyclic.
By contrast if $n$ is even, $\burn rn$ contains arbitrarily long direct products of dihedral groups.
Nevertheless the global strategy to study those groups remains the same.

\paragraph{A sequence of approximation groups.}
Let $n \in \N$ be a large exponent.
All known strategies for studying Burnside groups start in the same way: one produces by induction an \emph{approximation sequence} of groups
\begin{equation}
\label{eqn: intro - sequence}
	\free r = G_0 \onto G_1 \onto G_2 \dots \onto G_k \onto G_{k+1} \onto \dots
\end{equation}
whose direct limit is exactly $\burn rn$.
At each step $G_{k+1}$ is obtained from $G_k$ by adding new relations of the form $h^n = 1$, where $h$ runs over the set of all ``small'' infinite order elements of $G_k$.
The goal is to prevent this sequence to collapse to a finite group.
This is achieved by small cancellation arguments.
The novelty of our method is to use, in the context of torsion groups of even exponent, a geometric point of view on small cancellation \emph{à la} Delzant-Gromov.
To that end, we use several algebraic/geometric invariants which contain all the necessary informations we need to control about the groups $G_k$.

\paragraph{Geometric small cancellation.}
Let $S$ be a finite set and $R$ a collection of cyclically reduced words of the free group $\F(S)$.
Assume for simplicity that $R$ is invariant under taking cyclic permutations and inverses and write $\ell$ for the length of its shortest element.
Given $\lambda \in (0, 1)$, the group presentation $\left<S | R\right>$ satisfies the classical $C''(\lambda)$ small cancellation condition\footnote{The more classical small cancellation condition $C'(\lambda)$ requires that the length of any piece $u$ contained in a relation $r$ satisfies $|u| < \lambda |r|$. We use here the stronger, uniform condition $C''(\lambda)$ where the length of pieces is small compared to \emph{any} relation in $R$.} if the length $\abs u$ of every prefix $u$ of two distinct relations in $R$ is less than $\lambda \ell$.
For small values of $\lambda$, one understands precisely the properties of the corresponding group $\bar G = \F(S)/ \normal R$.
For instance if $\lambda \leq 1/6$, then $\bar G$ is hyperbolic.
The $C''(\lambda)$ condition can advantageously be reformulated as follows.
Let $X$ be the Cayley graph of $\F(S)$ with respect to $S$.
The presentation $\left<S | R\right>$ satisfies the $C''(\lambda)$ condition if for every distinct $r_1, r_2 \in R$, the overlap between the respective axes of $r_1$ and $r_2$ has length less then $\lambda \ell$, where $\ell$ is also the smallest translation length of an element in $R$.

With this idea in mind, one can extend the small cancellation theory to the context of hyperbolic groups \cite{Gromov:1987tk,Olshanskii:1993dr,Delzant:1996it} (or more generally of groups acting on a hyperbolic space).
Let $G$ be a non-elementary group acting properly co-compactly by isometries on a hyperbolic space $X$ and $R$ a subset of $G$, which is invariant under conjugation.
Roughly speaking we will say that $R$ satisfies a small cancellation condition if given any two distinct $r_1, r_2 \in R$, the length $\Delta(r_1,r_2)$ on which the respective axes of $r_1$ and $r_2$ fellow travel is very small compare to the translation lengths $\norm{r_1}$ and $\norm{r_2}$ (see \autoref{fig: overlap relations}).
\begin{figure}[htbp]
	\centering
	\includegraphics[page=1, width=0.8\textwidth]{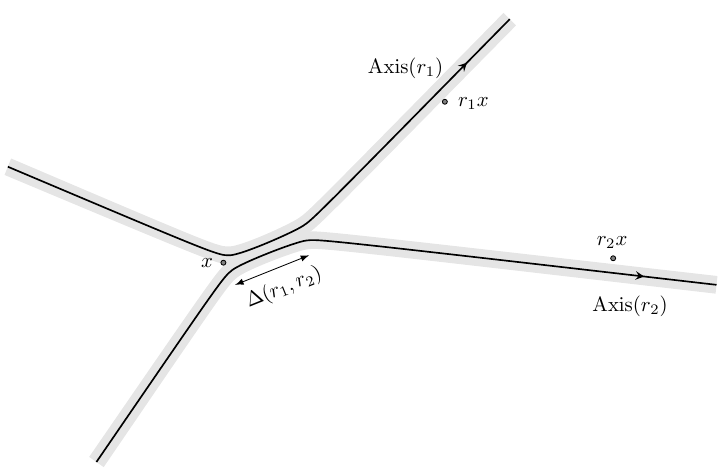}
	\caption{
		Overlap between two relations seen in the hyperbolic space $X$.
		The translation length $\norm {r_i}$ of $r_i$ is roughly the distance between $x$ and $r_ix$.
	}
	\label{fig: overlap relations}
\end{figure}
In this situation, the quotient $\bar G = G / \normal R$ is still a non-elementary hyperbolic group \cite{Gromov:1987tk,Delzant:1996it,Olshanskii:1993dr}.
Under this hypothesis, Gromov explains in \cite{Gromov:2001wi} how to let this group act on a hyperbolic space $\bar X$ whose geometry is finer than the one of the Cayley graph of $\bar G$, see also \cite{Delzant:2008tu,Coulon:2011il,Coulon:2014fr}.
Assume for instance that $\bar G$ is a quotient of the form $\bar G = G/ \normal{h^n}$ where $h$ is a loxodromic element and $n$ a (large) exponent.
Then $\bar M = \bar X / \bar G$ can be seen as an orbifold (whose fundamental group is $\bar G$ and universal cover is $\bar X$) and comes with an analog of Margulis' thin/thick decomposition for hyperbolic manifolds.
The thin part corresponds to the neighborhood of a single singular point whose isotropy group is exactly the maximal finite subgroup $\bar F \subset \bar G$ containing the image of $h$.
The pre-image in $\bar X$ of the thin part is roughly speaking the collection of all $\bar G$-translates of an $\bar F$-invariant hyperbolic disc $\mathcal D \subset \bar X$.
Moreover there exists a natural map $q \colon \bar F \to \dihedral[n]$, where $\dihedral[n]$, is the dihedral group of order $2n$, such that the action of $\bar F$ on $\mathcal D$ is identified via $q$ to the natural action of $\dihedral[n]$ on the disc.

Adopting this point of view, we associate to each approximation group $G_k$ in (\ref{eqn: intro - sequence}) a hyperbolic space $X_k$ on which $G_k$ acts properly co-compactly.
The goal will be to prove that, at each step, the new relations defining $G_{k+1}$ will satisfy a small cancellation condition in the above sense (i.e.\ relative to the action of $G_k$ on the space $X_k$).
It has the following main advantage: almost every needed property of the relations defining $G_k$ is captured by the hyperbolicity of $X_k$.
Consequently when studying the quotient map $G_k \onto G_{k+1}$ one can completely forget the relations defining $G_k$ and rely only on the geometry of $X_k$.
Following Delzant-Gromov \cite{Delzant:2008tu}, this allows us to formulate -- unlike in \cite{Novikov:1968tp,Olshanskii:1982ts,Lysenok:1996kw,Ivanov:1994kh} -- the induction hypothesis used to build the approximation sequence (\ref{eqn: intro - sequence}) in a rather compact form (see \autoref{res: induction step}).

\paragraph{A Margulis' lemma.}
As mentioned above, the main challenge when building the approximation sequence (\ref{eqn: intro - sequence}) is to make sure that $G_k$ is not eventually finite.
This will not happen if, at each step, the relations (of the form $h^n = 1$) used to define $G_{k+1}$ from $G_k$ satisfy a small cancellation condition.
Therefore, given any two loxodromic elements $g_1, g_2 \in G_k$ which do not generate an elementary subgroup, one needs to control \emph{uniformly, independently of $k$}, the ratio
\begin{equation}
\label{eqn: intro - margulis ratio}
	\frac{\Delta(g_1,g_2)}{\max\left\{ \norm {g_1}, \norm {g_2} \right\} }
\end{equation}
where $\Delta(g_1,g_2)$ measures the ``overlap'' between the respective axes of $g_1$ and $g_2$ in $X_k$ (see \autoref{fig: overlap relations}).
If $X_k$ was a simply connected manifold with \emph{pinched} negative sectional curvature, such estimate would follow from Margulis' Lemma.
However hyperbolicity only provides an upper bound for the curvature of the space.
To bypass this difficulty, one usually uses assumptions on the action of the group (e.g. the fact that the action is proper co-compact or acylindrical).
For instance, a first (naive) attempt to bound the ratio (\ref{eqn: intro - margulis ratio}) could work as follows.
Suppose that $g_1,g_2 \in G_k$ are two loxodromic elements such that $\Delta(g_1, g_2) > N \max \{\norm{g_1}, \norm{g_2}\}$.
It is a standard exercise of hyperbolic geometry to see that there exists a point $x \in X_k$ such that for every $i \in \intvald 0{N-1}$, the commutator $[g_1, g_2^i]$ moves $x$ by at most $100 \delta_k$ (where $\delta_k$ is the hyperbolicity constant of $X_k$), see \autoref{fig: naive attempt}.
\begin{figure}[htbp]
	\centering
	\includegraphics[page=2, width=0.8\textwidth]{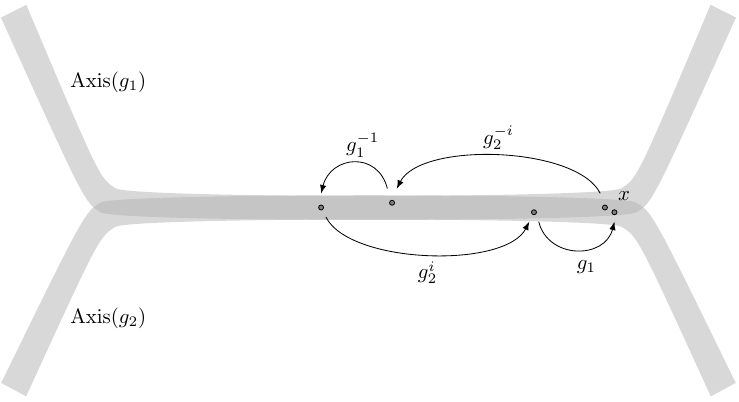}
	\caption{
		The point $x$ is hardly moved by the commutator $[g_1, g_2^i]$.
	}
	\label{fig: naive attempt}
\end{figure}
In particular, if $N$ exceeds the number of elements in the set $\set{u \in G_k}{\dist {ux}x \leq 100\delta_k}$, then $g_1$ commutes with a power of $g_2$, thus $\group{g_1,g_2}$ is elementary.
So, roughly speaking, the ratio (\ref{eqn: intro - margulis ratio}) is bounded above by the cardinality of the (almost) stabilizers of points in $X_k$.
This strategy has a major weakness though: if $n$ is an even exponent, the cardinality of finite subgroups is not uniformly bounded along the sequence $(G_k)$.
During the process we will indeed encounter points in $X_k$ with arbitrarily large stabilizers.
Hence this method cannot be used to keep the ratio (\ref{eqn: intro - margulis ratio}) uniformly bounded.
Any refinement of the above argument using acylindrical actions of $G_k$ on $X_k$ -- see for instance \cite{Hull:2016aa,Dahmani:2017ef} -- shall fail in the same way.

To bypass this difficulty one associates to the action of $G_k$ on $X_k$ several numerical invariants:
\begin{enumerate}
	\item $A(G_k,X_k,d)$ is characterized as follows: if $S$ is any finite subset of $G_k$ generating a non-elementary subgroup, then the set of points in $X_k$ which are moved by a distance at most $d$ by every element of $S$ has diameter at most $A(G_k, X_k,d)$ (see \autoref{def: acyl inv}).
	\item $\nu(G_k,X_k)$ is the smallest integer $m$ with the following property: let $g,h \in G_k$ with $h$ loxodromic . 
	If $g$, $hgh^{-1}$, $h^2gh^{-2}$, \dots, $h^mgh^{-m}$ generate an elementary subgroup, then so do $g$ and $h$ (see \autoref{def: nu inv}).
\end{enumerate}
The quantity $A(G_k, X_k,d)$ can be thought of as a local version of the ratio (\ref{eqn: intro - margulis ratio}).
Indeed, its definition only involves ``small'' elements.
Combined with the $\nu$-invariant one recovers the following global analogue of Margulis' Lemma.
See for instance \cite[Proposition~3.34]{Coulon:2016if} or \autoref{res: margulis lemma}.

\begin{prop}
	If $g_1,g_2 \in G_k$ do not generate an elementary subgroup, then
	\begin{equation*}
		\Delta(g_1,g_2) \leq \left[\nu(G_k, X_k) + 2\right] \max\left\{ \norm {g_1}, \norm {g_2} \right\}+ A(G_k,X_k,420\delta_k)+ 1000\delta_k,
	\end{equation*}
\end{prop}

Consequently, in order to make sure that for every $k \in \N$, the relations defining $G_{k+1}$ from $G_k$ satisfy a suitable small cancellation condition, it suffices to control (among others) the values of $A(G_k,X_k, 400\delta_k)$ and $\nu(G_k,X_k)$ all along the approximation sequence (\ref{eqn: intro - sequence}).
This was done in \cite{Coulon:2016if} in the absence of even torsion.

As soon as even torsion is involved, the situation becomes much more delicate.
In particular, the $\nu$-invariant does not behave very well when passing to a quotient (see for instance the discussion and examples at the beginning of \autoref{sec: invariants quotient - mixed}).
It results from the fact that the algebraic structure of finite subgroups of $\burn rn$ is rather intricate, see for instance Lysenok \cite{Lysenok:2007ks}.

\begin{rema}
	Note that Hull also developed a small cancellation theory in the context of acylindrically hyperbolic groups \cite{Hull:2016aa}. 
	However this work does not provide the necessary tools to control the small cancellation parameters along the sequence (\ref{eqn: intro - sequence}) and thus to build infinite torsion groups with \emph{bounded} exponents.
	As we explained above, it is not possible to control the acylindricity parameters for the action of $G_k$ on $X_k$.
	The purpose of this article is precisely to develop the implements needed to bypass this difficulty.	
\end{rema}

\paragraph{Structure of elementary subgroups.}
As we mentioned earlier, if $n$ is odd, then every maximal finite subgroup of $\burn rn$ is isomorphic to the cyclic group $\cyclic[n]$.
Indeed, as no even torsion is added along the sequence $(G_k)$ one never creates infinite dihedral subgroups. 
More precisely, one can prove by induction on $k$ that every infinite elementary subgroup of $G_k$ is actually cyclic.
Therefore, all the new finite subgroups we create along the way can only be cyclic as well.
Moreover these finite subgroups ``stabilize'' along the approximation sequence (\ref{eqn: intro - sequence}).
More precisely we have the following property: if $F$ is a finite subgroup of some $G_k$ that is maximal for the inclusion, then for every $k' \geq k$, every finite subgroup of $G_{k'}$ containing the image of $F$ actually coincides with the image of $F$.
By contrast, if $n$ is even, finite subgroups may ``grow'' when taking successive quotients.
Let us illustrate this fact with the following toy example.

\begin{exam}
\label{exa: intro growing elliptic}
	Assume that $n$ is a (large) exponent divisible by $4$.
	Start with the free group $G_0 = \free 2$ generated by $a$ and $b$ and set 
	\begin{equation*}
		G_1 = G_0 / \normal{a^n,b^n, (ba)^n}.
	\end{equation*}
	In $G_1$ the elements $s_0 = a^{n/2}$, $s_1= b^{n/2}$ and $s_2 = (ba)^{n/2}$ all generate a subgroup isomorphic to $\cyclic[2]$.
	Consequently in the quotient 
	\begin{equation*}
		G_2 = G_1/\normal{(s_0s_1)^n, (s_0s_2)^n}
	\end{equation*}
	$s_0$ and $s_1$ generate a subgroup isomorphic to the dihedral group $\dihedral[n]$ of order $2n$.
	In particular, $s_0$ commutes with the involution $u_1 = s_0 (s_0s_1)^{n/2}$.
	Similarly $s_0$ commutes with the involution $u_2 = s_0(s_0s_2)^{n/2}$.
	Form now the quotient 
	\begin{equation*}
		G_3 = G_2/\normal{(u_1u_2)^n, [a^{n/4}, u_1]^n}.
	\end{equation*}
	Note that $s_0$, $u_1$ and $u_2$ generate a subgroup isomorphic to $\cyclic[2] \times \dihedral[n]$.
	In this example, $F_1 = \group{s_0} = \cyclic[2]$ is a finite subgroup of $G_1$.
	Its image in $G_2$ (\resp $G_3$) embeds in $F_2 = \group{s_0,s_1} = \dihedral[n]$ (\resp $F_3 = \group{s_0, u_1,u_2} = \cyclic[2] \times \dihedral[n]$).
	However $F_2$ (\resp $F_3$) is not the image of a finite subgroup of $G_1$ (\resp $G_2$).
	We stopped our example after three steps. 
	However one can proceed further embeds $F_0$ in an arbitrarily large product of the form $\cyclic[2] \times \dots \times \cyclic[2] \times \dihedral[n]$.
	
	So far we have not used the relation $ [a^{n/4}, u_1]^n = 1$.
	We added it to illustrate how more complicated elliptic subgroups can arise.
	Indeed, using again the fact that $u_1$ commutes with $s_0 = a^{n/2}$, we observe that $a^{n/4}$ and $u_1 a^{n/4}u_1^{-1}$ generate a subgroup that is isomorphic to the non-trivial central extension of $\dihedral[n]$ by $\cyclic[2]$ (i.e.\ the quotient of $\cyclic[4] \ast_{\cyclic[2]}\cyclic[4]$ by $(\mathbf x\mathbf x')^n$, where $\mathbf x$ and $\mathbf x'$ are generators of the two cyclic factors $\cyclic[4]$).
	Such subgroup still contains a copy of $F_1$ but cannot be lifted in $G_2$ either.
\end{exam}

The previous example suggests that $G_k$ -- and thus $\burn rn$ -- contains nested copies of dihedral groups.
As mentioned above, one of the advantages of the space $X_k$ (compare to the Cayley graph of $G_k$) is that one ``sees'' some of these dihedral groups acting as the isometry group of a disc.
Unfortunately this is not sufficient to capture all the properties of finite subgroups of $G_k$.
To sort a bit this nested structure, we introduce the notion of \emph{dihedral germ}.
A dihedral germ of $G_k$ is an elliptic subgroup $C$ (for its action on $X_k$) containing a subgroup $C_0$ which is normalized by a loxodromic element and such that $[C: C_0]$ is a power of $2$.
As suggested by the terminology, the dihedral germs are exactly the finite subgroups of $G_k$ that may eventually grow in $G_{k+1}$, i.e.\ be embedded in a larger finite subgroup of $G_{k+1}$ that does not come from $G_k$.
In \autoref{exa: intro growing elliptic - continued} we carry a detailed analysis of the structure of the elementary subgroups appearing in \autoref{exa: intro growing elliptic}.
For the moment, let us show with an easier example how a dihedral germ can grow when adding even torsion.

\begin{exam}
	Assume that $A$ is a finite subgroup of $G_k$ and $C$ a subgroup of $A$ of index $2$ which is normalized by a loxodromic element $h \in G_k$.
	In particular, $A$ is a dihedral germ.
	Suppose for simplicity that $h^n$ is trivial in $G_{k+1}$.
	Let $a \in A \setminus C$.
	Seen in $G_{k+1}$, the group $C$ has index $2$ in both $A = \group{C, a}$ and $B = \group{C, h^{n/2}}$.
	In particular, $A$ and $B$ generate an elementary subgroup $E$ of $G_{k+1}$ which is (most of the time) isomorphic to $E = A \ast_C B$ and such that $E/C = \dihedral$.
	As an element of $G_{k+1}$ the product $t = a h^{n/2}$ has infinite order.
	However, since $\burn rn$ is the direct limit of $(G_k)$, we have $t^n = 1$ in $G_\ell$ for some $\ell > k+1$.
	The image in $G_\ell$ of $E$ is actually isomorphic to $E/\group{t^n}$ which is a finite group that strictly contains the dihedral germ $A$.
	Nevertheless there is no finite subgroup $F$ of $G_k$ containing $A$ such that the canonical quotient map $G_k \onto G_\ell$ induces an isomorphism from $F$ onto $E/\group{t^n}$.
\end{exam}

It turns out that the dihedral germs of $G_k$ are exactly its $2$-subgroups (when studying general periodic quotients different from free Burnside groups, those dihedral germs are slightly more complicated).
A careful analysis of dihedral germs allows us to prove that every finite subgroup of $G_k$ embeds in a direct product of the form $\dihedral[n] \times \dihedral[n_2] \times \dots \times \dihedral[n_2]$ where $n_2$ is the largest power of $2$ dividing $n$, see also  \cite{Ivanov:1994kh,Lysenok:1996kw}.
In particular, finite subgroups of $G_k$ share many identities with finite dihedral groups.
Those identities can be used to control a variation on the $\nu$-invariant which captures both geometric and algebraic features of the groups $G_k$ and behaves better when taking quotients (see \autoref{def: nu inv stg}).
Once we control the $\nu$-invariant along the sequence (\ref{eqn: intro - sequence}), a uniform estimate of the quantity $A(G_k, X_k,d)$ follows rather easily (\autoref{res: local-to-global acyl}).
Those two invariants (together with the injectivity radius of $G_k$ acting on $X_k$) provide a sufficient control to show that each group $G_{k+1}$ is actually a small cancellation quotient of $G_k$.
Therefore it is hyperbolic and non-elementary, which ensures that at the limit $\burn rn$ is infinite.

\paragraph{Critical exponent.}
All these arguments actually only work provided $n$ is divisible by a large power of $2$, namely $128$ for free Burnside groups.
Nevertheless, given any two integers $p,n \in \N$, the group $\burn r{pn}$ naturally maps onto $\burn rn$.
Since Burnside groups of large odd exponents are known to be infinite, we can conclude that free Burnside groups of sufficiently large exponents are infinite.
The works of Ivanov \cite{Ivanov:1994kh} and Lysenok \cite{Lysenok:1996kw} have a similar restriction. 
They require $n$ to be divisible by $2^9$ and $16$ respectively.
Using \cite{Lysenok:2007ks} as a ``black box'' our proof can be adapted for large exponents $n$ which are only divisible by $16$. 
However for completeness and simplicity we preferred to detail our own understanding of finite subgroups of $\burn rn$.

According to \autoref{mth: free burnside}, there exists a critical exponent $n_0$ such that for every integer $n \geq n_0$, the group $\burn rn$ is infinite.
In our method, $n_0$ is directly related to the parameters of the Small cancellation Theorem, see Equations (\ref{eqn: recale - hyp})--(\ref{eqn: recale - inj}).
Since our approach to small cancellation is \emph{qualitative} we do not provide an explicit value of $n_0$.
Nevertheless, for a general group $G$ we stress how this critical exponent depends on the action of $G$ on a hyperbolic space $X$ (see \autoref{res: partial periodic quotient - gal}).
An interested reader could go through all the arguments with a \emph{quantitative} point of view to get an estimate of $n_0$.
However the resulting $n_0$ would most likely be very large.

%
\subsection*{Outline of the paper}
%
The proof that we present here is essentially self-contained.
Besides hyperbolic geometry, the arguments only rely on geometrical small cancellation theory which is now well understood.
See for instance \cite{Coulon:2016uz,Guirardel:2014aa} for a survey on the topic.
For the benefit of the reader we did not attempt to write the shortest possible proof.
In particular, we added in the course of the article numerous discussions, examples and figures to highlight the main difficulties and illustrate the important results.

In \autoref{sec: hyp} we make a short review of hyperbolic geometry.
We define in \autoref{sec: invariants} all the geometric and algebraic invariants needed to control the small cancellation parameters when building the approximation sequence (\ref{eqn: intro - sequence}).
In \autoref{sec: sc} we first review the main properties of small cancellation theory.
Given a group $G$ acting on a hyperbolic space, the goal is to understand the properties of the quotient $\bar G$ obtained from $G$ by adjoining relations of the form $h^n = 1$, where $n$ is a large even integer.
In particular, we study the elementary subgroups of $\bar G$ (\autoref{sec: sc - elem subgroups}) as well as the geometric/algebraic invariants of $\bar G$ (\autoref{sec: invariants quotient}).
This section follows closely the work of Delzant-Gromov \cite{Delzant:2008tu}.
\autoref{sec: periodic} collects all the previous work. 
We first state and prove the induction hypothesis used to produce the approximation sequence (\ref{eqn: intro - sequence}), see \autoref{res: induction step}.
Then we apply our main result (\autoref{res: partial periodic quotient - gal}) to various examples (\autoref{sec: examples}) such as free Burnside groups, periodic quotients of hyperbolic groups, etc.

\paragraph{Acknowledgment.}
The author is grateful to the \emph{Centre Henri Lebesgue} ANR-11-LABX-0020-01 for creating an attractive mathematical environment.
He acknowledges support from the Agence Nationale de la Recherche under the grants \emph{Dagger} ANR-16-CE40-0006-01 and \emph{GoFR} ANR-22-CE40-0004 as well as the Région Bourgogne-Franche-Comté under the grant \emph{ANER 2024 GGD}.
His current institute receives support from the EIPHI Graduate School (contract ANR-17-EURE-0002).

The author is greatly indebted to Thomas Delzant for his valuable comments and suggestions that helped improving the exposition of the article.
He also warmly thanks the anonymous referees whose careful reading and insightful remarks clarified many parts or this text.

%
\section{Hyperbolic geometry}
%
\label{sec: hyp}

We recall here a few basic facts about hyperbolic spaces in the sense of Gromov \cite{Gromov:1987tk}.
A reader familiar with the subject can directly jump to \autoref{sec: invariants} where we define the important invariants associated to the action of a group on a hyperbolic space.
We included precise references for the \emph{quantitative} results.
Some of the them only provides a proof in the context of geodesic metric spaces. 
However, by relaxing if necessary the constants, which we do here, the same arguments work in the more general setting of length spaces.
For the rest, we refer the reader to Gromov's original article \cite{Gromov:1987tk} or the numerous literature on the subject, e.g. \cite{Coornaert:1990tj,Ghys:1990ki,Bowditch:1991wl,Bridson:1999ky}.

\begin{nota*}
	In this article, $\N$ stands for the set of all non-positive integers (including zero).
	In addition if $p,q \in \Z$, with $p \leq q$, then $\intvald pq$ stands for the discrete interval $\{p, p+1, p+2, \dots,q\}$.
\end{nota*}

%
\subsection{General facts}
%
\label{sec: hyp - gal facts}

\paragraph{Four point inequality.}
Let $X$ be a metric length space.
In this article all the paths are rectifiable and parametrized by arc length.
Given two points $x,y \in X$, we write $\dist[X] xy$ or simply $\dist xy$ for the distance between them.
The Gromov product of three points $x,y,z \in X$ is defined as
\begin{equation*}
	\gro xyz = \frac 12 \left( \dist xz + \dist yz - \dist xy \right).
\end{equation*}
Let $\delta \in \R_+^*$.
We assume that $X$ is $\delta$-hyperbolic, i.e.\ for every $x,y,z,t \in X$ we have
\begin{equation}
\label{eqn : hyp four points - 1}
	\gro xyt \geq \min\left\{\gro xzt , \gro zyt \right\} - \delta.
\end{equation}
or equivalently
\begin{equation}
\label{eqn : hyp four points - 2}
	\dist xy + \dist zt \leq \max \left\{ \dist xz + \dist yt, \dist xt + \dist yz \right\} + 2\delta.
\end{equation}
In this context, the Gromov product has the following useful interpretation.
For every $x,y,z \in X$, the quantity $\gro yzx$ is roughly the distance between $x$ and any geodesic $\geo yz$ between $y$ and $z$.
More precisely, we have
\begin{equation*}
	\gro yzx  \leq d(x, \geo yz) \leq \gro yzx  + 4\delta,
\end{equation*}
see for instance \cite[Chapitre~3, Lemme~2.7]{Coornaert:1990tj}.


\paragraph{Boundary at infinity.}
Let $o$ be a base point in $X$.
A sequence $(y_n)$ of points of $X$ \emph{converges at infinity} if $\gro{y_n}{y_m}o$ diverges to infinity as $n$ and $m$ tend to infinity.
The set $\mathcal S$ of all such sequences is endowed with an equivalence relation defined as follows: two sequences $(y_n)$ and $(y'_n)$ are equivalent if
\begin{equation*}
	\lim_{n \to \infty} \gro{y_n}{y'_n}o = +\infty.
\end{equation*}
The \emph{boundary at infinity} of $X$, denoted by $\partial X$ is the quotient of $\mathcal S$ by this equivalence relation.
If the sequence $(y_n)$ is an element in the class of $\xi \in \partial X$, we say that $(y_n)$ \emph{converges} to $\xi$ and write
\begin{equation*}
	\lim_{n \to \infty} y_n = \xi.
\end{equation*}
The Gromov product of three points can be extended to the boundary:
given $x\in X$ and $y,z \in X \cup \partial X$, we define 
\begin{equation*}
	\gro yz x = \inf \set{\liminf_{n\rightarrow + \infty} \gro {y_n}{z_n}x}{\lim_{n \to \infty} y_n = y,\ \lim_{n \to \infty} z_n = z}
\end{equation*}
Let $x \in X$.
Let $(y_n)$ and $(z_n)$ be two sequences of points of $X$ respectively converging to $y$ and $z$ in $X \cup \partial X$.
It follows from (\ref{eqn : hyp four points - 1}) that
\begin{equation}
\label{eqn: estimate gromov product boundary}
	\gro yzx \leq \liminf_{n\rightarrow + \infty} \gro {y_n}{z_n}x \leq \limsup_{n\rightarrow + \infty} \gro {y_n}{z_n}x \leq \gro yzx + k\delta,
\end{equation}
where $k$ is the number of points in $\{y,z\}$ which belong to $\partial X$.
Moreover, for every $t \in X$, for every $x,y,z \in X \cup \partial X$, the four point inequality (\ref{eqn : hyp four points - 1}) leads to
\begin{equation}
\label{eqn: hyperbolicity condition with boundary}
	\gro xzt \geq \min\left\{\gro xyt, \gro yzt \right\} - \delta.
\end{equation}
The isometry group of $X$ naturally acts on $\partial X$ preserving Gromov's products.

\paragraph{Busemann cocycles.}
To every point $\xi \in \partial X$, we would like to associate a Busemann cocycle.
However the space $X$ is neither locally compact, nor geodesic. 
To that end, we proceeds as follows.
Given any point $\xi \in \partial X$ and a base point $o \in X$, we first define a map $b \colon X \to \R$ by
\begin{equation*}
	 b(x) = \gro o\xi x - \gro x\xi o
\end{equation*}
and then let 
\begin{equation*}
	\begin{array}{rccc}
		c \colon & X \times X & \to & \R \\
		&(x,y) & \to & b(x) - b(y).
	\end{array}
\end{equation*}
The map $c$ is obviously a cocycle, i.e.\ 
\begin{equation*}
	c(x_1,x_3) = c(x_1,x_2) + c(x_2,x_3),\quad \forall x_1,x_2,x_3 \in X.
\end{equation*}
We call $c$ a \emph{Busemann cocycle at $\xi$} (\emph{based at $o$}).
Note that $c$ \emph{does} depend on the base point $o$.
Nevertheless for every $x,y \in X$,
\begin{equation}
\label{eqn: effective computation busemann}
		\abs{c(x,y) - \left[\gro y\xi x - \gro x\xi y\right]} \leq 3\delta.
\end{equation}
In particular, any two cocycles at $\xi$ differ by at most $6\delta$.
Moreover $c$ is \emph{almost $1$-Lipschitz}, i.e.
\begin{equation}
\label{eqn: busemann cocycle lipschitz}
	\abs{c(x,y)}\leq \dist xy + 2\delta, \quad \forall x,y \in X.
\end{equation}

\begin{lemm}
\label{res: variation four point w/ cocycle}
	Let $\xi \in \partial X$ and $c$ a Busemann cocycle at $\xi$.
	For every $x,y,y'\in X$ we have
	\begin{equation*}
		\dist y{y'} \leq \abs{c(y,y')} + 2 \max \left\{ \gro x\xi y, \gro x\xi{y'}\right\} + 5\delta.
	\end{equation*}
\end{lemm}

\begin{proof}
	Let $(z_n)$ be a sequence of points of $X$ converging to $\xi$.
	It follows from the four point inequality that for every $n \in \N$,
	\begin{align*}
		\dist y{y'} 
		& \leq \abs{\dist y{z_n} - \dist{y'}{z_n}} + 2 \max \left\{ \gro x{z_n} y, \gro x{z_n}{y'}\right\} + 2\delta \\
		& \leq \abs{\gro{y'}{z_n}y - \gro y{z_n}{y'}} + 2 \max \left\{ \gro x{z_n} y, \gro x{z_n}{y'}\right\} + 2\delta
	\end{align*}
	see for instance \cite[Lemma~2.2(2)]{Coulon:2014fr}.
	The conclusion follows by taking the limit and applying (\ref{eqn: estimate gromov product boundary}).
\end{proof}

We denote by $\partial_hX$ the set of all Busemann cocycles obtained as above.
The isometry group of $X$ naturally acts on $\partial_hX$:
if $g$ is an isometry of $X$ and $c$ a Busemann cocycle at $\xi \in \partial X$, then the map $gc \colon X\times X \to \R$ defined by $(gc)(z,z') = c(g^{-1}z,g^{-1}z')$ is a Busemann cocycle at $g\xi$.

\paragraph{Quasi-geodesics.}
Let $\kappa \in \R_+^*$ and $\ell \in \R_+$.
A \emph{$(\kappa, \ell)$-quasi-isometric embedding} is a map $f \colon X_1 \to X_2$ between two metric spaces such that for every $x,x' \in X_1$,
\begin{equation*}
	\kappa^{-1}\dist x{x'} - \ell \leq \dist{f(x)}{f(x')} \leq \kappa \dist x{x'} + \ell.
\end{equation*}
A \emph{$(\kappa, \ell)$-quasi-geodesic} is a $(\kappa,\ell)$-quasi-isometric embedding $\gamma \colon I \to X$  from an interval $I$ of $\R$ into $X$.
Recall that all the paths we consider are rectifiable by arc length.
Hence, if $\gamma \colon I\to X$ is a $(\kappa, \ell)$-quasi-geodesic, we have the following more accurate inequalities:
\begin{equation*}
	\kappa^{-1} \dist st - \ell \leq \dist{\gamma(s)}{\gamma(t)} \leq \dist st, \quad \forall s,t \in I.
\end{equation*}
A path is an \emph{$L$-local $(\kappa, \ell)$-quasi-geodesic} if its restriction to any interval of length $L$ is a $(\kappa, \ell)$-quasi-geodesic.
If $\gamma \colon \R_+ \to X$ is $(\kappa, \ell)$-quasi-geodesic, then there exists a unique point $\xi \in \partial X$ such that for every sequence of real numbers $(t_n)$ diverging to infinity we have $\lim_{n \to \infty} \gamma(t_n) = \xi$.
We view $\xi$ as the endpoint at infinity of $\gamma$ and write $\xi  = \gamma(\infty)$.
In this article, we mostly work with local $(1, \ell)$-quasi-geodesics.
Therefore we use the version of the stability of local quasi-geodesics below.
We refer to \cite[Theorem~1.13]{Bridson:1999ky} for the proof.

\begin{prop}[Stability of quasi-geodesics]
\label{res: stability qg}
	Let $\ell, L \in \R_+$ and $\gamma \colon I \to X$ be an $L$-local $(1, \ell)$-quasi-geodesic.
	If $L > 4\ell + 8\delta$, then the following holds.
	\begin{enumerate}
		\item \label{enu:stability qg - nghbrhd}
		For every $s,t,t' \in I$, with $t \leq s \leq t'$, we have $\gro{\gamma(t)}{\gamma(t')}{\gamma(s)} \leq \ell/2 + 2\delta$.
		\item \label{enu:stability qg - qc}
		For every $x \in X$, for every $y,y' \in X$, lying on $\gamma$, we have $d(x,\gamma) \leq \gro y{y'}x + \ell/2 + 4\delta$.
		\item \label{enu:stability qg - global qg}
		The path $\gamma$ is a (global) $(\kappa, \ell)$-quasi-geodesic, where $\kappa = \frac L{L - 2(\ell + 2 \delta)}$.
	\end{enumerate}
	In particular, the Hausdorff distance between two $L$-local $(1, \ell)$-quasi-geodesics with the same endpoints (eventually in $\partial X$) is at most $\ell + 6\delta$.
\end{prop}

Although the space $X$ is not geodesic, its boundary satisfies a visibility property: for every $x \in X$, $\xi \in \partial X$, for every $\ell > 10\delta$ and $L \geq 0$, there exists an $L$-local $(1, \ell)$-quasi-geodesic $\gamma \colon \R_+ \to X$  (which is also a global quasi-geodesic) such that $\gamma(0) = x$ and $\gamma(\infty) = \xi$, compare with \cite[Lemma~2.9]{Coulon:2016if}.

\begin{lemm}
\label{res: computing cocycle from qg}
	Let $\ell \in \R_+$ and $L > 4 \ell + 8\delta$.
	Let $\gamma \colon \R_+ \to X$ be an $L$-local $(1, \ell)$-quasi-geodesic and $\xi = \gamma(\infty)$ its endpoint at infinity.
	Let $c$ be a Busemann cocycle at $\xi$.
	For every $s,t \in \R_+$ with $t \geq s$, we have 
	\begin{equation*}
		\abs{c(\gamma(s),\gamma(t)) - \dist{\gamma(s)}{\gamma(t)}} \leq \ell + 8\delta.
	\end{equation*}
\end{lemm}

\begin{proof}
	The lemma directly follows from (\ref{eqn: effective computation busemann}) and \autoref{res: stability qg}\ref{enu:stability qg - nghbrhd}.
\end{proof}

\paragraph{Quasi-convex subsets.}
Let $\alpha \in \R_+$.
A subset $Y$ of $X$ is \emph{$\alpha$-quasi-convex} if for every $x \in X$, for every $y,y' \in Y$, $d(x,Y) \leq \gro y{y'}x + \alpha$.
Assume now that $Y$ is connected by rectifiable paths.
The length metric on $Y$ induced by the restriction of $\distV[X]$ to $Y$ is denoted by $\distV[Y]$.
We say that $Y$ is \emph{strongly quasi-convex} if it is $2 \delta$-quasi-convex and for every $y,y' \in Y$,
\begin{equation}
\label{eqn: def strongly qc}
	\dist[X]y{y'} \leq \dist[Y]y{y'} \leq \dist[X]y{y'} + 8\delta.
\end{equation}
It follows from \autoref{res: stability qg}\ref{enu:stability qg - qc} that every $L$-local $(1, \ell)$-quasi-geodesic is $(\ell/2 + 4\delta)$-quasi-convex, provided $L > 4\ell + 8\delta$.
It turns out that a similar proof provides a slightly sharper result for (global) quasi-geodesics: any $(1, \ell)$-quasi-geodesic is $(\ell/2 + 2\delta)$-quasi-convex.
If $Y$ is an $\alpha$-quasi-convex subset of $X$, then for every $A \geq \alpha $, its $A$-neighborhood, that we denote by $Y^{+A}$, is $2\delta$-quasi-convex \cite[Proposition~2.13]{Coulon:2014fr}.
Similarly, for every $A > \alpha + 2\delta$, the \emph{open} $A$-neighborhood of $Y$, i.e.\ the set of points $x \in X$ such that $d(x,Y) < A$, is strongly quasi-convex \cite[Lemma~2.13]{Coulon:2016if}.

Let $x$ be a point of $X$.
A point $y \in Y$ is an \emph{$\eta$-projection} of $x$ on $Y$ is $\dist xy \leq d(x,Y) + \eta$.
A $0$-projection is simply called a \emph{projection}.

\begin{rema*}
	We adopt the convention that the diameter of the empty set is zero, whereas the distance from a point to the empty set is infinite.
\end{rema*}

\begin{lemm}[Compare with {\cite[Proposition~2.1]{Coornaert:1990tj}} or {\cite[Lemma~2.12]{Coulon:2014fr}}]
\label{res: proj qc}
	Let $\alpha \in \R_+$ and $Y$ an $\alpha$-quasi-convex subset of $X$.
	Let $x,x' \in X$.
	\begin{enumerate}
		\item If $p$ is an $\eta$-projection of $x$ on $Y$, then for every $y \in Y$, we have $\gro xyp\leq \alpha + \eta$.
		\item If $p$ and $p'$ are respectively $\eta$- and $\eta'$-projection of $x$ and $x'$ on $Y$ then 
		\begin{equation*}
			\dist p{p'} \leq \max \{ \dist x{x'} - \dist xp - \dist{x'}{p'} + 2 \epsilon, \epsilon \},
		\end{equation*}
		where $\epsilon = 2\alpha + \delta + \eta + \eta'$.
	\end{enumerate}
\end{lemm}

\begin{lemm}[{Compare with \cite[Lemme~2.2.2]{Delzant:2008tu} or \cite[Lemma~2.13]{Coulon:2014fr}}]
\label{res: intersection of thickened quasi-convex}
	Let $Y_1, \dots , Y_m$ be a collection of subsets of $X$ such that $Y_j$ is $\alpha_j$-quasi-convex, for every $j \in \intvald 1m$.
	For all $A \geq 0$, we have
	\begin{equation*}
		\diam \left( Y_1^{+A} \cap \dotsc \cap Y_m^{+A} \right) 
		\leq \diam \left( Y_1^{+\alpha_1+3\delta} \cap \dotsc \cap Y_m^{+\alpha_m+3\delta} \right) +2A + 4\delta.
	\end{equation*}
\end{lemm}

%
\subsection{Isometries}
%
\label{sec: hyp - isom}

An isometry $g$ of $X$ is either \emph{elliptic} (its orbits are bounded) \emph{parabolic} (its orbits admit exactly one accumulation point in $\partial X$) or \emph{loxodromic} (its orbits admit exactly two accumulation points in $\partial X$).
In order to measure the action of $g$ on $X$ we used the \emph{translation length} and the \emph{stable translation length} respectively defined by
\begin{equation*}	
	\norm[X] g = \inf_{x \in X}\dist {gx}x
	\quad \text{and} \quad
	\snorm[X] g = \lim_{n \to \infty} \frac 1n \dist{g^nx}x.
\end{equation*}
If there is no ambiguity, we will omit the space $X$ from the notations.
These lengths are related by
\begin{equation}
\label{eqn: regular vs stable length}
	\snorm g \leq \norm g \leq \snorm g + 8\delta,
\end{equation}
e.g.\ \cite[Chapitre~10, Proposition~6.4]{Coornaert:1990tj}.
In addition, $g$ is loxodromic if and only if $\snorm g > 0$.
In such a case, the accumulation points of $g$ in $\partial X$ are
\begin{equation*}
	g^- = \lim_{n \to \infty} g^{-n}x
	\quad \text{and} \quad
	g^+ = \lim_{n \to \infty} g^nx.
\end{equation*}
They are the only points of $X\cup\partial X$, fixed by $g$.

\begin{lemm}
\label{res: cyl in inv qc}
	Let $g$ be a loxodromic isometry of $X$.
	Let $\ell, L \in \R_+$, with $L > 4 \ell+ 8\delta$.
	Let $\gamma \colon \R \to X$ be a bi-infinite $L$-local $(1, \ell)$-quasi-geodesic between $g^-$ and $g^+$.
	Let $Y$ be a non-empty $\group g$-invariant $\alpha$-quasi-convex subset of $X$.
	Then $\gamma$ lies in the $(\alpha + \ell/2 + 4\delta)$-neighborhood of $Y$.
\end{lemm}

\begin{proof}
	Let $x$ be a point on $\gamma$.
	It follows from the stability of quasi-geodesics that $\gro{g^-}{g^+}x \leq \ell/2 +2\delta$ (\autoref{res: stability qg}).
	We fix a point $y\in Y$.
	Since $Y$ is $\alpha$-quasi-convex, we have $d(x,Y) \leq \gro{g^{-n}y}{g^ny}x + \alpha$, for every $n \in \N$.
	We pass to the limit as $n$ approaches infinity and use (\ref{eqn: estimate gromov product boundary}) to get 
	\begin{equation*}
		d(x,Y) \leq \gro{g^-}{g^+}x + \alpha + 2\delta \leq \alpha + \ell/2 + 4\delta. \qedhere
	\end{equation*}
\end{proof}

The next lemma is a weak variation on the quasi-convexity of the distance function in a $\delta$-hyperbolic space.
\begin{lemm}[See {\cite[Lemma~2.26]{Coulon:2014fr}}]
\label{res: quasi-convexity distance isometry}
	Let $x$, $x'$ and $y$ be three points of $X$.
	If $g$ is an isometry of $X$, then $\dist {gy}y \leq \max\left\{ \dist {gx}x, \dist {gx'}{x'} \right\} + 2 \gro x{x'}y  + 6 \delta$.
\end{lemm}

Let $S$ be a (possibly infinite) set of isometries of $X$.
Its \emph{energy} is defined by
\begin{equation*}
	\nrj S = \inf_{x \in X} \sup_{s \in S} \dist {sx}x.
\end{equation*}
Given a number $d \in \R_+$, we denote by $\fix{S,d}$ the set of points which are moved by $S$ by a distance less than $d$, i.e.\ 
\begin{equation}
\label{eqn: def mov}
	\fix{S,d} = \set{x \in X}{\forall g \in S,\ \dist{gx}x \leq d}.
\end{equation}
Allowing $S$ to be infinite will be useful later to deal with infinite elliptic subgroups.
If the set $S = \{g\}$ is reduced to a single isometry, then $\nrj S = \norm g$ and we simply write $\fix{g,d}$ for $\fix{S,d}$.

\begin{lemm}
\label{res: fix qc}
	Let $S$ be a set of isometries and $d >\max\{ \nrj S,5\delta\}$.
	Then $\fix{S, d}$ is non-empty and $8\delta$-quasi-convex.
	Moreover it satisfies the following properties.
	\begin{enumerate}
		\item \label{enu: fix qc - translation}
		For every $x \in X\setminus \fix{S,d}$, we have 
		\begin{equation*}
			\sup_{g \in S} \dist{gx}x \geq 2 d\!\left(x,\fix{S,d}\right) + d - 10\delta.
		\end{equation*}
		\item \label{enu: fix qc - neighborhood}
		Let $x \in X$ and $A \in \R_+$. 
		If $\dist{gx}x \leq d + 2A$ for every $g \in S$, then $x$ is $(A + 5\delta)$-close to  $\fix{S,d}$.
	\end{enumerate}
\end{lemm}

\begin{rema*}
	The \og converse\fg\ of \ref{enu: fix qc - neighborhood} is obvious.
	Indeed the $A$-neighborhood of $\fix{S,d}$ is contained in $\fix{S, d + 2A}$, for every $A \in \R_+$.
	Although the objects are defined in a slightly different way, the proof works verbatim as in \cite[Proposition~2.28]{Coulon:2014fr}.
	Nevertheless for completeness we reproduce it here.
\end{rema*}

\begin{proof}
	We first prove Point~\ref{enu: fix qc - translation} when $S$ is reduced to a single element, say $g$.
	We denote by $Y$ the ``open version'' of $\fix{g,d}$, i.e.
	\begin{equation*}
		Y = \set{x \in X}{\forall g \in S,\ \dist{gx}x < d}.
	\end{equation*}
	Since $d > \nrj S$, the set $Y$ is non-empty.
	It is also $\group g$-invariant and contained in $\fix{g,d}$.
	Let $x \in X \setminus \fix{g,d}$.
	Let $\eta > 0$ and $y$ be an $\eta$-projection of $x$ on $Y$.
	Since $x$ does not belongs to $\fix{g,d}$, one observes that $\dist{gy}y \geq d - 2\eta$.
	We now fix $\epsilon \in (0,\eta)$ such that $\dist {gy}y + \epsilon < d$ and choose a $(1, \epsilon)$-quasi-geodesic $\gamma \colon I \to X$ joining $y$ to $gy$, so that $\gamma$ is entirely contained in $Y$ (this is where the strict inequality in the definition of $Y$ plays a role).
	In particular, $y$ and $gy$ are respective $\eta$-projections of $x$ and $gx$ on $\gamma$, which is $(\epsilon/2 + 2\delta)$-quasi-convex.
	Consequently \autoref{res: proj qc} yields
	\begin{equation}
	\label{eqn: fix qc - proj}
		d - 2\eta \leq \dist{gy}y \leq \max \left\{\dist{gx}x - 2 \dist xy + 6\eta + 10\delta, 3\eta + 5\delta \right\}.
	\end{equation}
	Recall that $d >5\delta$.
	Taking $\eta > 0$ arbitrarily small leads to
	\begin{equation}
	\label{eqn: fix qc - single element}
		 \dist{gx}x \geq 2 d(x,\fix{g,d}) + d - 10\delta.
	\end{equation}
	We now prove Point~\ref{enu: fix qc - translation} for a general set $S$.
	Since $d > \nrj S$, the set $\fix{S,d}$ is non-empty.
	Let $x \in X\setminus \fix{S,d}$.
	Note that $\fix{S,d}$ is exactly the intersection of all $\fix{g,d}$ where $g$ runs over $S$.
	Hence there exists $g \in S$ such that $x$ does not belong to $\fix{g,d}$.
	Applying (\ref{eqn: fix qc - single element}) we get 
	\begin{equation*}
		 \dist{gx}x 
		 \geq 2 d(x,\fix{g,d}) + d - 10\delta 
		 \geq 2 d(x,\fix{S,d}) + d - 10\delta.
	\end{equation*}
	and Point~\ref{enu: fix qc - translation} follows.
	Point~\ref{enu: fix qc - neighborhood} is a direct consequence of Point~\ref{enu: fix qc - translation}.
	We are left to prove that $\fix{S,d}$ is quasi-convex.
	Let $y$ and $y'$ be two points of $\fix{S,d}$.
	Let $x$ be a point of $X$.
	\autoref{res: quasi-convexity distance isometry} yields for every $g \in S$,
	\begin{equation*}
		\dist {gx}x 
		\leq \max \left\{ \dist {gy}y, \dist{gy'}{y'}\right\} + 2\gro y{y'}x + 6 \delta 
		\leq d + 2 \gro y{y'}x + 6\delta.
	\end{equation*}
	It follows then from Point~\ref{enu: fix qc - neighborhood} that $d(x,\fix{S,d}) \leq \gro y{y'}x + 8\delta$.
\end{proof}

\begin{lemm}
\label{res: cyl in mov}
	Let $\ell, L \in \R_+$, with $L > 4\ell + 8\delta$.
	Let $g$ be a loxodromic isometry of $X$.
	Let $\gamma \colon \R \to X$ be a bi-infinite $L$-local $(1, \ell)$-quasi-geodesic joining $g^-$ to $g^+$.
	\begin{enumerate}
		\item 
		For every $d > \max\{ \norm g, 5\delta\}$, the path $\gamma$ is contained in $\fix{g, d + \ell + 24\delta}$.
		\item 
		Conversely, if $\norm g > 8\delta$, the for every $d \in \R_+$, the set $\fix{g,d}$ is contained in the $A$-neighborhood of $\gamma$, where
		\begin{equation*}
			A = \frac 12 (d - \norm g) + \frac 12\ell + 13\delta.
		\end{equation*}
	\end{enumerate}
\end{lemm}

\begin{proof}
	The first part of the statement is a consequence of \autoref{res: cyl in inv qc} applied to $Y = \fix{g,d}$.
	Let us focus on the second part.
	Without loss of generality we can assume that $\fix{g,d}$ is non-empty.
	Let $\eta > 0$ such that $\norm g > 8\eta + 8\delta$.
	Let $y \in X$ such that $\dist {gy}y \leq \norm g + \eta$.
	Consider $\nu \colon \intval 0L \to X$ be a $(1,\eta)$-quasi-geodesic from $y$ to $gy$.
	We extend this path to $\group g$-invariant $L$-local $(1, 2 \eta)$-quasi-geodesic $\nu \colon \R \to X$ by letting $\nu(mL + t) = g^m \nu(t)$ for every $m \in \Z$ and $t \in [0, L)$.
	Let $x \in \fix{g,d}$ and $p = \nu(t)$ a projection of $x$ on $\nu$.
	Since $\nu$ is $\group g$-invariant, $gp = \nu(t +L)$ is also a projection of $gx$ on $\nu$. 
	Moreover $\dist{gp}p \geq \norm g > 8\eta + 8\delta$.
	Recall that $\nu$ restricted to $\intval t{t+L}$ is $(\eta + 2 \delta)$-quasi-convex.
	It follows from the projection on a quasi-convex (\autoref{res: proj qc}) that 
	\begin{equation*}
		d \geq \dist {gx}x \geq 2 \dist xp + \norm g  - 4\eta - 10\delta
	\end{equation*}
	Note that $L \geq \norm g > 8\eta + 8\delta$.
	By stability of quasi-geodesics (\autoref{res: stability qg}) we get $\gro{g^-}{g^+}p \leq \eta + 2\delta$.
	Since $\gamma$ is $(\ell/2 + 4\delta)$-quasi-convex, we get 
	\begin{align*}
		d(x, \gamma) 
		\leq \gro{g^-}{g^+}x  + \ell/2 + 6\delta
		& \leq \dist xp + \gro{g^-}{g^+}p + \ell/2 + 6\delta \\
		& \leq \frac 12(d - \norm g) + \ell/2 + 3\eta + 13\delta.	
	\end{align*}
	This holds for every sufficiently small $\eta$, whence the result.
\end{proof}

\begin{lemm}
\label{res: translation length vs cocycle}
	Let $\xi \in \partial X$.
	Let $g$ be an isometry of $X$ fixing $\xi$.	
	There exists $\epsilon \in \{\pm 1\}$ with the following property.
	Let $c$ a Busemann cocycle at $\xi$.
	For every $x \in X$, for every $n \in \Z$, we have $\abs{c(x, g^nx) - \epsilon n\snorm g} \leq 6\delta$.
\end{lemm}

\begin{proof}	
	If $g$ is loxodromic we choose $\epsilon \in \{\pm 1\}$ such that $\xi$ is the attractive point of $g^\epsilon$.
	If $g$ is elliptic or parabolic, the choice of $\epsilon$ does not matter, as in this situation $\snorm g = 0$.
	Note that for every $n \in \N \setminus\{0\}$, we have $\snorm {g^n} = n \snorm g$.
	Moreover if $g$ is loxodromic, then $\xi$ is still the attractive point of $g^{\epsilon n}$.
	Consequently, up to replacing $g$ by $g^n$, it suffices to prove the statement for $n \in \{-1, 0, 1\}$.
	Note that the lemma is obvious if $n = 0$.
	Since $g$ fixes $\xi$, for every Busemann cocycle $c$ at $\xi$, its translate $gc$ is also a Busemann cocycle at $\xi$.
	Hence it suffices to prove the statement for $n=1$.

	Let $c$ be a Busemann cocycle at $\xi$.
	Let $x \in X$.
	Let $m \in \N$.
	Observe that
	\begin{equation*}	
			c(x, g^mx) 
			= \sum_{k=0}^{m-1} c(g^kx, g^{k+1}x)
			= \sum_{k=0}^{m-1} g^{-k} c (x,gx).
	\end{equation*}
	Since $g$ fixes $\xi$, for every $k \in \N$, the map $g^{-k}c$ is a Busemann cocycle at $\xi$, and therefore differs from $c$ by at most $6\delta$.
	Thus $\abs{c(x, g^mx) - mc(x, gx)} \leq 6m \delta$.
	As Busemann cocycles are almost $1$-Lipschitz, we get
	\begin{equation*}
		\abs{c(x, gx)} \leq \frac 1m \abs{c(x, g^nx)} + 6\delta \leq \frac 1m \dist{g^nx}x + \left(6 + \frac 2m\right)\delta.
	\end{equation*}
	Taking the limit yields $\abs{c(x, gx)} \leq \snorm g + 6\delta$.
	In particular, the result holds if $g$ is either elliptic or parabolic.

	Assume now that $g$ is loxodromic.
	Fix $\eta > 0$.
	Note that $\norm{g^m} > 8\eta + 8\delta$, for every sufficiently large $m \in \N$.
	We fix such an exponent $m$ and write $h = g^{\epsilon m}$.
	Let $y\in X$ be a point such that $\dist {hy}y \leq \norm h+ \eta$.
	We choose a $(1, \eta)$-quasi-geodesic $\gamma \colon \intval 0L \to X$ joining $y$ to $h y$ and extend $\gamma$ to a bi-infinite path $\gamma \colon \R \to X$ as follows: for every $t \in [0, L)$, for every $k \in \Z$, we let $\gamma(kL + t) = h^k\gamma(t)$.
	It follows from our choice of $y$ that $\gamma$ is an $L$-local $(1, 2\eta)$-quasi-geodesic from $h^-$ to $h^+ = \xi$.
	Applying \autoref{res: computing cocycle from qg}, we get 
	\begin{equation}
	\label{eqn: translation length vs cocycle}
		\abs{c(y,hy) - \dist{hy}y} \leq 2\eta +8\delta.
	\end{equation}
	Observe that $h^{-1}c$ and $c$ are two cocycles at $\xi$ (since $h$ fixes $\xi$), hence they differ by at most $6\delta$.
	The cocycle property yields $\abs{c(x, g^mx) -  \epsilon c(y,hy)} \leq 12\delta$.
	Thus (\ref{eqn: translation length vs cocycle}) becomes $\abs{c(x, g^mx) - \epsilon\dist{g^my}y} \leq 2\eta + 20\delta$.
	Recall that $c(x, g^mx)$ and $mc(x, gx)$ differs by at most $6m\delta$.
	Hence
	\begin{equation*}
		\abs{c(x, gx) - \frac \epsilon m \dist{g^my}y} 
		\leq \frac 1m \abs{c(x, g^mx) - \epsilon \dist{g^my}y} + 6\delta
		\leq 6\delta + \frac 1m (2\eta + 20\delta).
	\end{equation*}
	The result follows by taking the limit as $m$ approaches infinity.
\end{proof}

\begin{lemm}
\label{res: isom fixing xi moving geo}
	Let $x \in X$ and $\xi \in \partial X$.
	Let $\ell, L \in \R_+$ with $L > 4 \ell + 8\delta$.
	Let $\gamma \colon \R_+ \to X$ be an $L$-local $(1, \ell)$-quasi-geodesic ray from $x$ to $\xi$.
	For every $d \in \R_+$, there exists $t_d \in \R_+$ with the following property: if $g$ is an isometry fixing $\xi$ and satisfying $\dist {gx}x \leq d$, then there exists $\epsilon \in \{\pm 1\}$ such that for every $t\geq t_d$ we have
	\begin{equation*}
		\dist{\gamma(t + \epsilon \snorm g)}{g \gamma(t)} \leq 2\ell + 25\delta.
	\end{equation*}
\end{lemm}

\begin{proof}
	The path $\gamma$ is a global quasi-geodesic (\autoref{res: stability qg}) thus there exists $t_d  \in \R_+$ such that for every $t \geq t_d$, we have 
	$\dist{\gamma(t)}x > d + \ell/2 + 3\delta$.
	Let $g$ be an isometry fixing $\xi$ such that $\dist{gx}x \leq d$.
	It follows from the triangle inequality that $\gro {gx}x{g\gamma(t)} > \ell/2 + 3\delta$ whenever $t \geq t_d$.
	Let $c$ be a Busemann cocycle at $\xi$.
	According to \autoref{res: translation length vs cocycle} there exists $\epsilon \in \{\pm 1\}$ such that for every $x \in X$, we have $\abs{c(x,gx) - \epsilon \snorm g} \leq 6\delta$.
	Let $t \geq t_d$.
	For simplicity we write $y_1 = \gamma(t)$ and $y_2 = \gamma(t+ \epsilon \snorm g)$.
	Suppose first that $\epsilon = 1$ so that $t + \epsilon \snorm g \geq t$.
	Since $\gamma$ is an $L$-local $(1, \ell)$-quasi-geodesic, we know from \autoref{res: computing cocycle from qg} that
	\begin{equation*}
		\abs{c(y_1,y_2) -\epsilon \snorm g} 
		\leq \abs{c(y_1,y_2) - \dist{\gamma(t)}{\gamma(t + \epsilon \snorm g)}} 
		\leq \ell + 8\delta.
	\end{equation*}
	If $\epsilon = -1$, we permute the role of $y_1$ and $y_2$ in the above argument and still get that $c(y_1, y_2)$ differs from $\epsilon \snorm g$ by at most $\ell + 8\delta$.
	Consequently 
	\begin{equation*}
		\abs{c(gy_1,y_2)}
		\leq \abs{ c(gy_1, y_1) + c(y_1, y_2)}
		\leq \ell + 14\delta.
	\end{equation*}
	By \autoref{res: stability qg}, $\gro x\xi{y_2} \leq \ell/2 + 2\delta$ and $\gro{gx}\xi{gy_1} \leq \ell/2 + 2\delta$.
	The four point inequality (\ref{eqn: hyperbolicity condition with boundary}) yields
	\begin{equation*}
		\min\left\{ \gro{gx}x{gy_1} , \gro {x}\xi{gy_1}\right\} \leq \gro {gx}\xi{gy_1} + \delta \leq \ell/2 + 3\delta.
	\end{equation*}
	It follows from our choice of $t_d$ that the minimum cannot by achieved by $\gro{gx}x{gy_1}$, hence $\gro x\xi{gy_1} \leq \ell/2 + 3\delta$.
	Applying \autoref{res: variation four point w/ cocycle} we get
	\begin{equation*}
		\dist{gy_1}{y_2} 
		\leq \abs{c(gy_1,y_2)} + 2\max\left\{ \gro x\xi{gy_1}, \gro x\xi{y_2}  \right\} + 5\delta
		\leq 2\ell + 25\delta. 
		\qedhere
	\end{equation*}
\end{proof}

The same arguments can be used to prove the following lemma.

\begin{lemm}
\label{res: loxo moving geo}
	Let $g$ be a loxodromic isometry.
	Let $L, \ell \in \R_+$, with $L > 4\ell +8\delta$.
	Let $\gamma \colon \R_+ \to X$ be an $L$-local $(1, \ell)$-quasi-geodesic from $g^-$ to $g^+$.
	Then for every $t \in \R$, for every $n \in \Z$, we have
	\begin{equation*}
		\dist{\gamma(t + n\snorm g)}{g^n \gamma(t)} \leq 2\ell + 25\delta. 
	\end{equation*}
\end{lemm}

%
\subsection{Group action}
%
\label{sec: hyp - gp action}

\paragraph{Classification of actions.}
Let $G$ be a group acting by isometries on $X$.
Its \emph{limit set} $\Lambda(G)$ is the set of accumulation points in $\partial X$ of some (hence any) orbit of $G$.
The action of $G$ on $X$  is \emph{elliptic} (\resp \emph{parabolic}, \emph{loxodromic}, \emph{non-elementary}) if $\Lambda(G)$ is empty (\resp contains exactly $1$ point, exactly $2$ points, at least $3$ points).
If there is no ambiguity regarding the action, we simply say that $G$ is \emph{elliptic} (\resp \emph{parabolic}, \emph{loxodromic}, \emph{non-elementary}).

\paragraph{Elliptic action.}
Even though $X$ is not necessarily locally compact, a group $G$ is elliptic if and only if its orbits are bounded \cite[Proposition~3.5]{Coulon:2016if}.
Elliptic groups actually have very small orbits.
\begin{lemm}[Compare with {\cite[Proposition~2.3.4]{Delzant:2008tu}} or {\cite[Corollary~2.38]{Coulon:2014fr}}]
\label{res: fix set elliptic}
	Let $G$ be an elliptic group of isometries of $X$.
	The set $\fix{G, 5\delta}$ is non-empty.
	Moreover if $Y$ is a non-empty $G$-invariant $\alpha$-quasi-convex subset of $X$, then $\fix{G,10\delta}$ intersects the $\alpha$-neighborhood of $Y$. 
\end{lemm}

\paragraph{Loxodromic action.}
Let $G$ be a loxodromic group.
In particular, it contains a loxodromic isometry, say $g$ \cite[Proposition~3.6]{Coulon:2016if}.
Note that $g^-$ and $g^+$ are the two points of $\Lambda(G)$.
Moreover every element of $G$ preserves $\{g^-, g^+\}$.
We denote by $G^+$ the subgroup of $G$ fixing \emph{pointwise} $\{g^-, g^+\}$.
It has index at most $2$ in $G$.
If $G = G^+$ we say that $G$ is \emph{preserves the orientation}.

Let $\Gamma$ be the union of all $L$-local $(1, \delta)$-quasi-geodesics from $g^-$ to $g^+$ with $L> 12\delta$.
The \emph{cylinder} $Y$ of $G$ is the set
\begin{equation}
\label{eqn: def cylinder}
	Y = \set{x \in X}{ d(x,\Gamma) < 20\delta}.
\end{equation}
It is a strongly quasi-convex subset of $X$, see \cite[Lemma~3.13]{Coulon:2016if}.
By construction $\partial Y = \{g^-, g^+\} = \Lambda(G)$.

\begin{lemm}
\label{res: normal elliptic fixing cyl}
	Let $g \in G$ be a loxodromic element.
	Let $F$ be an elliptic subgroup of $G$ normalized by $g$.
	Let $\ell, L \in \R_+$, with $L > 4\ell + 8\delta$.
	Let $\gamma$ be an $L$-local $(1,\ell)$-quasi-geodesic from $g^-$ to $g^+$.
	Then $\gamma$ is contained in $\fix{F,\ell + 30\delta}$.
\end{lemm}

\begin{proof}
	Since $g$ normalizes $F$, the set $\fix{F,6\delta}$ is a $\group g$-invariant $8\delta$-quasi-convex subset (\autoref{res: fix qc}).
	By \autoref{res: cyl in inv qc}, the path $\gamma$ is contained in the $(\ell/2 + 12\delta)$-neighborhood of $\fix{F,6\delta}$.
\end{proof}

\paragraph{Non-elementary action.}
The next lemma is an improved version of the classical ping-pong argument.
It provides a simple criterion to ensure that a group is non-elementary.

\begin{lemm}[See {\cite[Lemma~3.24]{Coulon:2016if}}]
\label{res: non elementary subgroup sufficient condition}
	Let $A \geq 0$.
	Let $x \in X$.
	Let $G$ be a group of isometries of $X$ generated by two elements $u$ and $v$ such that 
	\begin{enumerate}
		\item $2\gro {u^{\pm 1}x}{v^{\pm 1}x}x < \min \{\dist {ux}x, \dist {vx}x \} -A - 8\delta$,
		\item $2\gro{ux}{u^{-1}x}x < \dist {ux}x + A$,
		\item $2\gro{vx}{v^{-1}x}x < \dist {vx}x + A$.
	\end{enumerate}
	Then $G$ is non-elementary.
\end{lemm}

\paragraph{Gentle action.}
In order to complete the description of loxodromic groups started above, we introduce a (harmless) additional assumption.
\begin{defi}
\label{def: gentle action}
	The action of $G$ on $X$ is \emph{gentle} if every loxodromic subgroup $H$ preserving the orientation splits as a semi-direct product $H = \sdp F\cyclic$ where $F$ consists exactly of all elliptic elements of $H$.
\end{defi}
If every loxodromic subgroup is virtually cyclic, then the action of $G$ is automatically gentle.
From now on we assume that the action of $G$ on $X$ is gentle.
Let $H$ be a loxodromic subgroup of $G$ and $H^+$ the subgroup of $H$ fixing pointwise $\Lambda(H)$.
Let $F$ be the set of all elliptic elements of $H^+$.
It follows from our assumption that $F$ is an elliptic normal subgroup of $H$ and is maximal for these properties.
The quotient $H/F$ is either isomorphic to $\cyclic$ if $H$ preserves the orientation (i.e.\ $H = H^+$) or the infinite dihedral group $\dihedral$ otherwise.
Observe that if $H$ is generated by two elliptic subgroups, then $H$ cannot preserve the orientation.

\begin{lemm}
\label{res: canonical proj to dihedral}
	Let $H$ be a loxodromic subgroup of $G$.
	If $p \colon H \to \dihedral$ is a morphism whose kernel is elliptic, then this kernel is exactly the maximal normal elliptic subgroup $F$ of $H$.
\end{lemm}

\begin{proof}	
	By assumption the kernel of $p$ is an elliptic normal subgroup of $H$, hence it is contained in $F$.
	Let us prove the other inclusion.
	We fix a loxodromic element $h \in H$.
	According to our assumption, the pre-image under $p$ of any finite subgroup of $\dihedral$ is elliptic (as a finite extension of an elliptic subgroup).
	Hence $p(h)$ belongs to $\cyclic\setminus\{0\} \subset \dihedral$.	
	Observe then that $p(F)$ does not contain any element of $\cyclic\setminus\{0\}$.
	Indeed otherwise there would exist $m \in \cyclic\setminus\{0\}$ and $u \in F$ such that $p(h^m) = p(u)$.
	Since the kernel of $p$ is contained in $F$, the element $h^m$ should belong to $F$ which contradicts the fact that $h$ is loxodromic.
	Hence $p(F)$ is a finite subgroup of $\dihedral$.
	As $h$ normalizes $F$, its image $p(h)$ normalizes $p(F)$ which forces $p(F)$ to be trivial.
\end{proof}

Given a loxodromic element $g \in G$, we write $E(g)$ for the subgroup of $G$ preserving $\{g^-,g^+\}$.
It is the maximal elementary subgroup of $G$ containing $g$.
The group $E^+(g)$ stands for the maximal subgroup of $E(g)$ fixing pointwise $\{g^-,g^+\}$.

\begin{defi}(Primitive element)
\label{def: primitive}
	Let $g \in G$ be a loxodromic element.
	Let $F$ be the maximal normal elliptic subgroup of $E(g)$.
	We say that $g$ is \emph{primitive} if its image in $E^+(g)/F\equiv \cyclic$ generates the group.
\end{defi}

%
\section{Invariants of a group action}
%
\label{sec: invariants}

Let $X$ be a $\delta$-hyperbolic length space and $G$ a group acting gently by isometries on $X$ (see \autoref{def: gentle action}).
Note that for the moment, we have not made any serious assumption on the group $G$ or the space $X$.
In order to study the action of $G$ on $X$ we define several numerical invariants.
Those quantities will be useful later to estimate the small cancellation parameters needed to run the induction leading to the infiniteness of Burnside groups.
We define two types of invariants.
The first kind, namely the \emph{injectivity radius} $\inj[X]G$, the \emph{acylindricity constant} $A(G,X)$ as well as the \emph{$\nu$-invariant} $\nu(G,X)$ are purely \emph{geometric}.
Those invariants (or some variations of them) already appeared in \cite[D\'efinition~2.4.1]{Delzant:2008tu} and \cite[Definition~3.40]{Coulon:2016if}.
Unfortunately they are not sharp enough to handle even torsion.
More precisely the $\nu$-invariant, does not behave well when passing to quotient.
Therefore we also define (among others) a strong variation $\nu_{\rm{stg}}(G,X)$ of the $\nu$-invariant, which has a mixed nature: it reflects both the \emph{geometric} and \emph{algebraic} features of $G$.

%
\subsection{Geometric invariants}
%
\label{sec: invariants - geometric}

\begin{defi}[Injectivity radius]
\label{def: inj rad}
	The \emph{injectivity radius} of $G$ on $X$ is the quantity
	\begin{equation*}
		\inj[X]G = \inf \set{\snorm[X]g}{g \in G \ \text{loxodromic}}
	\end{equation*}
\end{defi}

\begin{defi}[Acylindricity]
\label{def: acyl inv}
	Let $d \in \R_+$.
	The \emph{acylindricity parameter at scale $d$}, denoted by $A(G,X,d)$, is
	\begin{equation*}
		 A(G,X,d) = \sup_{S \subset G} \diam \left(\fix{S,d}\right),
	\end{equation*}
	where $S$ runs over all subsets of $G$ generating a \emph{non-elementary} subgroup.
\end{defi}

\begin{defi}
\label{res: weak acylindricity}
	The action of $G$ on the space $X$ is \emph{weakly acylindrical} if the map ${d \mapsto A(G,X,d)}$ is bounded above by an affine function of $d$.
\end{defi}

The previous definition is designed to compensate for the fact that we do not control from below the curvature of $X$.
Its definition is modeled on the Margulis lemma for manifolds with pinched negative curvature

\begin{defi}[$\nu$-invariant]
\label{def: nu inv}
	The quantity $\nu(G,X)$ is the smallest integer $m$ with the following property:
	For every $g, h \in G$, with $h$ loxodromic, if $\{g,hgh^{-1}, h^2gh^{-2}, \dots, h^m gh^{-m}\}$ generate an elementary subgroup, then so do $g$ and $h$ 
\end{defi}

The $\nu$-invariant is useful to prove the following local-to-global phenomenon.

\begin{prop}
\label{res: local-to-global acyl}
	For every $d \in \R_+$, we have
	\begin{equation*}
		A(G,X,d) \leq \left[\nu(G,X)+3\right]d +  A(G,X, 420\delta) +24\delta.	
	\end{equation*}
	In particular, if $A(G,X,420\delta)$ and $\nu(G,X)$ are finite, then the action of $G$ on $X$ is weakly acylindrical.
\end{prop}

Before proving \autoref{res: local-to-global acyl} we focus on the following lemmas.

\begin{lemm}
\label{res: margulis lemma}
	Fix $d_1 = 340\delta$ and $d_2 = 420\delta$.
	Let $g$ and $h$ be two elements of $G$ which generate a non-elementary subgroup.
	\begin{enumerate}
		\item \label{enu: margulis lemma - short g}
		Assume that $h$ is loxodromic.
		Let $L \in \R_+$, with $L > 12\delta$ and $\gamma \colon \R \to X$ be an $L$-local $(1, \delta)$-quasi-geodesic from $h^-$ to $h^+$.
		Then
		\begin{equation*}
			\diam \left( \fix{g,d_1} \cap \gamma\right) \leq  \nu(G,X)\norm h + A(G,X, d_2) + 2\delta.
		\end{equation*}
		\item \label{enu: margulis lemma - large g}
		Assume that both $g$ and $h$ are loxodromic.
		There exists $L > 12 \delta$, with the following property.
	 	If $\gamma_g, \gamma_h \colon \R \to X$ are two $L$-local $(1, \delta)$-quasi-geodesics from $g^-$ to $g^+$, and $h^-$ to $h^+$ respectively, then
		\begin{equation*}
			\diam \left( \gamma_g^{+8\delta} \cap \gamma_h^{+8\delta} \right) \leq  \norm g + \norm h + \nu(G,X)\max\left\{\norm g, \norm h\right\} + A(G,X, d_2) + 20\delta.
		\end{equation*}
	\end{enumerate}
\end{lemm}

\begin{rema*}
	A similar statement is proved in \cite{Coulon:2016if} closely following the ideas of Delzant and Gromov \cite{Delzant:2008tu}.
	For completeness, we reproduce it here.
\end{rema*}

\begin{proof}
	For simplicity we let $\nu = \nu(G,X)$.
	We start with Point~\ref{enu: margulis lemma - short g}.
	Assume that contrary to our claim,
	\begin{equation*}
		\diam\left(\fix{g,d_1}\cap \gamma\right)  >  \nu\norm h + A(G,X,d_2) + 2\delta.
	\end{equation*}
	In particular, there exist $x = \gamma(s)$ and $x'= \gamma(s')$ lying in $\fix{g,d_1}$ such that $\dist x{x'} >  \nu\norm h + A(G,X,d_2) + 2\delta$.
	Without loss of generality we can assume that $s < s'$, so that
	\begin{equation*}
			s' - s \geq \dist x{x'} >  \nu\norm h + A(G,X, d_2) + 2\delta.
	\end{equation*}
	Since $\fix{g,d_1}$ is $8\delta$-quasi-convex (\autoref{res: fix qc}), $\gamma$ restricted to $\intval s{s'}$ lies entirely in the $11\delta$-neighborhood of $\fix{g,d_1}$.
	We now fix $t = s + A(G,X,d_2) +2\delta$.
	Let $r \in \intval st$ and $k \in \intvald 0\nu$.
	Note that $r_k = r + k \snorm h$ belongs to $\intval s{s'}$.
	Thus $\dist{g\gamma(r_k)} {\gamma(r_k)}\leq d_1 + 22\delta$.
	Applying \autoref{res: loxo moving geo} we obtain
	\begin{equation*}
		\dist{h^{-k}gh^{k}\gamma(r)}{\gamma(r)}
		\leq \dist{gh^k \gamma(r)}{h^k\gamma(r)}
		\leq \dist{g \gamma(r_k)}{\gamma(r_k)} +54\delta
		\leq d_1 + 76\delta.
	\end{equation*}
	In other words, the restriction of $\gamma$ to $\intval st$ is contained in $\fix{S, d_2}$, where $S$ is the set $S = \{g,h^{-1}gh, \dots, h^{-\nu}gh^\nu\}$.
	Consequently the diameter of $\fix{S, d_2}$ is larger than $A(G,X,d_2)$, and thus $S$ generate an elementary subgroup.
	Recall that $h$ is loxodromic.
	It follows from the definition of $\nu$ that $g$ and $h$ generate an elementary subgroup which contradicts our assumption.

	We now focus on Point~\ref{enu: margulis lemma - large g}.
	Up to permuting $g$ and $h$, we can assume that $\norm h \geq \norm g$.
	We fix 
	\begin{equation*}
		\ell = \nu\norm h + A(G,X, d_2) + 4\delta
		\quad \text{and} \quad
		M =  \ell + \norm g + \norm h.
	\end{equation*}
	Let $\gamma_g, \gamma_h \colon \R \to X$ as in the statement of Point~\ref{enu: margulis lemma - large g}.
	Assume that contrary to our claim we have
	\begin{equation*}
			\diam \left(  \gamma_g^{+8\delta} \cap \gamma_h^{+8\delta}  \right) 
			>  \norm g + (\nu+1)\norm h + A(G,X,d_2) + 20\delta
			\geq M + 16\delta.
	\end{equation*}
	We fix two points $x,y \in X$ lying in the $8\delta$-neighborhood of both $\gamma_g$ and $\gamma_h$ such that 
	\begin{equation*}
			\dist xy >  M+ 16\delta.
	\end{equation*}
	Up to changing the origin of $\gamma_g$ and $\gamma_h$ we can assume that $\gamma_g(0)$ and $\gamma_h(0)$ are projections of $x$ on $\gamma_g$ and $\gamma_h$ respectively.
	We write $\gamma_g(s)$ and $\gamma_h(s')$ for projections of $y$ on $\gamma_g$ and $\gamma_h$ respectively.
	Note that $s,s'\in(L, \infty)$.
	
	We claim that $\dist{\gamma_g(r)}{\gamma_h(r)} \leq 57\delta$, for every $r \in \intval 0L$.
	Since $\gamma_g$ is an $L$-local $(1, \delta)$-quasi-geodesic, we have $\gro{\gamma_g(0)}{\gamma_g(s)}{\gamma_g(r)} \leq 3\delta$ (\autoref{res: stability qg}), thus $\gro xy{\gamma_g(r)} \leq 19\delta$.
	Similarly $\gro xy{\gamma_h(r)} \leq 19\delta$.
	Moreover, the quantities $\dist{\gamma_g(r)}x$ and $\dist{\gamma_h(r)}x$ differ by at most $17\delta$.
	It follows from the four point inequality -- see for instance \cite[Lemma~2.2(2)]{Coulon:2014fr} -- that $\dist{\gamma_g(r)}{\gamma_h(r)}\leq 57\delta$,
	which completes the proof of our claim.
	
	According to \autoref{res: loxo moving geo}, $g$ (\resp $h$) acts on $\gamma_g$ (\resp $\gamma_h$) almost like a translation of length $\snorm g$ (\resp $\snorm h$). 
	Hence for every $r \in \R_+$, such that  $r \leq\ell$,
	\begin{equation*}
		\dist{gh \gamma_h(r)}{hg \gamma_h(r)}\leq 336\delta \leq d_1,
	\end{equation*}
	compare with \autoref{fig: naive attempt}.
	Consequently the path $\gamma_h$ restricted to $\intval 0{\ell}$ is contained in $\fix{u, d_1}$ where $u = g^{-1}h^{-1}gh$.
	Thus, applying \autoref{res: intersection of thickened quasi-convex} we get
	\begin{equation*}
		\diam \left( \fix{u,d_1}\cap \gamma_h\right) \geq \ell -\delta> \nu\norm h + A(G,X,d_2) + 2\delta.
	\end{equation*}
	It follows from the previous discussion that $u$ and $h$ generates an elementary subgroup.
	Hence $g^{-1}hg$ and $h$ generate an elementary subgroup.
	Since $h$ is loxodromic, $g$ fixes $h^-$ and $h^+$, therefore $g$ and $h$ generate an elementary subgroup, which contradicts our assumption.
\end{proof}

\begin{lemm}
\label{res: margulis lemma bis}
	Fix $d_2 = 420\delta$.
	Let $g$ and $h$ be two elements of $G$ which generate a non-elementary subgroup and set $S = \{g,h\}$.
	If $\norm h \geq 28\delta$, then for every $d \in \R_+$ we have
	\begin{equation*}
		\diam \left( \fix{S,d} \right) \leq \left[ \nu(G,X) + 3\right] d + A(G,X,d_2) + 24\delta.
	\end{equation*}
\end{lemm}

\begin{proof}
	For simplicity we write $\nu$ for $\nu(G,X)$.
	Without loss of generality we can assume that $d > \lambda(S)$. 
	In particular, $d > \max\{\norm g, \norm h\}$.
	
	Assume first that $\norm g < 28\delta$.
	We fix $L > 12 \delta$ and an $L$-local $(1, \delta)$-quasi-geodesic $\gamma_h$ from $h^-$ to $h^+$.
	It follows from \autoref{res: cyl in mov}, that $\fix{h,d}$ lies in the $d/2$-neighborhood of $\gamma_h$.
	On the other hand $\fix{g,d}$ lies in the $d/2$-neighborhood of $\fix{g,28\delta}$.
	Combining these observations with \autoref{res: intersection of thickened quasi-convex} we get
	\begin{align*}
		\diam \left( \fix{S,d} \right) 
		& \leq \diam \left( \fix{g,28\delta}^{+d/2} \cap \gamma_h^{+d/2} \right) \\
		& \leq \diam \left(\fix{g,28\delta}^{+11\delta} \cap \gamma_h^{+8\delta}\right) + d + 4\delta \\
		& \leq \diam \left(\fix{g,66\delta} \cap \gamma_h\right) + d + 20\delta
	\end{align*}
	It follows then from \autoref{res: margulis lemma} that 
	\begin{equation*}
		\diam \left( \fix{S,d} \right)  \leq (\nu + 1)d + A(G,X,d_2) + 22\delta.
	\end{equation*}
	Assume now that $\norm g \geq 28\delta$.
	In particular, $g$ is loxodromic.
	By \autoref{res: margulis lemma} we can find $\gamma_g$ and $\gamma_h$ two $L$-local $(1,\delta)$-quasi-geodesics with $L> 12\delta$, joining $g^-$ to $g^+$ and $h^-$ to $h^+$ respectively such that 
	\begin{equation*}
		\diam \left( \gamma_g^{+8\delta} \cap \gamma_h^{+8\delta} \right) \leq   (\nu + 2)d + A(G,X, d_2) + 20\delta.
	\end{equation*}
	It follows from \autoref{res: cyl in mov}, that $\fix{g,d}$ and $\fix{h,d}$ lies in the $d/2$-neighborhood of $\gamma_g$ and $\gamma_h$ respectively.
	Using \autoref{res: intersection of thickened quasi-convex} as above we get
	\begin{align*}
		\diam \left( \fix{S,d} \right) 
		&\diam \left( \gamma_g^{+d/2} \cap \gamma_h^{+d/2} \right)\\
		& \leq\diam \left( \gamma_g^{+8\delta} \cap \gamma_h^{+8\delta} \right) + d + 4\delta \\
		& \leq  (\nu + 3)d + A(G,X, d_2) + 24\delta.\qedhere
	\end{align*}
\end{proof}

\begin{proof}[Proof of \autoref{res: local-to-global acyl}]
	For simplicity we write $\nu$ for $\nu(G,X)$.
	Let $d \in \R_+$.
	Let $S \subset G$ be a set of elements which do not generate an elementary subgroup.
	Without loss of generality we can assume that $S$ is finite.
	We distinguish two cases.
	Assume first that $\norm g < 28\delta$, for every $g \in S$.
	It follows from \autoref{res: fix qc} that $\fix{g,d}$ is contained in the $d/2$-neighborhood of $\fix{g, 28\delta}$, for every $g \in S$.
	Combined with \autoref{res: intersection of thickened quasi-convex} we get
	\begin{align*}
		\diam\left(\fix{S,d}\right)
		& \leq \diam\left(  \bigcap_{g \in S} \fix{g, 28\delta}^{+11\delta} \right) + d + 4\delta\\
		& \leq \diam\left(  \bigcap_{g \in S} \fix{g, 50\delta} \right) + d + 4\delta \\
		&\leq A(G,X,50\delta) + d + 4\delta.
	\end{align*}
	Assume now that there exists $h \in S$ such that $\norm h \geq 28\delta$.
	In particular, $h$ is loxodromic.
	It follows from our assumption that there exists $g \in S$ such that $g$ and $h$ do not generate an elementary subgroup.
	Applying \autoref{res: margulis lemma bis}, we get
	\begin{equation*}
		\diam\left(\fix{S,d}\right)
		\leq \diam\left(\fix{\{g,h\},d}\right)
		\leq  (\nu + 3)d + A(G,X, 420\delta) + 24\delta.
	\end{equation*}
	In both cases, the diameter of $\fix{S,d}$ is bounded above by $A(G,X,420\delta) + 24\delta$, whence the result.	
\end{proof}

\begin{lemm}
\label{res: stab of para pt is para}
	Let $P$ be a parabolic subgroup of $G$.
	Let $\xi \in \partial X$ be the unique accumulation point of $P$.
	If the action of $G$ on $X$ is weakly acylindrical, then $\stab \xi$ is parabolic as well.
\end{lemm}

\begin{proof}
	For simplicity we let $E = \stab \xi$.
	Assume that contrary to our claim that $E$ is not parabolic.
	In particular, $E$ contains a loxodromic element $h$.
	Up to replacing $h$ by a power of $h$ we can assume that $\norm h > 8\delta$.
	Let $g \in P$.
	We set $S = \{g,h\}$ and $L =  \norm h + 100\delta$.
	Note that $\fix{S,L}$ has infinite diameter (\autoref{res: isom fixing xi moving geo}).
	It follows from the definition of weak acylindricity that the subgroup of $G$ generated by $g$ and $h$ is elementary.
	Consequently $P$ is contained in the maximal elementary subgroup $E(h)$ of $G$ containing $h$.
	Nevertheless the subgroups of $E(h)$ are either elliptic or loxodromic, which contradicts our assumption.
\end{proof}

%
\subsection{Mixed invariants}
%
\label{sec: invariants - mixed}

As explained in the previous section, the combination of the acylindricity parameter $A(G,X,420\delta)$ and the $\nu$-invariant provides a useful substitute to the Margulis lemma.
Nevertheless the latter invariant does not behave well when passing to quotient (see \autoref{sec: invariants quotient}).
To bypass this difficulty we consider a stronger version of the $\nu$-invariant whose mixed nature combines both geometric and algebraic features of $G$.
More precisely the algebraic part captures the properties of a special class of elementary subgroups that we define now.

\paragraph{Dihedral germs and dihedral pairs.}
Recall that $\dihedral$ stands for the infinite dihedral group.
Given $m \in \N$, we denote by 
\begin{equation*}
	\dihedral[m] = \left< \mathbf s, \mathbf r \mid \mathbf s^2, (\mathbf s\mathbf r)^2, \mathbf r^m \right>
\end{equation*}
the dihedral group of order $2m$ and by $\cyclic[m]$ the cyclic group of order $m$.
Note that $\dihedral[1] = \cyclic[2]$.
By convention $\dihedral[0]$ is the trivial group.
We think of $\dihedral[m]$ as the isometry group of the plane preserving a regular $m$-gon.
This motivates the following terminology.
The subgroup $\group{\mathbf r}$ is a normal subgroup called the \emph{rotation subgroup}.
Its elements are also called \emph{orientation preserving}.
The \emph{signature} is the morphism $\epsilon \colon \dihedral[m] \to \cyclic[2]$, where $\cyclic[2]$ is the quotient of $\dihedral[m]$ by the rotation subgroup.
An element of $\dihedral[m]$ that does not preserve the orientation is called a \emph{reflection}.

We adopt a similar terminology for $\dihedral$.
In particular, its \emph{rotation subgroup} (or \emph{translation subgroup}) is the maximal subgroup isomorphic to $\cyclic$.
If $m \neq 2$, the rotation subgroup of $\dihedral[m]$ is algebraically completely determined: it is the unique cyclic subgroup of order $m$.
Otherwise it should be thought an implicit piece of information attached to $\dihedral[2]$.

\begin{defi}
\label{def: dihedral germ}
	A subgroup $C$ of $G$ is called a \emph{dihedral germ} if it contains an elliptic subgroup $C_0$ which is normalized by a loxodromic element and such that $[C:C_0]$ is a power of $2$.
\end{defi}

Note that dihedral germs are elliptic.
Being a dihedral germ is invariant under conjugation.
Without any further assumption on the structure of loxodromic subgroups, it is not true in general that being a dihedral germ is invariant by taking subgroup.

\begin{defi}
\label{def: dihedral pair}
	A \emph{dihedral pair} is a pair $(E,C)$ of subgroups such that $C$ is a dihedral germ which is also normal in $E$ and $E/C$ embeds in a (finite or infinite) dihedral group.
	A subgroup $E$ of $G$ has \emph{dihedral shape} if there exists a subgroup $C$ such that $(E,C)$ is a dihedral pair.
\end{defi}

Every subgroup with dihedral shape is elementary.
Indeed such a group is virtually the extension of an elliptic subgroup by a cyclic group.
Note that the morphism from $E/C$ to a dihedral group is in general not unique.

\begin{lemm}
\label{res: loxo have dihedral shape}
	Let $E$ be a loxodromic subgroup of $G$ and $C$ a subgroup of $E$.
	Then $(E,C)$ is a dihedral pair if and only if $C$ is the maximal elliptic normal subgroup of $E$.
	In particular, every loxodromic subgroup has dihedral shape.
\end{lemm}

\begin{proof}
	Assume that $(E,C)$ is a dihedral pair.
	In particular, $E/C$ embeds in a dihedral group.
	Note that this dihedral group cannot be finite.
	Indeed otherwise $E$ would be a finite extension of the elliptic subgroup $C$, hence an elliptic subgroup as well.
	Consequently $E/C$ embeds in $\dihedral$.
	It follows from \autoref{res: canonical proj to dihedral} that $C$ is the maximal elliptic normal subgroup of $E$.
	The converse statement is obvious.
\end{proof}

\paragraph{Strong $\nu$-invariant.}
For our next invariant, we adopt the following terminology borrowed from Lysenok \cite{Lysenok:1996kw}.
A \emph{chain of length $m$} is a tuple $\mathcal C = (g_0, \dots, g_m)$ of elements of $G$ for which there exists $h \in G$ such that for every $k \in \intvald 0{m-1}$, we have $g_{k+1} = hg_kh^{-1}$.
The element $h$ is called a \emph{conjugating element} of $\mathcal C$.
Note that such an element is not necessarily unique.
With this terminology the definition of the $\nu$-invariant $\nu(G,X)$ can be restated as follows: it it the smallest integer $m$ such that if $\mathcal C = (g_0,\dots, g_m)$ is a chain of length $m$ generating an elementary subgroup and $h$ a \emph{loxodromic} conjugating element of $\mathcal C$, then $\group{g_0,h}$ is elementary.
We now introduce a variation of this $\nu$-invariant.

\begin{defi}[strong $\nu$-invariant]
\label{def: nu inv stg}
	The quantity $\nu_{\rm{stg}}(G,X)$ is the smallest integer $m$ with the following property:
	if $\mathcal C = (g_0,\dots, g_m)$ is a chain generating an elementary subgroup and $h$ a conjugating element of $\mathcal C$ such that
	\begin{itemize}
		\item either $h$ is loxodromic,
		\item or $\group{g_0, \dots, g_{m-1}}$ is contained in a dihedral germ,
	\end{itemize}
	then $\group{g_0,h}$ is elementary with dihedral shape.
\end{defi}

	One observes easily that $\nu(G,X) \leq \nu_{\rm{stg}}(G,X)$.
	Let us mention an example where these two invariants are not equal.
	
\begin{exam}
\label{exa: strict inequality nu}		
	Observe first that if $G = G_1 \ast G_2$ is a free product acting on its Bass-Serre tree $T$, then $\nu(G, T) \leq 2$.
	Consider indeed $g, h \in G$ with $h$ loxodromic such that the subgroup $E = \group{g,hgh^{-1},h^2gh^{-2}}$ is elementary.
	Without loss of generality we can assume that $g$ is non trivial.
	We first claim that the subgroup $E_0 = \group{g, hgh^{-1}}$ cannot be elliptic.
	Assume on the contrary that $E_0$ fixes a point say $x \in T$.
	As $G$ is a free product, $x$ is the \emph{unique} fixed point of $g$.
	Nevertheless $hgh^{-1}$ also fixes $x$ (as it belongs to $E_0$), hence $g$ fixes $h^{-1}x$.
	This forces $hx = x$ which contradicts the fact that $h$ is loxodromic and completes the proof of the claim.
	Since $T$ is a tree, $G$ does not admit any finitely generated parabolic subgroup, hence $E_0$ is loxodromic.
	Observe now that the elementary subgroup $E$ is generated by $E_0$ and $hE_0h^{-1}$.
	Consequently $h$ necessarily belongs to the maximal elementary loxodromic subgroup containing $E_0$.
	Therefore $g$ and $h$ generate an elementary subgroup.
	This proves that $\nu(G,T) \leq 2$ as announced.
	
	In this setting, every elliptic subgroup which is normalized by a loxodromic element is trivial.
	Hence a subgroup of $G$ is a dihedral germ if an only if it is a finite $2$-group.
	Let us now consider a more precise example.
	We fix $m \in \N$ and let $A = \cyclic[2]^{m+1}$.
	For every $i \in \intvald 0m$, we write $g_i$ for a generator of the $i$-th factor $\cyclic[2]$ in $A$.
	We denote by $G_1$ the following HNN extension of $A$
	\begin{equation*}
		G_1 = \left< A, h \mid hg_ih^{-1} = g_{i+1},\ \forall i \in \intvald 0{m-1}\right>.
	\end{equation*}
	Let $G = G_1 \ast \cyclic$.
	It follows from the construction that 
	\begin{equation*}
		\group{g_0, hg_0h^{-1}, \dots, h^mg_0h^{-m}} = \group{g_0, \dots, g_m} = A
	\end{equation*}
	is a finite $2$-groups so that $\group{g_0, \dots, g_m}$ is elementary and $\group{g_0, \dots, g_{m-1}}$ is a dihedral germ.
	On the other hand, the subgroup $\group{g_0,h}$ corresponds to $G_1$ which is not virtually cyclic, thus it cannot have dihedral shape.
	This shows that $\nu_{\rm{stg}}(G,T) > m$.
	In particular, if $m \geq 2$, then $\nu(G,T) < \nu_{\rm{stg}}(G,T)$.
	
	Note that in this example, the difference between $\nu(G,X)$ and $\nu_{\rm{stg}}(G,X)$ comes from the algebraic structure of elliptic subgroups.
	It emphases the fact that $\nu_{\rm{stg}}(G,X)$ is not a purely geometric invariant.
\end{exam}

\paragraph{Model collections.}
As we will see later, controlling the strong $\nu$-invariant is a key ingredient to handle even torsion and which was not needed to study free Burnside groups of odd exponents.
It requires a fine understanding of the structure of dihedral pairs.
We complete this section by a last notion designed to describe those subgroups.

A \emph{model collection} is a family $\boldsymbol{\mathcal E}$ of (abstract) torsion groups.
Its \emph{exponent} $\mu = \mu(\boldsymbol{\mathcal E})$ is the smallest positive integer such that $\mathbf g^\mu = 1$, for every $\mathbf E \in \boldsymbol{\mathcal E}$, for every $\mathbf g \in \mathbf E$.

\begin{defi}
\label{def: elemen structure}
	Let $p \in \N$ and $\boldsymbol{\mathcal E}$ be a model collection.
	We say that a dihedral pair $(E,C)$ \emph{has type $(\boldsymbol{\mathcal E},p)$} if there exist $k \in \N$ and a morphism $\varphi \colon E \to \mathbf E$, where $\mathbf E \in \boldsymbol{\mathcal E}$ such that the map $\varphi$ and the natural projection $E \onto E/C$ extend to an embedding from $E$ into $E/C \times \dihedral[p]^k \times \mathbf E$.
\end{defi}

\begin{rema*}	
	A reader only interested in \emph{free} Burnside groups can read the entire article by taking for $\boldsymbol{\mathcal E}$ the collection that consists only of the trivial group. 
	The exponent of this trivial model collection is $1$.
\end{rema*}

We now fix an integer $p \in \N$ and a model collection $\boldsymbol{\mathcal E}$ and write $\mu = \mu(\boldsymbol{\mathcal E})$ for its exponent.
Saying that $(E, C)$ has type $(\boldsymbol{\mathcal E},p)$ means that, up to a residual factor $\mathbf E \in \boldsymbol{\mathcal E}$, the group $E$ essentially embeds into a direct product of dihedral groups.
In particular, we can exploit the \emph{algebraic identities} of dihedral groups to recover information about $E$.
The next two statements give simple but essential examples of this idea.
Other applications will arise later in the article.

\begin{prop}
\label{res: dihedral complexity}
	Let $(E,C)$ be a dihedral pair with type $(\boldsymbol{\mathcal E}, p)$.
	Let $\mathcal C = (g_0,g_1, \dots, g_{\mu+2})$ be a chain of $G$ and $h$ a conjugating element of $\mathcal C$.
	If $g_0$ and $h$ belong to $E$, then $g_{\mu +2}$ belongs to $\group{g_0,g_1, \dots, g_{\mu+1}}$.
\end{prop}

\begin{proof}
	By assumption, there exist $k \in \N$, a group $\mathbf E \in \boldsymbol{\mathcal E}$, and a morphism $\varphi \colon E \to \mathbf E$, such that $\varphi$ extends to an embedding $E\into E/C  \times \dihedral[p]^k \times \mathbf E$.
	For every $i \in \intvald 03$ we let 
	\begin{equation*}
		u_i = g_ig_{i+1} \cdots g_{i+\mu-1}.
	\end{equation*}
	Note that it suffices to prove that $u_3u_2^{-1}u_0u_1^{-1} = 1$.
	To that end, we have to check that this identity holds in every factor of $E/C \times \dihedral[p]^k \times \mathbf E$.
	It was observed by Lysenok \cite[Proposition~15.10]{Lysenok:1996kw} that if $\mathbf x$ and $\mathbf y$ are two elements of $\dihedral$ then 
	\begin{equation*}
		\left(\mathbf y^3\mathbf x\mathbf y^{-3}\right)\left(\mathbf y^2\mathbf x^{-1}\mathbf y^{-2}\right)\mathbf x\left(\mathbf y\mathbf x^{-1}\mathbf y^{-1}\right)
		=\left[ \mathbf y^2, \left[\mathbf y,\mathbf x\right] \right]
		=1.
	\end{equation*}
	Hence the identity $u_3u_2^{-1}u_0u_1^{-1} = 1$ holds in $E/C$ as well as in any factor $\dihedral[p]$.
	On the other hand, we observe that $u_0 = (g_0h)^\mu h^{-\mu}$.
	By the very definition of the exponent $\mu$, the element $\varphi(u_0) \in \mathbf E$ is trivial and thus so are its conjugates $\varphi(u_1)$, $\varphi(u_2)$ and $\varphi(u_3)$.
	Hence the identity $u_3u_2^{-1}u_0u_1^{-1} = 1$ also holds in $\mathbf E$, and the proof is complete.
\end{proof}

Let $n \in \N$.
Let $\boldsymbol \Pi = \dihedral[p_1] \times \dots \times \dihedral[p_k] \times\mathbf E$ be a direct product of dihedral groups with some $\mathbf E \in \boldsymbol{\mathcal E}$ where $p_i$ divides $n$ for every $i \in \intvald 1k$.
The signature $\epsilon_i \colon \dihedral[p_i] \to \cyclic[2]$ induces a morphism $\boldsymbol \Pi \to (\cyclic[2])^k$, whose kernel is $\boldsymbol \Pi_+ = \cyclic[p_1] \times \dots \times \cyclic[p_k] \times \mathbf E$ is the \emph{pure rotation subgroup}.
Given a subgroup $\mathbf A$ of $\boldsymbol \Pi$, its \emph{reflection rank} is the dimension of the image of $\mathbf A$ in $\boldsymbol \Pi/\boldsymbol \Pi_+$ (seen as a $\cyclic[2]$-vector space).
The next lemma is a variation on Ivanov \cite[Lemma~16.2]{Ivanov:1994kh}.

\begin{lemm}
\label{res: norm in prod of dihedral}
	Let $\mathbf A$ be a subgroup of $\boldsymbol \Pi$ and $r$ its reflection rank.
	We assume that $2^{r + 3}\mu$ divides $n$.
	For every $\mathbf h \in \boldsymbol \Pi$, normalizing $\mathbf A$, there exists $\mathbf a \in \mathbf A$ with the following properties.
	\begin{enumerate}	
		\item \label{enu: norm in prod of dihedral - cent A}
		$\mathbf h^{n/4}\mathbf a^{-1}$ centralizes $\mathbf A$.
		\item \label{enu: norm in prod of dihedral - cent P}
		$[\mathbf h^{n/4}\mathbf a^{-1},\mathbf b]$ centralizes $\boldsymbol \Pi$, for every $\mathbf b \in \boldsymbol \Pi$.
	\end{enumerate}
\end{lemm}
\begin{proof}
	By assumption, there exist $\mathbf s_1, \dots, \mathbf s_r \in \mathbf A$ such that $\mathbf A$ is generated by $\mathbf s_1, \dots , \mathbf s_r$ and $\mathbf A\cap \mathbf \Pi_+$.
	One can choose for instance pre-images in $\mathbf A$ of elements forming a basis of $\mathbf A / \mathbf A \cap \boldsymbol \Pi_+$.
	We let 
	\begin{equation*}
		\mathbf a = \prod_{(\epsilon_1, \dots, \epsilon_r)\in \{0, 1\}^r} \left[\mathbf r, \mathbf s_1^{\epsilon_1}\dots \mathbf s_r^{\epsilon_r} \right],
		\quad \text{where} \quad \mathbf r = \mathbf h^{2^{-(r+2)}n}.
	\end{equation*}
	Since $\mathbf h$ normalizes $\mathbf A$, the element $\mathbf a$ belongs to $\mathbf A$.
	It is sufficient to check that in each factor of $\mathbf \Pi$ the image of $\mathbf h^{n/4}\mathbf a^{-1}$ satisfies the announced properties.
	Since the exponent of $\mathbf E$ divides $2^{-(r+2)}n$, the elements $\mathbf h^{n/4}$ and $\mathbf a$ are trivial in $\mathbf E$.
	Thus so is $\mathbf h^{n/4}\mathbf a^{-1}$.
	Hence \ref{enu: norm in prod of dihedral - cent A} and \ref{enu: norm in prod of dihedral - cent P} hold in $\mathbf E$.

	We now focus on the dihedral factors.
	Let $i \in \intvald 1k$.
	Assume first that the image of $\mathbf A$ in $\dihedral[p_i]$ is contained in the rotation group $\cyclic[p_i]$.
	Note that $2$ divides $2^{-(r+2)}n$.
	Hence (the image of) $\mathbf r$ lies in the rotation group of $\dihedral[p_i]$. 
	Thus $\mathbf a$ is trivial in $\dihedral[p_i]$, while $\mathbf h^{n/4}$ belongs to $\cyclic[p_i]$.
	Consequently (the image of) $\mathbf h^{n/4}\mathbf a^{-1}$ centralizes $\cyclic[p_i]$ and thus (the image of) $\mathbf A$.
	Moreover for every $\mathbf b \in \mathbf \Pi$, (the image of) $[\mathbf h^{n/4}\mathbf a^{-1},\mathbf b]$ which coincides with (the image of) $[\mathbf h^{n/4},\mathbf  b]$, centralizes $\dihedral[p_i]$.
	Assume now that the image of $\mathbf A$ in $\dihedral[p_i]$ contains a reflection.
	By construction every element of $\mathbf A\cap \mathbf \Pi_+$ is mapped to the rotation subgroup $\cyclic[p_i]$.
	Without loss of generality we can assume that $\mathbf s_1$ is mapped to a reflection of $\dihedral[p_i]$.
	Let $\Omega$ be the subset of all tuples $(\epsilon_1, \dots, \epsilon_r) \in \{0,1\}^r$ such that $\mathbf s_1^{\epsilon_1}\dots\mathbf s_r^{\epsilon_r}$ is mapped to a reflection in $\dihedral[p_i]$.
	As previously the image of $\mathbf r$ in $\dihedral[p_i]$ is a rotation.
	Hence, seen in $\dihedral[p_i]$, we have
	\begin{equation*}
		\left[\mathbf r, \mathbf s_1^{\epsilon_1}\dots \mathbf s_r^{\epsilon_r} \right] = 
		\left\{
			\begin{split}
				\mathbf r^2 & \quad\text{if}\ (\epsilon_1, \dots, \epsilon_r) \in \Omega \\
				1 & \quad \text{otherwise}
			\end{split}
		\right.
	\end{equation*}	
	Consequently we get in $\dihedral[p_i]$
	\begin{equation*}
		\mathbf a = \mathbf r^{2\card\Omega} = \mathbf h^{2^{-(r+1)}\card \Omega n}.
	\end{equation*}
	We claim that the cardinality of $\Omega$ is $2^{r-1}$.
	Recall that $s_1$ is mapped to a reflection in $\dihedral[p_i]$. 
	Thus $s_1s_2^{\epsilon_2}\dots s_r^{\epsilon_r}$ (\resp $s_2^{\epsilon_2}\dots s_r^{\epsilon_r}$) is mapped to a reflection in $\dihedral[p_i]$ is and only if $s_2^{\epsilon_2}\dots s_r^{\epsilon_r}$ (\resp $s_1s_2^{\epsilon_2}\dots s_r^{\epsilon_r}$) is mapped to a rotation in $\dihedral[p_i]$.
	Consequently the involution $\chi$ of $\{0,1\}^r$ swapping $(0,\epsilon_2, \dots, \epsilon_r)$ and $(1, \epsilon_2, \dots, \epsilon_r)$ induces an bijection from $\Omega$ onto $\{0,1\}^r \setminus \Omega$.
	Hence $2 \card \Omega = \card{\{0, 1\}^r}$ which completes the proof of our claim.
	It follows that $\mathbf h^{n/4}$ and $\mathbf a$ coincide in $\dihedral[p_i]$.
	Thus \ref{enu: norm in prod of dihedral - cent A} and \ref{enu: norm in prod of dihedral - cent P} hold in $\dihedral[p_i]$.
\end{proof}

%
\section{Small cancellation theory}
%
\label{sec: sc}

Let us recall the main strategy to study the free Burnside group $\burn rn$.
Starting from the free group $\free r$, we are going to build a sequence of non-elementary hyperbolic groups
\begin{equation*}
	\free r = G_0 \onto G_1 \onto G_2 \onto \dots \onto G_k \onto G_{k+1} \onto \dots 
\end{equation*}
whose directly limit is exactly $\burn rn$.
The group $G_{k+1}$ is obtained from $G_k$ by adjoining relations of the form $h^n = 1$ where $h$ runs over a subset of \og small \fg\ loxodromic elements of $G_k$.
The main difficulty is to make sure that $G_{k+1}$ remains a non-elementary hyperbolic group although the exponent $n$ has been fixed in advance.
In this article, we achieve this by using a \emph{geometric} approach of small cancellation theory \emph{à la} Delzant-Gromov \cite{Delzant:2008tu}.

In this section we focus on a \emph{single} step $G_k \onto G_{k+1}$.
We present first an overview of small cancellation theory and develop later the required additional material.
For proving infiniteness of Burnside groups we only need to consider relations of the form $h^n = 1$.
Nevertheless we start our study in a slightly more general setting as the intermediate results are of independent interest.

%
\subsection{General setting}
%
\label{sec: sc - general setting}

Let $X$ be a $\delta$-hyperbolic length space and $G$ a group acting gently by isometries on $X$.
Let $\mathcal Q$ be a collection of pairs $(H,Y)$ where $H$ is a subgroup of $G$ and $Y$ an $H$-invariant strongly quasi-convex subset of $X$.
We assume that $\mathcal Q$ is invariant under the action of $G$ defined by $g \cdot (H,Y) = (gHg^{-1},gY)$, for every $(H,Y) \in \mathcal Q$ and every $g \in G$.
We denote by $K$ the (normal) subgroup of $G$ generated by all $H$ where $(H,Y)$ runs over $\mathcal Q$.
The goal is to study the quotient $\bar G = G/K$.
To that end, we define two parameters $\Delta(\mathcal Q,X)$ and $\inj[X]{\mathcal Q}$ which play the role of the lengths of the largest piece and the smallest relation respectively.
\begin{eqnarray*}
	\Delta(\mathcal Q, X) &  = & \sup \set{\diam\left(Y_1^{+5\delta}\cap Y_2^{+ 5\delta}\right)}{(H_1,Y_1) \neq (H_2, Y_2) \in \mathcal Q}, \\
	\inj[X]{\mathcal Q} & = & \inf \set{ \norm h}{h \in H, (H,Y) \in \mathcal Q}.
\end{eqnarray*}

\begin{rema*}
	As explained above, we will later focus on a particular set of relations.
	More precisely the collection $\mathcal Q$ will be of the form 
	\begin{equation*}
		\mathcal Q = \set{\left(Y_h, \group{h^n}\right)}{h \in S}
	\end{equation*}
	where $n$ is a large integer and $S$ a subset of \og small\fg\ loxodromic elements of $G$, which is invariant under conjugation.
	Assuming that $\Delta(\mathcal Q, X)$ is finite will automatically imply that $\group{h^n}$ is normal in $\stab{Y_h}$, for every $h \in S$.
\end{rema*}

We now fix once and for all a number $\rho \in \R_+^*$.
Its value will be made precise later (see \autoref{res: small cancellation}).
It should be thought of as a very large distance.

\paragraph{Cones.}
Let $(H,Y) \in \mathcal Q$. 
The \emph{cone of radius $\rho$ over $Y$}, denoted by $Z_\rho(Y)$ or simply $Z(Y)$, is the quotient of $Y\times [0,\rho]$ by the equivalence relation that identifies all the points of the form $(y,0)$.
The equivalence class of $(y,0)$, denoted by $v$, is called the \emph{apex} or \emph{cone point} of $Z(Y)$. 
By abuse of notation, we still write $(y,r)$ for the equivalence class of $(y,r)$.
The map $\iota \colon Y \rightarrow Z(Y)$ that sends $y$ to $(y,\rho)$ provides a natural embedding from $Y$ to $Z(Y)$.
The \emph{radial projection} $p : Z(Y) \setminus\{v\} \rightarrow Y$ is the map sending $(y,r)$ to $y$.
We denote by $\distV[Y]$ the length metric on $Y$ induced by the restriction of $\distV$ to $Y$.
This cone $Z(Y)$ can be endowed with a metric as described below.
	
\begin{prop}{\rm \cite[Chapter I.5, Proposition 5.9]{Bridson:1999ky}} \quad
\label{res: def distance cone}
	The cone $Z(Y)$ is endowed with a metric characterized in the following way. 
	Let $x=(y,r)$ and $x'=(y',r')$ be two points of $Z(Y)$ then
	\begin{equation}
	\label{eqn: def metric cone}
		\cosh \dist[Z(Y)] x{x'} = \cosh r \cosh r' - \sinh r\sinh r' \cos \theta(y,y'),
	\end{equation}
	where $\theta(y,y')$ is the \emph{angle at the apex} defined by 
	\begin{equation*}
		\theta(y,y') = \min \left\{ \pi , \frac{\dist[Y]y{y'}}{\sinh \rho}\right\}.
	\end{equation*}
\end{prop}

The distance between two points $x=(y,r)$ and $x'=(y',r')$ of $Z(Y)$ has the following geometric interpretation.
Consider a geodesic triangle in the hyperbolic plane $\H_2$ such that the lengths of two sides are respectively $r$ and $r'$ and the angle between them is $\angle y{y'}$.
According to the law of cosines, $\dist x{x'}$ is exactly the length of the third side of the triangle (see \autoref{fig: distance cone}).
\begin{figure}[htbp]
	\centering
	\includegraphics[width=0.9\textwidth]{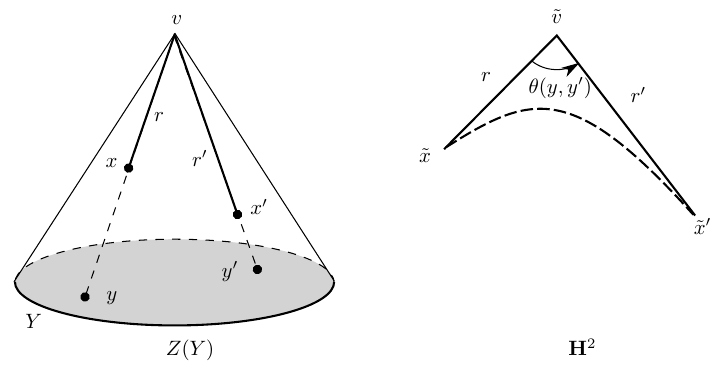}
	\caption{Geometric interpretation of the distance in the cone.}
	\label{fig: distance cone}
\end{figure}

\begin{exam}
	If $Y$ is a circle whose perimeter is $2\pi \sinh \rho$ and endowed with the length metric, then $Z(Y)$ is the closed hyperbolic disc of radius $\rho$.
	If $Y$ is the real line, then $Z(Y) \setminus\{v\}$ is the universal cover of the punctured hyperbolic disc of radius $\rho$.
\end{exam}

In order to compare the metric of $Y$ and $Z(Y)$, we use the map $\mu \colon \R_+ \to \R_+$ characterized as follows
\begin{equation*}
	\cosh \mu(t) = \cosh^2\rho - \sinh^2\rho \cos \left( \min \left\{ \pi, \frac t{\sinh \rho}\right\} \right), \quad \forall t \in \R_+,
\end{equation*}
so that for every $y,y' \in Y$ we have
\begin{equation*}
	\dist[Z(Y)]{\iota(y)}{\iota(y')} = \mu ( \dist[Y]y{y'}).
\end{equation*}
The next proposition summarizes the properties of $\mu$.

\begin{prop}
\label{res: map mu}
	The map $\mu$ is continuous, concave, and non-decreasing.
	In addition, for all $t \in \intval 0{\pi \sinh \rho}$, we have $t \leq \pi \sinh (\mu(t)/2)$.
\end{prop}

We complete this part with a useful tool to compare two cones.

\begin{lemm}
\label{res: qi between cones}
	Let $f \colon Y_1 \to Y_2$ a $(1,\ell)$ quasi-isometric embedding between two metric spaces.
	The map $Z(Y_1) \to Z(Y_2)$ sending $(y,r)$ to $(f(y),r)$ is again a $(1, \ell)$-quasi-isometric embedding.
\end{lemm}

\begin{proof}
	Let $x_1 = (y_1, r)$ and $x'_1 = (y'_1,r')$ be two points in $Z(Y_1)$.
	Let $y_2 = f(y_1)$, $y'_2 = f(y'_1)$, $x_2 = (y_2,r)$, and $x'_2 = (y'_2,r')$.
	In addition, set
	\begin{equation*}
		\theta_i = \min \left\{ \pi, \frac {\dist[Y_i]{y_i}{y'_i}}{\sinh \rho}\right\} 
		\quad \text{and} \quad 
		d_i = \dist[Z(Y_i)]{x_i}{x'_i}.
	\end{equation*}
 	Consider now the following figure in the regular hyperbolic plane $\mathbf H^2$.
	We fix a pair of points $o, \tilde x \in \mathbf H^2$ at distance $r$ from one another.
	Let $\tilde x'_1, \tilde x'_2 \in \mathbf H^2$ be the points at distance $r'$ from $o$ such that the algebraic angle from $\tilde x$ to $\tilde x'_i$ seen from $o$ is exactly $\theta_i$.
	It follows from the geometric interpretation of the metric in the cone, together with the triangle inequality in $\mathbf H^2$, that $\abs{d_1 - d_2} \leq \dist[\mathbf H^2]{\tilde x'_1}{\tilde x'_2}$.
	Note that $\tilde x'_1$ and $\tilde x'_2$ lie on the circle of radius $r'$ based at $o$.
	In particular, the distance between them is at most the length of the arc of circle joining them, i.e.\ $\abs{\theta_1 - \theta_2} \sinh r'$.
	Since $f$ is an $(1, \ell)$-quasi-isometric embedding, we have $\abs{\theta_1 - \theta_2} \leq \ell / \sinh \rho$.
	Consequently $\abs{d_1 - d_2} \leq \ell$.
\end{proof}

\begin{rema}
	An alternative construction for $Z(Y)$ would be to consider a variation of the combinatorial horoball defined by Groves and Maning \cite[Section~3.1]{Manning:2008aa}. 
	We chose to work with this definition of cone as the above geometric interpretation is very handy for the proofs.
	See for instance \autoref{res: qi between cones} or \autoref{res: local action of apex stab}
\end{rema}

\paragraph{The cone-off space $\dot X$.}
The \textit{cone-off of radius $\rho$ over $X$ relative to $\mathcal Q$} denoted by $\dot X_\rho(\mathcal Q)$ (or simply $\dot X$) is obtained by attaching for every $(H,Y) \in \mathcal Q$, the cone $Z_\rho(Y)$ on $X$ along $Y$ according to $\iota$.
The subset of $\dot X$ consisting of all apices of the cones is denoted by $\mathcal V$.
We endow $\dot X$ with the largest pseudo-metric $\distV[\dot X]$ for which all the maps $X \to \dot X$ and $Z(Y) \to \dot X$ -- where $(H,Y)$ runs over $\mathcal Q$ -- are $1$-Lipschitz.
It turns out that this pseudo-distance is a length metric on $\dot X$ \cite[Proposition 5.10]{Coulon:2014fr}.
The next lemmas detail the relationship between the metrics of $X$ and $\dot X$.

\begin{lemm}[{\cite[Lemma 5.8]{Coulon:2014fr}}]
\label{res: loose comparison metric X and dot X}
	For every $x,x' \in X$, we have
	\begin{equation*}
		\mu\left(\dist[X] x{x'}\right) \leq \dist[\dot X] x{x'} \leq \dist[X] x{x'}.
	\end{equation*}
\end{lemm}

\begin{lemm}[{\cite[Lemma 5.7]{Coulon:2014fr}}]
\label{res: metric on  dot X and Zi coincide}
	Let $(H,Y)\in \mathcal Q$.
	Let $x \in Z(Y)$.
	Let $d(x,Y)$ be the distance between $x$ and $\iota(Y)$ computed with $\distV[Z(Y)]$.
	For all $x' \in \dot X$, if $\dist[\dot X] x{x'} < d(x,Y)$ then $x'$ belongs to $Z(Y)$.
	Moreover $\dist[\dot X] x{x'} = \dist[Z(Y)] x{x'}$.
\end{lemm}

Let $v$ be the apex of $Z(Y)$.
It follows from the lemma that, as a set, the ball $B(v,\rho)$ (for the metric of $\dot X$) is nothing but $Z(Y)\setminus \iota(Y)$.
Moreover the metrics $\distV[\dot X]$ and $\distV[Z(Y)]$ coincide on $B(v, \rho/3)$.

\paragraph{The quotient space $\bar X$.}
The action of $G$ on $X$ naturally extends to an action by isometries on $\dot X$ as follows.
Let $(H,Y) \in \mathcal Q$.
For every $g \in G$, for every $x = (y,r)$ in $Z(Y)$, we define $gx$ to be the point of $Z(gY)$ given by $gx = (gy,r)$.
The space $\bar X$ is the quotient $\bar X = \dot X/K$.
The metric on $\dot X$ induces a pseudo-metric on $\bar X$.
We write $\zeta \colon \dot X \rightarrow \bar X$ for the canonical projection from $\dot X$ to $\bar X$.
The quotient $\bar G = G/K$ naturally acts by isometries on $\bar X$.
We denote by $\bar {\mathcal V}$ the image of $\mathcal V$ in $\bar X$.
For every $x \in \dot X$, we usually write $\bar x$ for its image in $\bar X$.

\paragraph{Small cancellation theorem.}
The next statement is a combination of Proposition~6.4, Proposition~6.7, Corollary~ 3.12 and Proposition~3.15 in \cite{Coulon:2014fr}.
See also \cite[Th\'eor\`emes~5.2.5 et 5.5.2]{Delzant:2008tu}.

\begin{theo}
\label{res: small cancellation}
	There exist $\delta_0, \delta_1, \Delta_0, \rho_0 \in \R_+^*$, which do not depend on $X$, $G$ or $\mathcal Q$, with the following property.
	Assume that $\rho \geq \rho_0$.
	If $\delta \leq \delta_0$, $\Delta(\mathcal Q,X) \leq \Delta_0$ and $\inj[X]{\mathcal Q} \geq 10\pi \sinh \rho$, then the following holds
	\begin{enumerate}
		\item \label{enu: small cancellation - hyp cone-off}
		The cone-off space $\dot X$ is $\dot \delta$-hyperbolic with $\dot \delta \leq \delta_1$.
		\item \label{enu: small cancellation - hyp}
		The quotient space $\bar X$ is $\bar \delta$-hyperbolic with $\bar \delta \leq \delta_1$.
		\item \label{enu: small cancellation - local embedding}
		Let $(H,Y) \in \mathcal Q$.
		Let $\bar v$ be the image in $\bar X$ of the apex $v$ of $Z(Y)$.
		The subgroup $\stab{\bar v} \subset \bar G$ is isomorphic to the quotient $\stab Y/H$.
		Moreover the projection $\zeta \colon \dot X \to \bar X$ induces an isometry from $B(v, \rho/2)/H$ onto $B(\bar v, \rho/2)$.
		\item \label{enu: small cancellation - local isom}
		For every $r \in (0, \rho/20]$, for every $x \in \dot X$, if $d(x, \mathcal V) \geq 2r$, then the projection $\zeta \colon \dot X \to \bar X$ induces an isometry from $B(x,r)$ onto $B(\bar x, r)$.
		\item \label{enu: small cancellation - translation kernel}
		For every $x \in \dot X$ for every $g \in K \setminus\{1\}$, we have $\dist[\dot X]{gx}x \geq \min \{2r, \rho/5\}$, where $r = d(x, \mathcal V)$.
		In particular, $K$ acts freely on $\dot X \setminus \mathcal V$.
		Moreover, the projection $\zeta \colon \dot X \to \bar X$ induces a covering map $\dot X \setminus \mathcal V \to \bar X \setminus \bar{\mathcal V}$.
	\end{enumerate}
\end{theo}

\begin{rema*}
	Note that the constants $\delta_0$ and $\Delta_0$ (\resp $\rho_0$) can be chosen arbitrarily small (\resp large).
	From now on, we will always assume that $\rho_0 > 10^{20} \delta_1$ whereas $\delta_0, \Delta_0 < 10^{-10}\delta_1$.
	These estimates are absolutely not optimal.
	We chose them very generously to ensure that all the inequalities which we might need later will be satisfied.
	What really matters is their orders of magnitude recalled below.
	\begin{equation*}
		\max\left\{\delta_0, \Delta_0\right\} \ll \delta_1  \ll \rho \ll \pi \sinh \rho.
	\end{equation*}
	An other important point to remember is the following.
	The constants $\delta_0$, $\Delta_0$ and $\pi \sinh \rho$ are used to describe the geometry of $X$ whereas $\delta_1$ and $\rho$ refers to the one of $\dot X$ or $\bar X$.
	From now on and until the end of \autoref{sec: sc} we assume that $X$, $G$ and $\mathcal Q$ are as in \autoref{res: small cancellation}.
	In particular, $\dot X$ and $\bar X$ are respectively $\dot\delta$- and $\bar \delta$-hyperbolic.
	Up to increasing one constant or the other, we can actually assume that $\dot \delta = \bar \delta$.
	Nevertheless we still keep two distinct notations, to remember which space we are working in.
\end{rema*}

\begin{nota*}
	In this section we work with three metric spaces namely $X$, its cone-off $\dot X$ and the quotient $\bar X$.
	Since the map $X \into \dot X$ is an embedding we use the same letter $x$ to designate a point of $X$ and its image in $\dot X$.
	We write $\bar x$ for its image in $\bar X$.
	Unless stated otherwise, we keep the notation $\distV$ (without mentioning the space) for the distances in $X$ or $\bar X$.
	The metric on $\dot X$ will be denoted by $\distV[\dot X]$.
\end{nota*}

%
\subsection{A few additional facts regarding the cone-off space}
%
\label{sec: sc - cone-off}

\paragraph{Radial projection.}
The \emph{radial projection} $p \colon \dot X \setminus \mathcal V \to X$ is defined as follows.
Its restriction to $X$ is the identity.
Given any $(H,Y) \in \mathcal Q$, the restriction of $p$ to $Z(Y) \setminus\{v\}$, where $v$ stands for the apex of $Z(Y)$, coincides with the radial projection defined in the previous paragraph.
This map is $G$-equivariant.
Observe that $\dist[\dot X] x{p(x)} \leq \rho$, for every $x \in \dot X\setminus \mathcal V$.

\begin{prop}
\label{res: radial proj qi - prelim}
	Let $x,x' \in X$ such that  $\gro x{x'}v > 0$, for every $v \in \mathcal V$ (here the Gromov product is computed in $\dot X$).
	Then
	\begin{equation*}
		\dist[\dot X] x{x'} \leq \dist[X]x{x'} \leq \frac{\pi \sinh \rho}{2\rho} \dist[\dot X]x{x'}.
	\end{equation*}
\end{prop}

\begin{proof}
	In this proof all the Gromov products are computed in $\dot X$.
	The first inequality directly follows from the fact that the embedding $X \to \dot X$ is $1$-Lipschitz.
	Let us focus on the second inequality.
	Let $\eta > 0$ and $\gamma \colon \intval ab \to \dot X$ be a $(1, \eta)$-quasi-geodesic from $x$ to $x'$.
	According to our assumption, up to decreasing $\eta$ we can assume that for every $(H,Y) \in \mathcal Q$, the diameter of $\gamma \cap Z(Y)$ is less than $2\rho$.	
	Consequently there exists a partition $t_0 = a \leq t_1 \leq \dots \leq t_m = b$ of $\intval ab$ such that
	\begin{enumerate}
		\item $\gamma(t_i)$ belongs to $X$ for every $i \in \intvald 0m$;
		\item $\dist[\dot X]{\gamma(t_{i+1})}{\gamma(t_i)} < 2\rho$, for every $i \in \intvald 0{m-1}$.
	\end{enumerate}
	\autoref{res: loose comparison metric X and dot X} combined with the concavity of the map $\mu$ tells us that 
	\begin{align*}
		\frac{2\rho}{\pi \sinh \rho}\dist x{x'}
		\leq \frac{2\rho}{\pi \sinh \rho} \sum_{i = 0}^{m-1} \dist{\gamma(t_{i+1})}{\gamma(t_i)}
		& \leq \sum_{i = 0}^{m-1} \mu\left(\dist{\gamma(t_{i+1})}{\gamma(t_i)}\fantomB\right) \\
		& \leq \sum_{i = 0}^{m-1} \dist[\dot X]{\gamma(t_{i+1})}{\gamma(t_i)} \\
		& \leq \dist[\dot X] x{x'} + \eta.
	\end{align*}
	This inequality holds for every sufficiently small $\eta >0$, hence the result.
\end{proof}

\begin{coro}
\label{res: radial proj qi}
	Let $Z$ be a subset of $\dot X$ such that $\gro z{z'}v > 2\dot \delta$ for every $z,z' \in Z$ and $v \in \mathcal V$ (here the Gromov product is computed in $\dot X$).
	Then the radial projection $p \colon \dot X \setminus \mathcal V \to X$ restricted to $Z$ is a quasi-isometric embedding.
\end{coro}

\begin{proof}
	In this proof all the Gromov products are computed in $\dot X$.
	Let $z,z' \in Z$.
	Let $y,y' \in X$ be the radial projections of $z$ and $z'$ respectively.
	It follows from the triangle inequality that $\dist[\dot X] z{z'}$ and $\dist[\dot X] y{y'}$ differ by at most $2\rho$.
	In view of \autoref{res: radial proj qi - prelim} it is sufficient to prove that $\gro y{y'}v > 0$ for every $v \in \mathcal V$.
	The four point inequality (\ref{eqn : hyp four points - 1}) applied in $\dot X$ gives
	\begin{equation*}
		\gro y{y'}v \geq \min \left\{ \fantomB\gro yzv, \gro z{z'}v , \gro {z'}{y'}v\right\} - 2\dot\delta.
	\end{equation*}
	Assume that $\gro yzv \leq 2 \dot \delta$.
	Then $z$ necessarily belongs to the cone $Z(Y)$ for some $(H,Y) \in \mathcal Q$.
	Indeed otherwise $z= y$ is a point of $X$, and thus $\gro yzv \geq 2\rho$.
	It follows the from the definition of the radial projection and \autoref{res: metric on  dot X and Zi coincide} that $z$ lies on a geodesic between $y$ and the apex of $Z(Y)$.
	As the distance between two apices is at least $2\rho$, the point $v$ is necessarily the apex of $Z(Y)$.
	Hence 
	\begin{equation*}
		\gro zzv = \dist[\dot X] vz = \gro yzv
	\end{equation*}
	is bounded above by $2\dot \delta$, which contradicts our assumption.
	We prove in the same way that $\gro {y'}{z'}v > 2\dot \delta$.
	On the other hand, according to our assumption we have $\gro z{z'}v > 2\dot \delta$.
	Thus $\gro y{y'}v > 0$.
\end{proof}

\paragraph{Parabolic subgroups.}

\begin{lemm}
\label{res: parabolic cone-off}
	Let $P$ be a subgroup of $G$.
	If $P$ is parabolic for its action on $\dot X$, then so is its action on $X$.
\end{lemm}

\begin{proof}
	Since the embedding $X \to \dot X$ is $1$-Lipschitz, $P$ cannot be elliptic for its action on $X$.
	Hence it suffices to prove that $P$ does not contain any loxodromic element (for its action on $X$).
	We denote by $\xi$ the unique point of $\Lambda(P) \subset \partial \dot X$.
	Let $L > 100\dot\delta$ and $\gamma : \R_+ \rightarrow \dot X$ be an $L$-local $(1,11\dot\delta)$-quasi-geodesic ray whose endpoint at infinity is $\xi$.
	Let $g \in P$. 
	By \autoref{res: isom fixing xi moving geo}, there is $t_0 \in \R_+$ such that for every $t \geq t_0$, we have $\dist[\dot X]{g\gamma(t)}{\gamma(t)} < 2\rho$.
	Since $\gamma$ is infinite, there exists $t \geq t_0$ such that $\gamma(t)$ belongs to $X$.
	It follows then from \autoref{res: loose comparison metric X and dot X} that 
	\begin{equation*}
		\mu \left( \dist{g\gamma(t)}{\gamma(t)}\right) \leq \dist[\dot X]{g \gamma(t)}{\gamma(t)} < 2\rho.
	\end{equation*}
	Thus $\dist{g\gamma(t)}{\gamma(t)} \leq \pi \sinh\rho$ (\autoref{res: map mu}).
	Consequently $\norm g \leq  \pi \sinh\rho$, for every $g \in P$.
	In particular, $P$ does not contain any loxodromic element for its action on $X$.
\end{proof}

%
\subsection{Apex stabilizer in the quotient space.}
%
\label{sec: sc - apex stab}

As we mentioned in the introduction the quotient space $\bar M = \bar X / \bar G$ can be seen as an orbifold, whose fundamental group is $\bar G$ \cite{Delzant:2008tu}.
Although this is not the point of view we adopted here, it is a great source of inspiration. 
According to \autoref{res: small cancellation}\ref{enu: small cancellation - local embedding}, for every $(H,Y) \in \mathcal Q$, the quotient $\stab Y/H$ embeds in $\bar G$, which basically means that $\bar M$ is developable, so that its universal cover is $\bar X$.
This orbifold $\bar M$ also comes with an analog of Margulis' thin/thick decomposition for hyperbolic manifolds.
The thin part corresponds to the neighborhood of the cone points (or more precisely their images in $\bar M$).
In particular, if $\bar x$ is point in a ball $B(\bar v, r)$ centered at a cone point $\bar v \in \bar{\mathcal V}$ and $\bar S$ a subset of $\bar G$ moving $\bar x$ by at most $2(\rho - r)$,
then the triangular inequality tells us that every element in $\bar S$ fixes $\bar v$, hence $\bar S$ generates an elliptic subgroup of $\bar G$.

In this section we study the structure of $\bar X$ around the apices.
In particular, we prove that the isotropy group of such a point locally acts as a dihedral group on a hyperbolic disc.
To that end, we make the following assumption (we refer to  \autoref{sec: hyp - gp action} for the definitions).
\begin{assu}[Cyclic relations]
\label{ass: even exponent}
	For every $(H,Y) \in \mathcal Q$, the group $H$ is loxodromic and $Y$ is its cylinder.
\end{assu}

Let $(H,Y) \in \mathcal Q$.
According to our small cancellation assumption, any non-trivial element in $H$ has a very large translation length.
Thus $H$ is necessarily a cyclic group generated by a loxodromic element.

\paragraph{Local classification of isometries.}
Let $(H,Y) \in \mathcal Q$.
Note that $\stab Y$ is the maximal loxodromic subgroup of $G$ containing $H$.
We write $\operatorname{Stab}^+(Y)$ for the subgroup of $\stab Y$ fixing pointwise $\partial Y$.
Its index in $\stab Y$ is at most $2$.
Since the action of $G$ on $X$ is gentle, the set $F$ of all elliptic elements of $\operatorname{Stab}^+(Y)$ is a normal subgroup of $\stab Y$.
Moreover $\operatorname{Stab}^+(Y)/F$ is isomorphic to $\cyclic$ while $\stab Y/F$ embeds in $\dihedral$.
In other words we have a short exact sequence
\begin{equation*}
	1 \to F \to \stab Y \xrightarrow q \mathbf L \to 1,
\end{equation*}
where $\mathbf L$ is either $\cyclic$ or $\dihedral$.

Denote by $\bar F$ the image of $F$ in $\bar G$.
Since $H$ is generated by a loxodromic element, its image in $\mathbf L$ is $n \cyclic$ for some $n \in \N\setminus\{0\}$.
We write $\mathbf L_n$ for $\mathbf L/n\cyclic$, i.e.\ $\mathbf L_n = \dihedral[n]$ if $\mathbf L = \dihedral$ and $\mathbf L_n = \cyclic[n]$, if $\mathbf L = \cyclic$.
Let $v$ be the apex of $Z(Y)$ and $\bar v$ its image in $\bar X$.
Recall that, according to the small cancellation theorem (\autoref{res: small cancellation}) the subgroup $\stab{\bar v}$ is isomorphic to $\stab Y/H$.
After taking the quotient by $H$ we get the following commutative diagram
\begin{equation*}
	\begin{tikzpicture}
		\matrix (m) [matrix of math nodes, row sep=2em, column sep=2.5em, text height=1.5ex, text depth=0.25ex] 
		{ 
			1	& F & \stab Y & \mathbf L & 1	\\
			1	& \bar F & \stab{\bar v} & \mathbf L_n & 1	\\
		}; 
		\draw[>=stealth, ->] (m-1-1) -- (m-1-2);
		\draw[>=stealth, ->] (m-1-2) -- (m-1-3);
		\draw[>=stealth, ->] (m-1-3) -- (m-1-4);
		\draw[>=stealth, ->] (m-1-4) -- (m-1-5);
		
		\draw[>=stealth, ->] (m-2-1) -- (m-2-2);
		\draw[>=stealth, ->] (m-2-2) -- (m-2-3);
		\draw[>=stealth, ->] (m-2-3) -- (m-2-4);
		\draw[>=stealth, ->] (m-2-4) -- (m-2-5);
		
		\path[->] (m-1-2) edge node[above,sloped,inner sep=0.5pt]{$\sim$} (m-2-2);

		\draw[>=stealth, ->] (m-1-3) -- (m-2-3) node[pos=0.5, right]{$\pi$};
		\draw[>=stealth, ->] (m-1-4) -- (m-2-4);
	\end{tikzpicture} 
\end{equation*}
where the horizontal lines are short exact sequences.
Note that $\stab{\bar v} \to \mathbf L_n$ is a well-defined map.
Indeed if $(H',Y')$ is another pair of $\mathcal Q$ such that $\bar v$ is this image of the apex $v'$ of $Z(Y')$, then there exists an element $u \in K$ such that $(H',Y') = (uHu^{-1},uY)$.
Thus the maps $q \colon \stab Y \to \mathbf L$ and $q' \colon \stab{Y'} \to \mathbf L$ differ at the source by the conjugation by $u$.

By analogy with singularities, the integer $n$ is called the \emph{order} of the cone point $\bar v$.
As we explained before $\mathbf L_n$ can be either $\cyclic[n]$ or $\dihedral[n]$.
In any case it embeds in $\dihedral[n]$.
We call the map $q_{\bar v} \colon \stab{\bar v} \to \dihedral[n]$ obtained in this way the \emph{geometric realization of $\stab{\bar v}$}.
Although it is not made explicit in the notation, we allow for the moment the order to be different from one apex to the other.
Recall that the elements of $\dihedral[n]$ are called \emph{rotations} or \emph{reflections} according to their action on the regular $n$-gon (see \autoref{sec: invariants - mixed}).
This allows us to define similar notions for the elements of $\stab{\bar v}$.
More precisely, we say that an element $\bar g \in \stab{\bar v}$ is a \emph{rotation} (\resp a \emph{reflection}, \emph{locally trivial}) \emph{at $\bar v$} if its image under $q_{\bar v}$ is a rotation (\resp a reflection, trivial).
A rotation at $\bar v$ is \emph{strict} if it does not belong to $\bar F$.
A \emph{central half-turn at $\bar v$} is a strict rotation at $\bar v$ which is an involution and centralizes $\stab{\bar v}$ (note that the existence of such a half-turn forces $n$ to be even).

Given a reflection $\mathbf x \in \mathbf L_n$, the pre-image under $q_{\bar v}$ of $\group{\mathbf x}$ is called a \emph{reflection group at $\bar v$}.
Similarly the kernel of $q_{\bar v}$ is called the \emph{local kernel at $\bar v$}.

\begin{rema*}
	Being a reflection at $\bar v$ is a \emph{local property}.
	Given two distinct apices $\bar v, \bar v' \in \bar{\mathcal V}$, an element $\bar g \in \bar G$ can be simultaneously a reflection at $\bar v$ and locally trivial at $\bar v'$.
	For instance, consider the hyperbolic group 
	\begin{equation*}
		G = \left< a,b,c \mid a^2, b^2, [b,c] \right> = \cyclic[2]\ast\left(\cyclic[2]\times \cyclic \right)
	\end{equation*}
	where the left factor $\cyclic[2]$ is generated by $a$, whereas the right factor $\cyclic[2]\times \cyclic$ is generated by $b$ and $c$.
	We consider the action of $G$ on the Bass-Serre tree $X$ associated to the given graph of group decomposition of $G$.
	\begin{center}
		\begin{tikzpicture}
			\node (A) at (0,0) {$\bullet$};
			\node (B) at (2,0) {$\bullet$};
			\draw (A) node[above]{$\group a$};
			\draw (B) node[above]{$\group b$};
			\draw[thick] (A) -- (B) node[pos=0.5, below]{$\{1\}$};
			\draw[thick] (B) .. controls +(right:1.5cm) and +(down:1.5cm) ..   (B) node[pos=0.5, below right]{$\group b$};
		\end{tikzpicture}
	\end{center}
	where $c$ is the stable letter associated to the loop.
	Fix now a large integer $n$ and define 
	\begin{equation*}
		\bar G = G / \normal{(ab)^n, c^n}.
	\end{equation*}
	One checks easily that $\bar G$ is a small cancellation quotient of $G$.
	Let $\bar v,\bar v' \in \bar X$ be the apices of the cones attached to the relations $(ab)^n$ and $c^n$ respectively.
	One observes that the image $\bar b$ of $b$ in $\bar G$ is a reflection at $\bar v$ but locally trivial at $\bar v'$.
	This subtlety is a source of difficulty when studying the strong $\nu$-invariant $\nu_{\rm{stg}}(\bar G, \bar X)$.
\end{rema*}

From now on, we make the following assumption.

\begin{assu}[Central half-turn]
\label{ass: even exponent}
	For every apex $\bar v \in \bar{\mathcal V}$, if the image of the geometric realization map $q_{\bar v} \colon \stab{\bar v} \to \dihedral[n]$ has some even torsion (i.e.\ contains an element with even order), then $\stab{\bar v}$ contains central half-turn at $\bar v$.
\end{assu}

\begin{rema*}
	Let us explain quickly how such an assumption can be satisfied.
	Later, when building the approximation sequence of $\burn rn$, we will see that every loxodromic subgroup of $G$ can be assumed to embed in a product of the form $\dihedral \times \dihedral[n] \times \dots \times \dihedral[n]$.
	In particular, if $g \in \stab Y$ is a primitive element of $E = \stab Y$, then $g^{n/2}$ is almost central: it commutes with every element in $E^+$ and anti-commutes with the ones of $E\setminus E^+$ (i.e.\ $ug^{n/2}u^{-1} = g^{-n/2}$, for every $u \in E \setminus E^+$).
	Consequently, if $H$ is the subgroup generated by $h = g^n$, then the image of $g^{n/2}$ in $\stab Y/H$ is a central half-turn.
\end{rema*}

\paragraph{Geometric realization.}
As suggested by the above terminology, the projection $q_{\bar v} \colon \stab {\bar v} \to \dihedral[n]$ captures how $\stab{\bar v}$ acts geometrically on the ball $B(\bar v, \rho)$.
To make this idea more precise, we are going to build a quasi-isometry between $B(\bar v, \rho)$ and a comparison hyperbolic cone $\mathcal D$ (endowed with the obvious action of $\dihedral[n])$ which is almost $q_{\bar v}$-equivariant.

We first define a morphism $\mathbf L\to \isom\R$.
Let $\xi$ be one of the endpoints at infinity of $Y$.
Let $h_0$ be a primitive element of $\stab Y$ whose attractive point is $\xi$.
\begin{itemize}
	\item If $\mathbf L = \cyclic$, then we map the positive generator $\mathbf t$ of $\mathbf L$ to the translation by $\snorm{h_0}$.
	\item If $\mathbf L$ is the dihedral group $\dihedral = \left< \mathbf  x, \mathbf  y\mid \mathbf x^2,\mathbf y^2 \right>$, then we map $\mathbf x$ to the symmetry at $0$ and $\mathbf  y$ to the symmetry at $\snorm{h_0}/2$.
	In particular, $\mathbf t = \mathbf x \mathbf y$ is mapped to the translation by $\snorm{h_0}$.
\end{itemize}
Note that the resulting morphism $\mathbf L \to \isom \R$ does not depend on the choice of $h_0$ (any two primitive elements in $\stab Y$ have the same stable translation length).
We write $\mathcal C$ for the quotient of $\R$ by the image of $n\cyclic$ in $\isom\R$.
It is a circle whose perimeter is $\ell = n \snorm{h_0}$.
We denote by $\mathcal D$ the cone of radius $\rho$ over $\mathcal C$ and write $o$ for its apex.
The action of $\mathbf L$ on $\R$ induces an action by isometries of $\mathbf L_n$ on $\mathcal D$ which fixes $o$.
Observe that the space $\mathcal D$ is a hyperbolic cone (i.e.\ with constant sectional curvature equal to $-1$ everywhere except maybe at the apex) whose total angle at the apex $o$ is
\begin{equation*}
	\Omega = \frac{n \snorm{h_0}}{\sinh \rho}.
\end{equation*}
It follows from the small cancellation assumption that $\Omega > 10\pi$.
Said differently $\mathcal D$ can be decomposed into $n$ copies of a sector of the hyperbolic disc of radius $\rho$ whose angle is $\snorm{h_0}/\sinh\rho$, so that $\dihedral[n]$ is the group of isometries of $\mathcal D$ preserving this decomposition (see \autoref{fig: comparison cone}).
Let us now compare the hyperbolic cone $\mathcal D$ to the ball $B(\bar v, \rho)$.
\begin{figure}[htbp]
	\centering
	\includegraphics[page=3, width=0.5\textwidth]{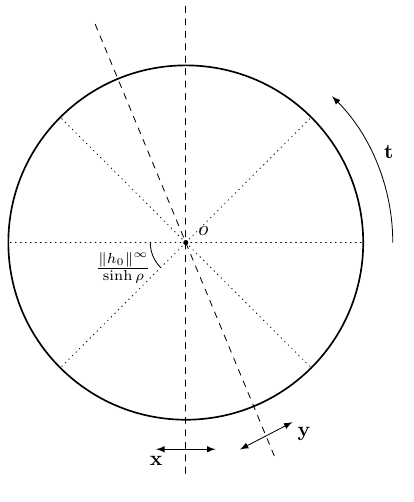}
	\caption{The comparison $\mathcal D$ cone for $n=8$.}
	\label{fig: comparison cone}
\end{figure}

Recall that $H$ is generated by a loxodromic element, say $h_1$.
Up to replacing $h_1$ by its inverse, we can assume that $\xi$ (which is the attractive point of $h_0$) is also the attractive point of $h_1$.
Recall that $h_0$ is a primitive element in $\stab Y$.
Hence there are $m _0\in \N\setminus\{0\}$ and $u_0 \in F$ such that $h_1 = h_0^{m_0} u_0$.

\begin{lemm}
\label{res: invariant Busemann cocycle}
	There is an $H$-invariant, almost $1$-Lipschitz cocycle $c \colon X \times X \to \R$ such that every Busemann cocycle at $\xi$ differs from $c$ by at most $6\delta$.
	Moreover $c(x,h_1^nx) = n\snorm {h_1}$, for every $n \in \Z$, and $x \in X$.
\end{lemm}

\begin{proof}
	The group $H = \group{h_1}$ is cyclic hence amenable.
	Therefore it admits an invariant mean, i.e.\ a positive bounded linear functional $M \colon \ell^\infty(H) \to \R$ such that $M(\mathbf 1) = 1$.
	The positivity together with this last assumption implies that $M$ has norm one.
	
	Fix a Busemann cocycle $c_\xi$ at $\xi$.
	Let $x,y \in X$.
	Define a map $f_{x,y} \colon H \to \R$ sending $h$ to $hc_\xi(x,y) = c_\xi(h^{-1}x,h^{-1}y)$.
	Recall that Busemann cocycles are almost $1$-Lipschitz, see (\ref{eqn: busemann cocycle lipschitz}).
	Therefore $f_{x,y}$ is bounded above by $\dist xy + 2\delta$.
	It allows us to set $c(x,y) = M(f_{x,y})$.
	Since $M$ has norm one, we get $\abs{c(x,y)} \leq \dist xy + 2\delta$.
	Observe that given $x,y,z \in X$, we have $f_{x,y} + f_{y, z} = f_{x,z}$.
	It follows that $c$ is a cocycle in the sense that 
	\begin{equation*}
		c(x,y) + c(y,z) = c(x,z), \quad \forall x,y,z \in X.
	\end{equation*}
	Let $h \in H$ and $x,y \in X$.
	Observe that $f_{hx,hy} = f_{x,y} \circ h^{-1}$.
	As $M$ is $H$-invariant $f_{hx,hy}$ and $f_{x,y}$ have the same mean, that is $c(hx,hy) = c(x,y)$.
	Hence the map $c \colon X \times X \to \R$ we have define isn an $H$-invariant, almost $1$-Lipschitz cocycle.

	Consider now another Busemann cocycle $c'_\xi$ at $\xi$.
	Recall that $h$ fixes $\xi$.
	Hence, $hc_\xi$ is also a Busemann cocycle at $\xi$.
	Therefore it differs from $c'_\xi$ by at most $6\delta$.
	In other words the uniform norm of the function $f_{x,y} - c'_\xi(x,y) \mathbf 1$ is at most $6\delta$.
	Consequently $\abs{c(x,y) - c'_\xi(x,y)} \leq 6 \delta$.
	Thus $c$ differs by at most $6\delta$ from any Busemann cocycle at $\xi$.
		
	Since $c$ is an $H$-invariant cocycle, the map $H \mapsto \R$ sending $h$ to $c(x,hx)$ is a homomorphism that does not depend on $x$.
	Indeed for every $h,h' \in H$, for every $x,y \in X$ the cocycle relation yields
	\begin{equation*}
		c(x,hh'x) = c(x,hx) + c(hx, hh'x) = c(x,hx) + c(x,h'x),
	\end{equation*}
	and 
	\begin{align*}
		c(x,hx) 
		= c(x,y) + c(y,hy) + c(hy,hx)
		& = c(x,y) + c(y,hy) - c(x,y) \\
		& = c(y,hy).
	\end{align*}
	Thus we are left to prove that $c(x,h_1x) = \snorm {h_1}$ for some (hence any) $x \in X$.
	Let $n \in \N$.
	Since $h \mapsto c(x,hx)$ is a homomorphism, we have 
	\begin{equation*}
		c(x,h_1^nx) = nc(x,h_1x).
	\end{equation*}
	Recall that $c_\xi$ is a Busemann cocycle at $\xi$.
	It follows from the properties of $c$ that 
	\begin{equation*}
		n\abs{c(x,h_1x) - \snorm{h_1}}
		\leq \abs{c(x,h_1^nx) - \snorm{h_1^n}}
		\leq  \abs{c_\xi(x,h_1^nx) - \snorm {h_1^n}} + 6 \delta.
	\end{equation*}
	Recall that $\xi$ is the attractive point of $h_1$.
	\autoref{res: translation length vs cocycle} applied with the cocycle $c_\xi$ and the element $g = h_1$ yields
	\begin{equation*}
		n\abs{c(x,h_1x) - \snorm {h_1}} \leq 12\delta, \quad \forall n \in \N,
	\end{equation*}
	whence the result.
\end{proof}

%

We now fix an arbitrary base point $y_0 \in Y$.
If $\mathbf L$ is the infinite dihedral group, we choose $y_0$ in $\fix{A,15\delta}$ where $A \subset \stab Y$ is the pre-image of $\group{\mathbf x}$ by $q$.
Such a point always exists by \autoref{res: fix set elliptic}.

\begin{lemm}
\label{res: comparing action axis vs R}
	The map  $\varphi \colon Y\to \R$ sending $y$ to $c(y_0,y)$ is an $H$-equivariant $(1, 150\delta)$-quasi-isometric embedding.
	Moreover, it is almost $q$-equivariant, in the sense that for every $y \in Y$, for every $g \in \stab Y$, we have 
	\begin{equation}
	\label{eqn: comparing action axis vs R - statement}
		\dist{\varphi(gy)}{q(g)\varphi(y)} \leq 200\delta.
	\end{equation}
\end{lemm}

\begin{proof}
	We first prove that $\varphi$ is a quasi-isometric embedding.
	Let $L > 20\delta$.
	Let $\gamma \colon \R \to X$ be an $L$-local $(1, \delta)$-quasi-geodesic joining the repulsive point of $h_0$ to its attractive point, i.e.\ $\xi$.
	It follows from \autoref{res: stability qg} that any point of $Y$ is $27\delta$-close to $\gamma$.
	Consider $y,y' \in Y$.
	Denote by $z = \gamma(t)$ and $z' = \gamma(t')$ their respective projections.
	Observe that $\dist{\varphi(y)}{\varphi(y')} = \abs{c(y,y')}$.
	Since $c$ is almost $1$-Lipschitz, we get
	\begin{align*}
		\abs{ \dist y{y'} - \abs{c(y,y')}} 
		& \leq \abs{ \dist z{z'} - \abs{c(z,z')}} + 2 \dist yz + 2\dist{y'}{z'} + 4\delta\\
		& \leq \abs{ \dist z{z'} - \abs{c(z,z')}}  + 112\delta.
	\end{align*}
	Recall that the point $z$ and $z'$ lies on the quasi-geodesic $\gamma$.
	It follows from \autoref{res: computing cocycle from qg} that
	\begin{equation*}
		\abs{ \dist y{y'} - \dist {\varphi(y)}{\varphi(y')}} \leq 121\delta.
	\end{equation*}
	Hence $\varphi$ is a $(1, 121\delta)$-quasi-isometric embedding.

	We now prove that $\varphi$ is $H$-equivariant.
	Let $y \in Y$ and $n \in \Z$.
	Using \autoref{res: invariant Busemann cocycle}, we get
	\begin{equation}
	\label{eqn: comparing action axis vs R}
		\varphi(h_1^ny) = c(y_0, h_1^ny) = c(y_0, y) + c(y,h_1^ny) = \varphi(y) + n \snorm{h_1}.
	\end{equation}
	We observed previously that $h_1$ can be written $h_1 = h_0^{m_0}u_0$ with $m_0 \in \N \setminus \{0\}$ and $u_0 \in F$. 
	In particular, $\snorm{h_1} = m \snorm {h_0}$.
	Recall that the action of $\mathbf L$ on $\R$ is defined in such a way that $q(h_0)$ acts as the translation by $\snorm{h_0}$, while $q(F)$ fixes $\R$ pointwise.
	Consequently, $q(h_1)$ acts on $\R$ as the translation by $\snorm{h_1}$.
	Hence (\ref{eqn: comparing action axis vs R}) exactly says that $\varphi$ is $H$-equivariant.

	Let us now prove that $\varphi$ is almost $q$-equivariant.
	We first focus on the subgroup $\operatorname{Stab}^+(Y) = q^{-1}(\cyclic)$, i.e.\ on the elements of the form $g = h_0^m u$ for some $m \in \Z$ and $u \in F$.
	The proof is very similar the previous one.
	Given $y \in Y$, we have
	\begin{equation}
	\label{eqn: comparing action axis vs R - almost equiv}
		\varphi(gy) = c(y_0, gy) = c(y_0, y) + c(y,gy) = \varphi(y) + c(y,gy).
	\end{equation}
	Denote by $c_\xi$ a Busemann cocycle at $\xi$.
	On the one hand, $c$ and $c_\xi$ differs by at most $6\delta$ (\autoref{res: invariant Busemann cocycle}).
	On the other hand, $c_\xi(y,gy)$ and $m \snorm{h_0}$ differs by at most $6\delta$ (\autoref{res: translation length vs cocycle}).
	By definition $q(g)$ acts on $\R$ as the translation by $m \snorm{h_0}$.
	Therefore (\ref{eqn: comparing action axis vs R - almost equiv}) yields $\dist{\varphi(gy)}{q(g) \varphi(y)} \leq 12\delta$.

	If $\mathbf L = \cyclic$, the proof is complete.
	Suppose now that that $\mathbf L = \dihedral$.
	Recall that $A$ is a pre-image by $q$ of $\group{\mathbf x}$ in $\stab Y$ and $y_0$ a point in $\fix{A, 15\delta}$.
	Let $a \in A \setminus F$, i.e.\ a pre-image by $q$ of $\mathbf x$.
	We claim that for every $y,y' \in Y$ we have 
	\begin{equation*}
		\abs{c(ay,ay') + c(y,y')} \leq 148\delta.
	\end{equation*}
	Denote by $\eta$ the repulsive point of $h_0$ so that $\eta$ and $\xi$ are the two accumulation points in $\partial X$ of $Y$.
	Fix $y,y' \in Y$ and denote by $z = \gamma(t)$ and $z' = \gamma(t')$ respective projections of $y$ and $y'$ on the quasi-geodesic $\gamma$.
	Up to swapping $y$ and $y'$ we can assume that $t \leq t'$.
	By definition of $Y$, we have $\dist yz \leq 27\delta$ and $\dist {y'}{z'} \leq 27\delta$.
	Consider a Busemann cocycle $c_\xi$ at $\xi$.
	The cocycle relation yields
	\begin{align*}
		\MoveEqLeft{c_\xi(ay,ay') + c_\xi(y,y')} \\
		& = \left[c_\xi(ay,az) + c_\xi(az,az') + c_\xi(az',ay')\right] + \left[c_\xi(y,z) + c_\xi(z,z') + c_\xi(z',y')\right].
	\end{align*}
	However Busemann cocycles are almost $1$-Lipschitz, hence 
	\begin{equation}
	\label{eqn: comparing action axis vs R - oppositie cocycles}
		\begin{split}
			\abs{c_\xi(ay,ay') + c_\xi(y,y')}  
			& \leq \abs{c_\xi(az,az') + c_\xi(z,z')} + 2 \dist yz + 2 \dist {y'}{z'} + 8\delta \\
			& \leq \abs{c_\xi(az,az') + c_\xi(z,z')} +116\delta.
		\end{split}
	\end{equation}
	Recall that (\ref{eqn: effective computation busemann}) provides a way to estimate Busemann cocycles, namely
	\begin{equation*}
		\abs{c_\xi(z,z') - \left[\gro {z'}\xi z - \gro z\xi {z'}\right]} \leq 3\delta.
	\end{equation*}
	However $z$ and $z'$ are two points on the quasi-geodesic $\gamma$.
	It follows from \autoref{res: stability qg} that 
	\begin{equation*}
		\abs{c_\xi(z,z') - \dist z{z'}} \leq 10\delta.
	\end{equation*}
	Recall that $a$ permutes the points $\xi$ and $\eta$ so that $a^{-1}c_\xi$ is a Busemann cocycle at $\eta$.
	Using the same argument as above we get
	\begin{equation*}
		\abs{c_\xi(az,az') + \dist z{z'}}
		\leq \abs{a^{-1}c_\xi(z,z') + \dist z{z'}} 
		\leq 10\delta.
	\end{equation*}
	Plugin in back these estimates in (\ref{eqn: comparing action axis vs R - oppositie cocycles}) we get $\abs{c_\xi(ay,ay') + c_\xi(y,y')}  \leq 136\delta$.
	However $c$ differs from $c_\xi$ by at most $6\delta$, whence our claim.
	We now claim that for every $y \in Y$ we have 
	\begin{equation*}
		\dist{\varphi(ay)}{q(a) \varphi(y)} \leq 121\delta.
	\end{equation*}
	Recall that $\mathbf x = q(a)$ acts on $\R$ as the symmetry at $0$.
	Combined with the cocycle relation it gives
	\begin{equation*}
		\varphi(ay) - q(a)\phi(y) 
		= c(y_0, ay) + c(y_0, y) 
		= \left[c(ay_0,ay) + c(y_0,y)\right] + c(y_0, ay_0)
	\end{equation*}
 	Our previous claim provides an estimate of the first term.
	Recall that $y_0$ belongs to $\fix{A, 15\delta}$.
	Since $c$ is almost $1$-Lipschitz, we have $\abs{c(y_0, ay_0)} \leq \dist {y_0}{ay_0} + 2\delta \leq 17\delta$.
	All together we obtain $\dist{\varphi(ay)}{q(a) \varphi(y)} \leq 200\delta$ as we announced.
	
	Note that any element of $\stab Y$ can be written either $g = h_0^mu$ or $g = ah_0^mu$ with $m \in \Z$ and $u \in F$.
	Hence the almost $q$-equivariance of $\varphi$ follows from our previous discussion combined with the triangle inequality.
\end{proof}

It follows from the previous statement that $\varphi$ induces a map $\bar \varphi \colon Y/H \to \mathcal C$, such that for every $\bar g \in \stab Y/H$, for every $\bar y \in Y/H$, 
\begin{equation*}
	\dist{\bar \varphi(\bar g \bar y)}{q_{\bar v}(\bar g)\bar \varphi(\bar y)} \leq 200\delta.
\end{equation*}
By \autoref{res: qi between cones}, $\bar\varphi$ induces a $(1, 150\delta)$-quasi-isometric embedding $Z(Y/H) \to \mathcal D$, that we again still denote $\bar\varphi$, so that for every $\bar g \in \stab Y/H$, for every $\bar x \in Z(Y/H)$, we have 
\begin{equation}
\label{eqn: cone comparison - almost equiv}
	\dist[\mathcal D]{\bar\varphi(\bar g \bar x)}{q_{\bar v}(\bar g)\bar\varphi(\bar x)} \leq 200\delta.
\end{equation}
Roughly speaking, this means that $\stab Y/H$ acts on $Z(Y/H)$ as $\mathbf L_n$ does on $\mathcal D$.

Note that $Z(Y/H)$ -- which is actually isometric to $Z(Y)/H$ -- is endowed here with the metric defined by (\ref{eqn: def metric cone}).
Although, as a set of points, $Z(Y)/H$ can be identified with the closed ball of $\bar X$ of radius $\rho$ centered at $\bar v$ (\autoref{res: small cancellation}), the distance we considered so far is not the exactly the one coming from $\bar X$.
Nevertheless the embedding $Z(Y) \to \dot X$ is $1$-Lipschitz.
It follows that the map $\bar\varphi \colon B(\bar v, \rho) \to \mathcal D$ induced by $\bar\varphi \colon Z(Y/H) \to \mathcal D$ is such that for every $\bar x, \bar x' \in B(\bar v, \rho)$ we have
\begin{equation}
\label{eqn: cone comparison - almost Lipschitz}
	\dist[\bar X]{\bar x}{\bar x'}\leq \dist[\mathcal D]{\bar\varphi(\bar x)}{\bar\varphi(\bar x')}  + 150\delta.
\end{equation}
As we observed previously, the metrics of $Z(Y)$ and $\dot X$ coincide on $B(v, \rho/3)$.
It follows that the metric on $Z(Y/H)$ and $\bar X$ coincide on $B(\bar v, \rho/3)$.
Hence the map $\bar\varphi \colon B(\bar v, \rho) \to \mathcal D$ is a $(1, 150\delta)$-quasi-isometric embedding when restricted to $B(\bar v, \rho/3)$.

\begin{prop}
\label{res: local action of apex stab}
	Let $\bar v \in \bar{\mathcal V}$.
	\begin{enumerate}
		\item \label{enu: local action of apex stab - triv}
		If $\bar g \in \stab{\bar v}$ is \emph{locally trivial at $\bar v$}, then $B(\bar v, \rho)$ is contained in $\fix{\bar g, \bar \delta}$.
		\item \label{enu: local action of apex stab - reflection}
		If $\bar A$ is a \emph{reflection group at $\bar v$}, then there exists a point $\bar x \in \fix{\bar A,\bar \delta}$ with $\dist{\bar v}{\bar x} = \rho$ such that $\fix{\bar A,\bar \delta} \cap B(\bar v, \rho/3)$ is non-empty and for every $\bar z \in \fix{\bar A,\bar \delta} \cap B(\bar v, \rho/3)$ we have 
		\begin{equation*}
			\min\left\{ \gro{\bar x}{\bar v}{\bar z}, \gro{\bar g \bar x}{\bar v}{\bar z} \right\} \leq \bar \delta,
		\end{equation*}
		where $\bar g$ is a central half-turn at $\bar v$.
		\item \label{enu: local action of apex stab - rot}
		If $\bar g \in \stab{\bar v}$ is a \emph{strict rotation at $\bar v$}, with order $m \in \N$,  then there exists $k \in \Z$, such that for every $i \in \N$, with 
		\begin{equation*}
			1 \leq i  \leq \min\left\{ m -1,  \frac{\Omega - \pi}{2\pi} \right\}, 
		\end{equation*}
		the subset  $\fix{\bar g^{ik}, \bar \delta}$ is non-empty and contained in the $\bar \delta$-neighborhood of $\bar v$.
		In particular, $\bar v$ is the unique vertex fixed by $\bar g$.
	\end{enumerate}
\end{prop}

\begin{rema}
	Roughly speaking Point~\ref{enu: local action of apex stab - reflection} is saying that any point of $B(\bar v, \rho/3)$ that is fixed by $A$ lies on the geodesic $\geo{\bar x}{\bar g \bar x}$ -- which goes through $\bar v$ by Point~\ref{enu: local action of apex stab - rot}.
	Nevertheless, in our setting, $\bar X$ does not need to be geodesic. 
	Thus a rigorous statement is the one formulated above.
\end{rema}

\begin{rema}
\label{rem: elliptic subgroup with strict rotation}
	According to the small cancellation assumption $\Omega > 10\pi$.
	Hence the condition on $i$ appearing in \ref{enu: local action of apex stab - rot} is satisfied whenever $1 \leq i \leq \min\{m-1, 4\}$.
	In particular, it is always satisfied by $i = 1$.

	It follows from \ref{enu: local action of apex stab - rot} that if $\bar E$ is an elliptic subgroup of $\bar G$ containing a strict rotation at $\bar v$, then $\fix{\bar E, 10 \bar \delta}$ is contained in $B(\bar v, 14\bar \delta)$.
	In particular, $\bar E$ is a subgroup of $\stab{\bar v}$.
\end{rema}

\begin{proof}
	We use the comparison map $\bar\varphi \colon B(\bar v, \rho) \to \mathcal D$ defined during the previous discussion.
	Assume first that $\bar g$ is locally trivial at $\bar v$, i.e.\ $q_{\bar v}(\bar g) = 1$.
	In other words $q_{\bar v}(\bar g)$ acts trivially on $\mathcal D$.
	Combining (\ref{eqn: cone comparison - almost equiv}) and (\ref{eqn: cone comparison - almost Lipschitz}) we get $\dist[\bar X]{\bar g\bar x}{\bar x} \leq 350\delta$, for every $\bar x \in B(\bar v, \rho)$.
	Hence $B(\bar v, \rho)$ is contained in $\fix{\bar g, \bar \delta}$, which completes the proof of \ref{enu: local action of apex stab - triv}.
	
	Assume now that $\bar g$ is a strict rotation at $\bar v$ with order $m$.
	For simplicity we write $I$ for the set of integers $i \in \N$ with 
	\begin{equation*}
		1 \leq i  \leq \min\left\{ m -1,  \frac{\Omega - \pi}{2\pi} \right\}, 
	\end{equation*}
	We claim that there exists $j \in \N\setminus\{0\}$ such for every $i \in I$, we have
	\begin{equation*}
		\pi \leq \frac {ij\Omega}m \leq \Omega - \pi.
	\end{equation*}
	Indeed if $\Omega/m \geq \pi$, it suffices to take $j = 1$.
	Otherwise, we take for $j$ the smallest integer such that $j\Omega /m \geq \pi$.
	Since $\Omega/m < \pi$, we necessarily have $j \Omega/ m < 2 \pi$, and our claim follows.
	Denote by $\mathbf r$ the generator of $\cyclic[n] \subset \mathbf L_n$ acting on $\mathcal D$ as a rotation of angle $\Omega / n$.
	It follows from our assumption that $m$ divides $n$ and $q_{\bar v}$ sends $\group{\bar g}$ onto $\group{ \mathbf r^{n/m}}$.
	Let $k \in \Z$ such that $\bar g^k$ is a pre-image of $\mathbf r_j = \mathbf r^{jn/m}$.
	Let $i \in I$.
	By construction $\mathbf r_j^i= \mathbf r^{ijn/m}$ acts on $\mathcal D$ as a rotation centered at $o$ whose angle belongs to $[\pi, \Omega - \pi]$.
	In particular, for every $\bar x \in \mathcal D$, the angle at $o$ between $\bar x$ and $\mathbf r^i_j\bar x$ is larger than $\pi$.
	Consequently 
	\begin{equation*}
		\dist[\mathcal D]{\mathbf r^i_j\bar x}{\bar x} = 2 \dist[\mathcal D] {\bar x}o.
	\end{equation*}
	Recall that $\bar \varphi$ induces an almost $q_{\bar v}$-equivariant $(1,150\delta)$-quasi-isometric embedding from $B(\bar v, \rho/3)$ into $\mathcal D$.
	Hence for every $\bar x \in B(\bar v, \rho/3)$, 
	\begin{equation*}
		\dist[\bar X]{\bar g^{ik} \bar x}{\bar x} \geq 2 \dist[\bar X]{\bar x}{\bar v} - 500\delta.
	\end{equation*}
	In particular, $\fix{\bar g^{ik}, \bar \delta} \cap B(\bar v, \rho/3)$ is contained in $B(\bar v, \bar \delta)$.
	Since $\fix{\bar g^{ik}, \bar \delta}$ is $10\bar \delta$-quasi-convex (\autoref{res: fix qc}) the set $\fix{\bar g^{ik}, \bar \delta}$ is entirely contained in $B(\bar v, \bar \delta)$, which completes the proof of \ref{enu: local action of apex stab - rot}.

	We are left to prove Point~\ref{enu: local action of apex stab - reflection}.
	Let $\bar A$ be a reflection group at $\bar v$.
	Without loss of generality we can assume that $q_{\bar v}(A) = \group{\mathbf x}$.
	It follows from \autoref{ass: even exponent} that $n$ is even and $\stab{\bar v}$ contains a central-half turn $\bar g$.
	We write $\mathbf r$ for its image in $\mathbf L_n$.
	Recall that $y_0$ is a base point in $Y \cap \fix{A, 10\delta}$ chosen to define the map $\bar \varphi$.
	Let $\bar y_0$ be its image in $\bar X$.
	It follows from the construction that the set of fixed points of $\mathbf x$ is exactly the geodesic of $\mathcal D$ between $\bar\varphi(\bar y_0)$ and $\mathbf r\bar\varphi(\bar y_0)$.
	Note that this geodesic passes through $o$ as the angle $\Omega$ at the apex of $\mathcal D$ is larger that $2\pi$.
	Consequently for every $d \geq 0$, for every $\bar x \in B(\bar v, \rho/3)$, such that $\dist{\mathbf x \bar\varphi(\bar x)}{\bar\varphi(\bar x)} \leq d$, we have either $\gro o{\bar\varphi(\bar y_0)}{\bar \varphi (\bar x)} \leq d/2$ or $\gro o{\mathbf r\bar\varphi(\bar y_0)}{\bar \varphi (\bar x)} \leq d/2$.
	We carry again this observation to $\bar X$ using the map $\bar\varphi \colon B(\bar v, \rho) \to \mathcal D$ to get the conclusion of \ref{enu: local action of apex stab - reflection}.
\end{proof}

\begin{voca*}
	In view of the previous statement, we can say that an element $\bar g \in \bar G$ is a \emph{strict rotation} if there is an apex $\bar v$ such that $\bar g$ is a strict rotation at $\bar v$.
	Indeed in such a case, $\bar g$ cannot be locally trivial or a reflection at any other vertex.
	Note that being a strict rotation is invariant under conjugation.
\end{voca*}

%
\subsection{Lifting properties}
%

In \autoref{res: small cancellation}\ref{enu: small cancellation - local isom} we mention a very important fact: small cancellation does not affect the small scale geometry of the space.
More precisely the projection $\zeta \colon \dot X \to \bar X$ is an isometry when restricted to a small ball lying sufficiently far away from apices.
This is a key ingredient to lift several figures from $\bar X$ to $\dot X$.
We complete this picture with other properties of the map $\zeta \colon \dot X \to \bar X$.
Exceptionally, in this section all the distances are measured either in $\dot X$ or $\bar X$.

The first step is to explain how one can lift isometrically in $\dot X$ a quasi-convex subset $\bar Z \subset \bar X$ as well as its (partial) stabilizer, provided it stays far away from the apex set $\bar {\mathcal V}$.

\begin{lemm}
\label{res: isom prelim}
	Let $x,y \in \dot X$ such that $\gro xyv > 12\dot \delta$, for every $v \in \mathcal V$.
	Then $\dist[\dot X]xy = \dist{\bar x}{\bar y}$.
\end{lemm}

\begin{proof}
	In this proof all the distances are measured in $\dot X$ or $\bar X$.
	We first claim that $\gro{\bar x}{\bar y}{\bar v} \geq 10\bar \delta$, for every $\bar v \in \bar{\mathcal V}$.
	To that end, we fix $\eta >0$ and a $(1, \eta)$-quasi-geodesic $\gamma_1 \colon \intval{a_1}{b_1} \to \dot X$ joining $x$ to $y$.
	Note that $\gamma_1$ stays far from $\mathcal V$.
	Indeed, $d(v,\gamma_1) \geq \gro xyv -\eta/2$, for every $v \in \mathcal V$.
	It follows from \autoref{res: small cancellation}\ref{enu: small cancellation - local isom} that if $\eta$ is sufficiently small, then the image $\bar \gamma_1 \colon \intval{a_1}{b_1} \to \bar X$ of $\gamma_1$ in $\bar X$ is an $L$-local $(1, \eta)$-quasi-geodesic joining $\bar x$ to $\bar y$ for some $L > 4\eta + 8\bar \delta$.
	Let $v \in \mathcal V$.
	Applying the stability of quasi-geodesics to $\bar \gamma_1$ (\autoref{res: stability qg}) we get
	\begin{equation*}
		12\dot \delta - \eta/2
		\leq \inf_{g\in K} d(gv, \gamma_1)
		\leq d(\bar v, \bar \gamma_1)
		\leq \gro{\bar x}{\bar y}{\bar v} + \eta/2 + 2\bar \delta.
	\end{equation*}
	This inequality holds for every sufficiently small $\eta$, which completes the proof of our first claim.

	Let $\eta \in (0, \dot \delta)$ and $y' \in \dot X$ be a pre-image of $\bar y$ such that $\dist x{y'} \leq \dist{\bar x}{\bar y} + \eta$.
	In particular, $\gro x{y'}v \geq \gro{\bar x}{\bar y}{\bar v} - \eta/2$, for every $v \in \mathcal V$.
	We are going to prove that $y$ and $y'$ are very close, provided $\eta$ is small enough.
	Let $\gamma_2 \colon \intval{a_2}{b_2} \to \dot X$ be a $(1,\eta)$-quasi-geodesic joining $y$ to $y'$.
	Let $v \in \mathcal V$.
	Applying the four point inequality (\ref{eqn : hyp four points - 1}) in $\dot X$, we observe that 
	\begin{equation*}
		\gro y{y'}v 
		\geq \min \left\{ \gro xyv , \gro x{y'}v \right\} - \dot \delta
		\geq \min \left\{ \gro xyv, \gro{\bar x}{\bar y}{\bar v} \right\} - \dot \delta - \eta/2
		> 8\dot \delta.
	\end{equation*}
	Reasoning as previously we see that for a sufficiently small value of $\eta$, the image $\bar \gamma_2 \colon \intval{a_2}{b_2} \to \bar X$ of $\gamma_2$ in $\bar X$ is an $L$-local $(1, \eta)$-quasi-geodesic from $\bar y$ to $\bar y'$, where $L > 4\eta + 8\dot\delta$ does not depend on $\eta$.
	By \autoref{res: stability qg}, it is also a (global) $(\kappa, \eta)$-quasi-geodesic joining $\bar y$ to itself, where $\kappa$ can be chosen independently of $\eta$.
	Therefore $\dist y{y'} \leq \dist {b_2}{a_2} \leq \kappa\eta$ as we announced.
	Applying the triangle inequality we get  $\dist xy \leq \dist x{y'} + \kappa \eta \leq \dist {\bar x}{\bar y} + (\kappa + 1)\eta$.
	This holds for every sufficiently small $\eta>0$, hence $\dist xy \leq \dist{\bar x}{\bar y}$.
	The converse inequality follows from the fact that $\zeta \colon \dot X \to \bar X$ is $1$-Lipschitz.
\end{proof}

\begin{lemm}
\label{res: zeta isom on qc far from apices}
	Let $Z$ be a subset of $\dot X$ such that $\gro z{z'}v > 13\dot \delta$, for every $z,z' \in Z$ and every $v \in \mathcal V$.
	The map $\zeta \colon \dot X \to \bar X$ induces an isometry from $Z$ onto its image $\bar Z$.
	In addition, the following holds.
	\begin{enumerate}
		\item \label{enu: zeta isom on qc far from apices - single isom}
		Let $\bar g \in G$ and $z_1, z_2 \in Z$ such that $\bar g \bar z_1 = \bar z_2$.
		Then there exists a unique pre-image $g \in G$ of $\bar g$ such that $gz_1 = z_2$.
		Moreover for every $z,z' \in Z$, if $\bar g \bar z = \bar z'$, then $gz = z'$.
		\item \label{enu: zeta isom on qc far from apices - group}
		The projection $\pi \colon G \to \bar G$ induces an isomorphism from $\stab Z$ onto $\stab {\bar Z}$.
	\end{enumerate}
\end{lemm}

\begin{rema*}
	The statement applies in particular if $Z$ is $\alpha$-quasi-convex and satisfies $d(v,Z) > \alpha + 13\dot \delta$, for every $v \in \mathcal V$.
	This slightly weaker version will be more flexible for later use, though. 
\end{rema*}

\begin{proof}
	In this proof all the distances are measured in $\dot X$ or $\bar X$.
	By \autoref{res: isom prelim}, the projection $\zeta \colon \dot X \to \bar X$ induces an isometry from $Z$ onto $\bar Z$.
	Let $\bar g \in G$ and $z_1, z_2 \in Z$ such that $\bar g \bar z_1 = \bar z_2$.
	By the very definition of $\bar X$, there exists a pre-image $g \in G$ of $\bar g$, such that $gz_1 = z_2$.
	Uniqueness follows from the fact that $K$ acts freely on $\dot X \setminus \mathcal V$ -- see \autoref{res: small cancellation}\ref{enu: small cancellation - translation kernel}.
	We now prove that $g$ satisfies the announced property.
	Let $z,z' \in Z$ such that $\bar g \bar z = \bar z'$.
	Let $v \in \mathcal V$.
	Applying the four point inequality (\ref{eqn : hyp four points - 1}) in $\dot X$ we have
	\begin{equation*}
		\gro {gz}{z'}v 
		\geq \min\left\{ \gro {gz}{gz_1}v, \gro{z_2}{z'}v \right\} - \dot \delta
		\geq \min\left\{ \gro z{z_1}{g^{-1}v}, \gro{z_2}{z'}v \right\} - \dot \delta
	\end{equation*}
	Note that $z_1$, $z_2$, $z$ and $z'$ belongs to $Z$.
	Hence, it follows from our assumption that $\gro {gz}{z'}v > 12\dot\delta$, for every $v \in \mathcal V$.
	By \autoref{res: isom prelim}, we get that $\dist{gz}{z'} = \dist{\bar g\bar z}{\bar z'} = 0$.
	This completes the proof of \ref{enu: zeta isom on qc far from apices - single isom}.
	Point~\ref{enu: zeta isom on qc far from apices - group} follows directly from \ref{enu: zeta isom on qc far from apices - single isom}.
\end{proof}

The next two statements are a variation on \cite[Lemme~5.10.1]{Delzant:2008tu}.

\begin{lemm}
\label{res: lifting quasi-convex}
	Let $\bar Z$ be a subset of $\bar X$ such that $\gro{\bar z}{\bar z'}{\bar v} > 13\bar \delta$, for every $\bar z, \bar z' \in \bar Z$ and every $\bar v \in \bar{\mathcal V}$.
	Let $\bar z_0$ be a point of $\bar Z$ and $z_0 \in \dot X$ a pre-image of $\bar z_0$
	Then there exists a unique subset $Z$ of $\dot X$ containing $z_0$ such that the projection $\zeta \colon \dot X \to \bar X$, induces an isometry from $Z$ onto $\bar Z$.
	In particular, $\gro z{z'}v \geq \gro{\bar z}{\bar z'}{\bar v}$ for every $z,z' \in Z$ and every $v \in \mathcal V$.
\end{lemm}

\begin{rema*}
	Note that \autoref{res: zeta isom on qc far from apices} applies to the lifted set $Z$.
	Hence, we can lift any isometry $\bar g \in \bar G$ which (partially) preserves $\bar Z$ to an isometry $g \in G$ with the same properties.
	\autoref{res: lifting quasi-convex} holds in particular if $\bar Z$ is $\alpha$-quasi-convex and satisfies $d(\bar v, \bar Z) > \alpha + 13\bar \delta$, for every $\bar v \in \bar{\mathcal V}$, in which case one can prove that $Z$ is quasi-convex as well.
	Nevertheless we will not use this fact here.
\end{rema*}

\begin{proof}
	In this proof all the distances are measured in $\dot X$ or $\bar X$.
	We define $Z$ as the set of points $z \in \dot X$ being the pre-image of a point $\bar z \in \bar Z$ and such that  $\dist z{z_0} = \dist{\bar z}{\bar z_0}$.
	We claim that every $\bar z \in \bar Z$ has a pre-image in $Z$.
	Let $\bar z \in \bar Z$.
	We fix a pre-image $z \in \dot X$ of $\bar z$ such that $\dist {z_0}z \leq \dist{\bar z_0}{\bar z} + 2\bar \delta$.
	In particular, for every $v \in \mathcal V$, we have $\gro z{z_0}v \geq \gro{\bar z}{\bar z_0}{\bar v} - \bar \delta$, hence $\gro z{z_0}v > 12\dot \delta$.
	It follows from \autoref{res: isom prelim} that $\dist {z_0}z = \dist{\bar z_0}{\bar z}$.
	In other words $z$ belongs to $Z$, which completes the proof of our claim.
	Hence the projection $\zeta \colon \dot X \to \bar X$ maps $Z$ onto $\bar Z$.

	We now prove that $\gro z{z'}v > 12\dot \delta$, for every $z,z' \in Z$ and every $v \in \mathcal V$.
	It follows from the very definition of $Z$ that $\gro z{z_0}v \geq \gro{\bar z}{\bar z_0}{\bar v}$ and $\gro{z'}{z_0}v \geq \gro{\bar z'}{\bar z_0}{\bar v}$.
	Combining the four point inequality (\ref{eqn : hyp four points - 1}) with our assumption on $\bar Z$, we get
	\begin{equation*}
		\gro z{z'}v 
		\geq \min \left\{ \gro z{z_0}v , \gro{z'}{z_0}v \right\} - \dot \delta
		\geq \min \left\{ \gro{\bar z}{\bar z_0}{\bar v} , \gro{\bar z'}{\bar z_0}{\bar v} \right\} - \bar \delta
		> 12\dot  \delta.
	\end{equation*}
	It now follows from \autoref{res: isom prelim} that the projection $\zeta \colon \dot X \to \bar X$ induces an isometry from $Z$ onto its image, i.e.\ $\bar Z$.
	This proves the existence of the set $Z$.
	The uniqueness directly follows from the definition of $Z$.
	Since $Z \to \bar Z$ is an isometry, $\gro z{z'}v \geq \gro{\bar z}{\bar z'}{\bar v}$, for every $z,z' \in Z$, for every $v \in \mathcal V$.
\end{proof}

The previous statements explain how to lift in $\dot X$ a quasi-convex subset $\bar Z \subset \bar X$, as well as its (partial) stabilizer, as soon as it stays away from the apex set $\bar {\mathcal V}$.
In particular, it applies to any $(1, \eta)$-quasi-geodesic path of $\bar X$ that avoids the cone points.
We focus now on a more delicate operation which consists in lifting paths of $\bar X$ (and their almost stabilizers) going through one or several apices.
The next statement follows \cite[Proposition~5.13]{Coulon:2016if}.

\begin{prop}
\label{res: lifting crossing apices - w/o reflection}
	Let $x$ and $y$ be two points of $X$.
	Let $\gamma \colon \intval ab \to \dot X$ be a path from $x$ to $y$ whose image $\bar \gamma \colon \intval ab \to \bar X$ is a $(1, \bar \delta)$-quasi-geodesic.
	Let $S$ be a subset of $G$ and $\bar S$ its image in $\bar G$.
	We assume that $\dist[\dot X]{gx}x \leq \rho/100$ and $\dist{\bar g\bar y}{\bar y} \leq \rho/100$, for every $g \in S$.
	In addition, we suppose that for every apex $\bar v \in \bar{\mathcal V}$ satisfying $\gro{\bar x}{\bar y}{\bar v} \leq \rho/4$, the set $\bar S$ lies in the local kernel at $\bar v$.
	Then $\dist[\dot X]{gy}y = \dist{\bar g \bar y}{\bar y}$ for every $g \in S$.
\end{prop}

\begin{proof}
	Since $\bar \gamma$ is a $(1,\bar \delta)$-quasi-geodesic and the projection $\zeta \colon \dot X \to \bar X$ is $1$-Lipschitz, the path $\gamma$ is a $(1, \dot \delta)$-quasi-geodesic.
	Let $v_1, \dots, v_m$ be the apices of $\mathcal V$ which are $\rho/5$-close to $\gamma$.
	For every $j \in \intvald 1m$, we denote by $\gamma(c_j)$ a projection of $v_j$ on $\gamma$.
	By reordering the apices we can always assume that $c_1 \leq c_2 \leq \dots \leq c_m$. 
	For simplicity of notation, we put $c_0 = a$ and $c_{m+1} = b$.
	Let $j \in \intvald 1m$.
	Since $ \gamma$ is a $(1,\dot \delta)$-quasi-geodesic,  we can find $b_{j-1} \in (c_{j-1},c_j]$ and  $a_j \in [c_j,c_{j+1})$ with the following properties.
	\begin{enumerate}
		\item $\dist[\dot X]{v_j}{\gamma (b_{j-1})} = 9\rho/10$ and $\dist[\dot X]{v_j}{\gamma (a_j)} = 9\rho/10$,
		\item $\gamma \cap B(v_j, 4\rho/5)$ is contained in the image of  $\gamma$ restricted to $(b_{j-1},a_j)$
	\end{enumerate}
	In addition, we let $a_0 = a$, $b_m= a_{m+1} = b$ (see \autoref{fig: lift}).
	\begin{figure}[htbp]
	\centering
		\includegraphics[width=0.9\textwidth]{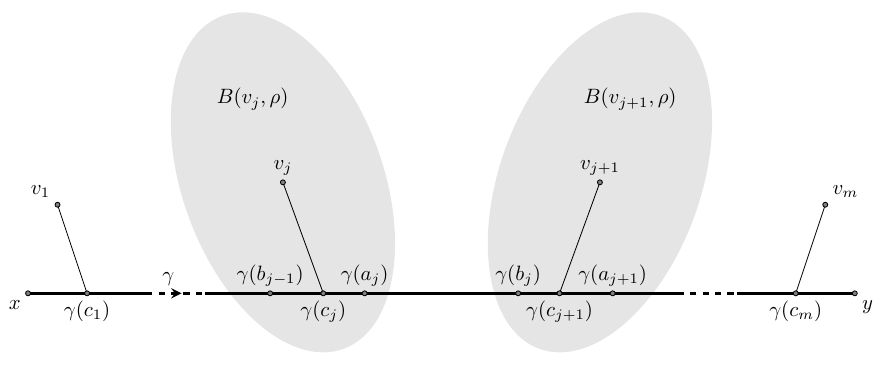}
	\caption{The cones intersecting $\gamma$.}
	\label{fig: lift}
	\end{figure}
	We claim that for every $j \in \intvald 0{m+1}$, for every $g \in S$,  we have
	\begin{equation*}
		\dist{\bar g\bar \gamma(a_j)}{\bar \gamma(a_j)} = \dist{g\gamma(a_j)}{\gamma(a_j)}.
	\end{equation*}
	The proof is by induction on $j$.
	If $j=0$ then $\gamma(a_j) = x$.
	The claim follows from the fact that the map $\zeta : \dot X \rightarrow \bar X$ induces an isometry from $B(x, \rho/20)$ onto $B(\bar x, \rho/20)$ -- see \autoref{res: small cancellation}\ref{enu: small cancellation - local isom}.
	Assume now that our claim is true for some $j \in \intvald 0m$. 
	Since $\gamma$ is a local quasi-geodesic, $a_j \leq b_j$.
	We denote by $\gamma_j$ the restriction of $\gamma$ to $\intval{a_j}{b_j}$ and by $\bar \gamma_j$ its image in $\bar X$.
	In addition we write $Z_j$ (\resp $\bar Z_j$) for the $\rho/50$-neighborhood of $\gamma_j$ (\resp $\bar \gamma_j$).
	It follows from the construction that $\gamma_j$ and thus $\bar \gamma_j$ stay away from any cone point.
	Consequently $\zeta \colon \dot X \to \bar X$ is an isometry when restricted to the $Z_j$ (\autoref{res: zeta isom on qc far from apices}).
	Moreover, it induces an isometry from $Z_j$ onto $\bar Z_j$ (\autoref{res: lifting quasi-convex}).

	Let $g \in S$.
	By assumption its image $\bar g$ in $\bar G$ moves $\bar x$ and $\bar y$ by at most $\rho/100$.
	Recall that we required $\bar \gamma$ to be $(1, \bar \delta)$-quasi-geodesic.
	It follows from \autoref{res: quasi-convexity distance isometry} that $\bar g$ moves $\bar \gamma(a_j)$ and $\bar \gamma(b_j)$ by at most $\rho/50$.
	In particular, $\bar g \bar \gamma(a_j)$ and $\bar g \bar \gamma(b_j)$ belongs to $\bar Z_j$.
	Nevertheless, according to our induction assumption, $g\gamma(a_j)$ is the (unique) lift of $\bar g \bar \gamma(a_j)$ that belongs to $Z_j$.
	The first step is to prove that $g$ does not move too much the point $\gamma(b_j)$.
	Applying \autoref{res: zeta isom on qc far from apices}\ref{enu: zeta isom on qc far from apices - single isom} we observe that $g\gamma(b_j)$ is the unique pre-image in $Z_j$ of $\bar g \bar \gamma(b_j)$.
	In particular, 
	\begin{equation*}
		\dist[\dot X]{g\gamma(b_j)}{\gamma(b_j)} = \dist{\bar g\bar \gamma(b_j)}{\bar \gamma(b_j)} \leq \rho/50.
	\end{equation*}
	If $j = m$, then $a_{m+1} = b_m$, thus the claim holds for $j+1$.
	Otherwise, the next step is to move accros $v_{j+1}$ and prove that $g$ does not move too much the point $\gamma(a_{j+1})$.
	Recall that $\dist{v_{j+1}}{\gamma(b_j)}=9\rho/10$.
	Since $\dist{g\gamma(b_j)}{\gamma(b_j)}  \leq \rho /50$, the element $g$ necessarily belongs to $\stab {v_{j+1}}$.
	Note also that
	\begin{equation*}
		\gro {\bar x}{\bar y}{\bar v_{j+1}} 
		\leq d(\bar v_{j+1}, \bar \gamma) + 4\bar \delta
		\leq d(v_{j+1}, \gamma) + 4\dot \delta
		\leq \rho/4,
	\end{equation*}
	see for instance by \cite[Lemma~2.4]{Coulon:2014fr} for the last inequality.
	It follows from our assumption that $\bar g$ is a locally trivial at $\bar v_{j+1}$.
	By construction, there is $(H_{j+1}, Y_{j+1}) \in \mathcal Q$, such that $v_{j+1}$ is the cone point of the cylinder over $Y_{j+1}$.
	Denote by $F_{j+1}$ the maximal elliptic normal subgroup of $\stab {Y_{j+1}}$.
	(Note that $\stab {Y_{j+1}} = \stab{v_{j+1}}$, however we stated the last sentence with $Y_{j+1}$ since the word elliptic refers there to the action of $G$ on $X$ and not the one on $\dot X$.)
	Observe that the pre-image in $\stab{v_{j+1}}$ of the local kernel at $\bar v_{j+1}$ is exactly $F_{j+1}H_{j+1}$.
	Hence we can write $g = uh$, with $u \in F_{j+1}$ and $h \in H_{j+1}$.
	According to \autoref{res: normal elliptic fixing cyl}, $u$ moves the points of $Y_{j+1}$ by at most $100 \delta$ (distances are measured in $X$ here).
	Going back to $\dot X$, the element $u$ moves any point of $B(v_{j+1}, \rho)$ by at most $\dot \delta$ (this is indeed a direct consequence of the definition of the metric on the cone over $Y_{j+1}$).
	Recall that  $g$ moves the point $\gamma(b_j) \in B(v_{j+1}, \rho)$ by at most $\rho/50$.
	It follows then from \autoref{res: small cancellation}\ref{enu: small cancellation - translation kernel} that $h = 1$, hence $g = u$ belongs to $F_{j+1}$.
	Therefore $g$ moves $\gamma(a_{j+1}) \in B(v_{j+1}, \rho)$ by at most $\dot \delta$.
	However, the map $\zeta : \dot X \rightarrow \bar X$ induces an isometry from the ball $B(\gamma(a_{j+1}), \rho/20)$ onto its image, hence $\dist[\dot X]{g\gamma(a_{j+1})}{\gamma(a_{j+1})} = \dist{\bar g \bar \gamma(a_{j+1})}{\bar \gamma(a_{j+1})}$.
	This proves our claim for $j+1$.
	The statement of the lemma follows from our claim for $j = m+1$.
\end{proof}

In the previous statement, we assumed that any isometry $\bar g \in \bar S$ which hardly move the endpoints of $\bar \gamma$, is actually in the local kernel of every vertex $\bar v$ lying close to $\bar \gamma$.
We now explore the situation where some element $\bar g$ might be a reflection at $\bar v$.

\begin{prop}
\label{res: lifting crossing apices - w/ reflection}
	Let $x$ and $y$ be two points of $X$.
	Let $S$ be a subset of $G$ and $\bar S$ its image in $\bar G$.
	We assume that $\dist {gx}x \leq \rho/100$ and $\dist{\bar g\bar y}{\bar y} \leq \rho/100$, for every $g \in S$.
	In addition, we suppose that for every apex $\bar v \in \bar{\mathcal V}$, the set $\bar S \cap \stab{\bar v}$ is contained in a reflection group at $\bar v$.
	Then there exists an element $u \in G$ with the following properties.
	\begin{enumerate}
		\item $\bar u$ commutes with every element in $\bar S$;
		\item $\dist[\dot X]{guy}{uy} = \dist{\bar g \bar y}{\bar y}$ for every $g \in S$;
		\item either $\bar u$ is trivial, or $\bar S$ lies in a reflection group at some apex $\bar v \in \bar{\mathcal V}$.
	\end{enumerate}
\end{prop}

\begin{rema}
\label{rem: assumption lifting crossing apices - w/ reflection}
	Note that if $\bar S \cap \stab{\bar v}$ is not contained in a reflection group at $\bar v$, then there is a strict rotation at $\bar v$ which is the product of at most two elements of $\bar S$.
	Indeed given any reflection $\mathbf x$ of $\dihedral[n]$, the geometric realization $q_{\bar v} \colon \stab{\bar v} \to \dihedral[n]$ cannot map $\bar S \cap \stab{\bar v}$ into $\group{\mathbf x}$.
	Consequently the image of $\bar S \cap \stab{\bar v}$ either contains a non-trivial rotation or two distinct reflections, whence the claim.
	This observation will be useful later to check that the assumptions of the proposition are fulfilled. 	
\end{rema}

\begin{proof}
	Assume first that for every apex $\bar v \in \bar{\mathcal V}$ satisfying $\gro{\bar x}{\bar y}{\bar v} \leq \rho/4$, the set $\bar S$ lies in the local kernel at $\bar v$.
	We fix $\epsilon > 0$ and  $u \in K$ such that $\dist[\dot X]x{uy} \leq \dist{\bar x}{\bar y} + \epsilon$.
	In addition we take a $(1, \epsilon)$-quasi-geodesic $\gamma \colon \intval ab \to \dot X$ from $x$ to $uy$.
	It follows from our choice of $u$ that the image $\bar \gamma \colon \intval ab \to \bar X$ of $\gamma$ is a $(1, 2\epsilon)$-quasi-geodesic from $\bar x$ to $\bar y$.
	Hence if $\epsilon$ is sufficiently small, \autoref{res: lifting crossing apices - w/o reflection} applies, which completes the proof.

	Assume now that there exists a vertex $\bar v \in \bar{\mathcal V}$ satisfying $\gro{\bar x}{\bar y}{\bar v} \leq \rho/4$ such that the set $\bar S$ does not lie in the local kernel at $\bar v$.
	Any element $\bar g \in \bar S$ moves $\bar x$ and $\bar y$ by at most $\rho /100$.
	Hence $\bar g$ moves $\bar v$ by at most $\rho$ (\autoref{res: quasi-convexity distance isometry}).
	It follows that $\bar S$ is contained in $\stab {\bar v}$.
	According to our assumption $\bar S$ is contained in a reflection group at $\bar v$.

	We now denote by $\mathcal U$ the set of all elements $u \in G$ whose image $\bar u$ in $\bar G$ commutes with $\bar S$.
	This set is non-empty as it contains the identity.
	We chose $u_0 \in \mathcal U$ such that $\dist{\bar u_0\bar y}{\bar x} \leq \dist{\bar u\bar y}{\bar x} + \bar \delta$, for every $u \in \mathcal U$.
	We are going to prove that for every $\bar v \in \bar {\mathcal V}$, if $\gro{\bar x}{\bar u_0\bar y}{\bar v} \leq \rho/4$, then $\bar S$ lies in the local kernel at $\bar v$.
	Consider indeed an apex $\bar v \in \bar {\mathcal V}$ such that $\gro{\bar x}{\bar u_0\bar y}{\bar v} \leq \rho/4$.
	As $\bar u_0$ commutes with $\bar S$, the distance $\dist{\bar g\bar u_0\bar y}{\bar u_0\bar y} = \dist{\bar g\bar y}{\bar y}$ is bounded above by $\rho/100$ for every $\bar g \in \bar S$.
	We prove as above that $\bar S$ is contained in a reflection group at $\bar v$.
	Suppose now that contrary to our claim $\bar S$ is not in the local kernel at $\bar v$.
	We fix $\bar s \in \bar S$ a reflection at $\bar v$.
	According to \autoref{res: local action of apex stab}\ref{enu: local action of apex stab - reflection}, there exists a point $\bar z_0 \in \bar X$ with $\dist{\bar v}{\bar z_0} = \rho$ such that for every $\bar z \in \fix{\bar S,\bar \delta} \cap B(\bar v, \rho/3)$ we have 
	\begin{equation*}
		\min\left\{ \gro{\bar z_0}{\bar v}{\bar z}, \gro{\bar h \bar z_0}{\bar v}{\bar z} \right\} \leq \bar \delta,
	\end{equation*}
	where $\bar h$ is a central half-turn at $\bar v$.
	Observe that $\bar u_1 = \bar h\bar u_0$ belongs to $\mathcal U$.
	Indeed, as $\bar h$ centralizes $\stab{\bar v}$, it commutes with $\bar S$.
	We now claim that $\dist{\bar u_1\bar y}{\bar x} < \dist{\bar u_0\bar y}{\bar x}  - 2\rho/15$.
	Let $\bar \gamma \colon \intval ab \to \bar X$ be a $(1, \bar \delta)$-quasi-geodesic from $\bar x$ to $\bar u_0 \bar y$.
	Let $\bar \gamma(t)$ be the projection of $\bar v$ onto $\bar \gamma$.
	By assumption $\gro{\bar x}{\bar u_0\bar y}{\bar v} \leq \rho/4$, hence $\dist{\bar \gamma(t)}{\bar v} \leq \rho/4 +4\bar \delta$, see for instance \cite[Lemma~2.4]{Coulon:2014fr}.
	Let
	\begin{align*}
		s_- & = \sup\set{s \in \intval at}{\dist{\bar \gamma(s)}{\bar v} \geq \rho/3}, \\
		s_+ & = \inf\set{s \in \intval tb}{\dist{\bar \gamma(s)}{\bar v} \geq \rho/3},
	\end{align*}
	so that $\bar y_- = \bar \gamma(s_-)$ and $\bar y_+ = \bar \gamma(s_+)$ are at distance exactly $\rho/3$ from $\bar v$ and 
	\begin{equation*}
		\dist{\bar y_-}{\bar y_+} \geq \rho/6 - 8\bar \delta.
	\end{equation*}
	Since any element of $\bar S$ moves the endpoint of $\bar \gamma$ by at most $\rho/100$, the path $\bar \gamma$ restricted to $\intval{s_-}{s_+}$ is contained in $\fix{S, \rho/50}$ (\autoref{res: quasi-convexity distance isometry}) hence in the $(\rho/100+7\bar \delta)$-neighborhood of $\fix{\bar S, \bar \delta}$ (\autoref{res: fix qc}).
	Hence for every $s \in (s_-,s_+)$ we have 
	\begin{equation*}
		\min\left\{ \gro{\bar z_0}{\bar v}{\bar \gamma(s)}, \gro{\bar h \bar z_0}{\bar v}{\bar \gamma(s)} \right\} \leq \rho/100 + 8\bar \delta.
	\end{equation*}
	See \autoref{fig: half-turn}.
	\begin{figure}[htbp]
    	\centering
    	\includegraphics[page=4, width=\textwidth]{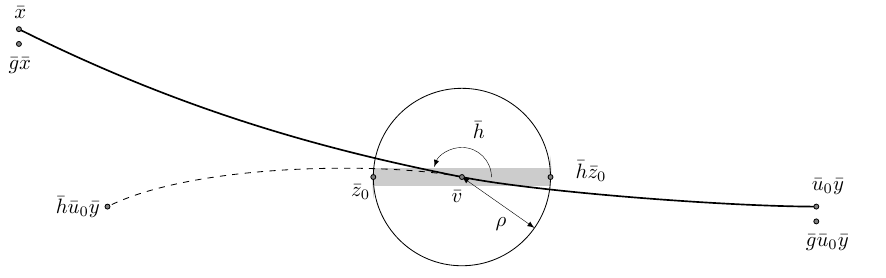}
    	\caption{
			The path $\bar \gamma$ going through the ball $B(\bar v, \rho)$. 
			The grey area corresponds to $\fix{\bar S, \bar \delta}\cap B(\bar v, \rho) $.
		}
    	\label{fig: half-turn}
    \end{figure}
	By continuity it also applies to $s_-$ and $s_+$.
	Recall that $\dist{\bar y_-}{\bar y_+} \geq \rho/10$.
	Up to permuting $\bar z_0$ and $\bar h \bar z_0$, it forces
	\begin{equation*}
		\gro{\bar z_0}{\bar v}{\bar y_-} \leq \rho/100 + 8\bar \delta,
		\quad \text{and} \quad
		\gro{\bar z_0}{\bar v}{\bar h\bar y_+} = \gro{\bar h \bar z_0}{\bar v}{\bar y_+} \leq \rho/100 + 8\bar \delta.
	\end{equation*}
	Hence $\dist{\bar y_-}{\bar h \bar y_+} \leq \rho/50 + 18 \bar \delta$ \cite[Lemma~2.2(2)]{Coulon:2014fr}.
	On the other hand, since $\bar y_-$ and $\bar y_+$ lie on a $(1, \bar \delta)$-quasi-geodesic between $\bar x$ and $\bar u_0\bar y$ we have
	\begin{align*}
		\dist{\bar x}{\bar u_0\bar y} 
		&\geq \dist{\bar x}{\bar y_-} + \dist{\bar y_-}{\bar y_+} + \dist{\bar y_+}{\bar u_0\bar y} - \bar \delta \\
		&\geq \dist{\bar x}{\bar y_-} + \dist{\bar y_+}{\bar u_0\bar y} + \rho/6 - 7 \bar \delta
	\end{align*}
	Combined with the triangle inequality, it yields
	\begin{equation*}
		\dist{\bar x}{\bar u_1\bar y}
		\leq \dist{\bar x}{\bar y_-} + \dist{\bar y_-}{\bar hy_+} + \dist{\bar y_+}{\bar u_0\bar y}
		\leq \dist{\bar x}{\bar u_0\bar y} - 2\rho/15.
	\end{equation*}
	which completes the proof our claim and contradicts the minimality of $u_0$.
	Consequently for every $\bar v \in \bar {\mathcal V}$, if $\gro{\bar x}{\bar u_0\bar y}{\bar v} \leq \rho/4$, then $\bar S$ lies the local kernel at $\bar v$.
	The conclusion now follows from the discussion at the beginning of the proof.
\end{proof}

%
\subsection{The action of $\bar G$ on $\bar X$}
%

We now study the general properties of the action of $\bar G$ on $\bar X$.

\begin{prop}
\label{res: sc - gentle action}
	The action of $\bar G$ on $\bar X$ is gentle.
\end{prop}

\begin{rema*}
	Recall that the action of $\bar G$ on $\bar X$ is gentle if every loxodromic subgroup $\bar E$ preserving the orientation splits as $\bar E = \sdp {\bar F} \cyclic$, where $\bar F$ is the set of all elliptic elements of $\bar E$.
\end{rema*}

\begin{proof}
	Let $\bar E$ be a loxodromic subgroup of $\bar G$, preserving the orientation and $\bar Z$ its cylinder (see \autoref{sec: hyp - gp action} for the definition).
	Assume first that there exists an apex $\bar v \in \bar {\mathcal V}$ such that $d(\bar v, \bar Z) \leq 20\bar \delta$.
	Let $\bar F$ be the \emph{set} of all elliptic elements of $\bar E$.
	Since $\bar E$ preserves the orientation, $\bar Z$ is contained in $\fix{\bar F, 100\bar \delta}$ (\autoref{res: normal elliptic fixing cyl}).
	It follows that $\dist{\bar u\bar v}{\bar v} < 2\rho$, for every $\bar u \in \bar F$.
	Consequently $\bar F=\bar E \cap \stab{\bar v}$ and is therefore a (normal) \emph{subgroup} of $\bar E$.
	We now prove that $\bar E/ \bar F$ is cyclic.
	To that end, we fix a loxodromic element $\bar g_0 \in \bar E$ such that $\norm{\bar g_0} \leq \norm{\bar g} +\bar \delta$, for every loxodromic element $\bar g \in \bar E$.
	Since $\bar g_0$ is a loxodromic element of $\bar E$, it sends $\bar v$ to a distinct apex.
	It follows that $\norm {\bar g_0} \geq \rho$.
	Let $\bar g \in \bar E$.
	The element $\bar g_0$ acts on $\bar Z$ by translation of length approximately $\norm{\bar g_0}$ (\autoref{res: cyl in mov}).
	Hence there exists $k \in \Z$ such that $\norm{\bar g_0^k\bar g} \leq \norm{\bar g_0}/2 + \rho/10 < \norm{\bar g_0} - \rho/10$.
	It follows from our choice of $\bar g_0$, that $\bar g_0^k\bar g$ is elliptic and thus belongs to $\bar F$.
	Hence $\bar E/\bar F$ is a cyclic group generated by the image of $\bar g_0$.

	Assume now that $d(\bar v, \bar Z) > 20\bar \delta$ for every apex $\bar v \in \bar {\mathcal V}$.
	Since $\bar Z$ is $2\bar \delta$-quasi-convex, there exists a subset $Z$ of $\dot X$ such that the projection $\zeta \colon \dot X \to \bar X$ induces an isometry from $Z$ onto $\bar Z$ (\autoref{res: lifting quasi-convex}).
	It follows then from \autoref{res: zeta isom on qc far from apices} that there exists a subgroup $E$ of $G$ such that $\pi \colon G \to \bar G$ induces an isomorphism from $E$ onto $\bar E$.
	By construction $\gro z{z'}v > 13\dot \delta$, for every $z,z' \in Z$, for every apex $v \in \mathcal V$ (the Gromov product is computed in $\dot X$ here).
	Consequently the radial projection $p \colon \dot X \setminus \mathcal V \to X$ induces an $E$-equivariant quasi-isometry from $Z$ onto $p(Z)$ (\autoref{res: radial proj qi}).
	This produces a $\pi_E$-equivariant quasi-isometry from $p(Z)$ to $\bar Z$ where $\pi_E$ stands for the map $\pi$ restricted to $E$.
	Since $\bar E$ is loxodromic and preserves the orientation, the same holds for $E$.
	Moreover if $F$ (\resp $\bar F$) stands for the set of elliptic elements of $E$ (\resp $\bar E$), then $\pi$ sends $F$  onto $\bar F$.
	As the action of $G$ is gentle, $E$ splits as $E = \sdp F \cyclic$.
	Hence $\bar E$ splits as well as $E = \sdp {\bar F} \cyclic$.
\end{proof}

\begin{lemm}
\label{res: two apices}
	The set $\bar{\mathcal V}$ contains at least two apices.
\end{lemm}

\begin{proof}
	In this proof all the distances are measured in $\dot X$ or $\bar X$.
	Let $v_1$ be an apex of $\mathcal V$ and $\bar v_1$ its image in $\bar X$.
	We denote by $v_2$ another apex such that $\dist{v_1}{v_2} \leq \dist{v_1}v + \dot \delta$, for every $v \in \mathcal V\setminus\{v_1\}$.
	We are going to prove that the image $\bar v_2$ of $v_2$ in $\bar X$ is distinct from $\bar v_1$.
	Let $\gamma \colon \intval {a_1}{a_2} \to \dot X$ be a $(1, \dot \delta)$-quasi-geodesic from $v_1$ to $v_2$.
	Let $b_1 = a_1 + \rho/4$ and $b_2 = a_2 - \rho/4$.
	For simplicity we write $x_1 = \gamma(b_1)$ and $x_2 = \gamma(b_2)$.
	Note that $\dist{x_1}{x_2} \geq 3\rho/2$.
	Moreover, $\gro{x_1}{x_2}v >12\dot \delta$, for every $v \in \mathcal V$.
	Indeed otherwise $v$ would be a cone point distinct from $v_1$ but much closer to $v_1$ than $v_2$.
	According to \autoref{res: isom prelim},	we have $\dist{\bar x_1}{\bar x_2} = \dist{x_1}{x_2}$, hence $\dist{\bar x_1}{\bar x_2} \geq 3\rho/2$.
	Combined with the triangle inequality we obtain $\dist{\bar v_1}{\bar v_2} \geq \rho$.
\end{proof}

\begin{prop}
\label{res: action quotient non elem}
	Assume that $\bar{\mathcal V}$ contains two distinct apices whose order is at least $3$.
	Then the action of $\bar G$ on $\bar X$ is non-elementary.
\end{prop}

\begin{proof}
	Let $\bar v_1$ and $\bar v_2$ be those apices.
	We fix a point $\bar x\in\bar X$ such that $\gro{\bar v_1}{\bar v_2}{\bar x} \leq \bar \delta$ whereas $\dist{\bar x}{\bar v_1} \geq \rho$ and $\dist{\bar x}{\bar v_2} \geq \rho$.
	According to \autoref{res: local action of apex stab}\ref{enu: local action of apex stab - rot} we see that $\stab{\bar v_i}$ contains a rotation $\bar g_i$ at $\bar v_i$ such that $\fix{\bar g_i, \bar \delta}$ and $\fix{\bar g^2_i, \bar \delta}$ are contained in $B(\bar v_i,\bar \delta)$, see \autoref{fig: non-elem}.
	\begin{figure}[htbp]
    	\centering
    	\includegraphics[page=5, width=0.8\textwidth]{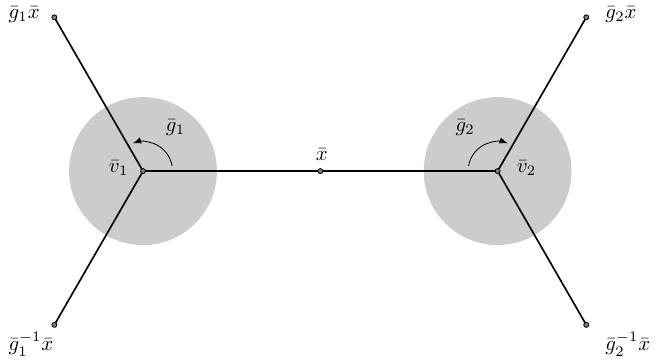}
    	\caption{
			The element $\bar g_1$ and $\bar g_2$ acting on $\bar x$.
			The shaded discs represent $B(\bar v_1, \rho)$ and $B(\bar v_2,\rho)$ respectively.
		}
    	\label{fig: non-elem}
    \end{figure}
	Applying \autoref{res: fix qc}, we get
	\begin{equation*}
		2\gro{\bar g_i\bar x}{\bar g_i^{-1}\bar x}{\bar x} \leq \dist{\bar g_i\bar x}{\bar x} + 30\bar \delta
		\quad \text{while} \quad
		 \dist{\bar g_i\bar x}{\bar x} \geq 2 \dist{\bar x}{\bar v_i} - 15\bar\delta \geq 2\rho - 15\bar \delta.
	\end{equation*}
	Combined with the four point inequality (\ref{eqn : hyp four points - 1}) it yields $\gro{\bar g_1^{\pm 1}\bar x}{\bar g_2^{\pm 1}\bar x}{\bar x} \leq 3\bar \delta$.
	It follows from \autoref{res: non elementary subgroup sufficient condition} that $\bar g_1$ and $\bar g_2$ generate a non-elementary subgroup.
	Hence $\bar G$ is non-elementary.
\end{proof}

%
\subsection{Structure of elementary subgroups}
%
\label{sec: sc - elem subgroups}

An important step to study further the action of $\bar G$ on $\bar X$ (and its invariants) is to understand the algebraic structure of its elementary subgroups.
As the map $X \to \bar X$ is $1$-Lipschitz, the projection $\pi \colon G \onto \bar G$ maps every elementary subgroup of $G$ to an elementary subgroup of $\bar G$.
More precisely it sends every elliptic (\resp parabolic, loxodromic) to an elliptic (\resp elliptic or parabolic, elementary) subgroup of $\bar G$.
However the nature of these subgroups (i.e.\ whether they are elliptic, parabolic, or loxodromic) may change.
Indeed given any element $g \in G$, there always exists $h \in K$ such that $gh$ is a loxodromic element of $G$. 
Thus $\group{gh}$ is loxodromic whereas its image $\group{\bar g}$ in $\bar G$ can be anything.
New elementary subgroups may also appear in $\bar G$.
This motivates the following definition.

\begin{defi}
\label{def: liftable subgroup}
	Let $\bar E$ be an elliptic (\resp parabolic, loxodromic) subgroup of $\bar G$.
	We say that $\bar E$ \emph{can be lifted} if there exists an elliptic (\resp parabolic, loxodromic) subgroup $E$ of $G$ (for its action on $X$) such that the quotient map $\pi \colon G \onto \bar G$ induces an isomorphism from $E$ onto $\bar E$.
	In this situation $E$ is a \emph{lift} of $\bar E$.
\end{defi}

Note that in this definition we ask $E$ and $\bar E$ to have the same nature.
The idea is that the subgroups of $\bar G$ that can be lifted are as \og easy \fg\ as the elementary subgroups of $G$.
Complicated algebraic structures necessarily come for the \og new \fg\ elementary subgroups.
In the next paragraphs we discuss whether an elementary subgroup of $\bar G$ can be lifted.
If not we use the geometry of small cancellation theory to describe its properties.

\paragraph{Comparing lifts.}
We start by proving that the lift of an elliptic subgroup of $\bar G$ is essentially unique.
More precisely, if $F_1$ and $F_2$ are two lifts of the same elliptic subgroup $\bar F\subset \bar G$, then $F_1$ and $F_2$ are conjugated.
This result is a particular case of a more general statement (see \autoref{res: uniqueness elliptic lift}) that allows to consider simultaneously elliptic and parabolic subgroups.

\begin{lemm}
\label{res: isom on non-loxo elem}
	Let $S$ be a subset of $G$ such that $\fix{S, \rho/10} \subset X$ is non-empty.
	Then the projection $\pi \colon G \onto \bar G$ is one-to-one when restricted to $S$.
	In particular, if $E$ is an elliptic or a parabolic subgroup (for its action on $X$), the projection $\pi$ induces an isomorphism from $E$ onto its image.
\end{lemm}

\begin{proof}
	The first part of the statement is a consequence of \autoref{res: small cancellation}\ref{enu: small cancellation - translation kernel}.
	If $E$ is elliptic or parabolic, then $\fix{\{1, g\},\rho/10}$ is non-empty for every $g \in E$ -- see for instance (\ref{eqn: regular vs stable length}).
	Hence the result.
\end{proof}

\begin{prop}
\label{res: lifting partial auto}
	Let $E$ be an elliptic or a parabolic subgroup of $G$ (for its action on $X$).
	Let $S_1$ be a subset of $E$ such that $\fix{S_1,\rho/100} \subset X$ is non-empty and $\bar S_1$ its image in $\bar G$.
	Let $\bar h \in \bar G$.
	Let $S_2$ be a pre-image in $G$ of $\bar h \bar S_1 \bar h^{-1}$ such that $\fix{S_2,\rho/100}$ is non-empty.
	Then there exists $h_0 \in G$ with the following properties
	\begin{enumerate}
		\item \label{res: lifting partial auto - lifting}
		For every $g \in S_1$, the element $h_0gh_0^{-1}$ is the (unique) pre-image of $\bar h \bar g \bar h^{-1}$ in $S_2$.
		\item \label{res: lifting partial auto - loxo}
		If $\bar h$ is loxodromic, then either $h_0$ is loxodromic, or $\bar S_1$ is contained in a reflection group at some vertex $\bar v \in \bar{\mathcal V}$.
	\end{enumerate}
\end{prop}

\begin{rema*}
	Item~\ref{res: lifting partial auto - lifting} tells us that the action of $\bar h$ on $\bar S_1$ coincide with the one of $h_0$ on $S_1$.
	Nevertheless $h_0$ is not necessarily the pre-image of $\bar h$.
	As we will see during the proof $\bar h_0$ and $\bar h$ may differ by a product of central half turns each of which centralizes $\bar S_1$.
	This situation is unavoidable.
	Suppose indeed that $\bar S_1$ forms a subgroup normalized by $\bar h$, which we assume to be loxodromic, so that $\bar E = \group{\bar S_1, \bar h}$ is a loxodromic subgroup of $\bar G$.
	If $h_0$ were to be a lift of $\bar h$, then $E = \group{S_1, h_0}$ would be a loxodromic subgroup of $G$ lifting $\bar E$.  
	Such a lift does not always exist though, see Examples~\ref{exa: lifting loxodromic sbgp w/ liftable elliptic} and \ref{exa: intro growing elliptic - continued} below.
\end{rema*}

\begin{proof}
	Let $h \in G$ be an arbitrary pre-image of $\bar h$.
	We fix two points $x_1, x_2 \in X$ lying respectively in $\fix{S_1, \rho/100}$ and $\fix{S_2, \rho/100}$.
	Note that both $\bar x_1$ and $\bar h^{-1} \bar x_2$ belongs to $\fix{\bar S_1, \rho/100}$.
	We claim that $\bar S_1 \cap \stab{\bar v}$ is contained in a reflection group at $\bar v$, for every $\bar v \in \bar{\mathcal V}$.
	Assume on the contrary that it is not the case.
	There exists $g \in E$ whose image $\bar g$ is a strict rotation (\autoref{rem: assumption lifting crossing apices - w/ reflection}).
	According to \autoref{res: local action of apex stab}\ref{enu: local action of apex stab - rot} there exists $k \in \N$ such that $\fix{\bar g^k,\bar \delta}$ is contained in $B(\bar v, \bar \delta)$.
	On the other hand, since $g$ belongs to $E$, the element $g^k$ is elliptic or parabolic (as an isometry of $X$).
	Hence there exists $x \in X$ such that $\dist{g^kx}x \leq 10 \delta$ -- see for instance (\ref{eqn: regular vs stable length}) -- and thus $\dist{\bar g^k\bar x}{\bar x} \leq \bar \delta$.
	This contradicts the previous point and completes the proof of our claim.
	It follows from \autoref{res: lifting crossing apices - w/ reflection} applied with $x = x_1$ and $y = h^{-1}x_2$ that there exists $u \in G$, such that $\bar u$ centralizes $\bar S_1$ and 	
	\begin{equation}
	\label{eqn: lifting partial auto}
		\dist[\dot X]{guh^{-1}x_2}{uh^{-1}x_2} = \dist{\bar g\bar h^{-1}\bar x_2}{\bar h^{-1}\bar x_2}, \quad \forall g \in S_1.
	\end{equation}
	Moreover either $\bar u$ is trivial or $\bar S_1$ lies in a reflection group.
	We let $h_0 = hu^{-1}$.
	Let $g \in S_1$ and $g'$ the (unique) pre-image of $\bar h \bar g \bar h^{-1}$ in $S_2$.
	It follows from (\ref{eqn: lifting partial auto}) that $h_0gh_0^{-1}$ and $g'$ are two pre-images of $\bar h \bar g\bar h^{-1}$ that move $x_2$ by at most $\rho/100$.
	Thus by \autoref{res: small cancellation}\ref{enu: small cancellation - translation kernel} $h_0gh_0^{-1}= g'$, which proves \ref{res: lifting partial auto - lifting}.
	Assume now that $\bar h$ is loxodromic.
	If $\bar u$ is trivial, then $h_0$ is a pre-image of $\bar h$, hence a loxodromic element (recall that $\zeta \colon X \to \bar X$ is $1$-Lipschitz).
	On the contrary if $\bar u$ is not trivial, then $\bar S_1$ is contained in a reflection group.
\end{proof}

\begin{coro}
\label{res: uniqueness elliptic lift}
	Let $F_1$ and $F_2$ be two subgroups of $G$.
	We assume that $F_1$ is elliptic and $F_2$ generated by a set $S_2$ such that $\fix{S_2,\rho/100}$ is non-empty.
	Let $\bar F_1$ and $\bar F_2$ be their respective images in $\bar G$.
	If $\bar F_1 = \bar F_2$, then there exists $u \in G$ whose image in $\bar G$ centralizes $\bar F_1$ and such that $F_2 = uF_1u^{-1}$.
\end{coro}

\begin{rema*}
	Note that the assumption on $F_2$ is automatically satisfied if $F_2$ is elliptic.
	In particular, if $F_1$ and $F_2$ are two elliptic subgroups of $G$ whose images in $\bar G$ coincide, then they are conjugate.
\end{rema*}

\begin{proof}
	Let $\bar S$ be the image of $S_2$ in $\bar F_1 = \bar F_2$ and $S_1$ the pre-image of $\bar S$ in $F_1$.
	Note that $\fix{S_1, \rho/100}$ is non-empty (\autoref{res: fix set elliptic}).
	According to \autoref{res: lifting partial auto} applied with $S_1$ and $S_2$, there exits $u \in G$ such that for every $s \in S_1$, the element $usu^{-1}$ is the pre-image of $\bar s$ in $S_2$.
	In particular, $\bar u$ commutes with $\bar S$, hence $\bar F_1$.
	Moreover since $S_2$ generates $F_2$, the group $u^{-1}F_2u$ is contained in $F_1$.
	Nevertheless the projection $\pi \colon G \to \bar G$ is one-to-one when restricted to $F_1$ (\autoref{res: isom on non-loxo elem}).
	Thus $F_2 = uF_1u^{-1}$.
\end{proof}

\begin{coro}
\label{res: lifting conj elliptic}
	Let $F_1$ and $F_2$ be two subgroups of $G$.
	We assume that $F_1$ is elliptic and $F_2$ generated by a set $S_2$ such that $\fix{S_2,\rho/100}$ is non-empty.
	Let $\bar F_1$ and $\bar F_2$ be their respective images in $\bar G$.
	If $\bar F_1$ and $\bar F_2$ are conjugated in $\bar G$, then so are $F_1$ and $F_2$ in $G$.
\end{coro}

\paragraph{Lifting elliptic subgroups.}
We now characterize the elliptic subgroups of $\bar G$ that can be lifted and explore the structure of the one that cannot be lifted.
The statement generalizes \cite[Lemme~5.10.2]{Delzant:2008tu}.
\begin{prop}[Lifting elliptic subgroups]
\label{res: dichotomy elliptic}	
	An elliptic subgroup $\bar F$ of $\bar G$ cannot be lifted if and only if it contains a strict rotation.
	In this case, $\bar F$ fixes an apex $\bar v \in \bar{\mathcal V}$. 
	Moreover, $\fix{\bar F,\bar \delta}$  is contained in $B( \bar v, \bar \delta)$.
	In particular, $\bar v$ is the only apex fixed by $\bar F$.
\end{prop}

\begin{proof}
	Recall that $\fix{\bar F,10 \bar \delta}$ is a non-empty $8\bar\delta$-quasi-convex subset of $\bar X$ (Lemmas~\ref{res: fix set elliptic} and \ref{res: fix qc}).
	Assume first that there exists a point $\bar x \in \fix{\bar F,\bar 10\delta}$ such that $d(\bar x, \bar {\mathcal V}) \geq \rho/3$.
	By \autoref{res: local action of apex stab}\ref{enu: local action of apex stab - rot}, $\bar F$ does not contain a strict rotation.
	We are going to prove that $\bar F$ can be lifted.
	We write $\bar Z$ for the $\bar F$-orbit of $\bar x$.
	It is $\bar F$-invariant and its diameter is at most $10 \bar \delta$.
	It follows that $\gro{\bar z}{\bar z'}{\bar v} \geq \rho/4$, for every $\bar z,\bar z' \in \bar Z$, for every $\bar v \in \bar {\mathcal V}$.
	According to Lemmas~\ref{res: lifting quasi-convex} and \ref{res: zeta isom on qc far from apices}, there exist a subgroup $F$ of $G$ and an $F$-invariant subset $Z$ of $\dot X$ with the following properties: 
	the map $\zeta \colon \dot X \to \bar X$ induces an isometry from $Z$ onto $\bar Z$; 
	the projection $\pi \colon G \to \bar G$ induces an isomorphism from $F$ onto $\bar F$;
	moreover $\gro z{z'}v > 13\dot \delta$ for every $z,z' \in Z$, for every $v \in \mathcal V$ (the Gromov product is computed in $\dot X$ here).
	In particular, $Z$ is bounded.
	It follows from \autoref{res: radial proj qi} that the radial projection $p(Z)$ is a bounded $F$-invariant subset of $X$.
	Hence $F$ is an elliptic subgroup of $G$ (for its action on $X$), lifting $\bar F$.
	
	Assume now that $d(\bar x, \bar {\mathcal V}) < \rho/3$ for every $\bar x \in \fix{\bar F, 10 \bar \delta}$.
	Since the set $\fix{\bar F, 10 \bar \delta}$ is $8\bar \delta$-quasi-convex, there exists $\bar v \in \bar {\mathcal V}$ such that $\fix{\bar F, 10 \bar \delta}$ is contained in $B(\bar v, \rho/3)$.
	In particular, $\bar F$ fixes $\bar v$.
	In addition, $\bar F$ cannot be lifted.
	Indeed if $F$ was a lift of $\bar F$, then the image in $\bar X$ of $\fix{F, 10\delta} \subset X$ would be contained in $\fix{\bar F, 10 \bar \delta}\setminus B(\bar v, \rho/3)$.
	We now claim that $\bar F$ contains a strict rotation $\bar g$ at $\bar v$.
	If it was not the case, then $\bar F$ would be contained in a reflection group at $\bar v$.
	Thus, there would exist a point $\bar x' \in \fix{\bar F, 10\bar \delta}$ such that $\dist {\bar x}{\bar v} > \rho/2$, see \autoref{res: local action of apex stab}\ref{enu: local action of apex stab - reflection},  which contradicts the previous observation.
	It follows then from \autoref{res: local action of apex stab}\ref{enu: local action of apex stab - rot} that $\fix{\bar F, \bar \delta}$  is contained in $B(\bar v,\bar \delta)$.
\end{proof}

\begin{coro}
\label{res: lifting elliptic in apex group}
	Let $(H,Y) \in \mathcal Q$.
	Let $\bar v$ be the image in $\bar X$ of the apex $v$ of $Z(Y)$.
	Let $\bar C$ be a subgroup of $\stab{\bar v}$.
	If $\bar  C$ can be lifted, then it admits a lift which is contained in $\stab Y$.
\end{coro}

\begin{proof}
	Let $F$ be the maximal elliptic normal subgroup of $\stab Y$ and $\bar F$ its image in $\bar G$.
	Recall that we have the following commutative diagram
	\begin{equation*}
		\begin{tikzpicture}
			\matrix (m) [matrix of math nodes, row sep=2em, column sep=2.5em, text height=1.5ex, text depth=0.25ex] 
			{ 
				1	& F & \stab Y & \mathbf L & 1	\\
				1	& \bar F & \stab{\bar v} & \mathbf L_n & 1	\\
			}; 
			\draw[>=stealth, ->] (m-1-1) -- (m-1-2);
			\draw[>=stealth, ->] (m-1-2) -- (m-1-3);
			\draw[>=stealth, ->] (m-1-3) -- (m-1-4);
			\draw[>=stealth, ->] (m-1-4) -- (m-1-5);
			
			\draw[>=stealth, ->] (m-2-1) -- (m-2-2);
			\draw[>=stealth, ->] (m-2-2) -- (m-2-3);
			\draw[>=stealth, ->] (m-2-3) -- (m-2-4) node[pos=0.5, above]{$q_{\bar v}$};
			\draw[>=stealth, ->] (m-2-4) -- (m-2-5);
			
			\path[->] (m-1-2) edge node[above,sloped,inner sep=0.5pt]{$\sim$} (m-2-2);

			\draw[>=stealth, ->] (m-1-3) -- (m-2-3) node[pos=0.5, right]{$\pi$};
			\draw[>=stealth, ->] (m-1-4) -- (m-2-4);
		\end{tikzpicture} 
	\end{equation*} 
	where $(\mathbf L, \mathbf L_n)$ is either $(\cyclic, \cyclic[n])$ or $(\dihedral, \dihedral[n])$.
	Since $\bar C$ can be lifted, its image under $q_{\bar v}$ does not contain a non-trivial rotation (\autoref{res: dichotomy elliptic}).
	The result follows from diagram chasing.
\end{proof}

We continue with a study of dihedral germs (see \autoref{def: dihedral germ}).
Let $F$ be an elliptic subgroup of $G$ and $\bar F$ its image in $\bar G$.
Observe that if $F$ is a dihedral germ, then the same does not necessarily hold for $\bar F$.
Indeed it may happen that the only loxodromic elements that are normalizing (a finite index subgroup of) $F$ became elliptic in $\bar G$.
Nevertheless the converse statement holds.
This is the aim of the next lemmas.

\begin{lemm}
\label{res: dihedral germ in relation group}
	Let $(H,Y) \in \mathcal Q$.
	Every elliptic subgroup of $\stab Y$ is a dihedral germ.
\end{lemm}

\begin{proof}
	For simplicity we let $E = \stab Y$.
	Let $E^+$ be the subgroup of $E$ fixing pointwise $\partial Y$ and $F$ the maximal elliptic subgroup of $E^+$.
	Let $C$ be an elliptic subgroup of $E$.
	The intersection $C_0 = C \cap E^+$ is a subgroup of $C$ with index at most $2$.
	Let $h$ be a (loxodromic) element in $H$.
	Let $c \in C_0$.
	It follows from the small cancellation assumption that $H$ is a normal subgroup of $E$.
	In particular, $chc^{-1} = h^k$ for some $k \in \Z$.
	Recall that $E^+/F$ is isomorphic to $\cyclic$.
	Pushing the previous identity in $\cyclic$, we get that $k = 1$.
	In other words $h$ commutes with $C_0$.
	Hence $C$ is a dihedral germ.
\end{proof}

\begin{lemm}
\label{res: reflection group provides germs}
	Let $C$ be an elliptic subgroup of $G$ (for its action on $X$).
	Let $\bar v \in \bar{\mathcal V}$.
	If the image of $C$ in $\bar G$ is contained in a reflection group at $\bar v$, then $C$ is a dihedral germ.
\end{lemm}

\begin{proof}
	Let $\bar C$ be the image of $C$ in $\bar G$.
	Let $(H,Y) \in \mathcal Q$ such that the apex $v$ of the cone $Z(Y)$ is a pre-image of $\bar v$.
	There exists an elliptic subgroup $C'$ of $\stab Y$ such that the projection $\pi \colon G \to \bar G$ maps $C'$ onto $\bar C$ (\autoref{res: lifting elliptic in apex group}).
	In other words $C$ and $C'$ are two lifts of $\bar C$, hence they are conjugated (\autoref{res: uniqueness elliptic lift}).
	Being a dihedral germ is invariant under conjugacy.
	Thus the conclusion follows from \autoref{res: dihedral germ in relation group}.
\end{proof}

\begin{lemm}[Lifting dihedral germs]
\label{res: lifting dihedral germ}
	Let $C$ be an elliptic subgroup of $G$ (for its action on $X$) and $\bar C$ its image in $\bar G$.
	If $\bar C$ is a dihedral germ, then so is $C$.
\end{lemm}

\begin{proof}
	By assumption there exists a subgroup $\bar C_0$ of $\bar C$ which is normalized by a loxodromic element, say $\bar h$, and such that $[\bar C: \bar C_0] = 2^k$ for some $k \in \N$.
	Since $C$ is elliptic, the projection $G \onto \bar G$ induces an isomorphism from $C$ onto $\bar C$ (\autoref{res: isom on non-loxo elem}).
	We write $C_0$ for the isomorphic pre-image of $\bar C_0$ in $C$.
	Note that $[C:C_0] = 2^k$.
	It follows from \autoref{res: lifting partial auto} applied with $S_1 = S_2 = C_0$ that there exists $h_0 \in G$ normalizing $C_0$.
	Moreover either $h_0$ is loxodromic or $\bar C_0$ is contained in a reflection group at some apex $\bar v \in \bar{\mathcal V}$.
	If $h_0$ is loxodromic, then $C$ is automatically a dihedral germ.
	Assume now that $\bar C_0$ is contained in a reflection group at $\bar v$.
	It follows from \autoref{res: reflection group provides germs} that $C_0$ is a dihedral germ.
	Hence it contains a subgroup $C_1$ which is normalized by a loxodromic element of $G$ and such that $[C_0:C_1]=2^m$ for some $m \in \N$.
	Thus $[C:C_1]= 2^{k+m}$ and $C$ is a dihedral germ.
\end{proof}	

The next lemma is formally not needed.
However it illustrates the role played by dihedral germs.
As we observed earlier, every elliptic subgroup $F$ of $G$ yields an elliptic subgroup $\bar F$ of $\bar G$.
However it could happen that $\bar F$ is strictly contained in an elliptic subgroup, which does not already come from a subgroup of $G$ containing $F$.
In this case, $F$ is necessarily a dihedral germ.
As suggested by the name, dihedral germs are exactly the elliptic subgroups of $G$ which can eventually \og grow \fg\ when passing to the quotient $\bar G$.

\begin{lemm}
\label{res: only dihedral germ can grow}
	Let $C$ be an elliptic subgroup of $G$ (for its action on $X$) and $\bar C$ its image in $\bar G$.
	Assume that there exists an elliptic subgroup $\bar A$ containing $\bar C$ which cannot be lifted.
	Then $C$ is a dihedral germ.
\end{lemm}

\begin{proof}
	This is just a reformulation of \autoref{res: reflection group provides germs}.
	Indeed according to \autoref{res: dichotomy elliptic}, there exists $\bar v \in \bar{\mathcal V}$ such that $\bar A$ is contained in $\stab{\bar v}$.
	Since $\bar C$ can be lifted, it does not contain a strict rotation at $\bar v$ (\autoref{res: dichotomy elliptic}) and thus lies in a reflection group at $\bar v$.
\end{proof}

We complete our discussion on elliptic subgroups with some preparatory work for the study of loxodromic subgroups.
If such a group $\bar E$ does not preserve the orientation, it can be decomposed as $\bar E = \bar A \ast_{\bar C} \bar B$, where $\bar C$ has index $2$ in both $\bar A$ and $\bar B$.
As the cylinder of $\bar E$ is contained in $\fix{\bar C, 100 \bar \delta}$ (\autoref{res: normal elliptic fixing cyl}), $\bar C$ can also be lifted.
We describe in this context what is the structure of $\bar A$ or $\bar B$.

\begin{lemm}
\label{res: lifting subgroup of index 2}
	Let $\bar A$ be an elliptic subgroup of $\bar G$.
	Assume that $\bar A$ contains a subgroup $\bar C$ of index $2$ that can be lifted.
	Let $\bar a \in \bar A \setminus \bar C$.
	Then there exists $\bar u \in \bar G$ such that 
	\begin{enumerate}
		\item $\group{\bar C, \bar u}$ is an elliptic subgroup that can be lifted;
		\item $\bar a^{-1} \bar u$ centralizes $\bar C$; and
		\item $\bar a^2 = \bar u^2$.
	\end{enumerate}
\end{lemm}

\begin{rema*}
	Observe that if $\bar u$ is trivial, then $\bar A$ is isomorphic to $\bar C \times \group{\bar a} = \bar C \times \cyclic[2]$.
	In general, the map $\bar a \mapsto \bar u$ extends to an (abstract) epimorphism from $\bar A$ onto $\group{\bar C,\bar u}$.
\end{rema*}

\begin{proof}
	If $\bar A$ can be lifted, then the statement obviously holds.
	Assume now that $\bar A$ cannot be lifted.
	There exists $\bar v \in \bar {\mathcal V}$ such that $\bar A$ is contained in $\stab {\bar v}$ (\autoref{res: dichotomy elliptic}).
	Let $q_{\bar v} \colon \stab{\bar v} \to \dihedral[n]$ be corresponding the geometric realization. 
	Since $\bar C$ can be lifted, $q_{\bar v}(\bar C)$ is either trivial or equal to $\group {\mathbf x}$ where $\mathbf x$ is a reflection of $\dihedral[n]$.
	Let $\mathbf t$ be a generator of the rotation group $\cyclic[n] \subset \dihedral[n]$.
	Recall that $\bar C$ has index $2$ in $\bar A$ and $\bar A$ cannot be lifted in $G$.
	It follows that $n$ is even and 
	\begin{equation*}
		q_{\bar v}(\bar A) = \group{q_{\bar v}(\bar C), \mathbf t^{n/2}}.
	\end{equation*}
	In particular, $q_{\bar v}(\bar a)\mathbf t^{n/2}$ belongs to $q_{\bar v}(\bar C)$.
	According to \autoref{ass: even exponent}, $\stab{\bar v}$ contains a central half-turn at $\bar v$ that we denote by $\bar g$.
	Note that $q_{\bar v}$ maps $\bar g$ to $\mathbf t^{n/2}$.
	We let $\bar u = \bar a\bar g$.
	We observe that $\bar a^{-1} \bar u = \bar g$ centralizes $\bar C$ and $\bar a^2 = \bar u^2$.
	By construction $\group{\bar C, \bar u}$ and $\bar C$ have the same image under $q_{\bar v}$.
	Consequently $\group{\bar C, \bar u}$ is contained in a reflection group at $\bar v$, hence can be lifted, which completes the proof.
\end{proof}

Another crucial ingredient to describe loxodromic subgroups of $\bar G$, is to understand the normalizer of elliptic subgroups that can be lifted, see for instance \autoref{res: lifting loxodromic sbgp}.
This is the purpose of the next proposition.

\begin{prop}[Lifting normalizer]
\label{res: lifting automorphism}
	Let $F$ be an elliptic subgroup of $G$ and $\bar F$ its image in $\bar G$.
	For every $\bar h \in \Norm{\bar F}$ there exists $h_0 \in \Norm F$ such that $\bar h^{-1} \bar h_0$ centralizes $\bar F$.
	If in addition $\bar h^2$ belongs to $\bar F$, then one can choose $h_0$ such that $h_0^2 \in F$ and $\bar h^2 = \bar h_0^2$.
\end{prop}

\begin{rema}
\label{rem: lifting automorphism}
	The result can be reformulated in the following way.
	Let $N(F)$ and $C(F)$ be the respective normalizer and centralizer of $F$ in $G$.
	We define $N(\bar F)$ and $C(\bar F)$ in the same way.
	The projection $\pi \colon G \onto \bar G$ does not necessarily map $N(F)$ onto $N(\bar F)$.
	Nevertheless it induces an epimorphism from $N(F)/C(F)$ onto $N(\bar F)/C(\bar F)$.
	Actually \autoref{res: isom on non-loxo elem} implies that this map is an isomorphism, but we will not need this fact here.
\end{rema}

\begin{proof}
	Applying \autoref{res: lifting partial auto} with $S_1 = S_2 = F$, we see that there exists $h_0 \in G$ such that for every $g \in F$, the element $h_0gh_0^{-1}$ is the pre-image of $\bar h \bar g \bar h^{-1}$ in $F$.
	In particular, $h_0$ normalizes $F$ and $\bar h^{-1}\bar h_0$ centralizes $\bar F$.
	
	Let us now focus on the second part of the statement, which requires a different construction of $h_0$.
	We assume that $\bar h \in \bar G$ normalizes $\bar F$ and $\bar h^2 \in \bar F$.
	In particular, $\bar F' = \group{\bar F, \bar h}$ is an elliptic subgroup of $\bar G$.
	Observe that there exists $\bar h_1 \in \bar G$ such that $\bar h ^{-1} \bar h_1$ centralizes $\bar F$, $\bar h^2 = \bar h_1^2$, and $\group{\bar F, \bar h_1}$ is an elliptic subgroup that can be lifted.
	Indeed, if $\group{\bar F, \bar h}$ can be lifted, it suffices to take $\bar h_1 = \bar h$, otherwise the conclusion follows from \autoref{res: lifting subgroup of index 2}.
	Let $F'_1$ be a lift of $\bar F'_1 = \group{\bar F, \bar h_1}$ and $h_1$ the pre-image of $\bar h_1$ in $F'_1$.
	There exists $u \in G$ such that $\bar u$ centralizes $\bar F$ and $uFu^{-1}$ is the pre-image of $\bar F$ in $F'_1$ (\autoref{res: uniqueness elliptic lift}).
	We choose $h_0 = u^{-1}h_1u$.
	By construction $h_1$ normalizes $uFu^{-1}$ and $h_1^2$ belongs to $uFu^{-1}$.
	Consequently $h_0$ normalizes $F$ and $h_0^2$ belongs to $F$.
	Since $\bar u$ centralizes $\bar F$, it commutes with $\bar h_1^2 = \bar h^2$, thus $\bar h^2 = \bar h_0^2$.
	Moreover $\bar h_1^{-1} \bar h_0$ centralizes $\bar F$, hence so does $\bar h^{-1}\bar h_0$.
\end{proof}

\paragraph{Lifting parabolic subgroups.}
The first result is not needed for the rest of the study. 
However, we believe that it may help the reader by clarifying the structure of parabolic subgroups in $G$.
As we mentioned earlier, since $\zeta \colon X \to \bar X$ is $1$-Lipschitz, the image by the projection $\pi \colon G \onto\bar G$ of a parabolic subgroup of $G$ is either elliptic or parabolic.
The next statement tells us that the former case does not happen.

\begin{lemm}
\label{res: image parabolic subgroup}
	Let $E$ be an elementary subgroup of $G$ and $\bar E$ its image in $\bar G$.
	Assume that there exist $d \in \R_+$ and a subset $S$ generating $E$ such that $\fix{S,d}$ is non-empty.
	If $E$ is parabolic (for its action on $X$) then $\bar E$ is parabolic (for its action on $\bar X$).
\end{lemm}

\begin{rema*}
	Note that the assumption automatically holds if $E$ is finitely generated.
\end{rema*}

\begin{proof}
	Since $E$ is parabolic, it has a unique fixed point in $\partial X$.
	Thus according to \autoref{res: isom fixing xi moving geo} we can assume that $\fix{S, \rho/100}$ is non-empty.
	We just need to prove that $\bar E$ cannot be elliptic.
	Assume on the contrary that it is.
	We distinguish two cases. 
	Suppose first that $\bar E$ can be lifted and let $E'$ be a lift of $\bar E$ (recall that by definition $E'$ is elliptic).
	Applying \autoref{res: uniqueness elliptic lift} with $S_2 = S$, we get that $E$ and $E'$ are conjugate, which contradicts the fact that $E$ and $E'$ have different nature.
	Suppose now that $\bar E$ cannot be lifted.
	In particular, there exists $g \in E$ whose image $\bar g$ in $\bar E$ is a strict rotation (\autoref{res: dichotomy elliptic}).
	By \autoref{res: local action of apex stab}\ref{enu: local action of apex stab - rot}, there exists $k \in \Z$ such that $\fix{\bar g^k, \bar \delta}$ is contained in $B(\bar v, \bar \delta)$ for some apex $\bar v \in \bar{\mathcal V}$.
	On the other hand, $g^k$ cannot be loxodromic, hence $\norm[X]{g^k} \leq 10 \delta$.
	Thus there exists $x \in X$ such that $\dist{g^kx}x \leq 10\delta$, thus $\bar x$ belongs to $\fix{\bar g^k, \bar \delta}\setminus B(\bar v, \bar \delta)$, which yields another contradiction.
\end{proof}

\begin{prop}[Lifting parabolic subgroups]
\label{res: lifting parabolic subgroups}
	Let $\bar P$ be a parabolic subgroup of $\bar G$.
	Assume that there exist $d \in \R_+$ and a subset $\bar S$ generating $\bar P$ such that $\fix{\bar S, d}$ is non-empty.
	Then $\bar P$ can be lifted.
\end{prop}

\begin{proof}
	Let $\bar \xi$ be the unique point of $\partial \bar X$ fixed by $\bar P$ and $\bar x$ be a point in the set $\fix{\bar S, d}$.
	Let $L > 100\bar \delta$ and  $\bar \gamma \colon \R_+ \to \bar X$ be an $L$-local $(1, 11\bar \delta)$-quasi-geodesic ray starting at $\bar x$ whose endpoint at infinity is $\bar \xi$.
	According to \autoref{res: isom fixing xi moving geo}, there exists $t_0$ such that for every $t \geq t_0$, for every $g \in S$, we have $\dist{\bar g \bar \gamma(t)}{\bar \gamma(t)} \leq 47\bar \delta$.
	It follows that $d(\bar \gamma(t),\bar {\mathcal V}) \geq \rho/2$, for every $t \geq t_0$.
	Indeed otherwise, there would exist $\bar v \in \bar{\mathcal V}$ such that $\bar S$, and thus $\bar P$, is contained in $\stab{\bar v}$, which contradicts our assumption.

	Let $\bar Z$ be the $\rho/10$-neighborhood of $\bar \gamma$ restricted to $[t_0, \infty)$.
	It is a $2\bar \delta$-quasi-convex subset.
	Observe that $\gro{\bar z}{\bar z'}{\bar v} > \rho/10$ for every $\bar z,\bar z' \in \bar Z$ and $\bar v \in \bar{\mathcal V}$.
	Indeed, as $\bar Z$ is $2\bar \delta$-quasi-convex, it is a consequence of our previous observation saying that $d(\bar \gamma(t),\bar {\mathcal V}) \geq \rho/2$, for every $t \geq t_0$, together with the triangle inequality.
	It follows from \autoref{res: lifting quasi-convex} that there exist a subset $Z$ of $\dot X$ such that the map $\zeta \colon \dot X \to \bar X$ induces an isometry from $Z$ onto $\bar Z$.
	Moreover $\gro z{z'}v > \rho/10$ for every $z,z' \in Z$ and $v \in \mathcal V$ (the Gromov product are computed in $\dot X$ here).
	In particular, there exists an $L$-local $(1, 11\dot \delta)$-quasi-geodesic ray $\gamma \colon [t_0, \infty) \to \dot X$ contained in $Z$ such that $\zeta \circ \gamma = \bar \gamma$ and $Z$ is its $\rho/10$ neighborhood.
	We write $\xi \in \partial \dot X$ for the endpoint  at infinity of $\gamma$.

	We now claim that $\bar P$ is contained in the image of $\stab\xi$ by the projection $\pi \colon G \to \bar G$.
	Let $\bar g \in \bar P$.
	As we observed, $\dist{\bar g\bar \gamma(t)}{\bar \gamma(t)} \leq 47\bar \delta$, for every $t \geq t_0$.
	It follows from \autoref{res: zeta isom on qc far from apices} applied to $Z$ that there exists a pre-image $g \in G$ of $\bar g$ such that for every $t \geq t_0$, we have $\dist[\dot X]{g\gamma(t)}{\gamma(t)} \leq 47\dot \delta$.
	In particular, $g$ fixes $\xi$, which completes the proof our first claim.
	
	We denote by $P$ the pre-image of $\bar P$ in $\stab \xi$.
	We now claim that $P$ is parabolic for its action on $X$.
	To that end, it suffices to show that $P$ is parabolic for its action on $\dot X$ (\autoref{res: parabolic cone-off}).
	We are going to prove that $P$ is not elliptic and does not contain a loxodromic element.
	As $\bar P$ is parabolic, it has unbounded orbits.
	The map $\zeta \colon \dot X \to \bar X$ being $1$-Lipschitz, the group $P$ has unbounded orbits, and thus cannot be elliptic (for its action on $\dot X$).
	Let $g \in P$ and $\bar g$ its image in $\bar P$.
	According to \autoref{res: isom fixing xi moving geo} there exist $t_1 \geq t_0$ and $\epsilon \in \{\pm 1\}$ such that for every $t \geq t_1$ we have $\dist[\dot X]{g\gamma(t)}{\gamma(t + \epsilon \snorm[\dot X]g)} \leq 47 \dot \delta$.	
	In particular, if $t$ is sufficiently large both $\gamma(t)$ and $g\gamma(t)$ belong to $Z$.
	Since $\zeta \colon \dot X \to \bar X$ is an isometry when restricted to $Z$, we get
	\begin{equation*}
		\norm[\dot X] g \leq \dist[\dot X]{g\gamma(t)}{\gamma(t)} \leq \dist{\bar g \bar \gamma(t)}{\bar \gamma(t)} \leq 47 \bar \delta.
	\end{equation*}
	This inequality holds for every $g \in P$.
	Therefore $P$ cannot contain a loxodromic element for its action on $\dot X$, which completes the proof of our second claim.
	As $P$ is a parabolic subgroup, the projection $\pi \colon G \to \bar G$ is one-to-one when restricted to $P$ (\autoref{res: isom on non-loxo elem}).
	Hence $P$ is a lift of $\bar P$.
\end{proof}

\paragraph{Lifting loxodromic subgroups.}
If $\bar G$ does not contain any element of order $2$, then one can prove that all its loxodromic subgroups can be lifted.
This is typically what is happening when studying Burnside groups of odd exponent.
As shown by the next example this is unfortunately no more the case in the presence of even torsion.

\begin{exam}
\label{exa: loxo cannot be lifted}
	Assume for instance that $\bar G$ is the group defined by 
	\begin{equation*}
		\bar G = \cyclic[n] \ast \cyclic[n]= \left< a,b \mid a^n = b^n = 1\right>.
	\end{equation*}
	If $n$ is a sufficiently large even integer, it is a small cancellation quotient of the free group $\free 2$ generated by $a$ and $b$.
	The subgroup $\group{a^{n/2},b^{n/2}} \equiv \dihedral$ is loxodromic but cannot be lifted.
\end{exam}

A more intricate configuration is described in \autoref{exa: lifting loxodromic sbgp w/ liftable elliptic}.
In general we show that every loxodromic subgroup of $\bar G$ is an (abstract) subdirect product of an elementary subgroup of $G$ and either $\cyclic$ or $\dihedral$.

\begin{prop}[Lifting loxodromic subgroups]
\label{res: lifting loxodromic sbgp}
	Let $\bar E$ be a loxodromic subgroup of $\bar G$ and $\bar F$ the maximal normal elliptic subgroup of $\bar E$.
	There exist a lift $F$ of $\bar F$, an elementary subgroup $E'$ of $G$ containing $F$, and an epimorphism $\theta \colon \bar E \onto E'$ with the following properties.
	\begin{enumerate}
		\item \label{enu: lifting loxodromic sbgp - dihedral pair}
		$(E',F)$ is a dihedral pair.
		\item \label{enu: lifting loxodromic sbgp - identity on F}
		The morphism $\theta$ restricted to $\bar F$ is a section of $\pi$ restricted to $F$.
		\item \label{enu: lifting loxodromic sbgp - embedding}
		The map $\theta$ induces an embedding from $\bar E$ into $\bar E/\bar F \times E'$.
	\end{enumerate}
\end{prop}

\begin{proof}
	Since the action of $\bar G$ on $\bar X$ is gentle (\autoref{res: sc - gentle action}), the group $\bar E$ fits in a short exact sequence
	\begin{equation*}
		1 \to \bar F \to \bar E \xrightarrow q \mathbf L \to 1,
	\end{equation*}
	where $\mathbf L$ is either $\cyclic$ or $\dihedral$.
	The cylinder $\bar Z$ of $\bar E$ is contained in $\fix{\bar F, 100\bar \delta}$ (\autoref{res: normal elliptic fixing cyl}).
	In particular, $\fix{\bar F, 100\bar \delta}\cap \zeta(X)$ is non-empty.
	It follows from \autoref{res: dichotomy elliptic} that $\bar F$ admits a lift in $G$ that we denote by $F$.
	The subgroup $\bar F$ is a dihedral germ, hence so is $F$ (\autoref{res: lifting dihedral germ}).
	We now claim that there exists an elementary subgroup $E'$ of $G$ containing $F$ as a normal subgroup such that the canonical section $\bar F \to F$ extends to an epimorphism $\theta \colon \bar E \onto E'$.
	To that end, we distinguish two cases.
	
	\paragraph{Case 1.} 
	Assume that $\bar E$ preserves the orientation, i.e.\ $\mathbf L = \cyclic$.
	Then $\bar E$ splits as a semi-direct product $\bar E = \sdp{\bar F}\mathbf L$.
	Let $\bar h$ be a primitive element of $\bar E$ (i.e.\ an element whose image under $q$ generates $\mathbf L$).
	According to \autoref{res: lifting automorphism} there exists $h_0 \in \Norm F$ such that $\bar h^{-1}\bar h_0$ centralizes $\bar F$.
	Let $E'$ be the subgroup of $G$ generated by $F$ and $h_0$.
	The canonical section $\bar F \to F$ extends to an epimorphism $\theta \colon \bar E \to E'$ sending $\bar h$ to $h_0$.
		
	\paragraph{Case 2.} 
	Assume that $\bar E$ does not preserve the orientation, so that $\mathbf L=\dihedral$.
	Let $\mathbf x_1,\mathbf x_2 \in \dihedral$ be two reflections generating $\mathbf L$.
	Let $\bar a_1, \bar a_2 \in \bar E$ be pre-images of $\mathbf x_1$ and $\mathbf x_2$ respectively.
	By construction $\bar a_i$ normalizes $\bar F$ and $\bar a_i^2$ belongs to $\bar F$.
	According to \autoref{res: lifting automorphism} there exists $b_i \in \Norm F$ with the following properties: $\bar a_i^{-1} \bar b_i$ centralizes $\bar F$; $b_i^2$ belongs to $F$; and $\bar a_i^2 = \bar b_i^2$.
	Let $E'$ be the subgroup of $G$ generated by $F$, $b_1$, and $b_2$.
	The canonical isomorphism $\bar F \to F$ extends to an epimorphism $\theta \colon \bar E \to E'$ sending $\bar a_i$ to $b_i$. 
		
	In both cases we have build the map announced in the claim.
	As $\theta$ extends the canonical section $\bar F \to F$, the composition $\pi \circ \theta$ is the identity when restricted to $\bar F$.
	Moreover $\theta$ induces an epimorphism $\bar E/\bar F \onto E'/F$.
	Any quotient of a dihedral group is still a dihedral group.
	Hence $(E', F)$ is a dihedral pair.
	In particular, $E'$ is elementary.
	One checks that the map $\bar E \to \bar E/\bar F \times E'$ induced by $\theta$ is an embedding.
\end{proof}

In the remainder of this section we revisit the previous statement and explore further the structure of certain loxodromic subgroups that cannot be lifted.
As suggested by \autoref{exa: loxo cannot be lifted} such a group $\bar E$ often does not preserves the orientation. 
In particular, it splits as $\bar E = \bar A \ast_{\bar C}\bar B$ where $\bar C$ is the maximal elliptic normal subgroup of $\bar E$ and has index $2$ in both $\bar A$ and $\bar B$.
If $\bar E$ could be lifted, then obviously so would $\bar A$ and $\bar B$.
Nevertheless the converse is false. 
This is the purpose of the first proposition.
The second one discusses the case where $\bar A$ or $\bar B$ cannot be lifted.

\begin{prop}
\label{res: lifting loxodromic sbgp w/ liftable elliptic}
	Let $A$ and $B$ be two elliptic subgroups of $G$.
	Denote by $\bar A$ and $\bar B$ their respective images in $\bar G$.
	Assume that their intersection $\bar C = \bar A \cap \bar B$ has index $2$ in both $\bar A$ and $\bar B$ so that $\bar E =\bar A \ast_{\bar C}\bar B$ is elementary.
	There exists $u \in G$ with the following properties.
	\begin{enumerate}
		\item The image $\bar u$ of $u$ in $\bar G$ centralizes $\bar C$.
		\item The subgroup $A \cap uBu^{-1}$ contains the pre-image of $\bar C$ in $A$, and $\group{A, uBu^{-1}}$ is elementary.
	\end{enumerate}
\end{prop}

\begin{proof}
	Recall that the quotient map $\pi \colon G \onto \bar G$ induces an isomorphism from $A$ and $B$ onto $\bar A$ and $\bar B$ respectively (\autoref{res: isom on non-loxo elem}).
	Let $C_A$ (\resp $C_B$) be the pre-image of $\bar C$ in $A$ (\resp $B$).
	According to \autoref{res: uniqueness elliptic lift}, there exits $u \in G$ whose image $\bar u$ in $\bar G$ centralizes $\bar C$ such that $C_A = uC_Bu^{-1}$.
	It follows that $C_A \subset A \cap uBu^{-1}$.
	Moreover this elliptic subgroup has index $2$ in both $A$ and $uBu^{-1}$, hence $E_u$ is elementary.
\end{proof}

Before moving to the case where $\bar A$ or $\bar B$ cannot be lifted, let us illustrate the previous statements with an example.

\begin{exam}
\label{exa: lifting loxodromic sbgp w/ liftable elliptic}
	Let $G$ be the group defined by 
	\begin{equation*}
		G 
		= \left< a_1,a_2, b, c \mid a_1^2, a_2^2, b^2, c^2, [a_1,c], [a_2,c] \right> 
		= \left(\dihedral \times \cyclic[2]\right) \ast_{\cyclic[2]} \dihedral.
	\end{equation*}
	acting on its Cayley graph $X$.
	In this description the elements $a_1$, $a_2$ and $c$ (\resp $b$ and $c$) generate the factor $\dihedral \times \cyclic[2]$ (\resp $\dihedral$).
	In particular, the amalgamated subgroup is $\cyclic[2] = \group c$.
	Set $s = a_1a_2$ and $r = bc$.
	If $n$ is a sufficiently large integer the group $\bar G = G/\normal{s^n, r^n}$ is a small cancellation quotient of $G$. 
	Assume in addition that $n$ is even.
	It follows that $\bar u = \bar r^{n/2}$ commutes with $\bar c$.
	Consequently the subgroup $\bar E_1$ generated by $\bar a_1$, $\bar u^{-1}\bar a_1\bar u$ and $\bar c$ is loxodromic and isomorphic to $\dihedral \times \cyclic[2]$.
	Note that $\bar E_1$ cannot be lifted in $G$.
	Indeed, the projection $G \onto \bar G$ maps the factors $\dihedral \times \cyclic[2] = \group{a_1,a_2,c}$ and $\dihedral = \group{b,c}$ to finite groups.
	Hence if $\bar E_1$ were to lift to a loxodromic subgroup of $G$, such a lift would not be conjugated into one of these factors.
	However  $\cyclic[2] = \group c$ is malnormal in $\dihedral = \group{b,c}$. 
	Thus every loxodromic subgroup of $G$, that is not conjugated into $\dihedral \times \cyclic[2] = \group{a_1,a_2,c}$ or $\dihedral = \group{b,c}$ is isomorphic to either $\cyclic$ or $\dihedral$.
	In particular, it cannot be isomorphic to $\bar E_1$.
	Observe that $\bar E_1$ also splits as $\bar E_1 = \bar A \ast_{\bar C}{\bar B_1}$ where $\bar A = \group{\bar a_1,\bar c}$, $\bar B_1 = \group{\bar u^{-1}\bar a_1\bar u, \bar c}$, and $\bar C = \group{\bar c}$.
	As $\bar u$ commutes with $\bar c$, the group $\bar B_1$ is actually $\bar B_1 = \bar u^{-1} \bar A \bar u$.
	We are in a configuration where both $\bar A$ and $\bar B_1$ are elliptic subgroups which can be lifted.
	
	As described in \autoref{res: lifting loxodromic sbgp w/ liftable elliptic}, a partial conjugation by $\bar u$ maps $\bar E_1$ to a new elementary subgroup  $\bar E'_1$ defined as $\bar E'_1 = \group{\bar A, \bar u \bar B_1 \bar u^{-1}}$ which is not necessarily loxodromic.
	In this precise example, it turns out that $\bar E'_1 = \bar A$  and thus can be lifted in $G$.
	This is not always the case though. 
	Indeed we can run the same construction with $\bar E_2 = \group{\bar a_1, \bar u^{-1}\bar a_2\bar u, \bar c}$.
	It is a loxodromic subgroup of $\bar G$ isomorphic to $\dihedral \times \cyclic[2]$.
	On the other hand, it splits as $\bar E_2 = \bar A*_{\bar C}\bar B_2$ where $\bar B_2 = \group{\bar u^{-1}\bar a_2\bar u, \bar c}$.
	In this case the group $\bar E'_2 = \group{\bar A, \bar u \bar B_2 \bar u^{-1}}$ is $\group{\bar a_1, \bar a_2, \bar c}$ hence an elliptic subgroup of $\bar G$, which is isomorphic to $\dihedral[n] \times \cyclic[2]$, and thus \emph{cannot} be lifted.
	Nevertheless, in both cases, $\bar E'_1$ or $\bar E'_2$ is the image (by the natural quotient map) of an \emph{elementary} subgroup of $G$, which was not the case for $\bar E_1$ or $\bar E_2$.
	
	This examples illustrates \autoref{res: lifting loxodromic sbgp} as well.
	Set $E'_1 = \group{a_1, c}$, $E'_2 = \group{a_1, a_2, c}$ and $C = \group{c}$.
	By construction, $(E'_1, C)$ and $(E'_2,C)$ are dihedral pairs with $E'_1$ elliptic and $E'_2$ loxodromic.
	Since $\bar u$ commutes with $\bar C$, we can define two maps
	\begin{align*}
		\theta_1 \colon \bar E_1 \to E'_1, &\quad \text{with}  \quad \theta_1(\bar a_1) = a_1, \ \quad\theta_1(\bar u^{-1} \bar a_1 \bar u) = a_1, \quad \text{and} \quad \theta_1(\bar c) = c \\
		\theta_2 \colon \bar E_2 \to E'_2, &\quad \text{with}  \quad \theta_2(\bar a_1) = a_1, \ \quad\theta_2(\bar u^{-1} \bar a_2 \bar u) = a_2, \quad \text{and} \quad \theta_2(\bar c) = c
	\end{align*}
	These maps are \emph{not} sections of $\pi \colon G \to \bar G$ (for instance $\theta_1$ is not injective).
	Nevertheless their restrictions to $\bar C$ are.
	In other words $\theta_i$ induces an embedding from $\bar E_i$ into $\bar E_i / \bar C \times E'_i$.
	Consequently $\theta_1$ and $\theta_2$ are examples of the maps given by  \autoref{res: lifting loxodromic sbgp}.
 \end{exam}

\begin{prop}
\label{res: lifting loxodromic sbgp w/ unliftable elliptic}
	Let $\bar E$ be a loxodromic subgroup of $\bar G$ that splits as $\bar E = \bar A \ast_{\bar C}\bar B$, where $\bar C$ is the maximal elliptic normal subgroup of $\bar E$ and has index $2$ in both $\bar A$ and $\bar B$.
	Assume that there exists $\bar v \in \bar{\mathcal V}$ such that $\bar A$ contains a strict rotation at $\bar v$.
	Let $q_{\bar v} \colon \stab{\bar v} \to \dihedral[n]$ be the associated geometric realization map and $\mathbf r \in \dihedral[n]$ a generator of the rotation group.
	If $n$ is divisible by $4$, then one of the following holds
	\begin{enumerate}
		\item $q_{\bar v} (\bar C)$ is trivial and 
		\begin{equation*}
			q_{\bar v}\left(\bar A\right) = \group{\mathbf r^{n/2}},
		\end{equation*}
		\item $q_{\bar v} (\bar C)$ is a reflection group generated by say $\mathbf x \in \dihedral[n]$ and 
		\begin{equation*}
			q_{\bar v}\left(\bar A\right) = \group{\mathbf x, \mathbf r^{n/4}\mathbf x\mathbf r^{-n/4}}.
		\end{equation*}
	\end{enumerate}
	Suppose now that there exist a subgroup $\bar E_0\subset\bar A$ and an element $\bar h \in \bar G$ such that $\bar h\bar E_0 \bar h^{-1}$ is contained in $\bar A$ and $\bar A=\group{\bar E_0,\bar C}$.
	Then either $\bar E_0$ contains a strict rotation, in which case $\bar h$ fixes $\bar v$ or the first case above fails and $q_{\bar v}$ maps $\bar E_0$ onto $\group{\mathbf r^{n/4}\mathbf x\mathbf r^{-n/4}}$.
\end{prop}

\begin{rema*}
	Our assumption on $\bar A$ exactly means that $\bar A$ cannot be lifted (\autoref{res: dichotomy elliptic}).
	It follows from \autoref{rem: elliptic subgroup with strict rotation} that $\bar A$ is contained in $\stab{\bar v}$, hence the image under $q_{\bar v}$ of $\bar C$ or $\bar A$ is well defined.
	It is important to note that in the second part of the statement $\bar h$ is not necessarily an element of $\stab{\bar v}$.
	In particular, if $\bar h$ does not fix $\bar v$, the last conclusion tells us that $\bar C$ and $\bar E_0$ are two reflection groups at $\bar v$;
	geometrically we can think that one is the conjugate of the other by a quarter turn at $\bar v$.
\end{rema*}

\begin{proof}
	The first part of the proof is essentially a variation on \autoref{res: lifting subgroup of index 2}.
	Since the cylinder of $\bar E$ is contained in $\fix{\bar C, 100\bar \delta}$ (\autoref{res: normal elliptic fixing cyl}) the subgroup $\bar C$ can be lifted and thus does not contain a strict rotation (\autoref{res: dichotomy elliptic}).
	In particular, $\bar C$ is either locally trivial at $\bar v$ or a reflection group at $\bar v$.
	Assume first that $q_{\bar v}$ maps $\bar C$ to the trivial group.
	By assumption $\bar A$ contains a strict rotation at $\bar v$ whereas $[\bar A: \bar C] = 2$, which forces
	\begin{equation*}
		q_{\bar v} (\bar A) = \group{\mathbf r^{n/2}}.
	\end{equation*}
	Assume now that $q_{\bar v}$ maps $\bar C$ to a reflection group generated by say $\mathbf x \in \dihedral[n]$.
	Reasoning as above we observe that
	\begin{equation*}
		q_{\bar v}(\bar A) = \group{\mathbf x, \mathbf r^{n/2}}.
	\end{equation*}
	Since $\mathbf x$ and $\mathbf r$ are respectively reflection and rotation in  the dihedral group $\dihedral[n]$,  we have 
	\begin{equation*}
		\mathbf r^{n/2} 
		= \left(\mathbf r^{n/4}\right)\left( \mathbf r^{n/4}  \right)
		= \mathbf  r^{n/4}\left(\mathbf  x \mathbf r^{-n/4} \mathbf x^{-1}\right)
		= \left(\mathbf  r^{n/4}\mathbf  x \mathbf r^{-n/4}\right) \mathbf x^{-1}
	\end{equation*}
	Hence 
	\begin{equation*}
		q_{\bar v}(\bar A) = \group{\mathbf x, \mathbf  r^{n/4}\mathbf  x \mathbf r^{-n/4}}.
	\end{equation*}
	This completes the first part of the statement.
	
	Let us focus now on the second half of the proposition.
	Suppose first that $\bar E_0$ contains a strict rotation (which is necessarily at $\bar v$).
	Then $\bar h \bar E_0 \bar h^{-1}$ contains a strict rotation at $\bar h \bar v$, which as an element of $\bar A$ has to fix $\bar v$.
	Strict rotations having a single fixed vertex (\autoref{res: local action of apex stab}) it yields $\bar h \bar v = \bar v$.
	Suppose now that $\bar E_0$ does not contain a strict rotation.
	In particular, $q_{\bar v}$ maps both $\bar E_0$ and $\bar C$ to subgroups of $\dihedral[n]$ which are trivial or reflection groups.
	Since $\bar E_0$ and $\bar C$ generates $\bar A$, the only possible configuration for $q_{\bar v}(\bar A)$ to contain a non-trivial rotation is the one where $q_{\bar v}(\bar E_0)$ and $q_{\bar v}(\bar C)$ are two \emph{distinct} reflection groups.
	Consequently the first case above fails, and $q_{\bar v}(\bar E_0) = \group{\mathbf r^{n/4}\mathbf x\mathbf r^{-n/4}}$, see \autoref{fig: travail leminaire}.
	\begin{figure}[htbp]
    	\centering
    	\includegraphics[page=8, width=0.4\textwidth]{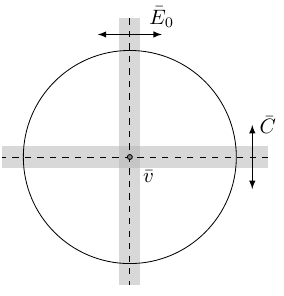}
    	\caption{
			The action of $\bar E$ on $B(\bar v, \rho)$.
			One assumes here that $\bar E_0$ does not contain a strict rotation.
			The shaded areas represent $\fix{\bar C, \bar \delta$} and $\fix{\bar E_0, \bar \delta$} respectively.
			}
    	\label{fig: travail leminaire}
    \end{figure}
\end{proof}

To summarize and illustrate the above study, let us come back to \autoref{exa: intro growing elliptic} and give a description of the elementary subgroups and dihedral germs appearing in this example.

\begin{exam}
\label{exa: intro growing elliptic - continued}
	We keep the same notations as in \autoref{exa: intro growing elliptic}.
	We assume here that $n$ is divisible by $8$ and write $n = n_2k$ where $n_2$ is the largest power of $2$ dividing $n$.
	Recall that 
	\begin{itemize}
		\item $G_0$ is the free group generated by $a$ and $b$, 
		\item $G_1 = G_0 / \normal{a^n,b^n, (ba)^n}$,
		\item $G_2 = G_1/\normal{(s_0s_1)^n, (s_0s_2)^n}$, where $s_0 = a^{n/2}$, $s_1= b^{n/2}$ and $s_2 = (ba)^{n/2}$,
		\item $G_3 = G_2/\normal{(u_1u_2)^n, [a^{n/4}, u_1]^n}$, where $u_i = s_0 (s_0s_i)^{n/2}$.
	\end{itemize}	
	Since $G_0$ is torsion-free, its elliptic subgroups are trivial while its loxodromic subgroups are isomorphic to $\cyclic$.
	In particular, every dihedral germ in $G_0$ is trivial.
	
	Since $G_0$ has no non-trivial elliptic subgroup, every such group in $G_1$ comes from the relations we added.
	More precisely, the elliptic subgroups of $G_1$ are, up to conjugacy, the subgroups of $\group a$, $\group b$, or $\group{ab}$.
	None of them can be normalized by a loxodromic element, unless they are trivial (an elliptic subgroup containing a strict rotation cannot be normalized by a loxodromic element). 
	Therefore the loxodromic subgroups are isomorphic to either $\cyclic$ or $\dihedral$.
	Consequently the dihedral germs in $G_1$ are the finite $2$-subgroups, namely any subgroup that is conjugated into $\group {a^k}$, $\group{b^k}$, or $\group {(ab)^k}$.
	
	Elliptic subgroups in $G_2$ fall in two categories depending whether they can be lifted in $G_1$ or not.
	The one that can be lifted are, up to conjugacy, the subgroups of $\group a$, $\group b$, or $\group{ab}$ as previously.
	The ones that cannot be lifted need to fix a cone point corresponding to the relations $(s_0s_1)^n$ and $(s_0s_2)^n$ (and their conjugates).
	In $G_1$ the maximal elementary subgroup containing $(s_0s_i)^n$ is $\group{s_0, s_i} = \dihedral$.
	Hence the cone point stabilizers in $G_2$ are isomorphic to $\dihedral[n]$, leading to a new kind of elliptic subgroups in $G_2$.
	Consider now a loxodromic subgroup $\bar E$ in $G_2$. 
	Its maximal elliptic normal subgroup $\bar F$ hardly move any point along the axis of $\bar E$ (\autoref{res: normal elliptic fixing cyl}).
	Therefore it can lift to an elliptic subgroup $F$ of $G_1$.
	There are now two options. 
	Either we can find a loxodromic element $h \in G_1$ normalizing $F$ and lifting a loxodromic element of $\bar E$, in which case $F$ is trivial.
	Or $\bar F$ is contained in a reflection group at some cone point, hence $\bar F$ is either trivial or $\cyclic[2]$.
	It follows that loxodromic subgroups of $G_2$ are of the form $\cyclic$, $\dihedral$, $\cyclic[2] \times \cyclic$, $\cyclic[2] \times \dihedral$, or $\cyclic[4] \ast_{\cyclic[2]}\cyclic[4]$.
	For instance, since $s_0 = a^{n/2}$ commutes with $u_i = s_0(s_0s_1)^{n/2}$ we have $\group {s_0, u_1, u_2} = \cyclic[2] \times \dihedral$ and $\group{a^{n/4}, u_1a^{n/4}u_1^{-1}} = \cyclic[4] \ast_{\cyclic[2]}\cyclic[4]$.
	It follows from this analysis that the dihedral germs of $G_2$ are exactly its finite $2$-subgroups.
	Therefore they are the subgroups conjugated into $\group {a^k}$, $\group{b^k}$, $\group {(ab)^k}$,  $\group{s_0, (s_0s_1)^k}$ or $\group{s_0, (s_0s_2)^k}$.
	
	We can then carry a similar analysis in $G_3$.
	The finite subgroups of $G_3$ are isomorphic to a subgroup of $\cyclic[n]$, $\dihedral[n]$, $\cyclic[2] \times \dihedral[n]$ and the non trivial central extension $\dihedral[n]$ by $\cyclic[2]$.
	The loxodromic subgroups of $G_3$ are extensions of $\cyclic$ or $\dihedral$ by the trivial group, $\cyclic[2]$, $\cyclic[2] \times \cyclic[2]$, or $\cyclic[4]$.
	Our example is too simple to realize each of these extensions as a loxodromic subgroup. 
	However this could be made possible by adjoining more relations of the form $x^n = 1$.
	The dihedral germs of $G_3$ are still its finite $2$-subgroups.
\end{exam}

%
\subsection{Invariants of $\bar G$ acting on $\bar X$}
%
\label{sec: invariants quotient}

This section is devoted to the study of the numerical invariants associated to the action of $\bar G$ on $\bar X$, namely $\inj[\bar X]{\bar G}$, $A(\bar G, \bar X,d)$ and $\nu_{\rm stg}(\bar G, \bar X)$ (see \autoref{sec: invariants} for the definitions).
As we explained earlier, the first two are purely geometric, whereas the last one has a mixed nature and captures both geometric and algebraic features of $\bar G$.

%
\subsubsection{Geometric invariants}
%
\label{sec: invariants quotient - geom}

\paragraph{The injectivity radius.}
The statement below is a (easy) generalization of \cite[Lemma~5.11.1]{Delzant:2008tu}.

\begin{prop}
\label{res: injectivity radius}
	Let $N$ be a subgroup of $G$ containing $K$ and $\bar N$ its image in $\bar G$.
	We denote by $\ell$ the infimum over the stable translation length (in $X$) of loxodromic elements of $N$ which do not belong to $\stab Y$ for some $(H,Y) \in \mathcal Q$.	
	Then 
	\begin{equation*}
		\inj[\bar X]{\bar N} \geq \min\left\{ \kappa \ell, \bar \delta\right\},
	\end{equation*}
	where $\kappa = \bar \delta /\pi \sinh (10\bar \delta)$.
\end{prop}

\begin{rema*}
	By convention, if $\bar N$ does not contain any loxodromic element then $\inj[\bar X]{\bar N}$ is infinite, in which case the statement is void.
\end{rema*}

\begin{proof}
	Let $\bar g$ be a loxodromic element of $\bar N$.
	We need to show that $\snorm{\bar g} \geq \min\left\{ \kappa \ell, \bar \delta\right\}$.
	By (\ref{eqn: regular vs stable length}) we have $m\snorm{\bar g} \geq \norm{\bar g^m} - 8\bar \delta$, for every $m \in \N$.
	Therefore it suffices to find an integer $m$ such that 
	\begin{equation*}
		\norm{\bar g^m} \geq  m\min\left\{ \kappa \ell, \bar \delta\right\} +8\bar \delta.
	\end{equation*}
	We denote by $m$ the largest integer satisfying $m \min\left\{ \kappa \ell, \bar \delta\right\} \leq 2\bar \delta$.
	Assume that $\norm{\bar g^m}$ is smaller than $m\min\left\{ \kappa \ell, \bar \delta\right\} +8\bar \delta$.
	In particular, $\norm{\bar g^m} < 10\bar \delta$.
	Thus $\fix{\bar g^m,10\bar \delta}$ is non empty.
	Moreover $d(\bar x, \bar {\mathcal V}) \geq \rho - 5 \bar \delta$, for every $\bar x \in \fix{\bar g^m, 10 \bar \delta}$.
	Indeed if it was not the case, $\bar g^m$ would fix an apex $\bar v \in \bar{\mathcal V}$ which contradicts the fact that $\bar g$ is loxodromic.
	Hence $\bar Z = \fix{\bar g^m, 20 \bar \delta}$ contains a point in $\zeta(X)$.
	Note also that $\bar Z$ is a $8\bar \delta$-quasi-convex subset.
	Thus $\gro{\bar z}{\bar z'}{\bar v} > \rho/2$, for every $\bar z, \bar z' \in \bar Z$ and $\bar v \in \bar{\mathcal V}$.
	By \autoref{res: lifting quasi-convex}, there exists a subset $Z$ of $\dot X$ such that the map $\zeta : \dot X \rightarrow \bar X$ induces an isometry from $Z$ onto $\bar Z$ and the projection $\pi : G \onto \bar G$ induces an isomorphism from $\stab Z$ onto $\stab {\bar Z}$.
	Observe that $\bar g$ preserves $\bar Z$.
	We denote by $g$ the pre-image of $\bar g$ in $\stab Z$.
	Since the kernel $K$ is contained in $N$, the element $g$ belongs to $N$.
	By construction $g$ is a loxodromic element which does not belong to any $\stab Y$ where $(H,Y) \in \mathcal Q$.
	Hence $\snorm g \geq \ell$.
	As we noticed before $\bar Z$ contains a point in $\bar z \in \zeta(X)$.
	Let $z \in X$ be its pre-image in $Z$.
	It follows from \autoref{res: loose comparison metric X and dot X}, that
	\begin{equation*}
		\mu\left(\dist[X]{g^mz}z\right) 
		\leq \dist[\dot X]{g^mz}z
		\leq \dist[\bar X]{\bar g^m \bar z}{\bar z}
		\leq 20 \bar \delta <2 \rho.
	\end{equation*}
	By \autoref{res: map mu},
	\begin{equation*}
		m\ell \leq m\snorm g\leq  \dist[X]{g^mz}z \leq \pi \sinh (10 \bar \delta) = \kappa^{-1}\bar\delta,
	\end{equation*}
	which contradicts the maximality of $m$.
\end{proof}

\paragraph{Acylindricity.}

The statement below is essentially a reformulation of \cite[Lemma~5.9.5]{Delzant:2008tu}.

\begin{prop}
\label{res: acyl quotient preparatory}
	For every $d_0 \in \intval0{\rho/10}$, we have
	\begin{equation*}
		A(\bar G, \bar X, d_0) \leq A(G, X, d_1) + d_0, 
		\quad\text{where}\quad 
		d_1 = \pi\sinh(d_0).
	\end{equation*}
\end{prop}

\begin{proof}
	Let $\bar S$ be a subset of $\bar G$ generating a non-elementary subgroup.
	Let $d_0 \in \intval 0{\rho/10}$.
	The goal is to bound from above the diameter of $\bar Z_0 = \fix{\bar S, d_0}$.
	Without loss of generality we can assume that this set is non-empty.
	Observe that $d(\bar x, \bar {\mathcal V}) \geq \rho - d_0/2$, for every $\bar x \in \bar Z$.
	Indeed if it was not the case, every element of $\bar S$ would fix a common apex $\bar v \in \bar{\mathcal V}$, contradicting the fact that $\bar S$ generates a non-elementary subgroup.
	Let $\bar Z$ be the $d_0$-neighborhood of $\bar Z_0$.
	Since $\bar Z_0$ is $8 \bar \delta$-quasi-convex (\autoref{res: fix qc}), we have $\gro{\bar z}{\bar z'}{\bar v} > \rho/4$, for every $\bar z, \bar z' \in \bar Z$ and $\bar v \in \bar{\mathcal V}$.
	According to \autoref{res: lifting quasi-convex} there exists a subset $Z$ of $\dot X$ such that the projection $\zeta \colon \dot X \to \bar X$ induces an isometry from $Z$ onto $\bar Z$.
	Moreover $\gro z{z'}v > \rho/4$, for every $z,z' \in Z$ and $v \in \mathcal V$ (the Gromov product is computed here in $\dot X$).
	We denote by $Z_0$ the pre-image of $\bar Z_0$ in $Z$.
	Let $\bar g \in \bar S$.
	By construction $\bar g\bar z$ belongs to $\bar Z$ for every $\bar z \in \bar Z_0$.
	Consequently there exists a (unique) $g \in G$ such that for every $z\in Z_0$ we have $\dist[\dot X]{gz}z = \dist{\bar g \bar z}{\bar z}$ (\autoref{res: zeta isom on qc far from apices}).
	We denote by $S$ the set of all $g \in G$ obtained in this way.
	Note that $S$ does not generate an elementary subgroup, otherwise so would $\bar S$.
	Let $\bar z \in \bar Z_0$ and $z \in \dot X$ be its pre-image in $Z_0$.
	Note that $gz$ belongs to $Z$.
	Let $y = p(z)$ be the radial projection of $z$.
	Since $\bar Z_0$ lies in the $d_0/2$-neighborhood of $\zeta(X)$ we have $\dist[\dot X] zy \leq d_0/2$.
	Combining the triangle inequality with \autoref{res: loose comparison metric X and dot X}, we get for every $g \in S$, 
	\begin{equation*}
		\mu\left( \dist{gy}y\right)
		\leq \dist[\dot X]{gy}y 
		\leq \dist[\dot X]{gz}z + d_0 
		\leq \dist{\bar g\bar z}{\bar z} + d_0
		< 2d_0 < 2\rho.
	\end{equation*}
	By \autoref{res: map mu}, we get that $y$ lies in $\fix{S,d_1} \subset X$.
	Assume that $\bar z'$ is another point in $\bar Z_0$.
	As previously we denote by $z' \in \dot X$ its pre-image in $Z_0$ and by $y' = p(z')$ the radial projection of $z'$.
	In particular, $y'$ also belongs to $\fix{S,d_1}\subset X$ and $\dist[\dot X]{z'}{y'} \leq d_0/2$.
	It follows from the triangle inequality that 
	\begin{equation*}
		\dist{\bar z}{\bar z'} 
		\leq \dist[\dot X]z{z'} 
		\leq \dist y{y'} + d_0 
		\leq \diam\left(\fix{S,d_1}\right) + d_0.
	\end{equation*}
	This inequality holds for every $\bar z, \bar z' \in \fix{\bar S, d_0}$.
	Hence
	\begin{equation*}
		\diam\left(\fix{\bar S,d_0}\right) \leq \diam\left(\fix{S,d_1}\right) +d_0. \qedhere
	\end{equation*}
\end{proof}

%
\subsubsection{Mixed invariants}
%
\label{sec: invariants quotient - mixed}

In view of Propositions~\ref{res: acyl quotient preparatory} and \ref{res: local-to-global acyl}, the $\nu$-invariant of $G$ can be used to control the acylindricity invariant $A(\bar G, \bar X,d)$ of the quotient $\bar G$.
If we want to iterate the procedure we need to control as well the $\nu$-invariant of $\bar G$.
Let us start with an informal discussion to emphasizes the difficulties that may arise along the way.
For simplicity let us assume that $G$ (hence $\bar G$) does not contain any parabolic subgroup.
Indeed those subgroups will not be a source of trouble.

If $\bar G$ has no even torsion, one can prove using only \emph{geometrical} arguments the following dichotomy -- see for instance the proof of \cite[Proposition~5.28]{Coulon:2016if}.
Given any chain $\bar {\mathcal C} = (\bar g_0, \dots, \bar g_m)$ generating an elementary subgroup of $\bar G$ and $\bar h \in \bar G$ a loxodromic conjugating element of $\bar {\mathcal C}$,
\begin{enumerate}
	\item either $\group{\bar g_0, \bar h}$ is elementary,
	\item or the chain $\bar {\mathcal C}$ can be lifted to a chain $\mathcal C = (g_0, \dots, g_m) $ in $G$ which generates an elementary subgroup of $G$ and  where one of the conjugating elements $h$ of $\mathcal C$ is a pre-image of $\bar h$.
\end{enumerate}
In the latter case, $h$ is necessarily loxodromic. 
If in addition $m \geq \nu(G,X)$, we can conclude that $\group{g_0,h}$ and thus $\group{\bar g_0, \bar h}$ is elementary.
It follows then that $\nu(\bar G, \bar X) \leq \nu(G, X)$.

Unfortunately this strategy fails in the presence of even torsion.
In \autoref{res: lifting automorphism} we observed the following phenomenon.
Let $F$ be an elliptic subgroup of $G$ and $\bar F$ its image in $\bar G$.
If $\bar h$ is an element of $\bar G$ normalizing $\bar F$, then there exists an element $h_0 \in G$ normalizing $F$ whose action by conjugation on $F$ coincide with the one of $\bar h$ on $\bar F$.
However $h_0$ is not necessarily a pre-image of $\bar h$.
In particular, if $\bar h$ is loxodromic, there is no reason that $h_0$ should be loxodromic as well.
The same issue arises when lifting chains. 
If a chain in $\bar G$ admits a loxodromic conjugating element, there is no reason that its lift in $G$ (provided it exists) has a loxodromic conjugating element.
This motivates the definition of the strong $\nu$-invariant (see \autoref{def: nu inv stg}).
However this is not the only obstruction.
As illustrated by the next example, the above dichotomy may fail anyway.

\begin{exam}
\label{exa: lysenok example}
	Start with the hyperbolic group 
	\begin{equation*}
		G = \left< a,b,c \mid a^2, b^2, c^2, [a,c] = 1 \right> = \dihedral \ast_{\cyclic[2]} \left(\cyclic[2] \times \cyclic[2]\right),
	\end{equation*}
	acting on its Cayley graph $X$.
	In this description $a$ and $b$ (\resp $a$ and $c$) generate the factor $\dihedral$ (\resp $\cyclic[2] \times \cyclic[2]$).
	The amalgamated group is $\cyclic[2] = \group a$.
	Set $r = ab$.
	Let $n \in \N$ be large integer divisible by $4$, so that the group 
	\begin{equation*}
		\bar G = G / \normal{r^n}.
	\end{equation*}
	is a small cancellation quotient of $G$.
	It follows from the additional relations that $\bar a$ and $\bar b$ generate a copy of $\dihedral[n]$.
	In particular, $\bar r^{-n/4}\bar a\bar r^{n/4}$ is an involution which commutes with $\bar a$, so that the subgroup
	\begin{equation*}
		\bar E = \group{\bar a,\bar r^{-n/4}\bar a\bar r^{n/4},\bar c} = \cyclic[2] \times \dihedral
	\end{equation*}
	is loxodromic.
	We denote by $\bar t = (\bar r^{-n/4}\bar a\bar r^{n/4})\bar c$ the translation in the dihedral factor of $\bar E$.
	Finally set
	\begin{equation*}
		\bar g_0 = \bar r^{-n/4}\bar a\bar r^{n/4}, \quad 
		\bar g_1 = \bar a, \quad \text{and} \quad 
		\bar g_2 = \bar t\bar g_0\bar t^{-1}
	\end{equation*}
	Note that $\group{\bar g_0, \bar g_1, \bar g_2}$ is a subgroup of $\bar E$ which is also isomorphic to $\cyclic[2] \times \dihedral$ (see \autoref{fig: pathological example}).
	\begin{figure}[htbp]
    	\centering
    	\includegraphics[page=6, width=\textwidth]{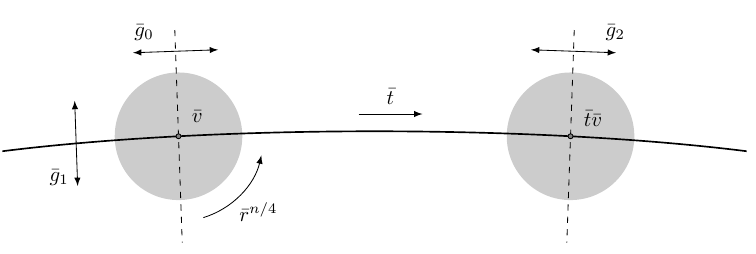}
    	\caption{
			The action of $\bar E$ and $\bar r^{n/4}$ on the space $\bar X$.
			The shaded discs respectively represent $B(\bar v, \rho)$ and $B(\bar t\bar v, \rho)$ where $\bar v$ is the cone point associated to the relation $\bar r^n = 1$.
		}
    	\label{fig: pathological example}
    \end{figure}
	As $\bar t$ commutes with $\bar g_1$, we observe that
	\begin{equation*}
		\bar g_1 = \bar h\bar g_0\bar h^{-1}
		\quad \text{and} \quad 
		\bar g_2 = \bar h\bar g_1\bar h^{-1}, 
		\quad \text{where} \quad
		\bar h = \bar t\bar r^{n/4}.
	\end{equation*}
	Hence $(\bar g_0, \bar g_1, \bar g_2)$ is a chain which generates an elementary subgroup of $\bar G$.
	One can checks that $\bar h$ is a loxodromic element. 
	Indeed, keeping the notation of \autoref{fig: pathological example}, one observes that $\bar h$ maps $\bar v$ to $\bar t \bar v$ and $\bar t\bar v$ to the point $\bar h \bar t \bar v$, whose projection on the axes of $\bar t$ is essentially $\bar t \bar v$.
	More generally, the overlap between the respective axis of $\bar t$ and $\bar h$ is roughly the segment $\geo{\bar v}{\bar t \bar v}$ depicted on \autoref{fig: pathological example}.
	Hence $\group{\bar g_0,\bar h}$ has more than two accumulation points in $\partial\bar X$, hence is not an elementary subgroup of $\bar G$.
	It cannot either be lifted to a \emph{chain} which generates an elementary subgroup of $G$.
\end{exam}

In the previous example, the difficulty comes from the fact that the subgroup $\group{\bar g_0, \bar g_1, \bar g_2}$ generated by the chain is a loxodromic subgroup of $\bar G$ that cannot be lifted.
Note that $\group{\bar g_0, \bar g_1}$ is an elliptic subgroup that cannot be lifted either.
As described in \autoref{res: lifting loxodromic sbgp w/ unliftable elliptic}, $\bar g_1$ is obtained from $\bar g_0$ by conjugation by a quarter turn.

This discussion suggests that the conjugation by a quarter turn rotation plays an important role.
This is the place where \emph{algebra} enters the stage.
Take $\bar v \in \bar{\mathcal V}$.
Let $\bar F$ be the kernel of the geometric realization map $q_{\bar v} \onto \dihedral[n]$.
Assume as in our example that $n$ is divisible by $4$.
Let $\bar r$ be a strict rotation at $\bar v$.
From a geometric point of view, conjugating $\bar F$ by $\bar r$ is not a trackable operation.
Indeed this will send an element $\bar g$ which is locally trivial at $\bar v$ to $\bar r \bar g \bar r^{-1}$ which is still locally trivial at $\bar v$.
Hence one cannot distinguish $\bar g$ and $\bar r \bar g \bar r^{-1}$ from their action on $B(\bar v, \rho)$.
This does not mean that $\bar g$ and $\bar r$ commutes though.
Nevertheless if we had a better understanding of the \emph{algebraic structure} of $\stab{\bar v}$, more precisely if we knew that $\stab{\bar v}$ is essentially a subproduct of dihedral groups, it could be possible to find a suitable quarter turn at $\bar v$ which truly commutes with a prescribed subset of $\bar F$ (see \autoref{res: norm in prod of dihedral}).

With additional algebraic hypotheses we are actually able to prove that if $m$ is sufficiently large, any chain $\bar {\mathcal C} = (\bar g_0, \dots, \bar g_m)$ will satisfy a \emph{variation} on the above dichotomy, which salvages the original strategy.

\paragraph{Additional assumption.}
We now begin a systematic study of the mixed invariants.
As we pointed out above, we require some additional hypothesis on the algebraic structure of elementary subgroups.

We fix once for all an even integer $n$ and write $n_2$ for the largest power of $2$ dividing $n$.
We also choose a model collection $\boldsymbol{\mathcal E}$, i.e.\ a family of (abstract) torsion groups and assume that its exponent $\mu = \mu(\boldsymbol{\mathcal E})$ divides $n$.
We suppose in addition that for every $\mathbf E \in \boldsymbol{\mathcal E}$, the exponent of $\mathbf E / Z(\mathbf E)$ divides $n/2$, where $Z(\mathbf E)$ stands for the center of $\mathbf E$.
From now on, we substitute \autoref{ass: even exponent} on vertex stabilizers for the following stronger hypotheses.
Recall that a dihedral pair $(E,C)$ has type $(\boldsymbol{\mathcal E}, n_2)$ if there exist $k \in \N$, and a morphism $\varphi \colon E \to \mathbf E$, where $\mathbf E \in \boldsymbol{\mathcal E}$ such that $\varphi$ extends to an embedding from $E$ into $E/C \times \dihedral[n_2]^k \times \mathbf E$ (see \autoref{def: elemen structure}).

\begin{assu}[Structure of elementary subgroups]
\label{ass: elem subgroup}
	Every dihedral pair of $G$ has type $(\boldsymbol{\mathcal E}, n_2)$.
\end{assu}

\begin{assu}[Relations]
\label{ass: relation}
	For every $(H,Y) \in \mathcal Q$, there exists a primitive element $g \in G$ such that $H = \group{g^n}$.
\end{assu}

Let us now mention a few consequences of these new assumptions.

\begin{lemm}
\label{res: dihedral germ in loxodromic subgroup}
	Every elliptic subgroup of a loxodromic subgroup of $G$ is a dihedral germ.
	Every subgroup of a dihedral germ is a dihedral germ.
\end{lemm}

\begin{proof}
	Let $E$ be a loxodromic subgroup of $G$ and $F$ its maximal elliptic normal subgroup.
	Let $g \in E$ be a loxodromic element.
	According to \autoref{ass: elem subgroup} there exists a group $E'$ of exponent $n$ such that $E$ embeds in $E/F \times E'$.
	Hence $g^n$ centralizes $F$.
	We conclude as in \autoref{res: dihedral germ in relation group} that every elliptic subgroup of $E$ is a dihedral germ.
	
	Consider now a dihedral germ $C$ and a subgroup $D$ of $C$.
	By definition there is an elliptic subgroup $C_0$ of $C$ whose index in $C$ is a power of $2$ that is normalized by a loxodromic element.
	It follows from the previous discussion that $C_0$ is actually \emph{centralized} by a loxodromic element.
	Consequently the elliptic subgroup $D_0 = D \cap C_0$ is also centralized by a loxodromic element.
	Moreover the index $[D : D_0]$ divides $[C : C_0]$.
	Hence $D$ is a dihedral germ.
\end{proof}

\begin{lemm}
\label{res : existence of central half-turn}
	Let $\bar v \in \bar{\mathcal V}$.
	The cone point $\bar v$ has order $n$.
	The group $\stab{\bar v}$ contains a central half-turn at $\bar v$.
	There exist $k \in \N$ and a morphism $\bar \varphi \colon \stab{\bar v} \to \mathbf E$, where $\mathbf E \in \boldsymbol{\mathcal E}$ such that the geometric realization $q_{\bar v} \colon \stab{\bar v} \to \dihedral[n]$ together with $\bar \varphi$ extend to an embedding from $\stab{\bar v}$ into $\dihedral[n] \times \dihedral[n_2]^k \times \mathbf E$.
\end{lemm}

\begin{rema*}	
	In particular, \autoref{ass: even exponent} holds.
\end{rema*}

\begin{proof}
	Fix $(H,Y) \in \mathcal Q$ such that $\pi \colon G \to \bar G$ maps $E=\stab Y$ onto $\stab{\bar v}$.
	Let $F$ be the maximal elliptic normal subgroup of $E$.
	According to \autoref{ass: elem subgroup}, there exist an integer $k \in \N$, and a morphism $\varphi \colon E \to \mathbf E$, where $\mathbf E \in \boldsymbol{\mathcal E}$, such that $\varphi$ extends to an embedding from $E$ into $E/F \times \dihedral[n_2]^k\times \mathbf E$.
	
	We write $\alpha \colon E \to \dihedral[n_2]^k\times \mathbf E$ for the map obtained by composing the embedding $E \into E/F \times \dihedral[n_2]^k\times \mathbf E$ with the natural projection onto $\dihedral[n_2]^k\times \mathbf E$.
	Note that $\alpha$ is one-to-one when restricted to $F$.
	Recall that $E/F$ is either $\cyclic$ or $\dihedral$.
	We denote by $\mathbf t$ a generator of the maximal infinite cyclic subgroup of $E/F$.
	By \autoref{ass: relation}, there exists a pre-image $g \in E$ of $\mathbf t$ such that $H = \group{g^n}$.
	In particular, $n$ is the order of $\bar v$.
	By assumption, the exponent of $\dihedral[n_2]^k\times \mathbf E$ divides $n$.
	Hence $\alpha(H)$ is trivial.
	It follows that $\alpha$ induces a map $\bar \alpha \colon \stab{\bar v} \to \dihedral[n_2]^k \times \mathbf E$ whose restriction to $\bar F$ (the image of $F$ in $\bar G$) is an embedding.
	Hence the morphism $\stab{\bar v} \to \dihedral[n] \times \dihedral[n_2]^k \times \mathbf E$ given by $q_{\bar v}$ and $\bar \alpha$ is an embedding.
	
	We are left to prove the existence of a central half turn.
	To that end, we claim that for every $u \in E$,
	\begin{equation}
	\label{eqn : existence of central half-turn}
		g^{n/2}ug^{-n/2} \in u H.
	\end{equation}
	Let $u$ in $E$.
	Recall that the exponents of $\mathbf E$ and $\mathbf E/Z(\mathbf E)$ respectively divide $n$ and $n/2$, hence 
	\begin{equation*}
		\mathbf g^{n/2}\mathbf u\mathbf g^{-n/2} = \mathbf u = \mathbf u\mathbf g^{-n},
		\quad  \text{where} \quad
		\mathbf g = \varphi(g)\  \text{and}\ \mathbf u = \varphi(u).
	\end{equation*}
	The same identities are also satisfied by the images of $g$ and $u$ is any factor $\dihedral[n_2]$.
	Assume now that $u$ belongs to $E^+$ (the maximal subgroup of $E$ preserving the orientation).
	Then $g^{n/2}$ commutes with $u$ (one checks indeed that in each factor of $E/F \times \dihedral[n_2]^k\times \mathbf E$, we have $g^{n/2}ug^{-n/2} = u$).
	Assume now that $u \in E\setminus E^+$.
	In other words the image of $u$ in $E/F$ is a reflection.
	Then we still have $g^{n/2}ug^{-n/2} = ug^n$ (one checks indeed as above, that this identity holds in each factor of $E/F \times \dihedral[n_2]^k\times \mathbf E$).
	This completes the proof of our claim.
	As we already observed $\alpha(g^n)=1$, hence the image of $g^{n/2}$ in $\stab{\bar v}$ is a non trivial rotation of order $2$ at $\bar v$.
	It follows from (\ref{eqn : existence of central half-turn}) that it is also central in $\stab{\bar v}$, thus it is a central half-turn at $\bar v$.
\end{proof}

\paragraph{Model collection for $\bar G$.}

We now prove that the elementary subgroups of $\bar G$ satisfy a condition similar to \autoref{ass: elem subgroup}.

\begin{prop}
\label{res: pushing assum elem}
	Dihedral pairs of $\bar G$ (for its action on $\bar X$) have type $(\boldsymbol{\mathcal E}, n_2)$.
\end{prop}

\begin{rema}
\label{rem: pushing assum elem}
	Note that the model collection $\boldsymbol{\mathcal E}$ is the same as the one of \autoref{ass: elem subgroup}.
	In particular, we deduce as in \autoref{res: dihedral germ in loxodromic subgroup} that every elliptic subgroup contained in a loxodromic subgroup of $\bar G$ is a dihedral germ.
	Moreover every subgroup of dihedral germ is a dihedral germ.
\end{rema}

\begin{proof}
	Let $(\bar E, \bar C)$ be a dihedral pair.
	We distinguish several cases.
	Assume first that $\bar E$ is either an elliptic subgroup that \emph{can} be lifted in $G$ or a parabolic subgroup.
	Note that in the latter case $\bar E$ can also be lifted in $G$.
	Indeed $\bar E$ is the extension of an elliptic subgroup, namely $\bar C$, by a finitely generated subgroup.
	Thus there exist $d \in \R_+$ and a subset $\bar S$ generating $\bar E$ such that $\fix{\bar S, d}$ is non-empty.
	It follows then from \autoref{res: lifting parabolic subgroups} that $\bar E$ can be lifted.
	Let $E$ be a lift of $\bar E$ and $C$ the pre-image of $\bar C$ in $E$.
	According to \autoref{res: lifting dihedral germ}, $C$ is a dihedral germ.
	It follows that $(E,C)$ is a dihedral pair.
	Hence the result follows from \autoref{ass: elem subgroup} applied to $(E,C)$.

	Assume now that $\bar E$ is a loxodromic subgroup.
	In particular, $\bar C$ is its maximal elliptic normal subgroup (\autoref{res: loxo have dihedral shape}).
	We write $q \colon \bar E \onto \bar E / \bar C$ for the corresponding projection.
	According to \autoref{res: lifting loxodromic sbgp} there exist a dihedral pair $(E',C)$ in $G$ where $C$ is an elliptic subgroup lifting $\bar C$ and an epimorphism $\theta \colon \bar E \onto E'$ with the following properties.
	\begin{enumerate}
		\item The morphism $\theta$ restricted to $\bar C$ is a section of $\pi$ restricted to $C$.
		\item The map $\theta$ induces an embedding from $\bar E$ into $\bar E/\bar C \times E'$.
	\end{enumerate}
	Our assumption applied to $(E',C)$ says that there exist $k \in \N$, and a morphism $\varphi \colon E' \to \mathbf E$ where $\mathbf E \in \boldsymbol{\mathcal E}$, which extends to an embedding $\psi \colon E' \into E'/C \times \dihedral[n_2]^k\times \mathbf E$.
	We claim that $\varphi \circ \theta \colon \bar E \to \mathbf E$ extends to a embedding $\psi_\theta \colon \bar E \to \bar E / \bar C \times \dihedral[n_2]^k\times \mathbf E$.
	For simplicity we write $\alpha \colon E' \to  \dihedral[n_2]^k \times \mathbf E$ for the composition of $\psi$ with the canonical projection form $E'/C \times \dihedral[n_2]^k\times \mathbf E$ onto  $\dihedral[n_2]^k\times \mathbf E$.
	Consider the map
	\begin{equation*}
		\begin{array}{lccc}
			\psi_\theta \colon & \bar E & \to & \bar E/\bar C \times \dihedral[n_2]^k \times \mathbf E \\
			& \bar g & \mapsto & \left(q(\bar g), \alpha \circ \theta(\bar g)\right).
		\end{array}
	\end{equation*}
	It follows from the construction that $\varphi \circ \theta$ coincides with $\psi_\theta$ composed with the projection $\bar E/\bar C \times \dihedral[n_2]^k \times \mathbf E \onto \mathbf E$.
	Hence it suffices to prove that $\psi_\theta$ is one-to-one.
	Consider an element $\bar g \in \bar E$ in the kernel of $\psi_\theta$.
	In particular, $q(\bar g) = 1$, that is $\bar g \in \bar C$. 
	Therefore $\theta(\bar g)$ belongs to $C$.
	Moreover $\alpha \circ \theta(\bar g)$ is trivial.
	Hence $\psi \circ \theta(\bar g) = 1$.
	Since $\psi \colon E' \into E'/C \times \dihedral[n_2]^k\times \mathbf E$ is one-to-one, the element $\theta(\bar g)$ is trivial.
	On the other hand, $\theta$ induces an embedding from $\bar E$ into $\bar E/ \bar C \times E'$.
	In other words, $\theta \colon \bar E \to E'$ is one-to-one when restricted to $\bar C$.
	Thus $\bar g = 1$, which completes the proof our claim.

	We finally assume that $\bar E$ is an elliptic subgroup that \emph{cannot} be lifted in $G$.
	In particular, there exists an apex $\bar v \in \bar{\mathcal V}$ such that $\bar E$ is contained in $\stab{\bar v}$ (\autoref{res: dichotomy elliptic}).
	According to \autoref{res : existence of central half-turn} there exist $k \in \N$, and a morphism $\bar \varphi \colon \stab{\bar v} \to \mathbf E$, where $\mathbf E \in \boldsymbol{\mathcal E}$, which combined with the geometric realization $q_{\bar v} \colon \stab{\bar v} \to \dihedral[n]$ provides an embedding from $\stab{\bar v}$ into $\dihedral[n] \times \dihedral[n_2]^k\times \mathbf E$.
	For simplicity we write $\bar \alpha \colon \stab{\bar v} \to \dihedral[n_2]^k\times \mathbf E$ for the composition of $\stab{\bar v} \into \dihedral[n] \times \dihedral[n_2]^k\times \mathbf E$ with the natural projection onto $\dihedral[n_2]^k\times \mathbf E$.
	Composing the geometric realization $q_{\bar v} \colon \stab{\bar v} \to \dihedral[n]$ with the canonical projection $\dihedral[n] \to \dihedral[n_2]$ leads to a morphism that we denote $q'_{\bar v} \colon \stab{\bar v} \to \dihedral[n_2]$.
	We are going to prove that $q'_{\bar v}$ and $\bar \alpha$ extend to an embedding from $\bar E$ into $\bar E/ \bar C \times \dihedral[n_2] \times \dihedral[n_2]^k \times \mathbf E$.

	Let $\bar F$ be the kernel of $q_{\bar v} \colon \stab{\bar v} \to \dihedral[n]$.
	We first claim that $[\bar C: \bar C \cap \bar F]$ is a power of $2$.
	Since $\bar C$ is a dihedral germ, it contains a subgroup $\bar C_0$ which is normalized by a loxodromic element and such that $[\bar C:\bar C_0] = 2^m$ for some $m \in \N$.
	Observe that $\bar C_0$ is contained in a reflection group at $\bar v$.
	Indeed otherwise, any element normalizing $\bar C_0$ would belong to $\stab{\bar v}$ thus they would be no \emph{loxodromic} element centralizing $\bar C_0$.
	In particular, $[\bar C_0 : \bar C_0 \cap \bar F]$ is at most $2$.
	On the other hand
	\begin{equation*}
		[\bar C : \bar C \cap \bar F] [\bar C \cap \bar F: \bar C_0 \cap \bar F] 
		= [\bar C : \bar C_0 \cap \bar F]
		= [\bar C : \bar C_0] [\bar C_0 : \bar C_0 \cap \bar F]
	\end{equation*}
	Consequently $[\bar C : \bar C \cap \bar F]$ divides $[\bar C : \bar C_0 \cap \bar F]$.
	In particular, it is a power of $2$ which completes the proof of our claim.

	Consider now $\bar g \in \bar E$ whose image in $\bar E/ \bar C \times \dihedral[n_2] \times \dihedral[n_2]^k \times \mathbf E$ is trivial.
	First observe that $\bar g$ belongs to $\bar C$.
	On the one hand, it follows from the previous claim that the order of $q_{\bar v}(\bar g)$ is a power of $2$.
	On the other hand, $n_2$ is a the largest power of $2$ dividing $n$.
	Hence the kernel of the projection $\dihedral[n] \to \dihedral[n_2]$, which contains $q_{\bar v}(\bar g)$ by assumption, consists only of elements with odd order.
	Therefore $q_{\bar v}(\bar g)$ is trivial, i.e.\ $\bar g$ belongs to $\bar F$.
	Observe that the map $\bar \alpha \colon \stab{\bar v} \to \dihedral[n_2]^k \times \mathbf E$ is an embedding when restricted to $\bar F$.
	Since $\bar \alpha(\bar g) = 1$, the element $\bar g$ is trivial.
	This shows that $\bar E$ embeds in $\bar E/ \bar C \times \dihedral[n_2] \times \dihedral[n_2]^k \times \mathbf E$.
\end{proof}

\paragraph{The strong $\nu$-invariant.}
We now start our study of the strong $\nu$-invariant.
Recall that  $\mu = \mu(\boldsymbol{\mathcal E})$ is the exponent of the model collection $\boldsymbol{\mathcal E}$.
The ultimate goal is to prove the following statement.
\begin{prop}
\label{res: nu-inv}
	Assume that $2^{\nu +2}\mu$ divides $n$ where $\nu = \nu_{\rm{stg}}(G, X)$.
	Then
	\begin{equation*}
		\nu_{\rm{stg}}(\bar G, \bar X) \leq \max\left\{\nu_{\rm{stg}}(G, X), \mu+4\right\}.
	\end{equation*}
\end{prop}

For simplicity we adopt the following terminology.
\begin{defi}
\label{def: strong chain}
	A chain $\mathcal C = (g_0, \dots, g_m)$ of $G$ is a \emph{strong chain} if it satisfies the following holds
	\begin{enumerate}
		\item $g_0,\dots, g_m$ generate an elementary subgroup of $G$ (for its action on $X$).
		\item either $\mathcal C$ admits a loxodromic conjugating element or $\group{g_0,\dots, g_{m-1}}$ is a dihedral germ.
	\end{enumerate}
\end{defi}
We define \emph{strong chains} of $\bar G$ in the exact same way.
A strong chain $\bar {\mathcal C} = (\bar g_0, \dots, \bar g_m)$ \emph{can be lifted} if there exists a strong chain $\mathcal C = (g_0, \dots, g_m)$ of $G$ such that the quotient map $\pi \colon G \onto \bar G$ sends $g_k$ to $\bar g_k$ for every $k \in \intvald 0m$.
In this situation we also say that $\mathcal C$ \emph{lifts} $\bar {\mathcal C}$.
As suggested at the beginning of this section we first prove the following dichotomy.

\begin{prop}
\label{res: nu-lift - dichotomy}
	Let $m \geq 3$ such that $2^{m+2}\mu$ divides $n$.
	Let $\bar{\mathcal C} = (\bar g_0, \dots, \bar g_m)$ be a strong chain of $\bar G$ and $\bar h$ a conjugating element of $\bar {\mathcal C}$.
	Then one of the following holds.
	\begin{enumerate}
		\item \label{enu: nu-inv elliptic}
		There exists $\bar v \in \bar{\mathcal V}$ such that $\group{\bar g_0, \bar h}$ is contained in $\stab{\bar v}$.
		\item \label{enu: nu-inv loxo}
		The subgroup $\group{\bar g_0, \bar h}$ is loxodromic.
		\item \label{enu: nu-inv lifted}
		There exists a strong chain $\bar {\mathcal C}' = (\bar g'_0, \dots, \bar g'_m)$ of $\bar G$ which can be lifted and such that $(\bar g'_1, \dots, \bar g'_{m-1}) = (\bar g_1, \dots, \bar g_{m-1})$.
	\end{enumerate}
\end{prop}

We split the proof into several lemmas depending on the nature of the group generated by $\bar{\mathcal C}$.

\begin{lemm}
\label{res: nu-lift - para and elliptic w/ lift}
	Let $\bar{\mathcal C} = (\bar g_0, \dots, \bar g_m)$ be a strong chain of $\bar G$.
	If the subgroup $\bar E$ of $\bar G$ generated by $\bar {\mathcal C}$ is either elliptic and \emph{can be lifted} or parabolic, then $\bar{\mathcal C}$ can be lifted.
\end{lemm}

\begin{proof}
	We first claim that $\bar E$ can always be lifted even if $\bar E$ is parabolic.
	Indeed since $\bar E$ is finitely generated \autoref{res: lifting parabolic subgroups} applies.
	Let $E$ be a lift of $\bar E$.
	For every $k \in \intvald 0m$, we denote by $g_k$ the pre-image of $\bar g_k$ in $E$.
	Note that, contrary to $(\bar g_0, \dots, \bar g_m)$, the tuple $(g_0, \dots, g_m)$ could not be a chain.
	Nevertheless, according to \autoref{res: lifting partial auto} applied with $S_1 = \{g_0, \dots, g_{m-1}\}$ and $S_2 = \{g_1, \dots, g_m\}$, there exists $h_0 \in G$ with the following properties
	\begin{enumerate}
		\item for every $k \in \intvald 0{m-1}$, we have $g_{k+1} = h_0g_kh_0^{-1}$.
		\item \label{enu: nu-inv - elliptic w/ lift - loxo}
		if $\bar h$ is loxodromic, then either $h_0$ is loxodromic or $\group{\bar g_0,\dots, \bar g_{\nu-1}}$ is contained in a reflection group.
	\end{enumerate}
	Thus $\mathcal C = (g_0, \dots, g_\nu)$ is actually a chain and $h_0$ a conjugating element of $\mathcal C$.
	Obviously $\mathcal C$ generates an elementary subgroup of $G$.
	Note that either $h_0$ is loxodromic or $\group{g_0, \dots, g_{\nu -1}}$ is a dihedral germ (Lemmas~\ref{res: reflection group provides germs} and \ref{res: lifting dihedral germ}) hence the proof is complete.
\end{proof}

\begin{lemm}
\label{res: nu-lift - prelim wrt dihedral}
	Let $\bar{\mathcal C} = (\bar g_0, \dots, \bar g_m)$ be a strong chain of $\bar G$ generating a subgroup $\bar E$.
	Let $\bar h$ be a conjugating element of $\bar{\mathcal C}$.	
	Let $q \colon \bar E \to \dihedral[p]$ be a surjective map from $\bar E$ onto a finite or infinite dihedral group.
	Suppose that the images $\bar E_0 = \group{\bar g_0, \dots, \bar g_{m-1}}$ and $\bar h\bar E_0\bar h^{-1} = \group{\bar g_1, \dots, \bar g_m}$ in $\dihedral[p]$ are contained in a reflection group.
	
	Then $q$ maps $\bar g_0$ and $\bar g_m$ to distinct refections while $\{ \bar g_0^2, \bar g_1, \dots, \bar g_{m-1}, \bar g_m^2\}$ is contained in its kernel.
	Moreover, if $p = \infty$, then the kernel of $q$ is exactly the normal subgroup of $\bar E$ generated by $\{ \bar g_0^2, \bar g_1, \dots, \bar g_{m-1}, \bar g_m^2\}$.
\end{lemm}

\begin{proof}
	Since the map $q$ is onto, it must map $\bar E_0$ and $\bar h\bar E_0\bar h^{-1}$ onto distinct reflection subgroups of $\dihedral[p]$, say $\group{\mathbf x}$ and $\group{\mathbf y}$.
	By definition $\bar g_1, \dots, \bar g_{m-1}$ belong to $\bar E_0 \cap \bar h\bar E_0\bar h^{-1}$, whose image in $\dihedral[p]$ is contained in $\group{\mathbf x} \cap \group{\mathbf y} = \{1\}$.
	Hence these elements belongs to the kernel of $q$, so that 
	\begin{equation*}
		q \left(\group{\bar g_0}\right) = q\left(\bar E_0\right) =  \group{\mathbf x}
	\end{equation*}
	Consequently $q(\bar g_0) = \mathbf x$ and $q(\bar g_0^2) = 1$.
	Similarly we prove that $q(\bar g_m) = \mathbf y$ and $q(\bar g_m^2) = 1$.
	This completes the proof of the first assertion.
	
	Suppose now that $p = \infty$.
	Denote by $\bar F$ the normal subgroup of $\bar E$ generated by $\{ \bar g_0^2, \bar g_1, \dots, \bar g_{m-1}, \bar g_m^2\}$.
	According to the previous discussion $q$ induces a map $\bar E / \bar F \onto \dihedral$.
	However by construction $\bar E / \bar F$ is isomorphic to a dihedral group generated by the images of $\bar g_0$ and $\bar g_m$.
	It follows that $\bar E / \bar F$ is infinite and the map  $\bar E / \bar F \onto \dihedral$ is an isomorphism.
	Consequently, $\bar F$ coincides with the kernel of $q$.
\end{proof}

\begin{lemm}
\label{res: nu-lift - elliptic w/o lift}
	Let $\bar{\mathcal C} = (\bar g_0, \dots, \bar g_m)$ be a strong chain of $\bar G$ which generates an elliptic subgroup $\bar E$ of $\bar G$ which \emph{cannot be lifted}.
	Let $\bar h$ be a conjugating element of $\bar{\mathcal C}$.	
	Then either $\bar E_0 = \group{\bar g_0, \dots, \bar g_{m-1}}$ contains a strict rotation and $\group{\bar g_0,\bar h}$ is contained in $\stab{\bar v}$ for some $\bar v \in \bar {\mathcal V}$, or $\bar{\mathcal C}$ can be lifted.
\end{lemm}

\begin{proof}
	There exists a (unique) apex $\bar v \in \bar{\mathcal V}$ such that $\bar E$ is a subgroup of $\stab{\bar v}$ (\autoref{res: dichotomy elliptic}).
	Let $q_{\bar v} \colon \stab{\bar v} \to \dihedral[n]$ be the geometric realization map.
	Assume first that $\bar E_0 = \group{\bar g_0, \dots, \bar g_{m-1}}$ contains a strict rotation at $\bar v$.
	Consequently $\bar h \bar E_0 \bar h^{-1}$ contains a strict rotation at $\bar h \bar v$.
	Since strict rotations fix a unique cone point we get $\bar h\bar v = \bar v$.
	Thus $\group{\bar g_0, \bar h}$ is contained in $\stab{\bar v}$.
	If $\bar h\bar E_0\bar h^{-1}$ contains a strict rotation at $\bar v$, then we can reverse the chain $\bar{\mathcal C}$ and the conclusion stays the same.
	
	Assume now that neither $\bar E_0$ nor $\bar h\bar E_0\bar h^{-1}$ does not contain a strict rotation at $\bar v$.
	Since $\bar E$ cannot be lifted, $q_{\bar v}$ maps $\bar E$ onto a dihedral subgroup of $\dihedral[n]$.
	According to \autoref{res: nu-lift - prelim wrt dihedral}, it sends $\bar g_0$ and $\bar g_m$ to two distinct reflections, and $\bar g_0^2, \bar g_1, \dots, \bar g_{m - 1}, \bar g_m^2$ to the identity.
	We denote by $\bar C$ the intersection of $\bar E$ with the kernel of $q_{\bar v}$ and let $\bar A = \group{\bar g_0, \bar C}$ and $\bar B = \group{\bar g_m, \bar C}$.
	Note that $\bar C$ is a subgroup of index $2$ in both $\bar A$ and $\bar B$.
	Let $(H,Y) \in \mathcal Q$ such that the projection $\pi \colon G \to \bar G$ maps $\stab Y$ onto $\stab {\bar v}$.
	We choose a lift $A$ (\resp $B$) of $\bar A$ (\resp $\bar B$) contained in $\stab Y$ so that $A \cap B$ is a lift of $\bar C$ that we denote by $C$ (see \autoref{res: lifting elliptic in apex group}).
	Let $a_0, \dots, a_{m-1}$ (\resp $b_1, \dots, b_m$) the pre-images of $\bar g_0, \dots, \bar g_{m-1}$ (\resp $\bar g_1, \dots, \bar g_m$) in $A$ (\resp $B$).
	As we already observed $\bar g_1, \dots, \bar g_{m-1}$ belong to $\bar C$, thus $a_k = b_k$, for all $k \in \intvald 1{m-1}$.
	Applying \autoref{res: lifting partial auto} with $S_1 = A$ and $S_2 = B$, we get that there exists $h_0 \in G$ such that for every $k \in \intvald 0{m-1}$ we have $b_{k+1} = h_0a_kh_0^{-1}$.
	If follows that $\mathcal C = (a_0, a_1, \dots, a_{\nu-1}, b_\nu)$ is a chain and $h_0$ a conjugating element of $\mathcal C$.
	In addition this chain generates a subgroup of $\stab Y$, which is therefore elementary.
	Recall that $\group{a_0, \dots, a_{\nu -1}}$ is an elliptic subgroup of $\stab Y$.
	Hence it is a dihedral germ (\autoref{res: dihedral germ in relation group}).
	Consequently $\mathcal C$ is a strong chain of $G$ lifting $\bar {\mathcal C}$.
\end{proof}

\begin{lemm}
\label{res: nu-lift - loxodromic}
	Let $m \geq 3$.
	Let $\bar {\mathcal C} = (\bar g_0, \dots, \bar g_m)$ be a strong chain generating a loxodromic subgroup $\bar E$ of $\bar G$.
	Let $\bar h$ be a conjugating element of $\bar {\mathcal C}$.
	If $2^{m + 2}\mu$ divides $n$ then one of the following holds
	\begin{enumerate}
		\item The subgroup $\group{\bar g_0, \bar h}$ is loxodromic.
		\item There exists a strong chain $\bar {\mathcal C}' = (\bar g'_0, \dots, \bar g'_m)$ which can be lifted and such that $(\bar g'_1, \dots, \bar g'_{m-1}) = (\bar g_1, \dots, \bar g_{m-1})$.
	\end{enumerate}
\end{lemm}

\begin{proof}
	Assume firs that $\bar E_0 = \group{\bar g_0, \dots, \bar g_{m-1}}$ contains a loxodromic element say $\bar t$.
	Since $\bar E_0$ and $\bar h\bar E_0 \bar h^{-1}$ generate an elementary subgroup, namely $\bar E$, both $\bar g_0$ and $\bar h$ are contained in the maximal elementary subgroup containing $\bar t$.
	In particular, $\group{\bar g_0, \bar h}$ is loxodromic.

	Assume now that $\bar E_0$ (and thus $\bar h\bar E_0\bar h^{-1}$) is elliptic (a subgroup of a loxodromic subgroup cannot be parabolic).
	Let $\bar C$ be the maximal normal elliptic subgroup of $\bar E$.
	Since $\bar E$ is generated by two elliptic subgroups (namely $\bar E_0$ and $\bar h\bar E_0\bar h^{-1}$) it does not preserve the orientation.
	Hence the quotient $\bar E/ \bar C$ is isomorphic to $\dihedral$.
	We write
	\begin{equation*}
		 1 \to \bar C \to \bar E \xrightarrow q \dihedral \to 1
	\end{equation*}
	for the corresponding short exact sequence.
	According to \autoref{res: nu-lift - prelim wrt dihedral}, $q$ maps $\bar g_0$ and $\bar g_m$ to two distinct reflections, while $\bar C$ is the normal subgroup of $\bar E$ generated by $\bar g_0^2, \bar g_1, \dots, \bar g_{m-1}, \bar g_m^2$.
	We let $\bar A = \group{\bar g_0, \bar C}$ and $\bar B = \group{\bar g_m, \bar C}$ so that $\bar E$ is isomorphic to $\bar A\ast_{\bar C} \bar B$.
	We now distinguish two cases.
	
	\paragraph{Case 1.}
	\emph{Assume first that both $\bar A$ and $\bar B$ can be lifted in $G$.}
	We denote by $A$ and $B$ a lift of $\bar A$ and $\bar B$ respectively.
	According to \autoref{res: lifting loxodromic sbgp w/ liftable elliptic} there exists $u \in G$ whose image $\bar u$ in $\bar G$ centralizes $\bar C$ such that $E_u = \group{A, uBu^{-1}}$ is elementary.
	We let $\bar g'_k = \bar g_k$, for every $k \in \intvald 0{m-1}$ and $\bar g'_m = \bar u \bar g_m \bar u^{-1}$.
	Since $\bar u$ centralizes $\bar C$ we observe that $\bar {\mathcal C}' = (\bar g'_0, \dots, \bar g'_m)$ is a chain and $\bar h_0 = \bar u \bar h$ a conjugating element of $\bar{\mathcal C}'$. 
	Note also that $\bar{\mathcal C}$ and $\bar{\mathcal C'}$ only eventually differ on the last element.
	We now focus on this new chain.
	Let $a'_0, \dots, a'_{m-1}$ be the lift of $\bar g'_0, \dots, \bar g'_{m -1}$ in $A$ and $b'_1, \dots, b'_m$ the lifts of $\bar g'_1, \dots, \bar g'_m$ in $uBu^{-1}$.
	We now proceed exactly as in the proof of \autoref{res: nu-lift - elliptic w/o lift}.
	We first observe that $\mathcal C' = (a'_0, a'_1, \dots, a'_{m-1}, b'_m)$ is a chain for some conjugating element $h'_0 \in G$.
	Moreover it generates a subgroup of $E_u$ which is therefore elementary.
	Recall that $\bar E_0 = \group{\bar g_0, \dots, \bar g_{m-1}}$ is an elliptic subgroup of the loxodromic subgroup $\bar E$, therefore it is a dihedral germ (\autoref{res: dihedral germ in loxodromic subgroup}).
	Hence $\bar {\mathcal C}'$ is a strong chain.
	Moreover, as a lift of $\group{\bar g_0, \dots, \bar g_{m-1}}$, the subgroup $\group{a'_0, \dots, a'_{m-1}}$ is a dihedral germ as well (\autoref{res: lifting dihedral germ}).
	Hence $\mathcal C'$ is a strong chain of $G$ lifting $\bar{\mathcal C}'$.
	
	\paragraph{Case 2.}
	\emph{Assume that either $\bar A$ or $\bar B$ cannot be lifted in $G$.}
	Up to replacing $(\bar g_0, \dots, \bar g_m)$ by $(\bar g_m, \dots, \bar g_0)$, we can assume that $\bar A$ cannot be lifted in $G$.
	In particular, there exists a (unique) apex $\bar v \in \bar {\mathcal V}$ such that $\bar A$ contains a strict rotation at $\bar v$ (\autoref{res: dichotomy elliptic}).
	We write $q_{\bar v} \colon \stab{\bar v} \to \dihedral[n]$ for the geometric realization map associated to $\bar v$ and $\mathbf r \in \dihedral[n]$ for a generator of the rotation group.
	Let $\bar E_1 = \group{\bar g_1, \dots, \bar g_{m-1}}$.
	Observe that $\bar A$ is generated by $\bar h^{-1}\bar E_1\bar h$ and $\bar C$.
	Moreover $\bar E_1$ is contained in $\bar A$.	
	It follows from \autoref{res: lifting loxodromic sbgp w/ unliftable elliptic} that either $\bar h^{-1}\bar E_1\bar h$ contains a strict rotation, in which case $\group{\bar g_0, \bar h} $ is a subgroup of $\stab{\bar v}$, or there exists a reflection $\mathbf x \in \dihedral[n]$ such that
	\begin{equation*}
		q_{\bar v}\left( \bar C\right) = \group{\mathbf x} 
		\quad \text{and} \quad
		q_{\bar v}\left( \bar h^{-1}\bar E_1\bar h\right) = \group{\mathbf x'},
		\quad \text{where} \quad \mathbf x' = \mathbf r^{n/4}\mathbf x\mathbf r^{-n/4},
	\end{equation*}
	See \autoref{fig: lifting chain}.
	Note that the first case cannot happen since $\bar E$, which is contained in $\group{\bar g_0, \bar h}$, is loxodromic by assumption.
	\begin{figure}[htbp]
    	\centering
    	\includegraphics[page=7, width=\textwidth]{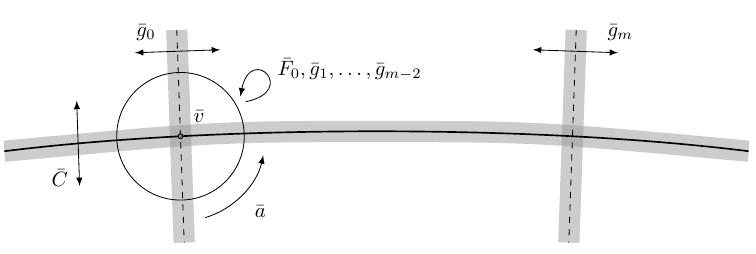}
    	\caption{
			Lifting chains. 
			The shaded areas represent $\fix{\bar g_0,10\bar \delta}$, $\fix{\bar g_m,10\bar \delta}$, and $\fix{\bar C,10\bar \delta}$ respectively.
			The isometry $\bar g_{m-1}$ is not on the picture.
			It is not clear a priori whether it is a reflection or locally trivial at $\bar v$.
		}
    	\label{fig: lifting chain}
    \end{figure}
	Let us explore further the second configuration.
	Observe first that  $\bar g_1, \dots, \bar g_{m-2}$ belong to $\bar C \cap \bar h^{-1}\bar E_1 \bar h$ so that there images under $q_{\bar v}$ are contained in $\group{\mathbf x} \cap \group{\mathbf x'} = \{1\}$.
	In other words $\bar g_1, \dots, \bar g_{m-2}$ belong to the kernel of $q_{\bar v}$ which we denote by $\bar F$.
	Consequently 
	\begin{equation*}
		q_{\bar v} \left( \group{\bar g_0}\right) = q_{\bar v}\left( \bar h^{-1}\bar E_1\bar h\right) = \group{\mathbf x'}.
	\end{equation*}
	Hence $q_{\bar v}(\bar g_0)=\mathbf x'$ and $\bar g_0^2$ also belongs to $\bar F$.
	Since $\bar g_{m-1}$ lies in $\bar C$, it image under $q_{\bar v}$ belongs to $\group{\mathbf x}$.
	Thus $\bar g_{m-1}^2$ belongs to $\bar F$.
	It is not clear though whether $\bar g_{m-1}$ is a reflection at $\bar v$ or not.
	This will force us to distinguish later two cases.
	We denote by $\bar F_0$ the \emph{normal} subgroup of $\stab{\bar v}$ generated by 
	\begin{itemize}
		\item $\bar g_0^2, \bar g_1, \dots, \bar g_{m-1}$, if $\bar g_{m-1}$ is locally trivial at $\bar v$, and
		\item $\bar g_0^2, \bar g_1, \dots, \bar g_{m-2}, \bar g_{m-1}^2$, if $\bar g_{m-1}$ is a reflection at $\bar v$.
	\end{itemize}
	Note that in both cases $\bar F_0$ is subgroup of $\bar F$.
	According to \autoref{res : existence of central half-turn}, there exist $k \in \N$, and an abstract group $\mathbf E \in \boldsymbol{\mathcal E}$, such that $\stab{\bar v}$ embeds in $\dihedral[n] \times \dihedral[n_2]^k\times \mathbf E$.
	Recall that the reflection rank of $\bar F_0$ is the rank as a $\cyclic[2]$-vector space of the image of $\bar F_0$ under the signature map $\dihedral[n] \times \dihedral[n_2]^k\times \mathbf E \onto \cyclic[2] \times \cyclic[2]^k$ that ``kills'' $\mathbf E$ as well as the rotation groups in all the dihedral factors  (see \autoref{sec: invariants - mixed}).
	Since $\cyclic[2] \times \cyclic[2]^k$ has exponent two, the images of $\bar g^2_0$ and $\bar g_{m-1}^2$ under this map are trivial.
	Hence the reflection rank of $\bar F_0$ is at most $m-1$.
	Denote by $\bar r$ a pre-image of $\mathbf r$ in $\stab{\bar v}$.
	By construction, $\bar r^{-1}$ normalizes $\bar F_0$.
	Recall that $2^{m+2}\mu$ divides $n$.
	According to \autoref{res: norm in prod of dihedral} there exists an element $\bar u_0 \in \bar F_0$ such that $\bar a = \bar r^{-n/4} \bar u_0^{-1}$ satisfies the following properties:
	\begin{enumerate}
		\item \label{enu: nu-inv - loxodromic - case 2.1 - same action as 1/4 turn}
		$\bar a$ centralizes $\bar F_0$
		\item \label{enu: nu-inv - loxodromic - case 2.1 - double comm}
		$[[\bar a, \bar u_1],\bar u_2]=1$, for every $\bar u_1, \bar u_2 \in \stab{\bar v}$.
	\end{enumerate}
	Recall that $\bar u_0$ belongs to $\bar F$, i.e.\ the kernel of $q_{\bar v}$.
	Hence $\bar a$ is a pre-image of $\mathbf r^{-n/4}$.
	We now let 
	\begin{equation*}
		\bar g'_0 = \bar a \bar g_0 \bar a^{-1}, \quad 
		\bar g'_k = \bar g_k, \quad \forall k \in \intvald 1{m-1},
		\quad \text{and} \quad
		\bar g'_m = \bar h \bar a^{-1} \bar g_{m-1} \bar a \bar h^{-1}.
	\end{equation*}
	By construction $\bar a$ commutes with $\bar g_1, \dots, \bar g_{m-2}$.
	Hence $\bar{\mathcal C}' = (\bar g'_0, \dots, \bar g'_m)$ is a chain with $\bar h_0  = \bar h \bar a^{-1}$ as conjugating element.	
	Note also that $\bar{\mathcal C}$ and $\bar{\mathcal C'}$ coincides except maybe on the first and last element.
	In addition we set $\bar A' = \group{\bar g'_0, \bar C}$ and $\bar B' = \group{\bar g'_m, \bar C}$ so that the subgroup of $\bar G$ generated by $\bar{\mathcal C}'$ is contained in $\bar E' = \group{\bar A', \bar B'}$.

	\medskip
	We first focus on the subgroup $\bar A'$.
	We claim that $\bar C$ has index at most two in $\bar A'$.
	Identity \ref{enu: nu-inv - loxodromic - case 2.1 - double comm} tells us that given $\bar c \in \bar C$ and $\epsilon \in \{\pm 1\}$, we have $[[\bar a, \bar g_0^\epsilon], \bar c] = 1$, which can be reformulated as
	\begin{equation}
	\label{eqn: nu-lift - loxodromic - normalize 0}
		(\bar g'_0)^{-\epsilon} \bar c (\bar g'_0)^\epsilon = \left(\bar a \bar g_0^{-\epsilon} \bar a^{-1}\right) \bar c \left( \bar a \bar g_0^\epsilon \bar a^{-1}\right) = \bar g_0^{-\epsilon} \bar c \bar g_0^\epsilon.
	\end{equation}
	As $\bar g_0$ normalizes $\bar C$, so does $\bar g'_0$.
	Since $\bar a$ commutes with $\bar g_0^2$, we have $(\bar g'_0)^2 = \bar g_0^2$, hence $(\bar g'_0)^2$ belongs to $\bar C$.
	Consequently $\bar C$ is a subgroup of index at most $2$ of $\bar A'$, which completes the proof or our claim.
	Recall that $\bar C$ is normalized by any loxodromic element in $\bar E$.
	Consequently $\bar A'$ is a dihedral germ.
	
	We now claim that $\bar A'$ can be lifted.
	On the one hand,  $\bar a$ is a pre-image of $\mathbf r^{-n/4}$ so that 
	\begin{equation*}
		q_{\bar v}(\bar g'_0) = \mathbf r^{-n/4}\mathbf x'\mathbf r^{n/4} = \mathbf x
	\end{equation*}
	On the other hand, $q_{\bar v}$ maps $\bar C$ to $\group{\mathbf x}$.
	Consequently $q_{\bar v}(\bar A') = \group{\mathbf x}$ and $\bar A'$ can be lifted in $G$ (\autoref{res: dichotomy elliptic}), which completes the proof of our claim.
	
	\medskip
	Our next focus is on the subgroup $\bar B'$.
	We are going to prove that $\bar C$ has index at most $2$ in $\bar B'$.	
	To that end, we distinguish two cases depending on the status of $\bar g_{m-1}$ at $\bar v$.

	\subparagraph{Case 2.1.} \emph{Suppose that $\bar g_{m-1}$ is locally trivial at $\bar v$}.
	Since it belongs to $\bar F_0$, the element $\bar g_{m-1}$ is centralized by $\bar a$.
	It follows from \ref{enu: nu-inv - loxodromic - case 2.1 - same action as 1/4 turn} that $\bar g'_m = \bar g_m$ and $\bar B' = \bar B$.
	Hence $\bar C$ has index $2$ in $\bar B'$, as we announced.

	\subparagraph{Case 2.1.} \emph{Suppose that $\bar g_{m-1}$ is a reflection at $\bar v$}.
	We first claim that $\bar B = \group{\bar g_m, C}$ is contained in $\stab{\bar h \bar v}$.
	As we already noticed $q_{\bar v}( \bar g_{m-1})$ belongs to $\group{\mathbf x}$.
	Since it is a reflection at $\bar v$, we actually have $q_{\bar v}(\bar g_{m-1}) = \mathbf x$.
	Recall that $\mathbf x$ and $\mathbf x'$ are distinct reflection at $\bar v$.
	Consequently $\bar g_0 \bar g_{m-1}$ is a strict rotation at $\bar v$.
	After conjugation by $\bar h$, we obtain that $\bar g_1\bar g_m$ is a strict rotation at $\bar h \bar v$ (recall that, unlike reflections, strict rotations ``behaves well'' under conjugation, see \autoref{res: local action of apex stab}).
	However  $\bar g_1\bar g_m$ belongs to $\bar B$.
	Hence according to \autoref{rem: elliptic subgroup with strict rotation}, the subgroup $\bar B$ is contained in $\stab{\bar h \bar v}$, which completes the proof of our claim.
	Note that, along the way, we proved that $\bar B$ cannot be lifted as it contains a strict rotation at $\bar h \bar v$ (\autoref{res: dichotomy elliptic}).
	This will play a role in our discussion later (see Case~2.B).

	According to our previous discussion $\bar h^{-1}\bar C \bar h$ is contained in $\stab{\bar v}$.
	Identity \ref{enu: nu-inv - loxodromic - case 2.1 - double comm} implies that given $\bar c \in \bar C$ and $\epsilon \in \{\pm 1\}$, we have $[[\bar a^{-1}, \bar g_{m-1}^\epsilon], \bar h^{-1}\bar c\bar h] = 1$, which yields
	\begin{equation}
	\label{eqn: nu-lift - loxodromic - normalize m}
		\begin{split}
			(\bar g'_m)^{-\epsilon} \bar c (\bar g'_m)^\epsilon 
			& = \bar h\left(\bar a^{-1} \bar g_{m-1}^{-\epsilon} \bar a\right) \left(\bar h^{-1} \bar c \bar h\right) \left( \bar a^{-1} \bar g_{m-1}^\epsilon \bar a\right) \bar h^{-1} \\
			& = \bar h \bar g_{m-1}^{-\epsilon} \left(\bar h^{-1} \bar c \bar h\right)  \bar g_{m-1}^\epsilon \bar h^{-1} \\
			& = (\bar g_m)^{-\epsilon} \bar c (\bar g_m)^\epsilon 
		\end{split}
	\end{equation}
	However $\bar g_m$ normalizes $\bar C$, hence so does $\bar g'_m$.
	Moreover $\bar a$ commutes with $\bar g_{m-1}^2$.
	Consequently $(\bar g'_m)^2 = \bar g_m^2$ belongs to $\bar C$.
	Hence $\bar C$ has index at most $2$ in $\bar B'$, as we announced.	
	\begin{rema*}
		Although we will not use this fact later, the reader can check, as we did for $\bar A'$, that in this case $q_{\bar v}(\bar h^{-1} \bar B' \bar h) = \group{\mathbf x'}$ so that $\bar h^{-1} \bar B' \bar h$ and thus $\bar B'$ can be lifted.
	\end{rema*}


	\medskip
	We now proceed further without any distinction on the status of $\bar g_{m-1}$ at $\bar v$.
	We proved that $\bar C$ has index at most $2$ in both $\bar A'$ and $\bar B'$.
	Hence the subgroup $\bar E' = \group{\bar A', \bar B'}$ is elementary as a cyclic or dihedral extension of the elliptic subgroup $\bar C$.
	Recall that $(\bar g'_0)^2 = \bar g_0^2$, $\bar g'_1 = \bar g_1$, \dots, $\bar g'_{m-1} = \bar g_{m-1}$, and $(\bar g'_m)^2 = \bar g_m^2$ belong to $\bar C$.
	Moreover, according to (\ref{eqn: nu-lift - loxodromic - normalize 0}) the action by conjugation of $\bar g'_0$ on $\bar C$ coincides with the one of $\bar g_0$.
	The same holds for $\bar g'_m$ and $\bar g_m$.
	This is indeed obvious if $\bar g'_m = \bar g_m$ as in Case~2.1 and follows from (\ref{eqn: nu-lift - loxodromic - normalize m}) in Case~2.2.
	Consequently, the subgroup generated by $\bar {\mathcal C}'$ is exactly $\bar E'$, hence is elementary.
	Recall that $\bar A'$ is a dihedral germ.
	Being a dihedral germ is stable under taking subgroups (\autoref{rem: pushing assum elem}), thus $\bar E'_0 = \group{\bar g'_0, \dots, \bar g'_{m-1}}$ is a dihedral germ as well.
	This shows that $\bar {\mathcal C}'$ is a strong chain.
	As we just observed above, $\bar{\mathcal C'}$ generates a subgroup $\bar E'$, which is either elliptic, parabolic or loxodromic.
	We distinguish again two cases depending on the nature of $\bar E'$.
	\begin{itemize}
		\item  Case 2.A. \emph{Assume that $\bar E'$ is not loxodromic.}
		Note that $\bar E'_0$ is contained in $\bar A'$.
		Since $\bar A'$ can be lifted $\bar E'_0$ does not contain a strict rotation (\autoref{res: dichotomy elliptic}).
		If $\bar E'$ is elliptic or parabolic, then Lemmas~\ref{res: nu-lift - para and elliptic w/ lift} and \ref{res: nu-lift - elliptic w/o lift} tell us that there exists a strong chain $\mathcal C'$ lifting $\bar{\mathcal C'}$.
		\item Case 2.B. \emph{Assume that $\bar E'$ is loxodromic.}
		The group $\bar E'$ is generated by two elliptic subgroups, namely $\bar A'$ and $\bar B'$.
		Thus $\bar C$ has index exactly $2$ in both $\bar A'$ and $\bar B'$ and $\bar E' / \bar C$ is isomorphic to $\dihedral$.
		Observe that $\bar E'_0$ does not contain a loxodromic element (it lies in the elliptic subgroup $\bar A'$).
		We rerun the previous discussion replacing $\bar{\mathcal C}$ and $\bar E = \bar A \ast_{\bar C} \bar B$ by $\bar {\mathcal C}'$ and $\bar E' = \bar A' \ast_{\bar C} \bar B'$.
		In particular, if $\bar B'$ can be lifted, we are back to Case~1.
		This means that there exists a strong chain $\bar {\mathcal C}''$ which can be lifted and which coincides with $\bar{\mathcal C}'$ except maybe for the first or the last term.
		Assume now that $\bar B'$ cannot be lifted.
		We now follow Case~2 with the reversed chain $(\bar g'_m, \dots, \bar g'_0)$, i.e.\ we permute the role of $\bar A$ and $\bar B$.
		We produce a new strong chain $\bar{\mathcal C}'' = (\bar g''_0, \dots, \bar g''_m)$ such that $(\bar g''_1, \dots, \bar g''_{m-1}) = (\bar g'_1, \dots, \bar g'_{m-1})$ which generates an elementary subgroup of the form $\bar E'' =  \group{\bar A'',\bar B''}$ where $\bar B''$ is an elliptic subgroups which can be lifted and containing $\bar C$ as a subgroup of index at most $2$.
		Recall that $\bar A'$ is an elliptic subgroup that can be lifted.
		Hence as we rerun Case~2, we are more precisely following Case~2.1 so that $\bar A'' = \bar A'$.
		Consequently both $\bar A''$ and $\bar B''$ are elliptic subgroups which can be lifted.
		Depending whether $\bar E''$ is loxodromic or not, we rerun again Case~1, or use Lemmas~\ref{res: nu-lift - para and elliptic w/ lift} and \ref{res: nu-lift - elliptic w/o lift} to produce a strong chain $\bar{\mathcal C}'''$ which can be lifted and coincides with $\bar{\mathcal C}''$ except maybe for the first or the last terms. \qedhere
	\end{itemize}
\end{proof}

Observe that \autoref{res: nu-lift - dichotomy} is a combination of  Lemmas~\ref{res: nu-lift - para and elliptic w/ lift}, \ref{res: nu-lift - elliptic w/o lift} and \ref{res: nu-lift - loxodromic}

\begin{proof}[Proof of \autoref{res: nu-inv}]
	Let $m \geq \max\{ \nu, \mu +4\}$.
	Let $\bar {\mathcal C} = (\bar g_0, \dots, \bar g_m)$ be a strong chain of $\bar G$.
	Let $\bar h$ be a conjugating element of $\bar{\mathcal C}$.
	According to \autoref{res: nu-lift - dichotomy} one of the following holds.
	\begin{enumerate}
		\item \label{enu: nu-inv elliptic}
		There exists $\bar v \in \bar{\mathcal V}$ such that $\group{\bar g_0, \bar h}$ is contained in $\stab{\bar v}$.
		\item \label{enu: nu-inv loxo}
		The subgroup $\group{\bar g_0, \bar h}$ is loxodromic,
		\item \label{enu: nu-inv lifted}
		There exists a strong chain $\bar {\mathcal C}' = (\bar g'_0, \dots, \bar g'_m)$ of $\bar G$ which can be lifted such that $(\bar g'_1, \dots, \bar g'_{m-1}) = (\bar g_1, \dots, \bar g_{m-1})$.
	\end{enumerate}
	We study each case separately.
	Assume first that $\group{\bar g_0, \bar h}$ is contained in $\stab{\bar v}$ for some $\bar v \in \bar{\mathcal V}$.
	Note that $\bar h$ cannot be loxodromic (it fixes $\bar v$).
	By the very definition of strong chain $\group{\bar g_0, \dots, \bar g_{m-1}}$ is a dihedral germ.
	According to \autoref{res : existence of central half-turn} there exist $k \in \N$ and $\mathbf E \in \boldsymbol {\mathcal E}$ such that $\stab{\bar v}$ embeds in $\dihedral[n] \times \dihedral[n_2]^k \times \mathbf E$.
	Since $m \geq \mu +2$, it follows from \autoref{res: dihedral complexity} that $\bar g_0$ and $\bar g_m$ respectively belong to $\group{\bar g_1, \dots, \bar g_m}$ and $\group{\bar g_0, \dots, \bar g_{m-1}}$.
	In other words $\bar h$ normalizes $\group{\bar g_0, \dots, \bar g_{m-1}}$.
	Hence $\group{\bar g_0, \bar h}$ is a cyclic extension of the dihedral germ $\group{\bar g_0, \dots, \bar g_{m-1}}$, therefore it has dihedral shape.
	
	Assume now that $\group{\bar g_0, \bar h}$ is loxodromic.
	Then it automatically has dihedral shape (\autoref{res: loxo have dihedral shape}).
	
	We are left with the last case.
	Let $\mathcal C' = (g'_0, \dots, g'_m)$ be a strong chain of $G$ lifting the chain $\bar {\mathcal C}'$ given by \ref{enu: nu-inv lifted}.
	Let $h_0$ be a conjugating element of $\mathcal C'$.
	Recall that $m \geq \nu$.
	Thus $\group{g'_0, h_0}$ is an elementary subgroup with dihedral shape.
	Recall that by \autoref{ass: elem subgroup} every dihedral pair of $G$ has type $(\boldsymbol{\mathcal E}, n_2)$.
	Since $m-2 \geq \mu +2$, it follows from \autoref{res: dihedral complexity} that $g'_1$ and $g'_{m-1}$ respectively belong to $\group{g'_2, \dots, g'_{m-1}}$ and $\group{g'_1, \dots, g'_{m-2}}$.
	Recall that $\bar{\mathcal C}$ and $\bar {\mathcal C}'$ coincide everywhere but except maybe on the first and the last element.
	Pushing the previous observation in $\bar G$ we get that $\bar g_1$ and $\bar g_{m-1}$ respectively belong to $\group{\bar g_2, \dots, \bar g_{m-1}}$ and $\group{\bar g_1, \dots, \bar g_{m-2}}$.
	Hence $\bar h$ normalizes $\group{\bar g_1, \dots, \bar g_{m-2}}$ and a fortiori $\group{\bar g_0, \dots, \bar g_{m-1}}$.
	In particular, $\group{\bar g_0, \bar h}$ is elementary.
	If $\bar h$ is loxodromic, then $\group{\bar g_0, \bar h}$ has automatically dihedral shape (\autoref{res: loxo have dihedral shape})
	Otherwise $\group{\bar g_0, \bar h}$ is a cyclic extension of the dihedral germ $\group{\bar g_0, \dots, \bar g_{m-1}}$, and thus it has dihedral shape.
\end{proof}

%
\section{Periodic groups}
%
\label{sec: periodic}

%
\subsection{Induction step}
%
\label{sec: periodic - induction step}

The next proposition will play the role of the induction step in the final induction (see \autoref{res: partial periodic quotient - gal}).
It is the generalization of \cite[Lemma~6.2.1]{Delzant:2008tu} in the presence of even torsion.

\begin{prop}
\label{res: induction step}
	There exist positive constants $\delta_1$, $C_0$ and $C_1$ such that for every positive integer $\nu_0$, there exists a critical exponent $N_0 \in \N$ with the following properties.
	Let $\boldsymbol{\mathcal E}$ be a model collection of abstract groups whose exponent $\mu = \mu(\boldsymbol{\mathcal E})$ is finite.
	Let $N_1 \geq N_0$ and $n \geq N_1$ be a multiple of $2^{\nu_0 +2}\mu$.

	Let $G$ be a group acting by isometries on a $\delta$-hyperbolic length space $X$ for some $\delta \leq \delta_1$ and satisfying the following assumptions.
	\begin{enumerate}
		\item \label{enu: induction step - non elem}
		The action of $G$ on $X$ is gentle and non-elementary.
		\item \label{enu: induction step - elem subgroup}
		Dihedral pairs of $G$ have type $(\boldsymbol{\mathcal E},n_2)$, where $n_2$ is the largest power of $2$ dividing $n$.
		\item \label{enu: induction step - param}
		$A(G,X,420\delta)  \leq (\nu_0+5)C_0$, \\
		$\inj[X]G\geq 1/C_0\sqrt{N_1}$,\\
		$\max\{\nu_{\rm{stg}}(G,X),\mu + 4\} \leq \nu_0$.
	\end{enumerate}
	We denote by $P$ the set of all primitive loxodromic elements $h \in G$ such that $\norm h \leq 10\delta_1$.
	Let $K$ be the (normal) subgroup of $G$ generated by $\set{h^n}{h \in P}$ and $\bar G$ the quotient of $G$ by $K$.
	We write $\pi \colon G \onto \bar G$ for the corresponding quotient map.

	Then there exists a $\bar \delta$-hyperbolic length space $\bar X$, with $\bar \delta\leq \delta_1$, on which $\bar G$ acts by isometries satisfying \ref{enu: induction step - non elem}-\ref{enu: induction step - param}.
	In addition there exists a $\pi$-equivariant map $X \to \bar X$ with the following properties.
	\begin{itemize}
		\item The map $X \to \bar X$ is $C_1/\sqrt{N_1}$-Lipschitz.
		\item If $F$ is an elliptic (\resp parabolic) subgroup of $G$, then $\pi$ induces an isomorphism from $F$ onto its image $\bar F$ which is also elliptic (\resp elliptic or parabolic).
		\item Any elliptic subgroup of $\bar G$, is either the isomorphic image of an elliptic subgroup of $G$, or has finite exponent dividing $n$.
		\item Any finitely generated parabolic subgroup of $\bar G$ is the isomorphic image of a parabolic subgroup of $G$.
		\item Let $F_1$ and $F_2$ be two subgroups of $G$.
		Assume that $F_1$ is elliptic and $F_2$ is generated by a set $S_2$ such that $\fix{S_2,100\delta_1}$ is non-empty.
		If the images of $F_1$ and $F_2$ are conjugated in $\bar G$, then $F_1$ and $F_2$ are conjugated in $G$.
	\end{itemize}
\end{prop}

\begin{voca*}
	Assume that $\nu_0$, $n$ and the model collection $\boldsymbol{\mathcal E}$ have been already fixed.
	If $G$ is a group acting on a metric space $X$ satisfying the assumptions of the proposition, including Points~\ref{enu: induction step - non elem}-\ref{enu: induction step - param}, we say that $(G,X)$ \emph{satisfies the induction hypotheses relative to $(n, \boldsymbol{\mathcal E})$}.
	The proposition says among others that if $(G,X)$ satisfies the induction hypotheses relative to $(n, \boldsymbol{\mathcal E})$, then so does $(\bar G,\bar X)$.
\end{voca*}

\begin{proof}
	We start by defining the various constants appearing in the statement.
	Let $\delta_0$, $\delta_1$, $\Delta_0$, and $\rho_0$ be the parameters given by the small cancellation theorem (\autoref{res: small cancellation}).
	Let $\rho \geq \rho_0$.
	Set $\kappa = \delta_1/\pi \sinh (10 \delta_1)$ (so that we can apply \autoref{res: injectivity radius}).
	We fix $C_0$ and $C_1$ as follows.
	\begin{align*}
		C_0 = \pi\sinh(420\delta_1)
		\quad \text{and} \quad
		C_1 & = \max \left\{ 10C_0\pi \sinh \rho , \frac 1{2C_0 \kappa \delta_1}\right\}.
	\end{align*}
	Observe that $\delta_1 \ll C_0 \ll \rho \ll C_1$.
	For every integer $N\in\N$ we define a rescaling parameter $\epsilon_N$ as follows
	\begin{equation*}
		\epsilon_N = \frac {C_1}{\sqrt N}
	\end{equation*}
	Let $\nu_0 \in \N$.
	The sequence $(\epsilon_N)$ converges to $0$ as $N$ tends to infinity.
	Therefore there exists a critical exponent $N_0\in \N$, such that for every integer $N \geq N_0$ we have
	\begin{align}
		\label{eqn: recale - hyp}
		\epsilon_N \delta_1 & \leq \delta_0, \\
		\label{eqn: recale - A}
		\epsilon_N (\nu_0 + 5)C_0 & \leq \min\{\Delta_0,C_0\} \\
		\label{eqn: recale - inj}
		\epsilon_N \kappa & \leq 1/2. 
	\end{align}
	Let $\boldsymbol{\mathcal E}$ be a collection of (abstract) groups and $\mu$ its exponent.
	We now fix $N_1 \geq N_0$.
	For simplicity we write $\epsilon$ instead of $\epsilon_{N_1}$.
	Let $n \geq N_1$ be a multiple of $2^{\nu_0 +2}\mu$.
	In particular, $\mu$ divides $n/2$.
	Consequently for every $\mathbf E \in \boldsymbol{\mathcal E}$, the exponents of $\mathbf E$ and $\mathbf E/Z(\mathbf E)$ respectively divide $n$ and $n/2$, which means that the model collection $\boldsymbol{\mathcal E}$ satisfies the assumptions stated in \autoref{sec: invariants quotient - mixed}.
	
	Let $G$ be a group acting on a $\delta$-hyperbolic space $X$ such that $(G,X)$ satisfies the induction hypotheses relative to $(n, \boldsymbol{\mathcal E})$.
	We denote by $P$ the set of all primitive loxodromic elements $h \in G$ such that $\norm h \leq 10\delta_1$.
	Let $K$ be the (normal) subgroup of $G$ generated by $\set{h^n}{h \in P}$ and $\bar G$ the quotient of $G$ by $K$.
	If $P$ is empty, then $\bar G = G$.
	Thus $\bar X = \epsilon X$ obviously satisfies the conclusion of the proposition.
	Otherwise, we are going to prove that $\bar G$ is a small cancellation quotient of $G$.
	To that end, we consider the action of $G$ on the rescaled space $\epsilon X$.
	According to (\ref{eqn: recale - hyp}) this space is $\epsilon\delta$-hyperbolic where $\epsilon \delta\leq \delta_0$.
	We define the family $\mathcal Q$ by
	\begin{equation*}
		\mathcal Q = \set{\left(\fantomB\group {h^n}, Y_h\right)}{h \in P}.
	\end{equation*}
	
	\begin{lemm}
	\label{res: induction - check sc hyp}
		The family $\mathcal Q$ satisfies the small cancellation hypotheses, i.e.\ $\Delta(\mathcal Q, \epsilon X) \leq \Delta_0$ and $\inj[\epsilon X]{\mathcal Q} \geq 10 \pi \sinh \rho$.
	\end{lemm}
	
	\begin{proof}
		We start with the upper bound of $\Delta(\mathcal Q, \epsilon X)$.
		Let $h_1$ and $h_2$ be two elements of $P$ such that $(\group {h_1^n}, Y_{h_1})$ and $(\group {h_2^n}, Y_{h_2})$ are distinct.
		We first claim that $h_1$ and $h_2$ generate a non-elementary subgroup.
		Assume on the contrary that it is not the case.
		Let $E$ be the maximal elementary subgroup containing $h_1$ and $h_2$. 
		This subgroup is necessarily loxodromic.
		We denote by $F$ its maximal elliptic normal subgroup, so that $(E,F)$ is a dihedral pair.
		According to our assumption there exists $k \in \N$ and $\mathbf E \in \boldsymbol{\mathcal E}$ such that $E$ embeds in $E/F \times \dihedral[n_2]^k\times \mathbf E$, where $n_2$ is the largest power of $2$ dividing $n$.
		Recall that $h_1$ and $h_2$ are primitive.
		Hence up to replacing $h_2$ by its inverse, we may assume that $h_1$ and $h_2$ have the same image in $E/F$.
		Since the exponents of $\dihedral[n_2]$ and $\mathbf E$ divide $n$, the images of $h_1^n$ and $h_2^n$ are trivial in $\dihedral[n_2]^n\times \mathbf E$.
		Consequently $h_1^n = h_2^n$ and thus $Y_{h_1} = Y_{h_2}$.
		This contradicts the fact that $(\group {h_1^n}, Y_{h_1})$ and $(\group {h_2^n}, Y_{h_2})$ are distinct and completes the proof of our claim.
		
		Recall that $h_i$ moves the points of $Y_{h_i}$ by at most $\norm[\epsilon X]{h_i} + 65\epsilon\delta$ (\autoref{res: cyl in mov}), while $\norm[\epsilon X]{h_i} \leq 10\epsilon\delta$.
		Consequently 
		\begin{equation*}
			Y_{h_1}^{+5\epsilon\delta} \cap Y_{h_2}^{+5\epsilon\delta} \subset \fix{\fantomB\{h_1,h_2\},85\epsilon\delta}.
		\end{equation*}
		Since $h_1$ and $h_2$ generate a non-elementary subgroup we obtain
		\begin{equation*}
			\diam\left(Y_{h_1}^{+5\epsilon\delta} \cap Y_{h_2}^{+5\epsilon\delta}\right) 
			\leq A(G,\epsilon X,420\epsilon\delta)
			\leq \epsilon A(G,X,420\delta)
			\leq \epsilon(\nu_0 + 5)C_0.
		\end{equation*}
		Using (\ref{eqn: recale - A}) we get $\Delta(\mathcal Q, \epsilon X) \leq \Delta_0$.
		Let us now focus on $\inj[\epsilon X]{\mathcal Q}$.
		It follows from our assumption on $\inj[X]{\mathcal Q}$ that 
		\begin{equation*}
			\inj[\epsilon X]G
			\geq \epsilon \inj[X]G
			\geq \frac {10C_0\pi \sinh \rho }{\sqrt{N_1}}\frac 1{C_0\sqrt{N_1}}
			\geq \frac{10\pi \sinh \rho}{N_1} 
			\geq \frac{10\pi \sinh \rho}n.
		\end{equation*}
		Let $(H,Y) \in \mathcal Q$.
		By construction, any element $g \in H$ is the $n$-th power of a loxodromic element of $G$.
		Consequently
		\begin{equation*} 
			\norm[\epsilon X]g
			\geq n \inj[\epsilon X]G
			\geq 10\pi \sinh \rho.
		\end{equation*}
		It follows that $\inj[\epsilon X]{\mathcal Q} \geq 10\pi \sinh \rho$.
	\end{proof}
	
	\paragraph{}On account of the previous lemma, we can now apply the small cancellation theorem (\autoref{res: small cancellation}) to the action of $G$ on the rescaled space $\epsilon  X$ and the family $\mathcal Q$.
	We denote by $\dot X$ the space obtained by attaching on $\epsilon X$ for every $(H,Y) \in \mathcal Q$, a cone of radius $\rho$ over the set $Y$.
	The space $\bar X$ is the quotient of $\dot X$ by $K$.
	According to \autoref{res: small cancellation}, $\bar X$ is a $\bar \delta$-hyperbolic length space with $\bar \delta \leq \delta_1$ and $\bar G$ acts by isometries on it.
	As usual, we write $\mathcal V$ for the set of apices in $\dot X$ and $\bar{\mathcal V}$ for its image in $\bar X$.
	We now prove that the action of $\bar G$ on $\bar X$ satisfies the induction hypotheses relative to $(n, \boldsymbol{\mathcal E})$.
	This action is gentle (\autoref{res: sc - gentle action}) and non-elementary (\autoref{res: action quotient non elem}), which provides \ref{enu: induction step - non elem}.
	In addition dihedral pairs of $\bar G$ have type $(\boldsymbol{\mathcal E},n_2)$ (\autoref{res: pushing assum elem}).
	Thus \ref{enu: induction step - elem subgroup} holds.
	Point~\ref{enu: induction step - param} is a consequence of the following lemma.

	\begin{lemm}
		The parameters $A(\bar G, \bar X)$, $\inj[\bar X]{\bar G}$ and $\nu_{\rm{stg}}(\bar G, \bar X)$ satisfy
		\begin{enumerate}
			\item $A(\bar G,\bar X,420\bar \delta)  \leq (\nu_0+5)C_0$;
			\item $\inj[\bar X]{\bar G}\geq 1/C_0\sqrt{N_1}$;
			\item $\max\{\nu_{\rm{stg}}(\bar G,\bar X),\mu + 4\} \leq \nu_0$.
		\end{enumerate}
	\end{lemm}
	
	\begin{proof}
		We start with the upper bound of $A(\bar G, \bar X)$.
		\autoref{res: acyl quotient preparatory} yields
		\begin{equation*}
			A(\bar G, \bar X,420\bar \delta) 
			\leq A(G, \epsilon X, d) + 420\bar \delta
		\end{equation*}
		where $d = \pi\sinh(420\bar \delta)$.
		Applying \autoref{res: local-to-global acyl} in $\epsilon X$, we obtain that
		\begin{equation*}
			A(\bar G, \bar X,420\bar \delta) 
			\leq  A(G,  \epsilon X, 420\epsilon\delta) + \left[ \nu(G,X) +4 \right] \pi \sinh(420\delta_1).
		\end{equation*}
		Since $\nu(G,X)$ is bounded above by $\nu_{\rm{stg}}(G,X)$, hence by $\nu_0$, we get that 
		\begin{equation*}	
			A(\bar G, \bar X) \leq \epsilon (\nu_0+5)C_0 + (\nu_0+4)C_0.
		\end{equation*}
		Using (\ref{eqn: recale - A}) we obtain $A(\bar G, \bar X) \leq (\nu_0+5)C_0$.
		We now focus on the injectivity radius of $\bar G$.
		Let $g$ be a loxodromic isometry of $G$.
		Recall that the action of $G$ on $X$ is gentle.
		Consequently, we can write $g = g_0^ku$ where $k$ is a positive integer, $g_0$ a primitive element and $u$ is contained in an elliptic subgroup of $G$ normalized by $g_0$.
		In particular, $\snorm[\epsilon X] g = k\snorm[\epsilon X]{g_0}$.
		Assume now that $g$ does not stabilize any cylinder $Y_h$, where $h \in P$.
		It follows that $g_0$ does not belong to $P$.
		Thus by (\ref{eqn: regular vs stable length})
		\begin{equation*}
			\snorm[\epsilon X] {g_0} 
			\geq \norm [\epsilon X]{g_0} - 8\epsilon\delta 
			\geq 2\epsilon\delta_1
		\end{equation*}
		\autoref{res: injectivity radius} applied with $N = G$ yields 
		\begin{equation*}
			\inj[\bar X]{\bar G}		
			\geq \min\left\{ 2\epsilon \kappa\delta_1, \delta_1 \right\}
		\end{equation*}
		Combined with (\ref{eqn: recale - inj}) we obtain 
		\begin{equation*}
			\inj[\bar X]{\bar G}		
			\geq 2\epsilon \kappa\delta_1
			\geq \frac 1 {C_0\sqrt{N_1}}.
		\end{equation*}
		The upper bound for $\nu_{\rm{stg}}(\bar G, \bar X)$ directly follows from \autoref{res: nu-inv}.
	\end{proof}

	We now study the properties of the projection $\pi \colon G \to \bar G$.
	Recall that the map $\zeta \colon \epsilon X \to \bar X$ is $1$-Lipschitz.
	Hence $X \to \bar X$ is $\epsilon$-Lipschitz (and $\pi$-equivariant by construction).
	If $F$ is an elliptic (\resp parabolic) subgroup of $G$, then it follows from \autoref{res: isom on non-loxo elem} that $\pi \colon G \to \bar G$ induces an isomorphism from $F$ onto its image $\bar F$.
	Moreover, $\bar F$ is elliptic (\resp elliptic or parabolic).
	Let $\bar F$ be an elliptic subgroup of $\bar G$.
	If $\bar F$ is not the isomorphic image of an elliptic subgroup of $G$, then there exists an apex $\bar v \in \bar{\mathcal V}$ such that $\bar F$ is contained in $\stab {\bar v}$ (\autoref{res: dichotomy elliptic}).
	On the other hand, there exist $k \in \N$ and $\mathbf E \in \boldsymbol{\mathcal E}$ such that $\stab{\bar v}$ embeds in $\dihedral[n] \times \dihedral[n_2]^k \times \mathbf E$ (\autoref{res : existence of central half-turn}).
	Since the exponent of $\mathbf E$ divides $n$, the exponent of $\bar F$ is finite and divides $n$ as well.
	By \autoref{res: lifting parabolic subgroups} any finitely generated parabolic subgroup of $\bar G$ is the isomorphic image of a parabolic subgroup of $G$.
	Let $F_1$ and $F_2$ be two subgroups of $G$.
	Assume that $F_1$ is elliptic and $F_2$ is generated by a finite set $S_2$ such that $\fix{S_2,100\delta_1}$ is non-empty.
	It follows from our choice of $\rho$, that $100\delta_1 \leq \rho/100$.
	Thus, if the respective images $\bar F_1$ and $\bar F_2$ are conjugated in $\bar G$, then so are $F_1$ and $F_2$ (\autoref{res: lifting conj elliptic}).
	We have checked all the announced properties of the projection $\pi \colon G \to \bar G$, and the proof of the proposition is completed.
\end{proof}

%
\subsection{Construction of periodic groups}
%
\label{sec: construction of periodic groups}

The number of variables in the next statement can be confusing at first sight.
Basically we are stating the fact that the critical exponent $N_1$ does not depend on the group $G$, but only on certain parameters related to its action on a hyperbolic space $X$.
More precisely $N_1$ is a function of 
\begin{itemize}
	\item the hyperbolicity constant $\delta$ of $X$;
	\item the invariants $\inj[X]G$, $\nu_{\rm stg}(G,X)$, and $A(G,X,d)$ for some appropriate value of $d$ ;
	\item the structure of subgroups of $G$ with dihedral shape.
\end{itemize}
For certain applications, which are not covered in this article, it can be crucial to make sure that the critical exponent $N_1$ does not depend on the group $G$ but only on the parameters of the action of $G$ on $X$.
Examples of such applications can be found in Coulon-Gruber \cite[Theorems~2.9 and 2.14]{Coulon:2019ac} or Gruber-Mackary \cite[Theorem~1.2]{Gruber:2021ug} where one needs to run (the analog for odd exponents of) \autoref{res: partial periodic quotient - gal} with the same exponent $n$ on several groups simultaneously.

\begin{theo}
\label{res: partial periodic quotient - gal}
	Let $\delta, r \in \R_+^*$, $\nu, \mu \in \N$ and set $\nu_1 = \max\{ \nu +2, \mu + 6\}$.
	There exist $N_1 \in \N$ such that for every integer $n \geq N_1$ which is a multiple of $2^{\nu_1}\mu$, the following holds.
	
	Let $\boldsymbol{\mathcal E}$ be a model collection of groups whose exponent divides $\mu$.
	Let $G$ be a group acting on a $\delta$-hyperbolic length space $X$ such that
	\begin{itemize}
		\item the action of $G$ on $X$ is gentle and non-elementary;
		\item for every dihedral pair $(E,C)$ the group $E$ embeds in $E/C \times \mathbf E$ for some $\mathbf E \in \boldsymbol{\mathcal E}$.
		\item $A(G,X,420\delta) \leq r$, $\nu_{\rm{stg}}(G,X) \leq \nu$, and  $\inj[X]G \geq 1/r$
	\end{itemize}
	Then there exists a quotient $Q$ of $G$ with the following properties.
	\begin{enumerate}
		\item \label{enu: partial periodic quotient - gal - non loxo}
		For every elliptic (\resp parabolic) subgroup $F$ of $G$, the projection $G \onto Q$ induces an embedding from $F$ into $Q$.
		\item \label{enu: partial periodic quotient - gal - torsion}
		For every $q \in Q$, either $q^n = 1$ or $q$ is the image of an elliptic or parabolic element of $G$.
		\item \label{enu: partial periodic quotient - gal - max}
		The projection $G \onto G/G^n$ induces an epimorphism $Q \onto G/G^n$.
		In particular, if $G$ has no parabolic element, and every elliptic element of $G$ has finite order dividing $n$, then $Q = G/G^n$.
		\item \label{enu: partial periodic quotient - gal - inj on small balls}
		For every $x \in X$, the map $G \to Q$ is one-to-one when restricted to 
		\begin{equation*}
			\set{g \in G}{\dist {gx}x < r}.
		\end{equation*}
		\item \label{enu: partial periodic quotient - gal - infinite}
		There are infinitely many conjugacy classes in $Q$ whose elements are not the image of an elliptic or a parabolic element of $G$.
		\item \label{enu: partial periodic quotient - gal - kernel}
		The kernel $K$ of $G \onto Q$ is purely loxodromic (i.e.\ all its non-trivial elements are loxodromic).
		As a normal subgroup $K$ is not finitely generated.
	\end{enumerate}
\end{theo}

\begin{proof}
	The main ideas of the proof are the following.
	Using \autoref{res: induction step} we construct by induction a sequence of groups $G = G_0 \rightarrow G_1\rightarrow G_2 \rightarrow \dots$ where $G_{k+1}$ is obtained from $G_k$ by adding new relations of the form $h^n$ where $h$ is a primitive element of $G$.
	Then we choose for the quotient $Q = G/K$ the direct limit of these groups.
	
	\paragraph{Critical exponent.}
	Let us define first all the parameters leading to the critical exponent.
	For simplicity we let $\nu_0 = \max\{\nu, \mu + 4\}$ and $\nu_1 = \nu_0 + 2$.
	The parameters $\delta_1$, $C_0$, $C_1$, and $N_0$ are the one given by \autoref{res: induction step}.
	We choose $\epsilon > 0$ and an integer $N_1 \geq N_0$ such that 
	\begin{equation*}
		\epsilon \delta \leq \delta_1, \quad
		\epsilon r \leq \min \left\{(\nu_0+5)C_0,50\delta_1\right\}, \quad
		\frac \epsilon r \geq \frac 1{C_0\sqrt{N_1}}, 
		\quad \text{and} \quad 
		\frac{C_1}{\sqrt{N_1}} < 1.
	\end{equation*}
	We now a fix an integer $n \geq N_1$ which is divisible by $2^{\nu_1}\mu$. 
	
	\paragraph{The initialization.}
	Let $\boldsymbol{\mathcal E}$ be a model collection of groups and $G$ be a group acting on a $\delta$-hyperbolic length space $X$ as in the theorem.
	Let $X_0$ be the space $X$ whose metric has been rescaled by $\epsilon$.
	It follows from our choice of $\epsilon$ and $N_1$ that $X_0$ is $\epsilon\delta$-hyperbolic where $\epsilon \delta \leq \delta_1$, $A(G,X_0,400\epsilon\delta) \leq (\nu_0+5)C_0$, and $\inj[X_0]G \geq 1/C_0\sqrt{N_1}$.
	In addition $\max\{\nu_{\rm stg}(G, X_0),\mu + 4\} \leq \nu_0$.
	In other words, if $G_0 = G$, then $(G_0,X_0)$ satisfies the induction hypotheses relative to $(n, \boldsymbol{\mathcal E})$.

	\paragraph{The induction step.} 
	Let $k \in \N$.
	We assume that we already constructed the group $G_k$ and the space $X_k$ such that $(G_k,X_k)$ satisfies the induction hypotheses relative to $(n, \boldsymbol{\mathcal E})$.
	We denote by $P_k$ the set of primitive loxodromic elements $h \in G_k$ such that $\norm[X_k]h \leq10\delta_1$.
	Let $K_k$ be the normal subgroup of $G_k$ generated by $\{h^n, h \in P_k\}$.
	We write $G_{k+1}$ for the quotient of $G_k$ by $K_k$.
	According to \autoref{res: induction step}, there exists a metric space $X_{k+1}$ such that $(G_{k+1},X_{k+1})$ satisfies the induction hypotheses relative to $(n, \boldsymbol{\mathcal E})$.
	Moreover $X_{k+1}$ comes with a $C_1/\sqrt{N_1}$-Lipschitz map $X_k \to X_{k+1}$ which is $\pi_k$-equivariant, where $\pi_k \colon G_k \onto G_{k+1}$ is the canonical projection, and fulfills the following properties.
	\begin{labelledenu}[P]
		\item \label{enu: partial periodic quotient - gal - proof - one to one}
		If $F$ is an elliptic (\resp parabolic) subgroup of $G_k$, then $\pi_k$ induces an isomorphism from $F$ onto its image which is also elliptic (\resp elliptic or parabolic).
		\item \label{enu: partial periodic quotient - gal - proof - elliptic}
		Any elliptic subgroup of $G_{k+1}$ is either isomorphic to an elliptic subgroup of $G_k$ or a finite group whose exponent divides $n$.
		\item \label{enu: partial periodic quotient - gal - proof - parabolic}
		Any finitely generated parabolic subgroup of $G_{k+1}$ is the isomorphic image of a parabolic subgroup of $G_k$.
		\item \label{enu: partial periodic quotient - gal - proof - conj}
		Let $F_1$ and $F_2$ be two subgroups of $G_k$.
		Assume that $F_1$ is elliptic and $F_2$ is generated by a finite set $S_2$ such that $\fix{S_2,100\delta_1}$ is non-empty.
		If the images of $F_1$ and $F_2$ are conjugated in $G_{k+1}$, then $F_1$ and $F_2$ are conjugated in $G_k$.
	\end{labelledenu}

	\paragraph{Direct limit.} 
	The direct limit of the sequence $(G_k)$ is a quotient $Q = G/K$ of $G$.
	We claim that this group satisfies the announced properties.
	Let $E$ be a subgroup of $G$ which is either elliptic or parabolic.
	A proof by induction on $k\in \N$ using \ref{enu: partial periodic quotient - gal - proof - one to one} shows that for every $k \in \N$, the map $G \onto G_k$ induces an isomorphism from $E$ onto its image which is either elliptic or parabolic for the action of $G_k$ on $X_k$.
	It follows that $G \onto Q$ induces an isomorphism from $E$ onto its image, which proves \ref{enu: partial periodic quotient - gal - non loxo}.
	
	A proof by induction on $k\in \N$ using \ref{enu: partial periodic quotient - gal - proof - elliptic} and \ref{enu: partial periodic quotient - gal - proof - parabolic} shows that if $g \in G_k$ is elliptic or parabolic (for its action on $X_k$) then either $g^n=1$ or $g$ is the image of an elliptic or a parabolic element of $G$ (for its action on $X$).
	Let $q \in Q$ and $g \in G$ be a pre-image of $q$.
	For simplicity we still write $g$ for the image of $g$ in $G_k$.
	Since the map $X_k \to X_{k+1}$ is $C_1/\sqrt{N_1}$-Lipschitz, we get for every $k \in \N$,
	\begin{equation*}
		\snorm[X_k]g \leq \left(\frac{C_1}{\sqrt{N_1}}\right)^k \snorm[X]g.
	\end{equation*}
	As $C_1 / \sqrt{N_1} < 1$, there exists $k \in \N$ such that 
	\begin{equation*}
		\snorm[X_k]g < \frac 1{C_0\sqrt{N_1}} \leq \inj[X_k]{G_k}.
	\end{equation*}
	Consequently $g$ is elliptic or parabolic as an element of $G_k$.
	It follows from the previous observation that one of the following holds.
	\begin{itemize}
		\item The element $g$ coincide in $G_k$ with an elliptic or a parabolic element of $G$, hence $q$ is the image of an elliptic or a parabolic element of $G$.
		\item We have $g^n  = 1$ (in $G_k$), hence $q^n = 1$.
	\end{itemize}
	This completes the proof of \ref{enu: partial periodic quotient - gal - torsion}.
	
	All the relations we added to built the sequence of groups $(G_k)$ have the form $h^n = 1$.
	Hence the projection $G \to Q$ induces an epimorphism $Q \to G/G^n$, which gives \ref{enu: partial periodic quotient - gal - max}.
	
	Let $g$ be an elliptic or a parabolic element of $K$.
	It follows from \ref{enu: partial periodic quotient - gal - non loxo} that the map $G \onto G/K$ induces an isomorphism from $\group g$ onto its image.
	Hence $g$ is trivial.
	Consequently $K$ is purely loxodromic.
	For every $k \in \N$, the action of $G_k$ on $X_k$ is non-elementary. 
	It follows that the sequence $(G_k)$ does not ultimately stabilize.
	Indeed, otherwise \ref{enu: partial periodic quotient - gal - torsion} would fail.
	Thus $K$ is infinitely generated as a normal subgroup, which completes the proof of \ref{enu: partial periodic quotient - gal - kernel}.

	Let $x \in X$.
	Let $g_1,g_2 \in G$ such that $\dist[X]{g_ix}x < r$.
	It follows from our choice of $\epsilon$ that $\fix{g,100\delta_1} \subset X_0$ is non empty, where $g = g_1^{-1}g_2$.
	Assume now that $g_1$ and $g_2$ have the same image in $Q$, i.e.\ $g$ is trivial in $Q$.
	In particular, there exists $i \in \N$ such that the image of $g$ in $G_i$ is trivial.
	Recall that the map $X_k \to X_{k+1}$ is $1$-Lipschitz for every $k \in \N$.
	In particular, $\fix{g,100\delta_1} \subset X_k$ is non-empty for every $k \in \N$.
	A proof by induction using \ref{enu: partial periodic quotient - gal - proof - conj} show that $g = 1$.
	Hence the quotient map $G \onto Q$ is one-to-one when restricted to the set 
	\begin{equation*}
		\set{g \in G}{\dist {gx}x < r},
	\end{equation*}
	whence \ref{enu: partial periodic quotient - gal - inj on small balls}.
	
	We are left to prove \ref{enu: partial periodic quotient - gal - infinite}.
	We proceed by contradiction. 
	Suppose that there are only finitely many conjugacy classes of elements in $Q$ which are not the image of an elliptic or a parabolic element of $G$.
	We denote by $S$ a set of representatives (for the action by conjugation of $Q$) of these conjugacy classes.
	Let $S_0$ be a finite pre-image of $S$ in $G_0$.
	In view of \ref{enu: partial periodic quotient - gal - torsion}, $q^n = 1$ for every $q \in S$.
	Since $S_0$ is finite, there is $i \in \N$ such that the image of any $u \in S_0$ in $G_i$ satisfies $u^n = 1$.
	In particular, the image of $S_0$ in $G_i$ is made of elliptic elements only.
	As we already observed, the sequence is $(G_k)$ is not ultimately constant.
	Consequently there exists $j \geq i$ such that $P_j$ is non-empty (recall that $P_j$ is the set of primitive loxodromic elements of $G_j$ such that $\norm[X_j] h \leq 10\delta_1$).
	We fix $g \in P_j$.
	We claim that if $g$ is conjugated in $Q$ to the image of an elliptic or parabolic element $u \in G_j$, then $g$ and $u$ are already conjugated in $G_j$.
	Indeed, consider an elliptic or parabolic element $u \in G_j$ such that $q$ and $u$ are conjugated in $Q$.
	There is an index $k \geq j$ such that $g$ and $u$ are conjugated in $G_k$.
	Note that the subsets $\fix{g,100\delta_1} \subset X_\ell$ is non empty, for every $\ell \geq j$.
	A proof by induction using \ref{enu: partial periodic quotient - gal - proof - conj} shows that $g$ and $u$ are conjugated as elements of $G_j$, which completes the proof of our claim.
	Our claim has two consequences for the image $q$ of $g$ in $Q$.
	Any elliptic or parabolic element of $G$ is still elliptic or parabolic in $G_j$.
	Thus $q$ cannot coincides with an elliptic or parabolic element of $G$, for otherwise $g$ would be conjugated in $G_j$ to an elliptic or parabolic element, which would contradict the fact that $g$ is loxodromic.
	Consequently $q$ is conjugated to an element of $S$.
	However all the elements of $S_0$ are elliptic in $G_j$.
	Using again our claim, it implies that $g$ is conjugated in $G_j$ to an element in the image of $S_0$.
	This contradicts again the fact that $g$ is loxodromic, and completes the proof of \ref{enu: partial periodic quotient - gal - infinite}.
\end{proof}

\paragraph{Free Burnside groups.}

\begin{theo}
\label{res: free burnside groups even}
	Let $r \geq 2$.
	There exists $N_1 \in \N$ such that for every integer $n \geq N_1$ that is a multiple of $128$, the free Burnside group $\burn rn$ is not finitely presented and therefore infinite.
	Moreover if $n_2$ stands for the largest power of $2$ dividing $n$, then every finite subgroup of $\burn rn$ embeds in $\dihedral[n] \times \dihedral[n_2]^k$ for some $k \in \N$.
\end{theo}

\begin{proof}
	Let $X$ be the Cayley graph of the free group $\free r$ of rank $r$.
	It is $\delta$-hyperbolic with $\delta \in \R_+^*$ such that $420\delta < 1$.
	It follows that $A(\free r, X,420\delta) = 0$. 
	Moreover $\inj[X]{\free r} \geq 1$.
	Every dihedral germ is trivial. 
	Hence $\nu_{\rm{stg}}(G,X) = 1$.
	We choose for $\boldsymbol{\mathcal E}$ the class of group reduced to the trivial one. 
	Its exponent $\mu = \mu(\boldsymbol{\mathcal E})$ is $1$.
	Every subgroup with dihedral shape is trivial or infinite cyclic.
	Hence the assumption of \autoref{res: partial periodic quotient - gal} are fulfilled.
	Note that $\nu_1 = \max \{ \nu_{\rm{stg}}(G,X)+2, \mu +6\} = 7$.
	Hence there exists $N_1 \in \N$ such that for every $n \geq N_1$ that is a multiple of $128$, the group $\burn rn$ is not finitely presented, hence infinite.
	We now prove the second assertion.
	Let 
	\begin{equation*}
		G_0 = \free r \onto G_1 \onto G_2 \onto \dots \onto G_k \onto G_{k+1} \onto \dots
	\end{equation*}
	be the sequence of of groups produced in the proof of \autoref{res: partial periodic quotient - gal}, whose direct limit is $\burn rn$.
	Let $F$ be a finite subgroup of $\burn rn$.
	By construction there exist $\ell \in \N$ and an elliptic subgroup $F_\ell$ of $G_\ell$ such that $G_\ell \onto \burn rn$ maps $F_\ell$ onto $F$.
	We assume that $\ell$ is the smallest for this property.
	Recall that each map $G_k \onto G_{k +1}$ is one-to-one when restricted to an elliptic subgroup.
	Hence $F_\ell$ is actually isomorphic to $F$.
	According to our choice of $\ell$, the subgroup $F_\ell$ cannot be lifted.
	Therefore there exists $v$ in the vertex set $\mathcal V_\ell$ of $X_\ell$ such that $F_\ell$ is contained in $\stab v$ (\autoref{res: dichotomy elliptic}).
	It follows from \autoref{res : existence of central half-turn} that $F_\ell$ embeds in $\dihedral[n] \times \dihedral[n_2]^k$ for some $k \in \N$.
\end{proof}

\begin{rema*}
	Although we have not written the details here, a careful reader can follow the induction in the proof of \autoref{res: partial periodic quotient - gal} to show, just as in Ivanov~\cite{Ivanov:1994kh} or Lysenok~\cite{Lysenok:1996kw}, that the word and the conjugacy problems are solvable in $\burn rn$.
	With the previous notations, the solution of the word problem is based on the following observation.
	Given an element $g \in \free r$, there exists $k \in \N$, \emph{which only depends on the word length} of $g$, such that $g$ is trivial in $\burn rn$ if and only if so is it in $G_k$.
	Hence it suffices to run the solution of the word problem in the hyperbolic group $G_k$.
	The conjugacy problem can be solved as follows.
	Given $g_1, g_2 \in \free r$, there exists $k \in \N$, \emph{which only depends on the word length} of $g_1$ and $g_2$, such that both $g_1$ and $g_2$ are elliptic in $G_k$.
	A proof by induction based on \autoref{res: lifting conj elliptic} shows that $g_1$ and $g_2$ are conjugated in $\burn rn$ if and only if so are they in $G_k$.
	Hence it suffices to run the solution of the conjugacy problem in $G_k$.
\end{rema*}

The next statement is similar to Lysenok~\cite[Theorem~1]{Lysenok:1996kw} and a direct consequence of Ivanov \cite[Theorem~A]{Ivanov:1994kh}.

\begin{coro}
\label{res: free burnside groups all}
	Let $r \geq 2$.
	There exists $N \in \N$, such that for every integer $n \geq N$, the group $\burn rn$ is infinite.
\end{coro}

\begin{proof}
	Denote by $N_1$ the parameter given by \autoref{res: free burnside groups even}.
	Recall that there is $N_2 \in \N$, such that for every odd exponent $n \geq N_2$, the free Burnside group of exponent $n$ is infinite, see for instance \cite{Delzant:2008tu,Coulon:2014fr} for a geometric proof.
	Consider now an integer $n \geq \max\{N_1, 128N_0\}$.
	If $n$ is divisible by $128$, then  $\burn rn$ is infinite, by  \autoref{res: free burnside groups even}.
	Otherwise $n$ can be write $n = pm$ where $m \geq N_0$ is an odd integer.
	In particular, $\burn rm$ is infinite.
	Since $\burn rn = \burn r{pm}$ maps onto $\burn rm$, the group  $\burn rn$ is infinite as well.
\end{proof}

%
\subsection{Other examples}
%
\label{sec: examples}

One source of examples comes from groups acting acylindrically on a $\delta$-hyperbolic length space $X$.
Let us recall first the definition of acylindricity.
For our purpose we need to keep in mind the parameters that appear in the definition.

\begin{defi}[Acylindrical action]
\label{def: acylindricity}
	Let $N,L,d \in \R_+^*$.
	The group $G$ acts \emph{$(d,L,N)$-acylindrically} on $X$ if the following holds:
	for every $x, y \in X$ with $\dist xy \geq L$, the number of elements $u \in G$ satisfying $\dist{ux}x \leq d$ and $\dist{uy}y \leq d$ is bounded above by $N$.
	The group $G$ acts \emph{acylindrically} on $X$ if for every $d >0$ there exist $N, L > 0$ such that $G$ acts $(d,L,N)$-acylindrically on $X$.
\end{defi}

Since $X$ is a hyperbolic space, one can decide whether an action is acylindrical by looking at a single value of $d$.

\begin{prop}[Dahmani-Guirardel-Osin {\cite[Proposition~5.31]{Dahmani:2017ef}}]
\label{res: acylindrical action vs hyp space}
	The action of $G$ on $X$ is acylindrical if and only if there exist $N,L >0$ such that the action is $(100\delta, L, N)$-acylindrical.
\end{prop}

\begin{rema*}
	Dahmani, Guirardel and Osin work in a class of geodesic spaces.
	Nevertheless, following the proof of \cite[Proposition~5.31]{Dahmani:2017ef} one observes that the statement also holds for length spaces.
	Moreover one gets the following \emph{quantitative} statement.
	Assume that the action of $G$ on $X$ is $(100\delta, L, N)$-acylindrical, then for every $d > 0$ the action is $(d, L(d), N(d))$-acylindrical where
	\begin{eqnarray*}
		L(d) & = &  L + 4d + 100\delta, \\
		N(d) & = & \left(\frac d{5\delta} + 3 \right)N.
	\end{eqnarray*}
\end{rema*}

Assume now that the action of $G$ on $X$ is $(100\delta, L, N)$-acylindrical.
There exist parameters $C_{\rm inj}(\delta, L,N)$, $C_\nu(\delta, L, N)$ and $C_A(\delta, L, N)$, which only depend on $\delta$, $L$ and $N$, that control the various invariants defined in \autoref{sec: invariants}.
More precisely
\begin{enumerate}
	\item $\inj[X]G \geq C_{\rm inj}(\delta, L,N) > 0$, \cite[Lemma~2.2]{Bowditch:2008bj} or  \cite[Lemma~3.9]{Coulon:2019ac}
	\item $\nu(G,X) \leq C_\nu(\delta, L, N)$ \cite[Lemmas~6.12]{Coulon:2016if}
	\item $A(G,X,420\delta) \leq C_A(\delta, L, N)$, \cite[Lemma~6.14]{Coulon:2016if}.
\end{enumerate}

Given a group $\mathbf E$, its \emph{holomorph} is the semi-direct product $\hol{\mathbf E} = \aut{\mathbf E}\ltimes\mathbf E$.
If $\boldsymbol{\mathcal E}$ stands for a collection of groups, we let
\begin{equation*}
	\hol{\boldsymbol{\mathcal E}}
	= \set{\hol{\mathbf E}}{\mathbf E \in \boldsymbol{\mathcal E}}.
\end{equation*}
The next statement is a particular case of \autoref{res: partial periodic quotient - gal} when the action of $G$ on $X$ is acylindrical.

\begin{theo}
\label{res: partial periodic quotient - acyl}
	Let $\delta, L, r \in \R_+^*$ and $N \in \N$.
	Let $\boldsymbol{\mathcal E}_0$ be a finite collection of finite groups.
	We write $\mu$ for the exponent of $\hol{\boldsymbol{\mathcal E}_0}$.
	There exist $\nu_1,N_1 \in \N$ such that for every integer $n \geq N_1$ which is a multiple of $2^{\nu_1}\mu$ the following holds.
	
	Let $G$ be a group acting by isometries on a $\delta$-hyperbolic length space $X$.
	We assume that this action is $(100\delta, L, N)$-acylindrical and non-elementary.
	In addition, we suppose that every finite subgroup of $G$ with dihedral shape is isomorphic to a group of $\boldsymbol{\mathcal E}_0$.
	Then there exists a quotient $Q$ of $G$ with the following properties.
	\begin{enumerate}
		\item \label{enu: partial periodic quotient - acyl - non loxo}
		For every elliptic subgroup $F$ of $G$, the projection $G \onto Q$ induces an embedding from $F$ into $Q$.
		\item \label{enu: partial periodic quotient - acyl - torsion}
		For every $q \in Q$, either $q^n = 1$ or $q$ is the image of an elliptic element of $G$.
		\item \label{enu: partial periodic quotient - acyl - max}
		The projection $G \onto G/G^n$ induces an epimorphism $Q \onto G/G^n$.
		\item \label{enu: partial periodic quotient - acyl - inj on small balls}
		For every $x \in X$, the map $G \to Q$ is one-to-one when restricted to 
		\begin{equation*}
			\set{g \in G}{\dist {gx}x < r}.
		\end{equation*}
		\item \label{enu: partial periodic quotient - acyl - infinite}
		There are infinitely many conjugacy classes in $Q$ whose elements are not the image of an elliptic element of $G$.
		\item \label{enu: partial periodic quotient - acyl - kernel}
		The kernel $K$ of $G \onto Q$ is purely loxodromic.
		As a normal subgroup $K$ is not finitely generated.
	\end{enumerate}
\end{theo}

\begin{proof}
	Most of the proof consists in showing that, thanks to he acylindricity of the action of $G$ on $X$, all the invariants we are interested in ($A$, $\nu_{\rm stg}$, ${\rm inj}$, etc.) are appropriately bounded.
	The statement becomes then a consequence of \autoref{res: partial periodic quotient - gal}.
	
	We denote by $M$ the cardinality of the biggest group in $\boldsymbol{\mathcal E}_0$.
	Up to replacing $r$ by a largest value, we can assume that 
	\begin{equation*}
		r \geq  C_A(\delta, L, N)
		\quad \text{and} \quad
		\frac 1r \leq  C_{\rm inj}(\delta, L,N).
	\end{equation*}
	Recall that $\mu$ is the exponent of $\hol{\boldsymbol{\mathcal E}_0}$.
	We now set 
	\begin{equation*}
		\nu = \max\{C_\nu(\delta, L, N), M+1\}\quad \text{and} \quad 
		\nu_1 = \max\{\nu +2 , \mu + 6\}.
	\end{equation*}
	and denote by $N_1$ the critical exponent given by \autoref{res: partial periodic quotient - gal}.
	Let $n \geq N_1$ be an integer divisible by $2^{\nu_1}\mu$.
	
	Let $G$ be a group acting on a $\delta$-hyperbolic length space $X$ as in the theorem.
	It follows from our previous discussion that $A(G, X,420\delta) \leq r$, $\inj[X]G \geq 1/r$ and $\nu(G,X) \leq \nu$.
	
	Observe that every finite dihedral germ $C$ is contained in $\boldsymbol{\mathcal E}_0$.
	Indeed, since the action of $G$ is acylindrical, every loxodromic subgroup of $G$ is virtually cyclic.
	Hence $C$ is finite.
	As a dihedral germ, $C$ also has dihedral shape.
	Consequently $C$ is isomorphic to a group in $\boldsymbol{\mathcal E}_0$.
	In particular its cardinality is at most $M$.

	Let us now prove that $\nu_{\rm{stg}}(G,X) \leq \nu$.
	Let $g,h \in G$ and $m \geq \nu$.
	For every $k \in \N$, we write $g_k = h^kgh^{-k}$.
	Suppose that $E = \group{g_0, \dots, g_m}$ is elementary.
	Assume first that $h$ is loxodromic.
	According to our choice of $\nu$, we have $m \geq \nu(G,X)$, thus the elements $g$ and $h$ generate an elementary subgroup of $G$.
	Note that this group is necessarily loxodromic, hence has dihedral shape.
	Assume now that $E_0 = \group{g_0, \dots, g_{m-1}}$ is a dihedral germ.
	As we noticed before, it contains at most $M$ element.
	Since $m \geq M+1$, there exist $i,j \in \intvald 0{m-1}$ with $i< j$ such that $g_i = g_j$.
	In particular, $g_{j-i}=g_0$.
	It follows that $h$ normalizes $\group{g_0, \dots, g_{m-1}}$.
	Hence the subgroup generated by $g$ and $h$ is elementary.
	More precisely, it is a cyclic extension of the dihedral germ $\group{g_0, \dots, g_{m-1}}$, thus it has dihedral shape.
	This proves that $\nu \geq \nu_{\rm{stg}(G,X)}$ and completes the proof of our claim.
	
	We now build an appropriate model collection $\boldsymbol{\mathcal E}$.
	Let $(E,C)$ be a dihedral pair where $E$ is infinite.
	In particular, $C$ belongs to $\boldsymbol{\mathcal E}_0$.
	Since $C$ is finite, it fits into the following short exact sequence 
	\begin{equation*}
		1 \to C \to E \to \mathbf L \to 1,
	\end{equation*}
	where $\mathbf L$ is either $\cyclic$ or $\dihedral$.
	We choose an element $g \in E$ whose image in $\mathbf L$ generates the maximal infinite cyclic subgroup of $\mathbf L$.
	Recall that $\mu$ is the exponent of $\hol{\boldsymbol{\mathcal E}_0}$.
	We claim that $\group{g^\mu}$ is a normal subgroup of $E$.
	Let $u \in E$.
	There exist $\epsilon \in \{\pm 1\}$ and $c \in C$ such that $ugu^{-1} = g^\epsilon c$.
	The value of $\epsilon$ depends whether the image of $u$ in $\mathbf L$ is a reflection or not.
	Since $g$ normalizes $C$ its action by conjugation on $C$ induces an automorphism of $C$ that we denote by $\varphi$.
	One checks that for every $p \in \N$,
	\begin{equation*}
		ug^pu^{-1} = (g^\epsilon c)^p = g^{\epsilon p}\varphi^{\epsilon(p-1)}(c)\cdots \varphi^{\epsilon}(c)c.
	\end{equation*}
	However in the holomorph $\hol C$, whose exponent divides $\mu$, we have
	\begin{equation*}
		1 = (\varphi^\epsilon,c)^\mu = (\varphi^{\epsilon\mu}, \varphi^{\epsilon(\mu-1)}(c)\cdots \varphi^{\epsilon}(c)c).
	\end{equation*}
	Hence $ug^\mu u^{-1} = g^{\epsilon \mu}$.
	Consequently $\group{g^\mu}$ is a normal subgroup of $E$ as we announced.
	Note also that the exponent of $E/\group{g^\mu}$ divides $\mu$.
	Moreover $E$ embeds in $E/C \times E/ \group{g^\mu}$.

	We denote by $\boldsymbol{\mathcal E}_1$ the collection of quotients $E/\group{g^\mu}$ obtained as above, where $(E,C)$ runs over all dihedral pairs with $E$ infinite.
	In addition we let $\boldsymbol{\mathcal E} = \boldsymbol{\mathcal E}_0 \cup \boldsymbol{\mathcal E}_1$.
	It follows from the construction that the exponent of $\boldsymbol{\mathcal E}$ divides $\mu$.
	Moreover, for every dihedral pair $(E,C)$ there exists $\mathbf E \in \boldsymbol{\mathcal E}$ such that $E$ embeds in $E/C \times \mathbf E$.
	All the assumptions of \autoref{res: partial periodic quotient - gal} are satisfied and the result follows.
\end{proof}

\paragraph{Periodic groups.}

The next statement provides an alternate proof of Ol'shanki\u\i-Ivanov \cite[Theorem~A, Items (a) and (d)]{Ivanov:1996va}.

\begin{theo}
\label{res: periodic quotients hyperbolic group}
	Let $G$ be a non-elementary hyperbolic group.
	There exist $p, N_1 \in \N$ such that for every integer $n \geq N_1$ that is a multiple of $p$, the quotient $G / G^n$ is infinite.
	Moreover 
	\begin{equation*}
		\bigcap_{n \geq 1} G^n =\{1\}.
	\end{equation*}

\end{theo}

\begin{proof}
	Let $\boldsymbol {\mathcal E}_0$ be the collection of all isomorphism classes of finite subgroups of $G$.
	Since $G$ is hyperbolic, $\boldsymbol {\mathcal E}_0$ is finite.
	We write $\mu$ for its exponent. 
	Let $X$ be the Cayley graph of $G$ relative to some finite generating set.
	The action of $G$ on this hyperbolic space is acylindrical	hence the assumptions of \autoref{res: partial periodic quotient - acyl} holds.
	Let $\nu_1,N_1 \in \N$ be the parameters given by \autoref{res: partial periodic quotient - acyl} and set $p = 2^{\nu_1}\mu$.
	Observe that the order of every elliptic element of $G$ divides $\mu$, hence $p$.
	Assume now that $n\geq N_1$ is a multiple of $p$.
	According to \autoref{res: partial periodic quotient - acyl} there exists an infinite quotient $Q$ of $G$ such that the projection $G \onto G/G^n$ induces a map $Q \onto G/G^n$.
	Moreover every element $q \in Q$ satisfies the following dichotomy.
	Either $q^n = 1$, or $q$ is the image of an elliptic element of $G$.
	In the latter case we have $q^p = 1$, and thus $q^n = 1$.
	In other words $Q$ is a quotient of $G/G^n$, hence $Q$ is isomorphic to $G/G^n$ which is therefore infinite.
	
	According to \autoref{res: partial periodic quotient - acyl} the quotient map $G \to G/G^n$ can be made one-to-one on arbitrarily large balls by enlarging the value of $n$.
	Hence the intersection of all the subgroups $G^n$, when $n$ runs over $\N\setminus\{0\}$ is trivial.
\end{proof}

\begin{rema}
	As for free Burnside groups, one can give a precise description of the finite subgroups of $G/G^n$, provided $n$ is sufficiently large and divisible by $p$.
	Note that the structure of finite subgroups stated in \cite[Theorem~A(c)]{Ivanov:1996va} also follows from our analysis based on model collections.
	One can also prove that the word and the conjugacy problem are solvable in these periodic quotients as in \cite[Theorem~A(b)]{Ivanov:1996va}.
\end{rema}

\begin{rema}
\label{res: exponent dependency hyperbolic group}
	Note that $p$ and $N_1$ depends on the group $G$ -- or at least on the parameters ($A$, $\nu$, ${\rm inj}$, etc.) of the action of $G$ on a hyperbolic space and the algebraic structure of its finite subgroups.
	Let us illustrate this fact with several examples.
		
	Fix an integer $n$ and consider the group 
	\begin{equation*}
		G_n = \group{a,b \mid aw_a(a^n,b^n), bw_b(a^n,b^n)} 
	\end{equation*}
 	where $w_a(x,y)$ and $w_b(x,y)$ are two words of the form
	\begin{align*}
		w_a(x,y) & = xy^Mxy^{M+1}\dots xy^{M + P-1} \\
		w_b(x,y) & = xy^{M+P}xy^{M + P +1}\dots xy^{M + 2P - 1}.
	\end{align*}
	For some very large values of $P$ and $M$ the presentation defining $G$ satisfies the classical $C'(1/6)$ small cancellation condition, so that $G_n$ is a non-elementary, torsion-free, hyperbolic group.
	By design the $n$-periodic quotient of $G_n$ is trivial.
	Hence for \autoref{res: periodic quotients hyperbolic group} to apply, we need to choose $N_1$ much larger then $n$.
	There reason is that we cannot find an action of $G_n$ on a $\delta$-hyperbolic space $X$ where the injectivity radius $\inj[X]G$ is not too small compare to $\delta$.
	Indeed, if we consider the Cayley graph $X$ of $G$, then $\inj[X]G$ is at most $1$, while $\delta$ has the same order of magnitude as the length of the relations, thus is much larger than $n$.
	An alternative would be to choose for $X$ the Cayley graph of $G$ in which we have electrified every subpath of any loop labelled by the relations.
	Then the graph would be $40$-hyperbolic, see Gruber-Sisto \cite[Remark~3.7]{Gruber:2014wo}.
	However since we have electrified segments labelled by $a^n$, the injectivity radius of $X$ is at most $1/n$.
	
	Let us move to the dependencies of $p$.
	Consider the hyperbolic group 
	\begin{equation*}
		G = {\rm PSL}(2, \Z) = \cyclic[2] \ast \cyclic[3] = \group{a,b \mid a^2, b^3}.
	\end{equation*}
	If $n$ is co-prime with $2$ or $3$, then $G/ G^n$ is a finite group as $a$ or $b$ would be trivial in such a quotient.
	For \autoref{res: periodic quotients hyperbolic group} to apply we need to choose exponents $n$ which are a multiple of $6$.
	
	As we illustrate in the next example, it is not enough though to choose for $p$ the least common multiple of the order of finite subgroups of $G$.
	Consider first a non-elementary hyperbolic group $A$ with the following properties: non-trivial finite order elements in $A$ have order $3$; $A$ is the normal closure of a single finite order element $a \in A$.
	Such a group can be obtained as a small cancellation quotient of $\cyclic[3] \ast\cyclic[3]$.
	Consider also the semi-direct product $B = C \rtimes \cyclic$, where $C =\cyclic[3]$ and the generator $t$ of $\cyclic$ acts on $C$ by the automorphism sending every element to its inverse.
	Form now the amalgamated product $G = A *_C B$ where $C$ is identified with $\group a \subset A$.
	The group $G$ is hyperbolic, non-elementary, and all its non-trivial finite order elements have order $3$.
	If $n$ is odd then $t^n c t^{-n} = c^{-1}$.
	Hence $c^2 = 1$ in $G/G^n$. 
	Since $c$ has order $3$ in $G$, it is trivial in $G/G^n$.
	Consequently the same holds for $\group a$ and thus $A$.
	Hence $G/G^n =\cyclic[n]$ is finite.
	Although $G$ has no even torsion, for \autoref{res: periodic quotients hyperbolic group} to apply we need to work here with even exponents $n$.
	The reason is that $C$ is normalized by an element (here $t$) inducing an automorphism of $C$ with order $2$.
	To (brutally) bypass this difficulty, we impose in \autoref{res: partial periodic quotient - acyl} that $p$ is divisible by the exponent of the holomorph of finite subgroups of $G$.
	
	If the group $G$ is torsion-free, then $\nu(G,X) = \nu_{\rm stg}(G,X) = 1$ and we can take $p = 128$ as for free groups (\autoref{res: free burnside groups even})
\end{rema}

\paragraph{Relatively hyperbolic groups.}
Since Gromov's original paper \cite{Gromov:1987tk}, several different definitions of relatively hyperbolic groups have emerged, see for instance \cite{Bowditch:2012ga,Farb:1998wz}. 
These definitions have been shown to be almost equivalent \cite{Bowditch:2012ga,Szczepanski:1998fo,Hruska:2010iw}.
For our purpose we will use the following one.

\begin{defi}[{\cite[Definition 3.3]{Hruska:2010iw}}]
\label{def: relatively hyperbolic}
	Let $G$ be a group and $\{ P_1, \dots, P_m\}$ be a collection of subgroups of $G$.
	We say that $G$ is \emph{hyperbolic relative to} $\{P_1, \dots, P_m\}$ if there exist a proper geodesic $\delta$-hyperbolic space $X$ and a collection $\mathcal Y$ of pairwise disjoint open horoballs satisfying the following properties.
	\begin{enumerate}
		\item $G$ acts properly by isometries on $X$ and $\mathcal Y$ is $G$-invariant.
		\item If $U$ stands for the union of the horoballs of $\mathcal Y$ then $G$ acts co-compactly on $X \setminus U$.
		\item $\{P_1, \dots, P_m\}$ is a set of representatives of the $G$-orbits of $\set{\stab Y}{Y \in \mathcal Y}$.
	\end{enumerate}
\end{defi}

The action of $G$ on the space $X$ given by \autoref{def: relatively hyperbolic} is not acylindrical.
Indeed the subgroups $P_j$ can be parabolic. 
This cannot happen with an acylindrical action \cite[Lemma 2.2]{Bowditch:2008bj}.
More generally, the elementary subgroups of $G$ are exactly the virtually cyclic subgroups of $G$ and the ones which are conjugated to a subgroup of some $P_j$.
As in the case of groups with an acylindrical action, one can prove that $\inj[X]G$ is positive whereas $\nu(G,X)$ and $A(G,X,420\delta)$ are finite.
Proceeding as in \autoref{res: partial periodic quotient - acyl} we get the following result.
\begin{theo}
\label{res: SC - partial periodic quotient - rel hyp case}
	Let $G$ be a group and $\{P_1, \dots, P_m\}$ be a collection of subgroups of $G$ such that $G$ is hyperbolic relatively to $\{P_1, \dots, P_m\}$.
	Assume that there are only finitely many isomorphism classes of finite subgroups with dihedral shape.
	There exist $p,N_1\in \N$ such that every integer $n \geq N_1$ multiple of $p$, there exists a quotient $Q$ of $G$ with the following properties.
	\begin{enumerate}
		\item if $E$ is a finite subgroup of $G$ or conjugated into some $P_j$, then the projection $G \onto Q$ induces an isomorphism from $E$ onto its image;
		\item for every element $g \in Q$, either $g^n=1$ or $g$ is the image a non-loxodromic element of $G$;
		\item there are infinitely many conjugacy classes in $Q$ whose elements do not belong to the image of an elementary non-loxodromic subgroup of $G$.
	\end{enumerate}
\end{theo}

\begin{rema}
	Another possible strategy is to consider the action of $G$ on its coned-off Cayley $\Gamma$ -- see Bowditch \cite{Bowditch:2012ga} -- which has the following properties : the elliptic subgroup for the action of $G$ on $\Gamma$ are precisely the elliptic and parabolic subgroups for the action of $G$ on $X$.
	The acylindricity of the action of $G$ on $\Gamma$ is proved in Osin \cite{Osin:2016gv}.
\end{rema}

\paragraph{Mapping class groups.}
Let $\Sigma$ be a compact surface of genus $g$ with $k$ boundary components.
In the rest of this paragraph we assume that its complexity $3g+k -3$ is larger than $1$.
The \emph{mapping class group} $\mcg \Sigma$ of $\Sigma$ is the group of orientation preserving self homeomorphisms of $\Sigma$ defined up to homotopy.
A mapping class $f \in \mcg \Sigma$ is
\begin{enumerate}
	\item \emph{periodic}, if it has finite order;
	\item \emph{reducible}, if it permutes a collection of essential non-peripheral curves (up to isotopy);
	\item \emph{pseudo-Anosov}, if there exists an homeomorphism in the class of $f$ that preserves a pair of transverse foliations and rescale these foliations in an appropriate way.
\end{enumerate}
It follows from Thurston's work that any element of $\mcg \Sigma$ falls into one these three categories \cite[Theorem 4]{Thurston:1988fa}.
The \emph{complex of curves} $X$ is a simplicial complex associated to $\Sigma$.
It has been first introduced by Harvey \cite{Harvey:1981tg}.
A $d$-simplex of $X$ is a collection of $d+1$ homotopy classes of curves of $\Sigma$ that can be disjointly realized.
Masur and Minsky proved that this new space is hyperbolic  \cite{Masur:1999hc}.
By construction, $X$ is endowed with an action by isometries of $\mcg \Sigma$.
Moreover Bowditch showed that this action is acylindrical \cite[Theorem 1.3]{Bowditch:2008bj}.
This action provides an other characterization of the elements of $\mcg \Sigma$.
An element of $\mcg \Sigma$ is periodic or reducible (\resp pseudo-Anosov) if and only it is elliptic (\resp loxodromic) for the action on the complex of curves \cite{Masur:1999hc}.
Recall that $\mcg \Sigma$ contains only finitely many conjugacy classes of finite subgroups \cite[Theorem~7.14]{Farb:2012ws}.
Hence the next statement is a direct application of \autoref{res: partial periodic quotient - acyl}.

\begin{theo}
\label{res: partial periodic quotient mcg}	
	Let $\Sigma$ be a compact surface of genus $g$ with $k$ boundary components such that $3g +k - 3 >1$.
	There exist $p,N_1 \in \N$ such that for every integer $n \geq N_1$ which is a multiple of $p$, there exists a quotient $Q$ of $\mcg \Sigma$ with the following properties.
	\begin{enumerate}
		\item If $E$ is a subgroup of $\mcg \Sigma$ that does not contain a pseudo-Anosov element, then the projection $\mcg \Sigma \onto Q$ induces an isomorphism from $E$ onto its image.
		\item \label{enu: partial periodic quotient mcg - torsion}	
		Let $f$ be a pseudo-Anosov element of $\mcg \Sigma$.
		Either $f^n = 1$ in $Q$ or $f$ coincide in $Q$ with a periodic or a reducible element.
		\item There are infinitely many conjugacy classes  in $Q$ whose elements are not the image of a periodic or reducible element of $\mcg \Sigma$.
		Any non-trivial element in the kernel of $\mcg \Sigma \onto Q$ is pseudo-Anosov.
	\end{enumerate}
\end{theo}

\begin{rema}
\label{rem: comparison mcg result}
	A similar statement was obtained in \cite[Theorem~1.2]{Coulon:2016if}, where the existence of a torsion-free finite index subgroup of $\mcg \Sigma$ is used to bypass the difficulties created by even torsion.
	The conclusion there is slightly weaker though.
	Indeed, instead of \ref{enu: partial periodic quotient mcg - torsion} it is only proved that $f^n$ coincides in $Q$ with a periodic or a reducible element.
\end{rema}

\begin{rema}
	Our strategy does not allow us to control $N_1$ independently of the surface $\Sigma$.
	For simplicity we restrict ourselves to a \emph{closed} surface $\Sigma$ and write $X$ for the corresponding curve graph.
	Even though there is a uniform bound (i.e.\ that does not depend on the genus of $\Sigma$) for the hyperbolicity constant of $X$ \cite{Hensel:20151s}, we do not control all the other parameters of the action of $\mcg{\Sigma}$ on $X$.
	For instance, it was proved by Farb, Leininger, and Margalit \cite[Theorem~1.5]{Farb:2008lo} that 
	\begin{equation*}
		\inj[X]{\mcg{\Sigma}} \leq \frac {4 \ln(2 + \sqrt 3)}{g \ln(g - 1/2)}, 
	\end{equation*}
	where $g$ stands for the genus of $\Sigma$.
	In particular, the injectivity radius converges to zero as $g$ tends to infinity.
	However the critical exponent $N_1$ used in the induction step (\autoref{res: induction step}) need satisfy 
	\begin{equation*}
		N_1 \geq  \frac 1{ \left[\inj[X]{\mcg{\Sigma}}\right]^2}, 
	\end{equation*}
	hence cannot be uniformly bounded.
%
\end{rema}

\paragraph{Amalgamated product.}
Let $G$ be a group.
A subgroup $H$ of $G$ is \emph{malnormal} if for every $g \in G$, we have $gHg^{-1} \cap H = \{1\}$ unless $g$ belongs to $H$.
The next statement is the extension to even exponents of \cite[Theorem~1.1]{Coulon:2016if}.

\begin{theo}
\label{res: quotient amalgamated products}
	Let $A$ and $B$ be two groups.
	Let $C$ be a subgroup of $A$ and $B$ malnormal in $A$ or $B$.
	Assume that there exists $M \in \N$ such that every subgroup of $A$ (\resp $B$) that is isomorphic to the extension of a $2$-group by a finite cyclic or dihedral group contains at most $M$ elements.
	There exist $p,N_1$ such that for every integer $n \geq N_1$ which is a multiple of $p$, there exists a quotient $Q$ of $A*_CB$ with the following properties.
	\begin{enumerate}
		\item The natural projection $A*_CB \onto Q$ induces an embedding of $A$ and $B$ into $Q$.
		\item For every $g \in Q$, if $g$ is not a conjugate of an element of $A$ or $B$ then $g^n=1$.
		\item There are infinitely many conjugacy classes in $Q$ whose elements are not conjugate of elements of $A$ or $B$.
	\end{enumerate}
\end{theo}

\begin{proof}
	We denote by $X$ the Bass-Serre tree associated to the amalgamated product $G= A*_CB$ \cite{Serre:1977wy}.
	As $C$ is malnormal in $A$ or $B$, the action of $G$ on $X$ is acylindrical.
	Moreover every loxodromic subgroup is either $\cyclic$ or $\dihedral$.
	Hence every dihedral germ is necessarily a $2$-group.
	Consequently, a finite group with dihedral shape is isomorphic to the extension of a $2$-group by a finite cyclic or dihedral group.
	Being finite, such a group is contained in a conjugate of $A$ or $B$.
	Thus is follows from our assumption that $G$ admits only finitely many isomorphism classes of finite subgroups with dihedral shape.
	The conclusion follows from \autoref{res: partial periodic quotient - acyl}.
\end{proof}

\bigskip 

\noindent
\emph{R\'emi Coulon} \\
Univ Rennes, CNRS \\
IRMAR - UMR 6625 \\
F-35000 Rennes, France\\
\texttt{remi.coulon@univ-rennes1.fr} \\
\texttt{http://rcoulon.perso.math.cnrs.fr}


\todos
\end{document}